\documentclass[a4paper,10pt,reqno]{amsart}

\UseRawInputEncoding

\usepackage{amsmath}
\usepackage{cases}

\usepackage{amsfonts}
\usepackage[colorlinks,linkcolor=blue,citecolor=blue]{hyperref}
\usepackage{latexsym, amssymb, amsmath, amsthm, bbm}
\usepackage[all]{xy}
\usepackage{pgfplots}
\usepackage{enumerate}

\DeclareSymbolFont{EulerExtension}{U}{euex}{m}{n}
\DeclareMathSymbol{\euintop}{\mathop} {EulerExtension}{"52}
\DeclareMathSymbol{\euointop}{\mathop} {EulerExtension}{"48}

\allowdisplaybreaks[4]

\def \id{\operatorname{id}}
\def \Id{\operatorname{Id}}

\def \C{\mathcal{C}}

\def \e{\varepsilon}
\def \M{\mathrm{M}}

\def \Z{\mathbb{Z}}

\def \k{\Bbbk}

\def \dim{\operatorname{dim}}

\def \Hom{\operatorname{Hom}}

\def \Id{\operatorname{Id}}

\def \Rep{\operatorname{Rep}}

\def \Rex{\operatorname{Rex}}

\def \C{\mathcal{C}}
\def \D{\Delta}

\def \e{\varepsilon}
\def \M{\mathrm{M}}

\def \End{\operatorname{End}}

\def \C{\mathcal{C}}
\def \D{\mathcal{D}}

\def \M{\mathcal{M}}

\def \Z{\mathcal{Z}}

\def \ev{\mathrm{ev}}
\def \coev{\mathrm{coev}}

\def \Vec{\mathsf{Vec}}
\def \1{\mathbf{1}}
\def \Rep{\mathsf{Rep}}
\def \op{\mathrm{op}}
\def \cop{\mathrm{cop}}
\def \biop{\mathrm{op\,cop}}

\def \intHom{\underline{\operatorname{Hom}}}

\usepackage{mathtools}
\usepackage{stmaryrd}

\def \pams{partially admissible mapping system}
\def \ams{admissible mapping system}
\def \pd{\boldsymbol}
\def \YD{\mathfrak{YD}}

\def \la{\langle}
\def \ra{\rangle}

\def \btl{\raisebox{0.2ex}{\,${\scriptstyle \blacktriangleleft}$}\,}
\def \btr{\raisebox{0.2ex}{\,${\scriptstyle \blacktriangleright}$}\,}

\def\joinrel{\mathrel{\mkern-4mu}}

\newcommand{\btd}{\mathop{\raisebox{0.2ex}{${\scriptstyle \blacktriangleright\joinrel\blacktriangleleft}$}}}

\newcommand{\dotltimes}
{\mathop{\raisebox{0.12ex}{$\shortmid$}\raisebox{0.2ex}{\makebox[0.86em][r]{${\scriptstyle\gtrdot\joinrel<}$}}}}

\newcommand{\dotrtimes}
{\mathop{\raisebox{0.2ex}{\makebox[0.86em][l]{${\scriptstyle>\joinrel\lessdot}$}}\raisebox{0.12ex}{$\shortmid$}}}

\numberwithin{equation}{section}

\newtheorem{theorem}{Theorem}[section]
\newtheorem{lemma}[theorem]{Lemma}
\newtheorem{proposition}[theorem]{Proposition}
\newtheorem{corollary}[theorem]{Corollary}
\newtheorem{definition}[theorem]{Definition}
\newtheorem{example}[theorem]{Example}
\newtheorem{remark}[theorem]{Remark}
\newtheorem{question}[theorem]{Question}

\newtheorem{conjecture}[theorem]{Conjecture}
\newtheorem{notation}[theorem]{Notation}

\begin{document}
\title[Partially dualized quasi-Hopf algebras]{Partially dualized quasi-Hopf algebras reconstructed from dual tensor categories to finite-dimensional Hopf algebras}

\author[K. Li]{Kangqiao Li}
\address{School of Mathematics, Hangzhou Normal University, Hangzhou 311121, China}
\email{kqli@hznu.edu.cn}



\thanks{$\dag$ This work was supported by National Natural Science Foundation of China [grant number 12301049].}

\date{}


\begin{abstract}
Let $H$ be a finite-dimensional Hopf algebra with a left coideal subalgebra $B$.
It is known that $\Rep(B)$, the category of finite-dimensional representations of $B$, is an indecomposable exact left $\Rep(H)$-module category.
This paper determines and systematically studies
a quasi-Hopf algebra structure $(H/B^+H)^\ast\#B$, called a (left) partial dual of $H$, which is reconstructed from the dual tensor category of $\Rep(H)$ with respect to $\Rep(B)$.
Consequently, $\Rep((H/B^+H)^\ast\#B)$ is categorically Morita equivalent to $\Rep(H)$.
As applications:
1) Our construction of partial duals unifies some classical results in the literature,
such as bismash products of matched pair of groups given by Takeuchi, bosonizations of dually paired Hopf algebras given by Heckenberger and Schneider, etc.
2) We show that any finite-dimensional Hopf algebra with coradical being an abelian extension is categorically Morita equivalent to a basic quasi-Hopf algebra.
3) We provide a process for constructing genuine quasi-Hopf algebras with an example.
\end{abstract}

\maketitle

\tableofcontents

\section{Introduction}

\subsection{Backgrounds and motivation: Reconstruction from dual tensor categories}\label{subsection:1.1}

Tensor categories play an important role in many areas of mathematics, and
the theory of (quasi-)Hopf algebras is one of the major sources and a closely related part of the theory of tensor categories.
A fundamental connection is that if $K$ is a finite-dimensional quasi-Hopf algebra, then the category $\Rep(K)$ of its finite-dimensional representations is canonically a finite tensor category. Conversely, it is more worthwhile to study when and how to reconstruct a finite-dimensional quasi-Hopf algebra $K$ from a finite tensor category $\C$, 
which means that there exists a tensor equivalence $\C\approx\Rep(K)$. This is referred to as the \textit{reconstruction theory} of quasi-Hopf algebras. Such questions were answered positively when $\C$ admits a quasi-fiber functor $\C\rightarrow\Vec$, where $\Vec$ is the category of finite-dimensional vector spaces (e.g. Majid \cite{Maj95}, Etingof and Schiffmann \cite{ES02}). Of course there are other ways to realize the reconstruction, such as the one provided by Etingof and Ostrik \cite{EO04} if $\C$ is integral.
Note that these conditions are indeed sufficient and necessary for an arbitrary finite tensor category $\C$ from which a quasi-Hopf algebra $K$ could be reconstructed.
On the other hand, one might also consider the reconstruction problem when $\C$ is defined in some particular ways.

A classical construction for finite tensor categories is the notion of \textit{dual tensor categories}, which was introduced and studied in \cite{Ost03,EO04}. Specifically, let $\M$ be an indecomposable exact left module category over a finite tensor category $\C$. The dual category of $\C$ with respect to $\M$ is defined as the finite tensor category
$$\C_\M^\ast:=\Rex_\C(\M)^\mathrm{rev}$$
of $\C$-module endofunctors of $\M$, but with reverse compositions as tensor products (opposite to the original definition in the references above) for convenience in this paper. In particular, this could be viewed as the categorical version of dual Hopf algebras, in the sense that
$\Rep(H)_\Vec^\ast\approx\Rep(H^\ast)$ holds for any finite-dimensional Hopf algebra $H$.
Furthermore, two finite tensor categories are said to be \textit{categorically Morita equivalent} (\cite{EO04,EGNO15}),
if there is an tensor equivalence $\D\approx\C_\M^\ast$ for some indecomposable exact left $\C$-module category $\M$.
Indeed, the categorically Morita equivalence was shown in \cite{Mug03} to be an equivalence relation. It has meaningful invariants (\cite{Sch01,ENO11,Shi12} etc.) and was also studied as an analogue of Morita equivalence between rings in the literature.

However, one might find that the dual category $\C_\M^\ast$ is defined in an abstract way, even if $\C$ is chosen as the category $\Rep(H)$ for a finite-dimensional Hopf algebra $H$. Fortunately, Andruskiewitsch and Mombelli \cite{AM07} classified indecomposable exact left $\Rep(H)$-module categories as $\Rep(B)$, where $B$ is an indecomposable exact left $H$-comodule algebra.
The first goal in this paper is to reconstruct from the dual category $\Rep(H)_{\Rep(B)}^\ast$ when $B$ is in particular regarded as a left coideal subalgebra of $H$.

For the purpose, we should introduce a certain kind of ``exact sequence'' with form
\begin{equation}\label{eqn:admissiblemapsys1.1}
\xymatrix{
B \ar@<.5ex>[r]^{\iota} & H \ar@<.5ex>@{-->}[l]^{\zeta}
\ar@<.5ex>[r]^{\pi\;\;\;\;\;\;\;\;\;\;\;\;}
& C:=H/B^+H\;. \ar@<.5ex>@{-->}[l]^{\gamma\;\;\;\;\;\;\;\;\;\;\;\;}  }
\end{equation}
Here, $\iota$ denotes the inclusion of the left coideal subalgebra $B\subseteq H$, and $\pi:H\twoheadrightarrow H/B^+H$ is the quotient map, while $\zeta$ and $\gamma$ satisfy several conditions listed in Definition \ref{def:PAMS}.
Although the requirements for (\ref{eqn:admissiblemapsys1.1}) seem a little bit complicated,
we remark that such a diagram would coincide with the following notions in the literature if some of the four maps $\iota,\pi,\zeta,\gamma$ are special:
\begin{itemize}
\item
If $C$ is a Hopf algebra and $\pi$, $\zeta$ are Hopf algebra maps, then
(\ref{eqn:admissiblemapsys1.1}) is an \textit{admissible mapping system} introduced by Radford \cite{Rad85};

\item
If $B$, $C$ are both Hopf algebras, and $\iota$, $\pi$ are Hopf algebra maps, then $B\xrightarrow{\iota}H\xrightarrow{\pi}C$ is referred to as a
(strictly) \textit{exact sequence of Hopf algebras} by Schneider \cite{Sch93},
or a
(cocleft) \textit{extension of Hopf algebras} by Masuoka \cite{Mas94}. In this situation,
(\ref{eqn:admissiblemapsys1.1}) is the mapping system mentioned by Schauenburg in \cite{Sch02,Sch02(b)};

\item
If $\gamma(C)$ is a right coideal subalgebra of $H$, then (\ref{eqn:admissiblemapsys1.1}) is the mapping system characterizing \textit{cross product Hopf algebras} which is introduced by Bespalov and Drabant in \cite{BD99}, and by Caenepeel, Ion, Militaru and Zhu in \cite{CMIZ99} as well.

\item
It is also related to the notion of \textit{wreath cleft algebras} introduced by Bulacu and Torrecillas in \cite{BT21} as well.
\end{itemize}

For convenience in this paper, let us call (\ref{eqn:admissiblemapsys1.1}) together with its linear dual diagram a \textit{\pams}, denoted by the pair $(\zeta,\gamma^\ast)$. Its detailed definition and basic properties would be collected in Section \ref{section:2}, in order to establish the main results.

\subsection{The main results}

Our first key result is the following one, which is a combination of Theorem \ref{thm:partialdual} and Corollary \ref{cor:catMoritaequiv}:

\begin{theorem}\label{thm:1.1}
Let $H$ be a finite-dimensional Hopf algebra over a field $\k$ with left coideal subalgebra $B$. Suppose $(\zeta,\gamma^\ast)$ is a {\pams} given in (\ref{eqn:admissiblemapsys1.1}), and
regard $C^\ast=(H/B^+H)^\ast$ as a right coideal subalgebra of $H^\ast$ by dualizing the quotient map $\pi$.
Then:
\begin{itemize}
\item[(1)]
The smash product algebra $C^\ast\#B$ has a structure of quasi-Hopf algebra, whose detailed operations are listed in Theorem \ref{thm:partialdual};
\item[(2)]
There is an equivalence between finite tensor categories:
\begin{equation}\label{eqn:1.1}
\Rep(H)_{\Rep(B)}^\ast\approx \Rep(C^\ast\#B).
\end{equation}
\end{itemize}
\end{theorem}

We would call $C^\ast\#B$ in Theorem \ref{thm:1.1}(1) a \textit{left partially dualized quasi-Hopf algebra} (or \textit{left partial dual}) of $H$ determined by the {\pams} $(\zeta,\gamma^\ast)$.
Meanwhile,  its dual structure is called a \textit{right partially dualized coquasi-Hopf algebra} (or \textit{right partial dual}) $C\btd B^\ast$ of $H$ in Definition \ref{def:rightpartialdual}.

As for the tensor equivalence (\ref{eqn:1.1}) in Theorem \ref{thm:1.1}(2),
roughly speaking,
it is formulated as the following composition functor:
\begin{equation}\label{eqn:1.3}
\Rep(H)_{\Rep(B)}^\ast
\approx {}_{C^\ast}\Rep(H){}_{C^\ast}
\cong {}_{C^\ast}\mathfrak{M}{}_{C^\ast}^{H^\ast}
\overset{\Phi}{\approx} {}_{C^\ast}\mathfrak{M}{}^{B^\ast}
\cong\Rep(C^\ast\#B).
\end{equation}
Here, ${}_{C^\ast}\mathfrak{M}{}_{C^\ast}^{H^\ast}$ and ${}_{C^\ast}\mathfrak{M}{}^{B^\ast}$ are categories of relative Doi-Hopf modules of respective types, whose analogues were studied as abelian categories in
\cite{Tak79,Mas92,CMZ97,Sch02,Skr07} etc.
Specifically,
the first equivalence is the canonical description (\cite{EO04}) of dual tensor categories when $C^\ast$ is regarded as a left $H$-module algebra via the hit action, and the last isomorphism is classical in \cite{Doi92}.

Therefore, the most complicated step within (\ref{eqn:1.3}) is to make the abelian category
${}_{C^\ast}\mathfrak{M}{}^{B^\ast}$ (or $\Rep(C^\ast\#B)$) a tensor category, as well as to
equip the linear abelian equivalence $\Phi:M\mapsto \overline{M}:= M/M\left(C^\ast\right)^+$ introduced in \cite{Tak79} with a suitable monoidal structure $J$.
This is also the reason why our formulation is highly dependent on the
properties of so-called {\pams}s $(\zeta,\gamma^\ast)$. In particular, such systems require the
\textit{cocleftness property} of finite-dimensional Hopf algebras $H$ over (left or right) coideal subalgebras, which was introduced by Masuoka and Doi \cite{MD92} and confirmed by Skryabin \cite{Skr07}. Specifically, since $H^\ast$ is right cocleft over its right coideal subalgebra $C^\ast$, there exists a (unitary and counitary) convolution invertible right $C^\ast$-module map $\gamma^\ast:H^\ast\rightarrow C^\ast$ retracting the injection $\pi^\ast:C^\ast\rightarrowtail H^\ast$. As a result, the desired monoidal structure $J$ for $\Phi$ could be defined as
\begin{equation}\label{eqn:1.4}
J_{M,N}:\overline{M}\otimes\overline{N}\cong\overline{M\otimes_{C^\ast}N},\;\;\;\;
\overline{m}\otimes\overline{n}
\mapsto\sum\overline{m_{(0)}\overline{\gamma}^\ast(m_{(1)})\otimes_{C^\ast}n},
\end{equation}
where $\overline{\gamma}^\ast$ is the convolution inverse of $\gamma^\ast$.
It helps us determine the whole structures of the left partial dual $C^\ast\#B$ as a quasi-Hopf algebra in Theorem \ref{thm:1.1}(1) with quite a long but constructive proof. In fact, the other related retraction $\zeta:H\rightarrow B$ admitting symmetric conditions with $\gamma^\ast$ is also required for this purpose.

Also, it should be remarked that (\ref{eqn:1.4}) coincides with a structure mentioned by Schauenburg \cite{Sch02} for the situation when $C^\ast$ is a Hopf subalgebra of $H^\ast$, and then $\Phi$ is also analogous to the \textit{de-equivariantization of the Hopf algebra} $H^\ast$ described by Angiono, Galindo and Mariana \cite{AGM14}.

Furthermore,
in order to study the classical examples (see Example \ref{ex:1.2} in the following subsection), we provide some elementary properties of partial dualizations with detailed proofs. The first one describes the \textit{opposite and coopposite structures} of left partial duals, which is a combination of the results given in Subsections \ref{subsection:5.1} and \ref{subsection:5.2}:

\begin{proposition}\label{prop:1.a}
Let $H$ be a finite-dimensional Hopf algebra with left partially dualized quasi-Hopf algebra $C^\ast\#B$. Then there are isomorphisms
\begin{eqnarray*}
& (C^\ast\#B)^\biop\cong B^\biop\# {C^\ast}^\biop, & \\
&(C^\ast\#B)^\op\cong {C^\ast}^\op\# B^\op
\;\;\;\;\text{and}\;\;\;\;
(C^\ast\#B)^\cop\cong {C^\ast}^\cop\# B^\cop&
\end{eqnarray*}
of quasi-bialgebras,
where $B^\biop\# {C^\ast}^\biop$, ${C^\ast}^\op\# B^\op$ and ${C^\ast}^\cop\# B^\cop$ are respectively left partial duals determined by certain {\pams}s (induced in Subsections \ref{subsection:2.2} and \ref{subsection:2.3}).
\end{proposition}

The other one is the conclusion of Subsection \ref{subsection:6.2} for the situation when partial duals are Hopf algebras (with trivial associators). It describes the structures obtained by applying multiple partial dualizations in canonical ways, and shows the
``partial self-dual property'' that left and right partial dualizations are mutually inverse:

\begin{proposition}\label{prop:1.b}
Suppose the left partial dual $C^\ast\#B$ of $H$ is a Hopf algebra (with trivial associator). Then:
\begin{itemize}
\item[(1)]
$H^\ast$ is a left partial dual of $C^\ast\#B$;
\item[(2)]
The right partial dual $C\btd B^\ast$ is a left partial dual of $H^\ast$;
\item[(3)]
$H$ is a left partial dual of the right partial dual of $C^\ast\#B$.
\end{itemize}
Similar statements hold for right partial dualizations.
\end{proposition}

\subsection{Classical examples: Unified structures in the literature}


The main result on (left and right) partial duals might be meaningful in a sense, as it is introduced in this paper and other papers that our construction unifies some classical structures of (quasi-)Hopf algebras as examples. We list some of them as follows, which also motivates us to study partial duals further:

\begin{example}\label{ex:1.2}
Suppose all the algebras, coalgebras and (quasi-)Hopf algebras in this example are over a field $\k$.
\begin{itemize}
\item[(1)]
Let $(F,G)$ be a matched pair of finite groups introduced by Takeuchi \cite{Tak81}, which determines a factorizable group $F\bowtie G$. Then:
\begin{itemize}
\item[i)]
The bismash product Hopf algebra $\k^G\#\k F$ (\cite{Tak81}) is a left partial dual of the group algebra $\k(F\bowtie G)$
\item[ii)]
The other bismash product Hopf algebra $\k G\#\k^F$ (e.g. \cite{BGM96}) is a right partial dual of the group algebra $\k(F\bowtie G)$.
\end{itemize}

\item[(2)]
Let $H$ and $K$ be two finite-dimensional Hopf algebras with a Hopf pairing $\sigma:K^\ast\otimes H\rightarrow\k$.
Then the generalized quantum double $K^{\ast\,\cop}\bowtie_\sigma H$ introduced by Doi and Takeuchi \cite{DT94} is a left partial dual of $K^\op\otimes H$.

- In particular, the Drinfeld double $D(H)$ (\cite{Dri86}) is a left partial dual of $H^\op\otimes H$.

\item[(3)]
Let $K\rightarrow H\rightarrow Q$ be a split extension of finite-dimensional Hopf algebras, which determines the bismash product $H=K\#Q$. Then the double cross product $Q\bowtie K^\ast$ studied by Majid \cite{Maj90} and Schauenburg \cite{Sch02} is a right partial dual of $H$.

\item[(4)]
Let $K\rightarrow H\rightarrow Q$ be an abelian extension of finite-dimensional Hopf algebras in the sense of \cite{Mas02}. Then each left partial dual of form $Q^\ast\#K$ is a commutative quasi-Hopf algebra, and hence $H$ is known to be group-theoretical (\cite{Nat03}).

\item[(5)]
Let $A$ be a finite-dimensional Hopf algebra. Suppose
$(B',B)$ is a dual pair of Hopf algebras in the category ${}^A_A\mathfrak{YD}$ of finite-dimensional (left-left) Yetter-Drinfeld modules introduced by Heckenberger and Schneider \cite{HS13}.
Then
the bosonization $B'\dotrtimes A$ is a right partial dual of $B\dotrtimes A$.
\end{itemize}
\end{example}


Although one could find that most of the examples of partial duals listed above are in fact Hopf algebras, we would provide in Subsection \ref{subsection:6.1} the conditions when a general left partially dualized quasi-Hopf algebra $C^\ast\#B$ has trivial associators.
If so, the corresponding {\pams} is said to be \textit{admissible}, which is closely related to the factorization problem of Hopf algebras (as cross products in the senses of \cite{BD99,BCT13}).
Consequently, we explain as another example that every partial dual of the 4-dimensional Taft algebra is a Hopf algebra isomorphic to itself.

Of course, there do exist examples of left partial dualized Hopf algebras $C^\ast\#B$ which are not Hopf algebras. Furthermore, some of them should not be gauge equivalent to any Hopf algebras. Therefore, our results could be used to construct new quasi-Hopf algebras as applications.

\subsection{Applications: Categorical Morita invariants and construction of genuine quasi-Hopf algebras}

As we have mentioned in Subsection \ref{subsection:1.1}, the notion of dual tensor categories is used to introduce the categorically Morita equivalence between finite tensor categories.
Here, let us remark that the left center (\cite{Maj91,JS91}) is an invariant under this relation (\cite{Sch01,EO04}).
Furthermore, two finite tensor categories $\C$ and $\D$ are categorically Morita equivalent if and only if their left centers $\Z(\C)$ and $\Z(\D)$ are braided tensor equivalent
(\cite{ENO11,EGNO15}).
In particular, if $\C=\Rep(K)$ for a finite-dimensional quasi-Hopf algebra $K$, then the category
${}^{K}_{K}\YD$ of Yetter-Drinfeld modules over $K$ and $\Rep(D(K))$
are also categorical Morita invariants (\cite{HN99,Sch02',BCP06}) as they are both braided equivalent to the left center $\Z(\Rep(K))$.

Thus, since
the tensor equivalence (\ref{eqn:1.1}) in Theorem \ref{thm:1.1}(2) implies that $\Rep(H)$ and $\Rep(C^\ast\#B)$ are categorically Morita equivalent, we could
obtain the corresponding properties of the left partially dualized quasi-Hopf algebra $C^\ast\#B$, which are found in Proposition \ref{prop:YDmodsequiv}. As applications to the structures in Example \ref{ex:1.2}, one might directly obtain the following results in the literature:

\begin{corollary}\label{cor:1.3}
\begin{itemize}
\item[(1)](\cite{BGM96})
Let $(F,G)$ be a matched pair of finite groups. Then the Drinfeld doubles $D(\k^G\#\k F)$ and $D(\k(G\bowtie F))$ are gauge equivalent, and $(\k^G\#\k F)^\ast\cong\k G\#\k^F$ as Hopf algebras.

\item[(2)](cf. \cite{Sch02})
Let $H$ be a finite-dimensional Hopf algebra fitting into a split abelian extension of Hopf algebras.
Then $\Rep(H)$ is categorically Morita equivalent to $\Rep(K)$ for some finite-dimensional cocommutative Hopf algebra $K$.

\item[(3)](\cite{HS13})
Let $A$ be a finite-dimensional Hopf algebra. Suppose $(B,B')$ is a dual pair of Hopf algebras in ${}^A_A\YD$. Then there is a braided tensor equivalence
$${}^{B\dotrtimes A}_{B\dotrtimes A}\YD\approx{}^{B'\dotrtimes A}_{B'\dotrtimes A}\YD$$
between the categories of finite-dimensional Yetter-Drinfeld modules over bosonizations.
\end{itemize}
\end{corollary}

On the other hand, as our main results on partial duals are provided constructively, they could be applied to construct new quasi-Hopf algebras with certain properties, especially genuine ones.
Recall that a quasi-Hopf algebra is said to be \textit{genuine}, if it is not gauge equivalent to any ordinary Hopf algebra.
To the best of the author's knowledge, including \cite{DPR90,EG05,Ang10,Liu14} etc., there seem to be not a large number of results on constructions of genuine quasi-Hopf algebras in the literature.
However, we would determine an example of left partially dualized quasi-Hopf algebra which is genuine.

Anyway in this paper, we provide two main applications for constructing quasi-Hopf algebras as follows, which are combinations of Propositions \ref{prop:coradicalabelext(MAMS)} and \ref{prop:PDofK8genuine(MAMS)}, as well as Example \ref{ex:D8(MAMS)}:

\begin{proposition}\label{prop:1.4}
\begin{itemize}
\item[(1)]
Suppose $H$ is a finite-dimensional Hopf algebra whose coradical $H_0$ is a Hopf subalgebra fitting into an abelian extension. Then $H$ has a left partially dualized quasi-Hopf algebra which is basic.

\item[(2)]
By determining a certain left partial dual of the $8$-dimensional Kac algebra over $\mathbb{C}$, we find a non-trivial $3$-cocycle $\pd{\phi}$ on the dihidral group $D_8$, such that $(\mathbb{C}D_8)^\ast$ is a  genuine (semisimple commutative) quasi-Hopf algebra.
\end{itemize}
\end{proposition}

We remark that Proposition \ref{prop:1.4}(1) is shown by a constructive proof based on our main result Theorem \ref{thm:1.1}(1).
As for Proposition \ref{prop:1.4}(2), the structures of $(\mathbb{C}D_8)^\ast$ as a quasi-Hopf algebra are found by detailed calculations, according to a suitable {\pams} as well as our constructive operations of left partial duals. However, in order to prove that it is genuine, some gauge invariants of fusion categories should be used, such as the exponent and the Grothendieck ring (\cite{Eti02,EO04}).

More generally, the process of the construction for Proposition \ref{prop:1.4}(2) might be applied to obtain a number of (genuine) quasi-Hopf algebras, including non-commutative ones and non-semisimple ones. However, it would require complicated calculations to construct such examples.

\subsection{Organization of the paper}

To begin with, in Section \ref{section:2}, we recall and conclude the cocleftness properties of coideal subalgebras in finite-dimensional Hopf algebras, and collect necessary properties of \textit{{\pams}s} for later uses.

Afterwards, the main theorem (Theorem \ref{thm:1.1}(1)) on the construction of \textit{left partially dualized quasi-Hopf algebras} $C^\ast\#B$ is stated in Section \ref{section:3}. Its detailed operations are listed, but for the moment we only verify that the ``comultiplication'' and ``counit'' are algebra maps (and the counit axiom) with some technical calculations.

As for the associator and antipodes, we formulate them in Section \ref{section:4} by establishing the suitable tensor product bifunctor on the finite abelian category $\Rep(C^\ast\#B)$ in (\ref{eqn:1.4}), as well as making $\Phi$ a tensor equivalence with monoidal structure $J$.
Meanwhile, the reconstruction result $\Rep(H)_{\Rep(B)}^\ast\approx \Rep(C^\ast\#B)$
is stated, and some consequences on their centers (as well as Yetter-Drinfeld modules, quantum doubles) are given.

Section \ref{section:5} is devoted to provide identifications (in Proposition \ref{prop:1.a}) of opposite and coopposite structures of left partially dualized quasi-Hopf algebras, as well as the notion of \textit{right partial dualized coquasi-Hopf algebras}.

In Section \ref{section:particularcases}, we provide conditions when partial duals are ordinary Hopf algebras (with trivial associators), and describe the structures obtained by applying multiple partial dualizations (Proposition \ref{prop:1.b}).
Then our general constructions are applied to describe bismash products of matched pair of groups and bosonizations of dually paired Hopf algebras as partial duals (within Example \ref{ex:1.2}), where some classical results follow as consequences (within Corollary \ref{cor:1.3}). An easy example on determining all the left partial duals of the $4$-dimensional Taft algebra is also considered.

Finally in Section \ref{section:genuinequasi-Hopf}, we focus on the {\pams}s which arise from extension of Hopf algebras, and provide basic properties of the partial duals determined by such systems.
As applications (Proposition \ref{prop:1.4}), we study Hopf algebras whose coradical fits into an abelian extension, as well as show how to construct a genuine quasi-Hopf algebra as a left partial dual of the 8-dimensional Kac algebra.

\section{Cocleftness property and {\pams}s}\label{section:2}

Throughout this paper, all vector spaces, algebras and Hopf algebras are over a field $\k$. The tensor product over $\k$ is denoted simply by $\otimes$.
Moreover for a Hopf algebra $H$, we always denote its antipode by $S$, and Sweedler notation $\Delta(h)=\sum h_{(1)}\otimes h_{(2)}$ is always used to denote the coproduct of $h\in H$.
We refer to \cite{Swe69,Mon93,Rad12} for basics of Hopf algebras.

\subsection{Preliminaries: Cocleftness of finite-dimensional Hopf algebras over coideal subalgebras}\label{subsection:cocleft}

The cocleftness property \cite[Definition 2.2]{MD92} of finite-dimensional Hopf algebras over (left or right) coideal subalgebras plays a role in theories of the extensions for Hopf algebras. It was introduced and studied in \cite{Mas92,MD92,Mas94,Skr07} etc., some of which would be recalled or concluded in this subsection.

Let $B$ be a left coideal subalgebra of a Hopf algebra $H$. We say that $H$ is left $B$-\textit{cocleft} if there exists a convolution invertible left $B$-module map $\zeta:H\rightarrow B$ (which is also called a \textit{cointegral}).
The notion of Hopf algebras being right cocleft over right coideal subalgebras is defined similarly.

In fact, it is known that the cocleftness property always holds for finite-dimensional Hopf algebras:

\begin{lemma}\label{lem:cocleftness}(\cite{Skr07})
Let $H$ be a finite-dimensional Hopf algebra. Then:
\begin{itemize}
\item[(1)]
Each left coideal subalgebra $B$ of $H$ is Frobenius, and $H$ is left cocleft over $B$;
\item[(2)]
There is a one-to-one correspondence between left coideal subalgebras $B$ of $H$ and right coideal subalgebras $A'$ of $H^\ast$, which is given by
\begin{equation}\label{eqn:correspondence}
B\mapsto (H/B^+H)^\ast \;\;\;\;\;\;\text{and}\;\;\;\;\;\;
A'\mapsto(H^\ast/H^\ast {A'}^+)^\ast.
\end{equation}
Here $B^+$ denotes the intersection of $B$ and the kernel of the counit, and ${A'}^+$ denotes similarly.
\end{itemize}
\end{lemma}

\begin{proof}
\begin{itemize}
\item[(1)]
Consider the biopposite finite-dimensional Hopf algebra $H^\biop$ with right coideal subalgebra $B^\biop$.
We can know by \cite[Theorem 6.1(i)]{Skr07} that $B^\biop$ is a Frobenius algebra, which implies that $H^\biop$ is right $B^\biop$-cocleft by \cite[Theorem 3.5]{MD92}.
This is equivalent to say that $B$ is Frobenius and $H$ is left $B$-cocleft.

\item[(2)]
It is sufficient to apply \cite[Corollary 6.5]{Skr07} to Hopf algebras $H^\cop$ and $H^{\ast\op}$, which are dual to each other. Specifically, note that all the left coideal subalgebras $B$ of $H$ and all the right coideal subalgebras $A'$ of $H^\ast$ are Frobenius according to the proof of (1) above. Therefore, \cite[Proposition 2.10(ii)]{Mas92} can be used on $H^\cop$ to obtain the correspondence (\ref{eqn:correspondence}).
\end{itemize}
\end{proof}

For any left coideal subalgebra $B$ of a finite-dimensional Hopf algebra $H$, we always denote respectively the corresponding inclusion map and quotient map by
\begin{equation}\label{eqn:coidealsubalgs}
\iota:B\rightarrowtail H \;\;\;\;\;\;\text{and}\;\;\;\;\;\; \pi:H\twoheadrightarrow H/B^+H,
\end{equation}
and one could find that the correspondence (\ref{eqn:correspondence}) sends $\iota$ to the injection
$\pi^\ast:(H/B^+H)^\ast\rightarrowtail H^\ast$
of right $H^\ast$-comodule algebras. Thus $(H/B^+H)^\ast$ can be regarded as a right coideal subalgebra of $H^\ast$ via $\pi^\ast$, and $H^\ast$ is also right $(H/B^+H)^\ast$-cocleft. Consequently, we will use notations for $b\in B$ and $f\in (H/B^+H)^\ast$ that
\begin{equation}\label{eqn:coidealsubalgsnotation(MAMS)}
\sum b_{(1)}\otimes b_{(2)}\in H\otimes B
\;\;\;\;\;\;\text{and}\;\;\;\;\;\;\sum f_{(1)}\otimes f_{(2)}\in (H/B^+H)^\ast\otimes H^\ast
\end{equation}
to represent the structures of the left and right coideals (or comodules) respectively.

For the sake of subsequent applications, we conclude more cocleftness properties of $\iota$ and $\pi^\ast$ as follows. The counits of $H$ and $H^\ast$ are denoted by $\e$ and $\langle-,1\rangle$ respectively, by which one could immediately write
\begin{equation}\label{eqn:piiota}
\pi\circ\iota=\la\e|_B,-\ra\pi(1).
\end{equation}
Moreover, we would denote the quotient right $H$-module structure of $H/B^+H$ by $\btl$.

\begin{lemma}\label{lem:cleftcocleft}
Let $H$ be a finite-dimensional Hopf algebra with left coideal subalgebra $B$. Then there exist convolution invertible maps $\zeta:H\rightarrow B$ and $\gamma:H/B^+H\rightarrow H$ such that
\begin{equation}\label{eqn:admissiblemapsys}
\begin{array}{ccc}
\xymatrix{
B \ar@<.5ex>[r]^{\iota} & H \ar@<.5ex>@{-->}[l]^{\zeta} \ar@<.5ex>[r]^{\pi\;\;\;\;\;\;}
& H/B^+H \ar@<.5ex>@{-->}[l]^{\gamma\;\;\;\;\;\;}  }
\end{array}
\end{equation}
where:
\begin{itemize}
\item[(1)]
$\iota$ is a map of left $H$-comodules and algebras, and $\zeta$ is a map of left $B$-modules;
\item[(2)]
$\pi$ is a map of right $H$-modules and coalgebras, and $\gamma$ is a map of right $H/B^+H$-comodules;
\item[(3)]
Denote the convolution multiplication on $\End_\k(H)$ by $\ast$. Then
\begin{equation}\label{eqn:convolution}
(\iota\circ\zeta)\ast(\gamma\circ\pi)=\id_H.
\end{equation}
Moreover, $\zeta\circ\gamma=\la\e_{H/B^+H},-\ra 1_B$ holds on $H/B^+H$.
\end{itemize}
\end{lemma}

\begin{proof}
Firstly, the desired properties for $\iota$ and $\pi$ are evident as long as we note that $B^+H$ is a coideal of $H$, and we could write
\begin{equation}\label{eqn:iotapi}
\sum b_{(1)}\otimes\iota(b_{(2)})=\sum \iota(b)_{(1)}\otimes\iota(b)_{(2)}\in H\otimes H
\;\;\;\;\text{and}\;\;\;\;\pi(hk)=\pi(h)\btl k\in H/B^+H
\end{equation}
with our notations for all $b\in B$ and $h,k\in H$.

According to Lemma \ref{lem:cocleftness}(1), there exists a left $B$-module map $\zeta:H\rightarrow B$ with the convolution inverse $\overline{\zeta}$, and it follows that
\begin{equation}\label{eqn:zeta}
\zeta[\iota(b)h]=b\zeta(h)\;\;\;\;\;\;(\forall b\in B,\;h\in H).
\end{equation}
Then consider the map
\begin{equation}\label{eqn:gammadef}
\gamma:H/B^+H\rightarrow H,\;\pi(h)\mapsto\sum\iota[\overline{\zeta}(h_{(1)})]h_{(2)}.
\end{equation}
It is well-defined, as one could compute for any $b\in B$ and $h\in H$ that
\begin{eqnarray*}
\gamma[\pi(\iota(b)h)]
&\overset{(\ref{eqn:gammadef})}{=}&
\sum\iota[\overline{\zeta}(\iota(b)_{(1)}h_{(1)})]\iota(b)_{(2)}h_{(2)}
~\overset{(\ref{eqn:iotapi})}{=}~
\sum\iota[\overline{\zeta}(b_{(1)}h_{(1)}]\iota(b_{(2)})h_{(2)}  \\
&=&
\sum\iota[\overline{\zeta}(b_{(1)}h_{(1)})b_{(2)}]h_{(2)}
~=~ \sum\iota[\overline{\zeta}(b_{(1)}h_{(1)})b_{(2)}\zeta(h_{(2)})\overline{\zeta}(h_{(3)})]h_{(4)}  \\
&\overset{(\ref{eqn:zeta})}{=}&
    \sum\iota[\overline{\zeta}(b_{(1)}h_{(1)})\zeta(\iota(b_{(2)})h_{(2)})\overline{\zeta}(h_{(3)})]h_{(4)}
~=~ \e(b)\sum\iota[\overline{\zeta}(h_{(1)})]h_{(2)}  \\
&\overset{(\ref{eqn:gammadef})}{=}&
\gamma[\pi(\e(b)h)].
\end{eqnarray*}
Furthermore, a straightforward verification shows that $\gamma$ has a convolution inverse
\begin{equation}\label{eqn:gammabardef}
\overline{\gamma}:H/B^+H\rightarrow H,\;\pi(h)\mapsto \sum S(h_{(1)})\iota[\zeta(h_{(2)})],
\end{equation}
which is also well-defined because
\begin{eqnarray*}
\overline{\gamma}[\pi(\iota(b)h)]
&\overset{(\ref{eqn:gammabardef})}{=}&
\sum S(b_{(1)}h_{(1)})\iota[\zeta(\iota(b_{(2)})h_{(2)})]
~\overset{(\ref{eqn:zeta})}{=}~ \sum S(h_{(1)})S(b_{(1)})\iota[b_{(2)}\zeta(h_{(2)})]  \\
&=& \sum S(h_{(1)})S(b_{(1)})\iota(b_{(2)})\iota[\zeta(h_{(2)})]
~\overset{(\ref{eqn:iotapi})}{=}~
\e(b)\sum S(h_{(1)})\iota[\zeta(h_{(2)})]  \\
&\overset{(\ref{eqn:gammabardef})}{=}&
\overline{\gamma}[\pi(\e(b)h)]
\end{eqnarray*}
for any $b\in B$ and $h\in H$.

Finally, the remaining properties for $\zeta$ and $\gamma$ are shown in the followings:

\begin{itemize}
\item[(1)]
The desired property of $\zeta$ has been obtained as (\ref{eqn:zeta}) at the beginning of this proof.

\item[(2)]
Let us verify that $\gamma$ defined by (\ref{eqn:gammadef}) preserves right $H/B^+H$-coactions, where the right $H/B^+H$-comodule structure on $H$ is chosen to be $(\id\otimes\pi)\circ\Delta$. Indeed, note in (\ref{eqn:piiota}) that $\pi\circ\iota$ is trivial on $B$, and hence for any $h\in H$,
\begin{eqnarray*}
\sum\gamma[\pi(h)]_{(1)}\otimes\pi\left(\gamma[\pi(h)]_{(2)}\right)
&\overset{(\ref{eqn:gammadef})}{=}&
\sum\iota[\overline{\zeta}(h_{(1)})]_{(1)}h_{(2)}
    \otimes\pi\left(\iota[\overline{\zeta}(h_{(1)})]_{(2)}h_{(3)}\right)  \\
&\overset{(\ref{eqn:iotapi})}{=}&
\sum\overline{\zeta}(h_{(1)})_{(1)}h_{(2)}\otimes
    \left[\pi\left(\iota[\overline{\zeta}(h_{(1)})_{(2)}]\right)\btl h_{(3)}\right]  \\
&\overset{(\ref{eqn:piiota})}{=}&
\sum\iota[\overline{\zeta}(h_{(1)})]h_{(2)}\otimes
    \left[\pi(1)\btl h_{(3)}\right]  \\
&\overset{(\ref{eqn:iotapi})}{=}&
\sum\iota[\overline{\zeta}(h_{(1)})]h_{(2)}\otimes \pi(h_{(3)})  \\
&\overset{(\ref{eqn:gammadef})}{=}&
\sum\gamma[\pi(h_{(1)})]\otimes\pi(h_{(2)})  \\
&=&
\sum\gamma[\pi(h)_{(1)}]\otimes\pi(h)_{(2)}.
\end{eqnarray*}

\item[(3)]
For any $h\in H$, we know by the definition of $\gamma$ in (\ref{eqn:gammadef}) that
\begin{eqnarray*}
\sum \iota[\zeta(h_{(1)})]\gamma[\pi(h_{(2)})]
&\overset{(\ref{eqn:gammadef})}{=}&
\sum \iota[\zeta(h_{(1)})]\iota[\overline{\zeta}(h_{(2)})]h_{(3)}  \\
&=&  \sum \iota[\zeta(h_{(1)})\overline{\zeta}(h_{(2)})]h_{(3)}
= h.
\end{eqnarray*}
The desired equation (\ref{eqn:convolution}) is obtained.

Also, one could know by \cite[Lemma 1]{DT86} that $\iota\circ\overline{\zeta}$ is the convolution inverse of $\iota\circ\zeta\in\End_\k(H)$, as $\iota$ is an algebra map. Consequently, it follows by (\ref{eqn:convolution}) that the equation $\gamma\circ\pi=(\iota\circ\overline{\zeta})\ast\id_H$ holds as well, which implies that
\begin{eqnarray*}
\zeta\left(\gamma[\pi(h)]\right)
&=& \sum\zeta\left(\iota[\overline{\zeta}(h_{(1)})]h_{(2)}\right)
\overset{(\ref{eqn:zeta})}{=}
\sum\overline{\zeta}(h_{(1)})\zeta(h_{(2)})  \\
&=& \la\e,h\ra1
=\la\varepsilon_{H/B^+H},\pi(h)\ra1
\end{eqnarray*}
for any $h\in H$. In other words, $\zeta\circ\gamma$ is trivial on $H/B^+H$.
\end{itemize}
\end{proof}

The properties for $\zeta^\ast:B^\ast\rightarrow H^\ast$ and $\gamma^\ast:H^\ast\rightarrow(H/B^+H)^\ast$ dual to those in Lemma \ref{lem:cleftcocleft} are more useful in this paper, which are listed in the following corollary. We denote by $\btr$ the left $H^\ast$-module structure on $B^\ast$ induced by $\iota^\ast$.

\begin{corollary}\label{cor:cleftcocleft*}
Suppose $\zeta:H\rightarrow B$ and $\gamma:H/B^+H\rightarrow H$ satisfy the properties in Lemma \ref{lem:cleftcocleft}. Then
\begin{equation}\label{eqn:admissiblemapsys*}
\begin{array}{ccc}
\xymatrix{
(H/B^+H)^\ast \ar@<.5ex>[r]^{\;\;\;\;\;\;\pi^\ast}
& H^\ast \ar@<.5ex>@{-->}[l]^{\;\;\;\;\;\;\gamma^\ast} \ar@<.5ex>[r]^{\iota^\ast}
& B^\ast \ar@<.5ex>@{-->}[l]^{\zeta^\ast}  },
\end{array}
\end{equation}
where:
\begin{itemize}
\item[(1)]
$\iota^\ast$ is a map of left $H^\ast$-modules and coalgebras, and $\zeta^\ast$ is a map of left $B^\ast$-comodules. Namely:
\begin{equation}\label{eqn:iota*}
\iota^\ast(h^\ast k^\ast)=h^\ast\btr \iota^\ast(k^\ast)
\;\;\;\;\;\;(\forall h^\ast,k^\ast\in H^\ast),
\end{equation}
and
\begin{equation}\label{eqn:zeta*}
\sum\iota^\ast[\zeta^\ast(b^\ast)_{(1)}]\otimes\zeta^\ast(b^\ast)_{(2)}
=\sum b^\ast_{(1)}\otimes\zeta^\ast(b^\ast_{(2)})
\;\;\;\;\;\;(\forall b^\ast\in B^\ast);
\end{equation}
\item[(2)]
$\pi^\ast$ is a map of right $H^\ast$-comodules and algebras, and $\gamma^\ast$ is a map of right $(H/B^+H)^\ast$-modules. Namely:
\begin{equation}\label{eqn:pi*}
\sum\pi^\ast(f)_{(1)}\otimes\pi^\ast(f)_{(2)}
=\sum\pi^\ast(f_{(1)})\otimes f_{(2)}
\;\;\;\;\;\;(\forall f\in (H/B^+H)^\ast),
\end{equation}
and
\begin{equation}\label{eqn:gamma*}
\gamma^\ast[h^\ast\pi^\ast(f)]=\gamma^\ast(h^\ast)f
\;\;\;\;\;\;(\forall h^\ast\in H^\ast,\;f\in (H/B^+H)^\ast);
\end{equation}
\item[(3)]
Denote also the convolution multiplication on $\End_\k(H^\ast)$ by $\ast$. Then
\begin{equation}\label{eqn:convolution*}
(\zeta^\ast\circ\iota^\ast)\ast(\pi^\ast\circ\gamma^\ast)=\id_{H^\ast}.
\end{equation}
Moreover,
\begin{equation}\label{eqn:gamma*zeta*}
\gamma^\ast\circ\zeta^\ast=\langle-,1\rangle \e_{H/B^+H}
\end{equation}
holds on $B^\ast$.
\end{itemize}
\end{corollary}

\begin{remark}\label{rmk:convolution}
Suppose that $\overline{\zeta}$ and $\overline{\gamma}$ are convolution inverses of $\zeta$ and $\gamma$ respectively.
As shown in the proof of Lemma \ref{lem:cleftcocleft}(3), we have in this situation additional equations in the convolution algebra $\End_\k(H)$, and especially their dual forms in $\End_\k(H^\ast)$. For examples:
\begin{equation}\label{eqn:convolution*1}
\zeta^\ast\circ\iota^\ast=\id_{H^\ast}\ast\;(\pi^\ast\circ\overline{\gamma}^\ast)
\end{equation}
as well as
\begin{equation}\label{eqn:convolution*2}
\pi^\ast\circ\overline{\gamma}^\ast=S\ast(\zeta^\ast\circ\iota^\ast),\;\;\;\;
\overline{\zeta}^\ast\circ\iota^\ast=(\pi^\ast\circ\gamma^\ast)\ast S\;\;\;\;
\text{and}\;\;\;\;
(\pi^\ast\circ\overline{\gamma}^\ast)\ast(\overline{\zeta}^\ast\circ\iota^\ast)=S
\end{equation}
and so on. Here the antipode of $H^\ast$ is also denoted by $S$ without confusions.
\end{remark}

Furthermore, we could assume without the loss of generality that the cointegral $\zeta:H\rightarrow B$ is \textit{unitary} and \textit{counitary}, which means that
\begin{equation*}
\zeta(1_H)=1_B\;\;\;\;\text{and}\;\;\;\;\e|_B \circ\,\zeta=\varepsilon
\end{equation*}
both hold. This is due to reasons appearing in the proofs of \cite[Theorem 9]{DT86} and \cite[Lemma 2.15]{Mas92}. Specifically, it is clear that
$\zeta':H\rightarrow B,\;h\mapsto\zeta(h)\overline{\zeta}(1)$
would be a unitary left $B$-module map. Moreover,
$\overline{\zeta'}:h\mapsto\zeta(1)\overline{\zeta}(h)$ is the convolution inverse of $\zeta'$ satisfying
\begin{equation}\label{eqn:zeta'bar}
\la\e|_B,\overline{\zeta'}(\iota(b)h)\ra
=\la\e|_B,b\overline{\zeta'}(h)\ra
\end{equation}
for any $b\in B$ and $h\in H$.
In fact, this could be obtained by following computations:
\begin{eqnarray*}
\la\e|_B,\overline{\zeta'}(\iota(b)h)\ra
&=&
\sum\la\e|_B,\overline{\zeta'}(\iota(b)_{(1)}h)\ra\la\e,\iota(b)_{(2)}\ra  \\
&\overset{(\ref{eqn:iotapi})}{=}&
\sum\la\e|_B,\overline{\zeta'}(b_{(1)}h_{(1)})\ra
  \la\e|_B,b_{(2)}\zeta'(h_{(2)})\overline{\zeta'}(h_{(3)})\ra  \\
&\overset{(\ref{eqn:zeta})}{=}&
\sum\la\e|_B,\overline{\zeta'}(b_{(1)}h_{(1)})\ra
  \la\e|_B,\zeta'[\iota(b_{(2)})h_{(2)}]\overline{\zeta'}(h_{(3)})\ra  \\
&\overset{(\ref{eqn:iotapi})}{=}&
\sum\la\e|_B,\overline{\zeta'}[\iota(b)_{(1)}h_{(1)}]\zeta'[\iota(b)_{(2)}h_{(2)}]
  \overline{\zeta'}(h_{(3)})\ra
~=~ \la\e|_B,b\overline{\zeta'}(h)\ra.
\end{eqnarray*}
Next we define
\begin{equation}\label{eqn:zeta''}
\zeta'':H\rightarrow B,\; h\mapsto \sum \zeta'(h_{(1)})\la\e|_B,\overline{\zeta'}(h_{(2)})\ra,
\end{equation}
which is clearly unitary and counitary, with convolution inverse
$$\overline{\zeta''}:h\mapsto
\sum\la\e|_B,\zeta'(h_{(1)})\ra\overline{\zeta'}(h_{(2)}).$$
Meanwhile, $\zeta''$ also preserves left $B$-actions, since we might calculate for any $b\in B$ and $h\in H$ that
\begin{eqnarray*}
\zeta''[\iota(b)h]
&\overset{(\ref{eqn:zeta''})}{=}&
\sum \zeta'(b_{(1)}h_{(1)})\la\e|_B,\overline{\zeta'}[\iota(b_{(2)})h_{(2)}]\ra \\
&\overset{(\ref{eqn:zeta'bar})}{=}&
\sum \zeta'(b_{(1)}h_{(1)})\la\e|_B,b_{(2)}\overline{\zeta'}(h_{(2)})\ra  \\
&=& \sum \zeta'[\iota(b)h_{(1)}]\la\e|_B,\overline{\zeta'}(h_{(2)})\ra
~\overset{(\ref{eqn:zeta})}{=}~
   \sum b\zeta'(h_{(1)})\la\e|_B,\overline{\zeta'}(h_{(2)})\ra  \\
&\overset{(\ref{eqn:zeta''})}{=}& b\zeta''(h).
\end{eqnarray*}

As a conclusion, the constructions for $\zeta$ and $\gamma$ in Lemma \ref{lem:cleftcocleft} could be unitary as well as counitary. For convenience, we say that a linear map is \textit{biunitary} if it preserves unit and counit at the same time.

\begin{lemma}\label{lem:cleftcocleft2}
Let $H$ be a finite-dimensional Hopf algebra with left coideal subalgebra $B$. Then there exist convolution invertible maps $\zeta:H\rightarrow B$ and $\gamma:H/B^+H\rightarrow H$ satisfying properties in Lemma \ref{lem:cleftcocleft} (and Corollary \ref{cor:cleftcocleft*}), as well as the followings:
\begin{itemize}
\item[(1)]
$\zeta$ is biunitary:
\begin{equation}\label{eqn:zetabiunitary}
\zeta(1)=1_B\;\;\;\;\text{and}\;\;\;\;\e|_B\circ\,\zeta=\e,
\end{equation}
and $\zeta\circ\iota=\id_B$;
\item[(2)]
$\gamma$ is biunitary:
\begin{equation}\label{eqn:gammabiunitary}
\gamma[\pi(1)]=1\;\;\;\;\text{and}\;\;\;\;\e\circ\gamma=\e_{H/B^+H},
\end{equation}
and $\pi\circ\gamma=\id_{H/B^+H}$;
\end{itemize}
Dually:
\begin{itemize}
\item[(3)]
$\zeta^\ast$ is biunitary, and
\begin{equation}\label{eqn:iota*zeta*}
\iota^\ast\circ\zeta^\ast=\id_{B^\ast};
\end{equation}
\item[(4)]
$\gamma^\ast$ is biunitary, and
\begin{equation}\label{eqn:gamma*pi*}
\gamma^\ast\circ\pi^\ast=\id_{(H/B^+H)^\ast}.
\end{equation}
\end{itemize}
\end{lemma}

\begin{proof}
\begin{itemize}
\item[(1)]
As mentioned before the lemma, we might assume that (\ref{eqn:zetabiunitary}) holds for the cointegral $\zeta:H\rightarrow B$. This also implies that $\zeta\circ\iota=\id_B$, since $\zeta$ preserves left $B$-actions. Specifically, we have
$$\zeta[\iota(b)]\overset{(\ref{eqn:zeta})}{=}b\zeta(1)=b$$
for any $b\in B$.

\item[(2)]
It is known by (1) that $\overline{\zeta}(1)=1_B$, because $1_H$ is a grouplike element in $H$. With the definition of $\gamma$ (\ref{eqn:gammadef}), we find that
$\gamma(\pi(1))=\iota[\overline{\zeta}(1)]1=1_B$ holds.

On the other hand, $\overline{\zeta}$ is counitary since $\zeta$ is so. Thus
\begin{eqnarray*}
(\varepsilon\circ\gamma)[\pi(h)]
&\overset{(\ref{eqn:gammadef})}{=}&
\sum\la\e,\iota[\overline{\zeta}(h_{(1)})]h_{(2)}\ra
=\la\e|_B,\overline{\zeta}(h)\ra  \\
&=& \la\e,h\ra=\la\e_{H/B^+H},\pi(h)\ra
\end{eqnarray*}
for all $h\in H$, which implies that $\gamma$ is counitary (\ref{eqn:gammabiunitary}).
Furthermore, note that $\gamma$ is a map of right $H/B^+H$-comodules. Then for any $x\in H/B^+H$, we have the equation
$$\sum\gamma(x)_{(1)}\otimes\pi[\gamma(x)_{(2)}]
=\sum \gamma(x_{(1)})\otimes x_{(2)},$$
whose image under $\e\otimes\id$ would become
\begin{eqnarray*}
\pi[\gamma(x)]&=&\sum\la\e,\gamma(x)_{(1)}\ra\pi[\gamma(x)_{(2)}]
~=~\sum\la\e,\gamma(x_{(1)})\ra x_{(2)}  \\
&\overset{(\ref{eqn:gammabiunitary})}{=}&
\sum\la\e_{H/B^+H},x_{(1)}\ra x_{(2)}~=~x.
\end{eqnarray*}

\item[(3)]
This is dual to (1).
\item[(4)]
This is dual to (2).
\end{itemize}
\end{proof}

With the properties in Lemma 2.5 (2) and (3), we could write the module structures $\btl$ and $\btr$ on $H/B^+H$ and $B^\ast$ respectively as follows:
\begin{equation}\label{eqn:btl}
x\btl h=\pi[\gamma(x)h]
\;\;\;\;\;\;\text{for all}\;\;h\in H,\;x\in H/B^+H
\end{equation}
and
\begin{equation}\label{eqn:btr}
h^\ast\btr b^\ast
=\iota^\ast[h^\ast\zeta^\ast(b^\ast)]
\;\;\;\;\;\;\text{for all}\;\;h^\ast\in H^\ast,\;b\in B^\ast.
\end{equation}

\subsection{Partially admissible mapping system for left coideal subalgebra}\label{subsection:2.2}

As shown in the previous subsection, for every left coideal subalgebra $B$ of a finite-dimensional Hopf algebra $H$, there always exists a pair of cointegrals
$$\zeta:H\rightarrow B\;\;\;\;\text{and}\;\;\;\;\gamma^\ast:H^\ast\rightarrow (H/B^+H)^\ast$$
satisfying all the properties in Lemma \ref{lem:cleftcocleft}, Corollary \ref{cor:cleftcocleft*}, as well as Lemma \ref{lem:cleftcocleft2}. Such a pair $(\zeta,\gamma^\ast)$ would be called a \textit{{\pams}} in this paper, which is a generalization of an admissible mapping system introduced in \cite[Section 2.8]{Rad85}, or the mapping system introduced in \cite[Section 1]{Sch02(b)} and \cite[Section 3.3]{Sch02(b)}. Specifically:

\begin{definition}\label{def:PAMS}
Let $H$ be a finite-dimensional Hopf algebra. Suppose that
\begin{itemize}
\item[(1)]
$\iota:B\rightarrowtail H$ is an injection of left $H$-comodule algebras, and
$\pi:H\twoheadrightarrow C$ is a surjection of right $H$-module coalgebras;
\item[(2)]
The image of $\iota$ equals the space of the coinvariants of the right $C$-comodule $H$ with structure $(\id_H\otimes\pi)\circ \Delta$.
\end{itemize}
Then the pair of $\k$-linear diagrams
\begin{equation}
\begin{array}{ccc}
\xymatrix{
B \ar@<.5ex>[r]^{\iota} & H \ar@<.5ex>@{-->}[l]^{\zeta} \ar@<.5ex>[r]^{\pi}
& C \ar@<.5ex>@{-->}[l]^{\gamma}  }
&\;\;\text{and}\;\;&
\xymatrix{
C^\ast \ar@<.5ex>[r]^{\pi^\ast}
& H^\ast \ar@<.5ex>@{-->}[l]^{\gamma^\ast} \ar@<.5ex>[r]^{\iota^\ast}
& B^\ast \ar@<.5ex>@{-->}[l]^{\zeta^\ast}  },
\end{array}
\end{equation}
is said to be a {\pams} for $\iota$, denoted by $(\zeta,\gamma^\ast)$ for simplicity, if all the following conditions
\begin{itemize}
\item[(3)]
$\zeta$ and $\gamma$ have convolution inverses $\overline{\zeta}$ and $\overline{\gamma}$ respectively;
\item[(4)]
$\zeta$ preserves left $B$-actions, and $\gamma$ preserves right $C$-coactions;
\item[(5)]
$\zeta$ and $\gamma$ are biunitary (where the counit of $B$ and unit of $C$ are induced by those of $H$ via $\iota$ and $\pi$ respectively);
\item[(6)]
$(\iota\circ\zeta)\ast(\gamma\circ\pi)=\id_H$
\end{itemize}
hold (and the dual forms of (1) to (6) hold equivalently).
\end{definition}

\begin{remark}\label{rmk:pams}
\begin{itemize}
\item[(1)]
When $B$ is a left coideal subalgebra of $H$, the quotient right $H$-module coalgebra $C$ satisfying (1) and (2) in Definition \ref{def:PAMS} must be isomorphic to $H/B^+H$. This is because of \cite[Theorem 1]{Tak79} and \cite[Proposition 3.10(1)]{Mas94(b)} applied to the Hopf algebra $H^\biop$ and the cleftness property.

Consequently, there must exist a {\pams} $(\zeta,\gamma^\ast)$ for the inclusion $B\subseteq H$:
\begin{equation*}
\begin{array}{ccc}
\xymatrix{
B \ar@<.5ex>[r]^{\iota} & H \ar@<.5ex>@{-->}[l]^{\zeta} \ar@<.5ex>[r]^{\pi\;\;\;\;\;\;}
& H/B^+H \ar@<.5ex>@{-->}[l]^{\gamma\;\;\;\;\;\;}  }
&\;\;\text{and}\;\;&
\xymatrix{
(H/B^+H)^\ast \ar@<.5ex>[r]^{\;\;\;\;\;\;\pi^\ast}
& H^\ast \ar@<.5ex>@{-->}[l]^{\;\;\;\;\;\;\gamma^\ast} \ar@<.5ex>[r]^{\iota^\ast}
& B^\ast \ar@<.5ex>@{-->}[l]^{\zeta^\ast}  }
\end{array}
\end{equation*}
according to the conclusions in the previous subsection.

\item[(2)]
The condition (3) in Definition \ref{def:PAMS} could be inferred from the conditions (1),(2),(4),(6). Specifically,
since the convolution algebra $\End_\k(H)$ is finite-dimensional, it follows that the transformation $\theta:=(\gamma\circ\pi)\ast S$ is the convolution inverse of $\iota\circ\zeta$. Now we claim that $\theta(H)$ is included in
$$\iota(B)=\{h\in H\mid \sum h_{(1)}\otimes \pi(h_{(2)})=h\otimes \pi(1)\}.$$
In fact, denote the right $H$-action on $C$ (or $H/B^+H$) by $\btl$, and we could use the assumptions
\begin{equation}\label{eqn:pi-rmk}
\pi(hk)=\pi(h)\btl k\;\;\;\;(\forall h,k\in H)
\end{equation}
and
\begin{equation}\label{eqn:gamma-rmk}
\sum\gamma(x)_{(1)}\otimes\pi[\gamma(x)_{(2)}]
=\sum \gamma(x_{(1)})\otimes x_{(2)}\;\;
(\forall x\in C)
\end{equation}
from (1),(4) to calculate for every $h\in H$ that
\begin{eqnarray*}
(\id\otimes\pi)\circ\Delta(\theta(h))
&=&
\sum\gamma[\pi(h_{(1)})]_{(1)}S(h_{(3)})
\otimes\pi\left(\gamma[\pi(h_{(1)})]_{(2)}S(h_{(2)})\right)  \\
&\overset{(\ref{eqn:pi-rmk})}=&
\sum\gamma[\pi(h_{(1)})]_{(1)}S(h_{(3)})\otimes
  \left[\pi\left(\gamma[\pi(h_{(1)})]_{(2)}\right)\btl S(h_{(2)})\right] \\
&\overset{(\ref{eqn:gamma-rmk})}=&
\sum\gamma[\pi(h_{(1)})_{(1)}]S(h_{(3)})\otimes
  \left[\pi(h_{(1)})_{(2)}\btl S(h_{(2)})\right] \\
&=&
\sum\gamma[\pi(h_{(1)})]S(h_{(4)})\otimes\left[\pi(h_{(2)})\btl S(h_{(3)})\right] \\
&\overset{(\ref{eqn:pi-rmk})}=&
\sum\gamma[\pi(h_{(1)})]S(h_{(4)})\otimes\pi(h_{(2)}S(h_{(3)})) \\
&=&
\sum\gamma[\pi(h_{(1)})]S(h_{(2)})\otimes\pi(1)
~=~ \theta(h)\otimes\pi(1).
\end{eqnarray*}
Thus $\theta=\iota\circ\overline{\zeta}$ holds for some linear map $\overline{\zeta}:H\rightarrow B$. However, note that $\iota$ is an algebra map, which implies that $\overline{\zeta}$ is the convolution inverse of $\zeta$.

Similarly, one could show that $\gamma$ is also convolution invertible with the help of the diagram
$$\xymatrix{
C^\ast \ar@<.5ex>[r]^{\pi^\ast}
& H^\ast \ar@<.5ex>@{-->}[l]^{\gamma^\ast} \ar@<.5ex>[r]^{\iota^\ast}
& B^\ast \ar@<.5ex>@{-->}[l]^{\zeta^\ast}  }.$$
\end{itemize}
\end{remark}

As a particular consequence of Remark \ref{rmk:pams}(1), one might also find an equation on the dimensions due to \cite[Theorem 2.1(6)]{Mas92}:

\begin{corollary}\label{cor:dims}
Suppose that $(\zeta,\gamma^\ast)$ is a {\pams} for $\iota:B\rightarrowtail H$. With notations in Definition \ref{def:PAMS}, we have
\begin{equation}\label{eqn:dims}
\dim(H)=\dim(B)\dim(C)=\dim(B)\dim(H/B^+H).
\end{equation}
\end{corollary}

Additional basic properties of {\pams}s could be concluded from Subsection \ref{subsection:cocleft}:

\begin{proposition}\label{prop:PAMS-comptrival}
Suppose that $(\zeta,\gamma^\ast)$ is a {\pams} for $\iota:B\rightarrowtail H$. Then:
\begin{itemize}
\item[(1)]
$\zeta\circ\iota=\id_B$ and $\pi\circ\gamma=\id_{C}$ both hold;
\item[(2)]
$\zeta\circ\gamma$ is trivial;
\item[(3)]
$\overline{\zeta}$ and $\overline{\gamma}$ are biunitary;
\item[(4)]
$\pi\circ S^{-1}\circ\iota$ is trivial.
\end{itemize}
Moreover, the dual forms of these properties hold as well.
\end{proposition}

\begin{proof}
\begin{itemize}
\item[(1)]
This is due to the same reason as the proof of Lemma \ref{lem:cleftcocleft2}.
\item[(2)]
As shown in the proof of Lemma \ref{lem:cleftcocleft}(3), one could know $\zeta\circ\gamma=\la\e_C,-\ra1_B$ due to Definition \ref{def:PAMS}(6).
\item[(3)]
We only verify that $\overline{\zeta}$ is biunitary, as $\overline{\gamma}$ is so for similar reasons.

In fact, note that $\zeta$ is biunitary. Therefore, $\overline{\zeta}$ is unitary because
\begin{eqnarray*}
\overline{\zeta}(1)&=&\overline{\zeta}(1)\zeta(1)~=~(\overline{\zeta}\ast\zeta)(1)~=~1_B.
\end{eqnarray*}
On the other hand, $\overline{\zeta}$ is counitary because
\begin{eqnarray*}
(\e|_B\circ\overline{\zeta})(h)
&=&
\sum\la\e|_B,\overline{\zeta}(h_{(1)})\ra\la\e,h_{(2)}\ra
~=~
\sum\la\e|_B,\overline{\zeta}(h_{(1)})\ra\la\e|_B,\zeta(h_{(2)})\ra  \\
&=&
\sum\la\e|_B,\overline{\zeta}(h_{(1)})\zeta(h_{(2)})\ra
~=~
\la\e,h\ra
\end{eqnarray*}
for all $h\in H$.
\item[(4)]
Assume that $B\subseteq H$ is a left coideal subalgebra without loss of generality. Then we could know that $S^{-1}(B^+)H=B^+H$ holds as subspaces of $H$, by applying \cite[Lemma 3.1]{Kop93} to the Hopf algebra $H^\cop$ with antipode $S^{-1}$. It follows that $\pi$ is exactly the quotient map
$$\pi:H\twoheadrightarrow H/S^{-1}(B^+)H=H/S^{-1}(B)^+H$$
as well, which implies that $\pi\circ S^{-1}\circ\iota$ is also trivial.
\end{itemize}
\end{proof}

Clearly, the notion of {\pams} is self-dual in the sense of the following proposition:

\begin{proposition}\label{prop:BCP-selfdual}
Let $H$ be a finite-dimensional Hopf algebra. Suppose that the diagram
$$B\xrightarrow{\iota}H\xrightarrow{\pi}C$$
satisfies that
\begin{itemize}
\item
$\iota:B\rightarrowtail H$ is an injection of left $H$-comodule algebras, and
$\pi:H\twoheadrightarrow C$ is a surjection of right $H$-module coalgebras;
\item
The image of $\iota$ equals to the space of the coinvariants of the right $C$-comodule $H$ with structure $(\id_H\otimes\pi)\circ \Delta$.
\end{itemize}
Then the followings are equivalent:
\begin{itemize}
\item[(1)]
$(\zeta,\gamma^\ast)$ is a {\pams} for $\iota:B\rightarrowtail H$;
\item[(2)]
$(\gamma^\ast,\zeta)$ is a {\pams} for $\pi^\ast:C^{\ast\,\op\,\cop}\rightarrowtail H^{\ast\,\op\,\cop}$.
\end{itemize}
\end{proposition}

\begin{proof}
It is sufficient to show that (1) implies (2), as the converse is by a similar argument on the diagram
$$C^{\ast\,\biop}\xrightarrow{\pi^\ast} H^{\ast\,\biop}
  \xrightarrow{\iota^\ast}B^{\ast\,\biop}.$$

Evidently, the map $\pi^\ast$ is an injection of left $H^{\ast\,\op\,\cop}$-comodule algebras, and $\iota^\ast$ is a surjection of right $H^\ast$-module coalgebras.
Also, the image of $\pi^\ast$ equals the space of the coinvariants of the right $B^{\ast\,\biop}$-comodule $H^{\ast\,\biop}$, according to the triviality of $\iota^\ast\circ\pi^\ast$ and counting dimensions:
$$\dim(B^{\ast\,\biop})=\dim(B)\overset{(\ref{eqn:dims})}{=}
\frac{\dim(H)}{\dim(C)}
=\frac{\dim(H^{\ast\,\biop})}{\dim(C^{\ast\,\biop})}.$$
\end{proof}

However, recalling the proofs of Lemmas \ref{lem:cleftcocleft} and \ref{lem:cleftcocleft2}, we find that a {\pams} $(\zeta,\gamma^\ast)$ for $\iota:B\rightarrowtail H$ is determined according to any given biunitary cointegral $\zeta:H\rightarrow B$. With the help of Proposition \ref{prop:BCP-selfdual}, it could be shown that:

\begin{corollary}\label{cor:BCP-exist}
Let $H$ be a finite-dimensional Hopf algebra with a diagram $B\xrightarrow{\iota}H\xrightarrow{\pi}C$ satisfying the assumptions in Proposition \ref{prop:BCP-selfdual}. Then:
\begin{itemize}
\item[(1)]
For any convolution invertible left $B$-module map $\zeta:H\rightarrow B$ which is biunitary,
there exists a unique map $\gamma$ such that $(\zeta,\gamma^\ast)$ is a {\pams};
\item[(2)]
For any convolution invertible right $C$-comodule map $\gamma:C\rightarrow H$ which is biunitary,
there exists a unique map $\zeta$ such that $(\zeta,\gamma^\ast)$ is a {\pams}.
\end{itemize}
\end{corollary}

\begin{proof}
Here we identify the right $H$-module coalgebra $C$ with $H/B^+H$ for convenience.
\begin{itemize}
\item[(1)]
The existence of $\gamma$ follows from the proofs of Lemmas \ref{lem:cleftcocleft} and \ref{lem:cleftcocleft2} as mentioned before this corollary. Moreover, the uniqueness of $\gamma$ is followed by the facts that $\gamma\circ\pi=(\iota\circ\overline{\zeta})\ast\id$ and that $\pi$ is surjective.

\item[(2)]
Suppose that $\gamma:H/B^+H\rightarrow H$ is a convolution invertible right $H/B^+H$-comodule map which is biunitary. Then its dual map $\gamma^\ast:H^{\ast\,\op\,\cop}\rightarrow(H/B^+H)^{\ast\,\op\,\cop}$ would also become convolution invertible and biunitary, but preserving left $(H/B^+H)^{\ast\,\op\,\cop}$-actions.

Consequently according to (1), there exists a {\pams} $(\gamma^\ast,\zeta)$ for the left coideal subalgebra $\pi^\ast:(H/B^+H)^{\ast\,\op\,\cop}\rightarrow H^{\ast\,\op\,\cop}$ for some cointegral
$$\zeta:H^{\op\,\cop}\cong(H^{\ast\,\op\,\cop})^\ast
\rightarrow
\left(H^{\ast\,\op\,\cop}/(H/B^+H)^{\ast\,\op\,\cop\,+}H^{\ast\,\op\,\cop}\right)^\ast
\overset{(\ref{eqn:correspondence})}{\cong} B^{\op\,\cop}.$$
Finally by Proposition \ref{prop:BCP-selfdual}, this is equivalent to say that $(\zeta,\gamma^\ast)$ is a {\pams} for $\iota:B\rightarrowtail H$.
\end{itemize}
\end{proof}

\textbf{Convention.}
We would assume $B\subseteq H$ and identify $C$ in Definition \ref{def:PAMS} with $H/B^+H$ from Subsection \ref{subsection:2.3} to Subsection \ref{subsection:5.3}. This is due to the statements in Remark \ref{rmk:pams}, and might make the contents easier to understand.

\subsection{Formulas on the convolution inverses of cointegrals, and consequent systems}\label{subsection:2.3}

In this subsection, we first explore more properties of a {\pams} $(\zeta,\gamma^\ast)$ for a left coideal subalgebra $B$ of a finite-dimensional Hopf algebra $H$.

For the subsequent applications, the formulas with the biunitary cointegral $\gamma^\ast$ and integral $\zeta^\ast$ would be in fact more useful (rather than $\zeta$ and $\gamma$), including Equations (\ref{eqn:iota*}) to (\ref{eqn:convolution*}) in Corollary \ref{cor:cleftcocleft*} as well as (\ref{eqn:iota*zeta*}) and (\ref{eqn:gamma*pi*}). Let us begin by pointing out some properties of their convolution inverses $\overline{\gamma}^\ast$ and $\overline{\zeta}^\ast$, where we would regard $B\subseteq H$ and $(H/B^+H)^\ast\subseteq H^\ast$ frequently to use simplified notations
$$\sum b_{(1)}\otimes b_{(2)}\in H\otimes B
\;\;\;\;\;\;\text{and}\;\;\;\;\;\;\sum f_{(1)}\otimes f_{(2)}\in (H/B^+H)^\ast\otimes H^\ast,$$
where $b\in B^\ast$ and $f\in(H/B^+H)^\ast$.

\begin{lemma}\label{eqn:f(1)gammabarf(2)}
\begin{itemize}
\item[(1)] For any $f\in(H/B^+H)^\ast$ and $h^\ast\in H^\ast$,
\begin{equation}\label{eqn:gamma*bar1}
\sum f_{(1)}\overline{\gamma}^\ast(h^\ast f_{(2)})
=\la f,1\ra\overline{\gamma}^\ast(h^\ast)\in(H/B^+H)^\ast;
\end{equation}
\item[(2)] For any $b\in B^\ast$,
\begin{equation}\label{eqn:zeta*bar1}
\sum(\overline{\zeta}^\ast(b^\ast_{(1)})_{(1)}\btr b^\ast_{(2)})\otimes\overline{\zeta}^\ast(b^\ast_{(1)})_{(2)}
=\e\otimes\overline{\zeta}^\ast(b^\ast)\in B^\ast\otimes H^\ast.
\end{equation}
\end{itemize}
\end{lemma}

\begin{proof}
\begin{itemize}
\item[(1)]
We compute directly that
\begin{eqnarray*}
&& \sum f_{(1)}\overline{\gamma}^\ast(h^\ast f_{(2)})  \\
&\overset{(\ref{eqn:gamma*pi*})}{=}&
\sum \gamma^\ast\left(\pi^\ast(f_{(1)})\right)\overline{\gamma}^\ast(h^\ast f_{(2)})
~\overset{(\ref{eqn:gamma*})}{=}~
\sum \gamma^\ast\left(\pi^\ast(f_{(1)})\pi^\ast[\overline{\gamma}^\ast(h^\ast f_{(2)})]\right)  \\
&\overset{(\ref{eqn:convolution*2})}{=}&
\sum \gamma^\ast\left(\pi^\ast(f_{(1)})
  S(h^\ast_{(1)}f_{(2)})\zeta^\ast[\iota^\ast(h^\ast_{(2)} f_{(3)})]\right)  \\
&\overset{(\ref{eqn:pi*})}{=}&
\sum \gamma^\ast\left(\pi^\ast(f)_{(1)}
  S(\pi^\ast(f)_{(2)})S(h^\ast_{(1)})\zeta^\ast[\iota^\ast(h^\ast_{(2)} \pi^\ast(f)_{(3)})]\right)  \\
&=&
\sum \gamma^\ast\left(S(h^\ast_{(1)})\zeta^\ast[\iota^\ast(h^\ast_{(2)} \pi^\ast(f))]\right)  \\
&\overset{(\ref{eqn:iota*})}{=}&
\sum \gamma^\ast\left(S(h^\ast_{(1)})\zeta^\ast[h^\ast_{(2)}\btr \iota^\ast(\pi^\ast(f))]\right)  \\
&\overset{(\ref{eqn:piiota})}{=}&
\la f,1\ra\sum \gamma^\ast\left(S(h^\ast_{(1)})\zeta^\ast(h^\ast_{(2)}\btr \e)\right)
~\overset{(\ref{eqn:btr})}{=}~
\la f,1\ra\sum \gamma^\ast\left(S(h^\ast_{(1)})\zeta^\ast[\iota^\ast(h^\ast_{(2)})]\right)  \\
&\overset{(\ref{eqn:convolution*2})}{=}&
\la f,1\ra\sum \gamma^\ast\left(\pi^\ast[\overline{\gamma}^\ast(h)]\right)
~=~
\la f,1\ra\overline{\gamma}^\ast(h^\ast).
\end{eqnarray*}

\item[(2)]
If we consider by Proposition \ref{prop:BCP-selfdual} the {\pams} $(\gamma^\ast,\zeta)$ for $\pi^\ast:(H/B^+H)^{\ast\,\op\,\cop}\rightarrowtail H^{\ast\,\op\,\cop}$,
then one could know according to (1) that
$$\sum\overline{\zeta}(b_{(1)}h)b_{(2)}=\la\e,b\ra\overline{\zeta}(h)\in B$$
holds for all $b\in B$ and $h\in H$. Let us compare values of the both sides of this equation under an arbitrary $b^\ast\in B^\ast$, which are:
\begin{eqnarray*}
\left\la b^\ast,\sum\overline{\zeta}(b_{(1)}h)b_{(2)}\right\ra
&=&
\sum\la b^\ast_{(1)},\overline{\zeta}(b_{(1)}h)\ra\la b^\ast_{(2)},b_{(2)}\ra  \\
&\overset{(\ref{eqn:iota*zeta*})}{=}&
\sum\la \overline{\zeta}^\ast(b^\ast_{(1)})_{(1)},b_{(1)}\ra
    \la \overline{\zeta}^\ast(b^\ast_{(1)})_{(2)},h\ra \la \iota^\ast[\zeta^\ast(b^\ast_{(2)})],b_{(2)}\ra  \\
&=&
\sum\la \overline{\zeta}^\ast(b^\ast_{(1)})_{(1)},b_{(1)}\ra
    \la \overline{\zeta}^\ast(b^\ast_{(1)})_{(2)},h\ra \la b^\ast_{(2)},\zeta[\iota(b_{(2)})]\ra  \\
&\overset{(\ref{eqn:iotapi})}{=}&
\sum\la \overline{\zeta}^\ast(b^\ast_{(1)})_{(1)},\iota(b)_{(1)}\ra
    \la \overline{\zeta}^\ast(b^\ast_{(1)})_{(2)},h\ra
    \la \zeta^\ast(b^\ast_{(2)}),\iota(b)_{(2)}\ra  \\
&=&
\sum\la \iota^\ast[\overline{\zeta}^\ast(b^\ast_{(1)})_{(1)}\zeta^\ast(b^\ast_{(2)})],b\ra
    \la \overline{\zeta}^\ast(b^\ast_{(1)})_{(2)},h\ra  \\
&\overset{(\ref{eqn:btr})}{=}&
\sum\la \overline{\zeta}^\ast(b^\ast_{(1)})_{(1)}\btr b^\ast_{(2)},b\ra
    \la \overline{\zeta}^\ast(b^\ast_{(1)})_{(2)},h\ra
\end{eqnarray*}
and
\begin{eqnarray*}
\left\la b^\ast,\la\e,b\ra\overline{\zeta}(h)\right\ra
&=&
\la \e,b\ra \la b^\ast,\overline{\zeta}(h) \ra
~=~
\la \e,b\ra \la \overline{\zeta}^\ast(b^\ast),h \ra.
\end{eqnarray*}
As a conclusion, the desired equation
$$\sum(\overline{\zeta}^\ast(b^\ast_{(1)})_{(1)}\btr b^\ast_{(2)})\otimes\overline{\zeta}^\ast(b^\ast_{(1)})_{(2)}
=\e\otimes\overline{\zeta}^\ast(b^\ast)$$
holds in $B^\ast\otimes H^\ast$.
\end{itemize}
\end{proof}

This lemma implies that $\overline{\gamma}^\ast\circ S^{-1}$ preserves left $(H/B^+H)^\ast$-actions, and that $S^{-1}\circ\overline{\zeta}^\ast$ preserves right $B^\ast$-coactions:

\begin{corollary}\label{cor:cleftcocleft*convinv}
\begin{itemize}
\item[(1)]
For any $f\in(H/B^+H)^\ast$ and $h^\ast\in H^\ast$,
\begin{equation}\label{eqn:gamma*bar2}
f\overline{\gamma}^\ast(h^\ast)=\overline{\gamma}^\ast[h^\ast S^{-1}(\pi^\ast(f))]
\end{equation}
holds in $(H/B^+H)^\ast$, namely, $\overline{\gamma}^\ast\circ S^{-1}:H^\ast\rightarrow(H/B^+H)^\ast$ is a left $(H/B^+H)^\ast$-module map;

\item[(2)]
For any $b^\ast\in B^\ast$,
\begin{equation}\label{eqn:zeta*bar2}
\sum\overline{\zeta}^\ast(b^\ast_{(1)})\otimes b^\ast_{(2)}
=\sum\overline{\zeta}^\ast(b^\ast)_{(2)}\otimes\iota^\ast[S^{-1}(\overline{\zeta}^\ast(b^\ast)_{(1)})]
\end{equation}
holds in $H^\ast\otimes B^\ast$, namely, $S^{-1}\circ\overline{\zeta}^\ast:B^\ast\rightarrow H^\ast$ is a right $B^\ast$-comodule map.
\end{itemize}
\end{corollary}

\begin{proof}
We verify the equations with the help of Lemma \ref{eqn:f(1)gammabarf(2)}:
\begin{itemize}
\item[(1)]
For any $f\in(H/B^+H)^\ast$ and $h^\ast\in H^\ast$,
\begin{eqnarray*}
f\overline{\gamma}^\ast(h^\ast)
&=&
\sum f_{(1)}\overline{\gamma}^\ast(h^\ast S^{-1}(f_{(3)})f_{(2)})
~\overset{(\ref{eqn:gamma*bar1})}{=}~
\sum \la f_{(1)},1\ra\overline{\gamma}^\ast(h^\ast S^{-1}(f_{(2)}))  \\
&\overset{(\ref{eqn:pi*})}{=}&
\sum \la\pi^\ast(f)_{(1)},1\ra\overline{\gamma}^\ast[h^\ast S^{-1}(\pi^\ast(f)_{(2)})]
~=~
\overline{\gamma}^\ast[h^\ast S^{-1}(\pi^\ast(f))].
\end{eqnarray*}
It follows that $\overline{\gamma}^\ast\circ S^{-1}$ preserves left $(H/B^+H)^\ast$-actions:
\begin{eqnarray*}
\overline{\gamma}^\ast[S^{-1}(\pi^\ast(f)h^\ast)]
&=&
\overline{\gamma}^\ast[S^{-1}(h^\ast)S^{-1}(\pi^\ast(f))]
~\overset{(\ref{eqn:gamma*bar2})}{=}~
f\overline{\gamma}^\ast(S^{-1}(h^\ast)).
\end{eqnarray*}

\item[(2)]
For any $b^\ast\in B^\ast$, we know by Equation (\ref{eqn:zeta*bar1}) that
\begin{equation}\label{eqn:zeta*bar1Delta}
\sum(\overline{\zeta}^\ast(b^\ast_{(1)})_{(1)}\btr b^\ast_{(2)})
\otimes\overline{\zeta}^\ast(b^\ast_{(1)})_{(2)}
\otimes\overline{\zeta}^\ast(b^\ast_{(1)})_{(3)}
=\e\otimes\overline{\zeta}^\ast(b^\ast)_{(1)}\otimes\overline{\zeta}^\ast(b^\ast)_{(2)},
\end{equation}
and consequently,
\begin{eqnarray*}
\sum\overline{\zeta}^\ast(b^\ast_{(1)})\otimes b^\ast_{(2)}
&=&
\sum\overline{\zeta}^\ast(b^\ast_{(1)})_{(3)}
    \otimes \left[S^{-1}(\overline{\zeta}^\ast(b^\ast_{(1)})_{(2)})
            \overline{\zeta}^\ast(b^\ast_{(1)})_{(1)}\btr b^\ast_{(2)}\right]  \\
&=&
\sum\overline{\zeta}^\ast(b^\ast_{(1)})_{(3)}
    \otimes \left[S^{-1}(\overline{\zeta}^\ast(b^\ast_{(1)})_{(2)})
            \btr\left(\overline{\zeta}^\ast(b^\ast_{(1)})_{(1)}\btr b^\ast_{(2)}\right)\right]  \\
&\overset{(\ref{eqn:zeta*bar1Delta})}{=}&
\sum\overline{\zeta}^\ast(b^\ast)_{(2)}
    \otimes \left(S^{-1}(\overline{\zeta}^\ast(b^\ast)_{(1)})\btr\e\right)  \\
&\overset{(\ref{eqn:btr})}{=}&
\sum\overline{\zeta}^\ast(b^\ast)_{(2)}\otimes\iota^\ast[S^{-1}(\overline{\zeta}^\ast(b^\ast)_{(1)})].
\end{eqnarray*}
It also follows that $S^{-1}\circ\overline{\zeta}^\ast$ preserves right $B^\ast$-coactions:
\begin{eqnarray*}
\sum S^{-1}(\overline{\zeta}^\ast(b^\ast))_{(1)}
\otimes\iota^\ast[S^{-1}(\overline{\zeta}^\ast(b^\ast))_{(2)}]
&=&
\sum S^{-1}(\overline{\zeta}^\ast(b^\ast)_{(2)})
\otimes\iota^\ast[S^{-1}(\overline{\zeta}^\ast(b^\ast)_{(1)})]  \\
&\overset{(\ref{eqn:zeta*bar2})}{=}&
\sum S^{-1}(\overline{\zeta}^\ast(b^\ast_{(1)}))\otimes b^\ast_{(2)}.
\end{eqnarray*}
\end{itemize}
\end{proof}

Now by the conclusions stated in Corollary \ref{cor:cleftcocleft*convinv}, some {\pams}s for left coideal subalgebras of $H^{\ast\,\cop}$ and $H^{\ast\,\op}$ could be provided as consequences:

\begin{proposition}\label{prop:cleftcocleftmore}
Suppose $(\zeta,\gamma^\ast)$ is a {\pams} for $\iota:B\rightarrowtail H$.
Then:
\begin{itemize}
\item[(1)]
$(\overline{\gamma}^\ast\circ S^{-1},\overline{\zeta})$ is a {\pams} for the left coideal subalgebra $\pi^\ast:(H/B^+H)^{\ast\,\cop}\rightarrowtail H^{\ast\,\cop}$:
$$\begin{array}{ccc}
\xymatrix{
(H/B^+H)^{\ast\,\cop} \ar@<.5ex>[r]^{\;\;\;\;\;\;\pi^\ast}
& H^{\ast\,\cop}
\ar@<.5ex>[l]^{\;\;\;\;\;\;\;\;\overline{\gamma}^\ast\circ S^{-1}}
\ar@<.5ex>[r]^{\;\;\iota^\ast\circ S^{-1}}
& B^{\ast\,\op} \ar@<.5ex>[l]^{\overline{\zeta}^\ast}  }
&\;\;\text{and}\;\;&
\xymatrix{
B^\cop \ar@<.5ex>[r]^{S^{-1}\circ\iota}
& H^\op
\ar@<.5ex>[l]^{\overline{\zeta}}
\ar@<.5ex>[r]^{\pi\;\;\;\;\;\;}
& (H/B^+H)^\op \ar@<.5ex>[l]^{S^{-1}\circ\overline{\gamma}\;\;\;\;\;\;}  },
\end{array}$$
where the right $H^{\ast\,\cop}$-module structure on $B^{\ast\,\op}$ is
$$B^{\ast\,\op}\otimes H^{\ast\,\cop}\rightarrow B^{\ast\,\op},\;\;
b^\ast\otimes h^\ast\mapsto S^{-1}(h^\ast)\btr b^\ast;$$

\item[(2)]
$(\overline{\gamma}^\ast,\overline{\zeta}\circ S^{-1})$ is a {\pams} for the left coideal subalgebra $S^{-1}\circ\pi^\ast:(H/B^+H)^{\ast\,\cop}\rightarrowtail H^{\ast\,\op}$:
$$\begin{array}{ccc}
\xymatrix{
(H/B^+H)^{\ast\,\cop} \ar@<.5ex>[r]^{\;\;\;\;\;\;\;\;S^{-1}\circ\pi^\ast}
& H^{\ast\,\op}
\ar@<.5ex>[l]^{\;\;\;\;\;\;\;\;\overline{\gamma}^\ast}
\ar@<.5ex>[r]^{\;\;\iota^\ast}
& B^{\ast\,\op} \ar@<.5ex>[l]^{S^{-1}\circ\overline{\zeta}^\ast}  }
&\;\;\text{and}\;\;&
\xymatrix{
B^\cop \ar@<.5ex>[r]^{\iota}
& H^\cop
\ar@<.5ex>[l]^{\overline{\zeta}\circ S^{-1}}
\ar@<.5ex>[r]^{\pi\circ S^{-1}\;\;\;\;\;\;}
& (H/B^+H)^\op \ar@<.5ex>[l]^{\overline{\gamma}\;\;\;\;\;\;}  },
\end{array}$$
where the left $H^{\ast\,\op}$-comodule structure on $(H/B^+H)^{\ast\,\cop}$ is
$$(H/B^+H)^{\ast\,\cop}\rightarrow H^{\ast\,\op}\otimes(H/B^+H)^{\ast\,\cop},\;\;
f\mapsto \sum S^{-1}(f_{(2)})\otimes f_{(1)}.$$
\end{itemize}
\end{proposition}

\begin{proof}
\begin{itemize}
\item[(1)]
Firstly, it follows from Corollary \ref{cor:cleftcocleft*}(2) that $\pi^\ast$ is clearly an injection of left $H^{\ast\,\cop}$-comodules and algebras from $(H/B^+H)^{\ast\,\cop}$ to $H^{\ast\,\cop}$. On the other hand, as we could compute that
$$\iota^\ast[S^{-1}(h^\ast k^\ast)]=\iota^\ast[S^{-1}(k^\ast)S^{-1}(h^\ast)]
\overset{(\ref{eqn:iota*})}{=}S^{-1}(k^\ast)\btr \iota^\ast[S^{-1}(h^\ast)]
$$
for any $h^\ast,k^\ast\in H^\ast$,
the surjection $\iota^\ast\circ S^{-1}$ preserves right $H^{\ast\,\cop}$-coactions. It is evidently a coalgebra map from $H^{\ast\,\cop}$ to $B^{\ast\,\op}$ as well.

Moreover, note that $\iota^\ast\circ S^{-1}\circ \pi^\ast$ is trivial according to Proposition \ref{prop:cleftcocleftmore}(4). Thus the image of $\pi^\ast$ in contained by the space
$${H^{\ast\,\cop}}_\mathrm{coinv}:=\left\{h^\ast\in H\mid
\sum h^\ast_{(2)}\otimes \iota^\ast(S^{-1}(h^\ast_{(1)}))=h^\ast\otimes\e\right\}$$
of the coinvariants of $H^{\ast\,\cop}$ as a right $B^{\ast\,\op}$-comodule. In fact they are equal, since we know by \cite[Proposition 3.10(1)]{Mas94(b)} that
$$\dim({H^{\ast\,\cop}}_\mathrm{coinv})
=\frac{\dim(H^{\ast\,\cop})}{\dim(B^{\ast\,\op})}
\overset{(\ref{eqn:dims})}{=}\dim((H/B^+H)^{\ast\,\cop}).$$

Finally, it is direct to show that $\overline{\gamma}^\ast\circ S^{-1}$ and $\overline{\zeta}^\ast$ have convolution inverses $\gamma^\ast\circ S^{-1}$ and $\zeta^\ast$ respectively. The rest of requirements for them are due to Corollary \ref{cor:cleftcocleft*convinv} and Proposition \ref{prop:cleftcocleftmore}(3), as well as the following computations:
\begin{eqnarray*}
\sum\pi^\ast[\overline{\gamma}^\ast(S^{-1}(h^\ast_{(2)}))]
  \overline{\zeta}^\ast[\iota^\ast(S^{-1}(h^\ast_{(1)}))]
&=&
\sum\pi^\ast[\overline{\gamma}^\ast(S^{-1}(h^\ast)_{(1)})]
  \overline{\zeta}^\ast[\iota^\ast(S^{-1}(h^\ast)_{(2)})]  \\
&\overset{(\ref{eqn:convolution*2})}{=}&
S(S^{-1}(h))
~=~ h.
\end{eqnarray*}

\item[(2)]
This is due to direct verifications, which are completely similar in the proof of (1).
\end{itemize}
\end{proof}

On the other hand, let us combine the results of Propositions \ref{prop:BCP-selfdual} and \ref{prop:cleftcocleftmore} to obtain analogous {\pams}s for left coideals subalgebras of $H^\op$ and $H^\cop$ respectively, which would be more useful in later sections:

\begin{corollary}\label{cor:PAMSopcop}
Suppose $(\zeta,\gamma^\ast)$ is a {\pams} for $\iota:B\rightarrowtail H$.
Then:
\begin{itemize}
\item[(1)]
$(\overline{\zeta}\circ S^{-1},\overline{\gamma}^\ast)$ is a {\pams} for the left coideal subalgebra $\iota:B^\op\rightarrowtail H^\op$:
$$\begin{array}{ccc}
\xymatrix{
B^\op \ar@<.5ex>[r]^{\iota}
& H^\op
\ar@<.5ex>[l]^{\overline{\zeta}\circ S^{-1}}
\ar@<.5ex>[r]^{\pi\circ S^{-1}\;\;\;\;\;\;}
& (H/B^+H)^\cop \ar@<.5ex>[l]^{\overline{\gamma}\;\;\;\;\;\;}  }
&\;\;\text{and}\;\;&
\xymatrix{
(H/B^+H)^{\ast\,\op} \ar@<.5ex>[r]^{\;\;\;\;\;\;\;\;S^{-1}\circ\pi^\ast}
& H^{\ast\,\cop}
\ar@<.5ex>[l]^{\;\;\;\;\;\;\;\;\overline{\gamma}^\ast}
\ar@<.5ex>[r]^{\;\;\iota^\ast}
& B^{\ast\,\cop} \ar@<.5ex>[l]^{S^{-1}\circ\overline{\zeta}^\ast}  },
\end{array}$$
where the right $H^\op$-module structure on $(H/B^+H)^\cop$ is
\begin{equation}\label{eqn:rightH^opmod}
(H/B^+H)^\cop\otimes H^\op\rightarrow (H/B^+H)^\cop,\;\;
x\otimes h\mapsto x\btl S^{-1}(h);
\end{equation}

\item[(2)]
$(\overline{\zeta},\overline{\gamma}^\ast\circ S^{-1})$ is a {\pams} for the left coideal subalgebra $S^{-1}\circ\iota:B^\op\rightarrowtail H^\cop$:
$$\begin{array}{ccc}
\xymatrix{
B^\op \ar@<.5ex>[r]^{S^{-1}\circ\iota}
& H^\cop
\ar@<.5ex>[l]^{\overline{\zeta}}
\ar@<.5ex>[r]^{\pi\;\;\;\;\;\;}
& (H/B^+H)^\cop \ar@<.5ex>[l]^{S^{-1}\circ\overline{\gamma}\;\;\;\;\;\;}  }
&\;\;\text{and}\;\;&
\xymatrix{
(H/B^+H)^{\ast\,\op} \ar@<.5ex>[r]^{\;\;\;\;\;\;\pi^\ast}
& H^{\ast\,\op}
\ar@<.5ex>[l]^{\;\;\;\;\;\;\;\;\overline{\gamma}^\ast\circ S^{-1}}
\ar@<.5ex>[r]^{\;\;\iota^\ast\circ S^{-1}}
& B^{\ast\,\cop} \ar@<.5ex>[l]^{\overline{\zeta}^\ast}  },
\end{array}$$
where the left $H^\cop$-comodule structure on $B^\op$ is
\begin{equation}\label{eqn:leftH^copcomod}
B^\op\rightarrow H^\cop\otimes B^\op,\;\;
b\mapsto\sum S^{-1}(b_{(1)})\otimes b_{(2)}.
\end{equation}
\end{itemize}
\end{corollary}

\begin{proof}
These two claims follow from (2) and (1) of Proposition \ref{prop:cleftcocleftmore} respectively.
\end{proof}

In the end of this subsection,
we list the six {\pams}s which are induced by a given one
\begin{equation*}
\begin{array}{ccc}
\xymatrix{
B \ar@<.5ex>[r]^{\iota} & H \ar@<.5ex>@{-->}[l]^{\zeta} \ar@<.5ex>[r]^{\pi}
& C \ar@<.5ex>@{-->}[l]^{\gamma}  }
&\;\;\text{and}\;\;&
\xymatrix{
C^\ast \ar@<.5ex>[r]^{\pi^\ast}
& H^\ast \ar@<.5ex>@{-->}[l]^{\gamma^\ast} \ar@<.5ex>[r]^{\iota^\ast}
& B^\ast \ar@<.5ex>@{-->}[l]^{\zeta^\ast}  },
\end{array}
\end{equation*}
in the following table, as a collection of Propositions \ref{prop:BCP-selfdual}, \ref{prop:cleftcocleftmore} and Corollary \ref{cor:PAMSopcop}, according to identification $C\cong H/B^+H$:

\begin{table}[!ht]
\centering
\caption{Induced {\pams}s}
\begin{tabular}{|c|c|c|}
\hline
Inclusion & Quotient map & Partially admissible mapping system \\
\hline
$B\overset{\iota}\longrightarrow H$ & $H\overset{\pi}\longrightarrow H/B^+H$ & $(\zeta,\gamma^\ast)$  \\
$B^\op\overset{\iota}\longrightarrow H^\op$
  & $H^\op\overset{\pi\circ S^{-1}}\longrightarrow C^\cop$
  & $(\overline{\zeta}\circ S^{-1},\overline{\gamma}^\ast)$  \\
$B^\op\overset{S^{-1}\circ\iota}\longrightarrow H^\cop$
  & $H^\cop\overset{\pi}\longrightarrow C^\cop$
  & $(\overline{\zeta},\overline{\gamma}^\ast\circ S^{-1})$  \\
$C^{\ast\,\biop}\overset{\pi^\ast}\longrightarrow H^{\ast\,\biop}$
  & $H^{\ast\,\biop}\overset{\iota^\ast}\longrightarrow B^{\ast\,\biop}$ & $(\gamma^\ast,\zeta)$  \\
$C^{\ast\,\cop}\overset{\pi^\ast}\longrightarrow H^{\ast\,\cop}$
  & $H^{\ast\,\cop}\overset{\iota^\ast\circ S^{-1}}\longrightarrow B^{\ast\,\op}$
  & $(\overline{\gamma}^\ast\circ S^{-1},\overline{\zeta})$  \\
$C^{\ast\,\cop}\overset{S^{-1}\circ\pi^\ast}\longrightarrow H^{\ast\,\op}$
  & $H^{\ast\,\op}\overset{\iota^\ast}\longrightarrow B^{\ast\,\op}$
  & $(\overline{\gamma}^\ast,\overline{\zeta}\circ S^{-1})$  \\
\hline
\end{tabular}
\end{table}

\section{Left partially dualized quasi-Hopf algebras}\label{section:3}

\subsection{Left partially dualized quasi-Hopf algebras determined by {\pams}s}

Let $(\zeta,\gamma^\ast)$ be a {\pams} for a left coideal subalgebra $\iota:B\rightarrowtail H$ in the sense of Definition \ref{def:PAMS}, and as usual $\pi^\ast:(H/B^+H)^\ast\rightarrowtail H^\ast$ is then regarded as a right coideal subalgebra.

As a result, there is in the literature a structure of algebra $(H/B^+H)^\ast\#B$, namely, the \textit{smash product} of $(H/B^+H)^\ast$ and $B$. Its underlying vector space is $(H/B^+H)^\ast\otimes B$, with unit $\e\#1$ and multiplication given by
\begin{eqnarray}\label{eqn:smashprod}
(f\#b)(g\#c)
&:=& \sum f(b_{(1)}\rightharpoonup g)\#b_{(2)}c
=\sum fg_{(1)}\#(b\leftharpoonup g_{(2)})c  \nonumber   \\
&=& \sum fg_{(1)}\#\la g_{(2)},b_{(1)}\ra b_{(2)}c
\end{eqnarray}
for any $f,g\in(H/B^+H)^\ast$ and $b,c\in B$, where $\leftharpoonup$ and $\rightharpoonup$ are hit actions. Note that here we also write
$$\sum f_{(1)}\otimes f_{(2)}\in (H/B^+H)^\ast\otimes H^\ast
\;\;\;\;\text{and}\;\;\;\;\sum b_{(1)}\otimes b_{(2)}\in H\otimes B$$
for $f\in(H/B^+H)^\ast$ and $b\in B$ as in Subsection \ref{subsection:cocleft}.
This structure is in fact a particular case of \cite[Remark 1.3(b)]{Doi92} on the Hopf algebra $H^\ast$. See also \cite{Tak80} and \cite[Remark 3.1(1)]{CMZ97}.

Our first main goal in this paper is to show that the algebra $(H/B^+H)^\ast\#B$ has a structure of quasi-Hopf algebra:

\begin{theorem}\label{thm:partialdual}
Let $H$ be a finite-dimensional Hopf algebra. Suppose that $B$ is a left coideal subalgebra of $H$ with a {\pams} $(\zeta,\gamma^\ast)$.
Then the smash product $(H/B^+H)^\ast\#B$ with algebra structure (\ref{eqn:smashprod}) is a quasi-Hopf algebra, whose structures are defined as follows: For a linear basis $\{b_i\}$ of $B$ with dual basis $\{b_i^\ast\}$ of $B^\ast$,
\begin{itemize}
\item[(1)]
The ``comultiplication''
$\pd{\Delta}:
(H/B^+H)^\ast\#B\rightarrow\left((H/B^+H)^\ast\#B\right)^{\otimes 2}$
satisfies that for any $f\in(H/B^+H)^\ast$ and $b\in B$,
\begin{equation}\label{eqn:Delta(f)}
\pd{\Delta}(f\#1)
=\sum_{i}\left(f_{(1)}\#b_i\right)\otimes\left(\gamma^\ast[f_{(2)}\zeta^\ast(b_i^\ast)]\#1\right),
\end{equation}
and
\begin{equation}\label{eqn:Delta(b)}
\pd{\Delta}(\e\#b)
=\sum_{i}\left(\e\#\zeta[\gamma(f_i^\ast)b_{(1)}]\right)\otimes\left(f_i\#b_{(2)}\right),
\end{equation}
where $\{f_i\}$ is a linear basis of $(H/B^+H)^\ast$ with dual basis $\{f_i^\ast\}$ of $H/B^+H$.

\item[(2)]
The ``counit''
$\pd{\e}:(H/B^+H)^\ast\#B\rightarrow\k$
satisfies that
\begin{equation}\label{eqn:epsilon}
\pd{\e}(f\#b)=\la f,1\ra\la\e,b\ra
\end{equation}
for any $f\in(H/B^+H)^\ast$ and $b\in B$. Moreover,
The equations
\begin{equation}\label{eqn:counitaxiom}
(\pd{\e}\otimes\id)\circ\pd{\Delta}=\id=(\id\otimes\pd{\e})\circ\pd{\Delta}
\end{equation}
hold on $(H/B^+H)^\ast\#B$.

\item[(3)]
The associator
\begin{equation}\label{eqn:associator}
\pd{\phi}=\sum_{i,j}\left(\e\#b_i\right)
\otimes\left(\e\#b_j\right)
\left(\overline{\gamma}^\ast[S^{-1}(\overline{\zeta}^\ast(b^\ast_i)_{(1)})]\#1\right)
\otimes
\left(\overline{\gamma}^\ast[S^{-1}(\overline{\zeta}^\ast(b^\ast_j)
  \overline{\zeta}^\ast(b^\ast_i)_{(2)})]\#1\right).
\end{equation}

\item[(4)]
Define a linear transformation $\pd{T}:(H/B^+H)^\ast\#B\rightarrow (H/B^+H)^\ast\#B$ such that for any $f\in(H/B^+H)^\ast$ and $b\in B$,
$$\pd{T}(f\#b)
=\sum_i (\e\#b_i)
\left(\overline{\gamma}^\ast[\pi^\ast(f)
  (\iota(b)\rightharpoonup\overline{\zeta}^\ast(b^\ast_i))]\#1\right).$$
Then every antipode $\pd{S}$ with its distinguished elements $\pd{\alpha}$ and $\pd{\beta}$ satisfies the equation
$$\pd{\beta}\pd{S}(-)\pd{\alpha}=\pd{T}(-).$$
In particular, if the element
\begin{equation}\label{eqn:upsilon}
\pd{\upsilon}:=\pd{T}(\e\#1)=
\sum_i (\e\#b_i)\left(\overline{\gamma}^\ast[\overline{\zeta}^\ast(b^\ast_i)]\#1\right)
\in(H/B^+H)^\ast\#B
\end{equation}
is invertible, then there are two antipodes with their distinguished elements as follows:
$$\pd{S}_1:=\pd{T}(-)\pd{\upsilon}^{-1}\;\;\;\;\text{with}\;\;\;\;
\pd{\alpha}_1:=\pd{\upsilon},\;\;\pd{\beta}_1:=\e\#1$$
and
$$\pd{S}_2:=\pd{\upsilon}^{-1}\pd{T}(-)\;\;\;\;\text{with}\;\;\;\;
\pd{\alpha}_2:=\e\#1,\;\;\pd{\beta}_2:=\pd{\upsilon}.$$
\end{itemize}
\end{theorem}

\begin{remark}\label{rmk:phi^-1}
The inverse of the associator $\pd{\phi}$ in $\left((H/B^+H)^\ast\#B\right)^{\otimes 3}$ is
\begin{equation}\label{eqn:phi^-1}
\pd{\phi}^{-1}=\sum_{i,j}\left(\e\#b_i\right)
\otimes\left(\gamma^\ast[\zeta^\ast(b^\ast_i)_{(1)}]\#b_j\right)\otimes
\left(\gamma^\ast[\zeta^\ast(b^\ast_i)_{(2)}\zeta^\ast(b^\ast_j)]\#1\right).
\end{equation}
where $\{b_i\}$ is a basis of $B$ with dual basis $\{b_i^\ast\}$ of $B^\ast$.
\end{remark}

The proofs of Theorem \ref{thm:partialdual} and Remark \ref{rmk:phi^-1} would be provided in the subsequent subsections, which are Subsections \ref{subsection:Thm(1-2)Pf}, \ref{subsection:Thm(3)Pf} and \ref{subsection:Thm(4)Pf}. Before that, we introduce the following definition:

\begin{definition}
The smash product $(H/B^+H)^\ast\#B$ with the structure in Theorem \ref{thm:partialdual} is called the \textit{left partially dualized quasi-Hopf algebra} (or \textit{left partial dual} for simplicity) of the finite-dimensional Hopf algebra $H$ determined by the {\pams} $(\zeta,\gamma^\ast)$.
\end{definition}

Of course, we might find equivalent formulas of the structures described in Theorem \ref{thm:partialdual}, some of which should be more convenient to use in the subsequent subsections:

\begin{remark}\label{rmk:equivDelta}
Let $\{b_i\}$ denote a linear basis of $B$ with dual basis $\{b_i^\ast\}$ of $B^\ast$, and let $\{f_i\}$ denote a linear basis of $(H/B^+H)^\ast$ with dual basis $\{f_i^\ast\}$ of $H/B^+H$ as usual.
\begin{itemize}
\item[(1)]
For all $f\in(H/B^+H)^\ast$ and $b\in B$,
it is supposed to define that
$$\pd{\Delta}(f\#b)=\pd{\Delta}(f\#1)\pd{\Delta}(\e\#b),$$
and then the ``comultiplication'' $\mathbf{\Delta}$ could be concluded as
\begin{equation}\label{eqn:Deltacomplete}
\mathbf{\Delta}(f\#b)
=\sum_{i,j}\left(f_{(1)}\#b_i\zeta[\gamma(f_j^\ast)b_{(1)}]\right)
\otimes\left(\gamma^\ast[f_{(2)}\zeta^\ast(b_i^\ast)]f_j\#b_{(2)}\right).
\end{equation}
\end{itemize}
One could also directly apply the dual bases to verify the following formulas:
\begin{itemize}
\item[(2)]
We could also write
\begin{eqnarray}\label{eqn:Delta(b)2}
\mathbf{\Delta}(f\#1)
&=&
\sum_i\left(f_{(1)}\#\zeta[\gamma(f_i^\ast)\leftharpoonup f_{(2)}]\right)\otimes(f_i\#1),  \nonumber  \\
\mathbf{\Delta}(\e\#b)
&=&
\sum_i(\e\#b_i)\otimes\left(\gamma^\ast[b_{(1)}\rightharpoonup\zeta^\ast(b_i^\ast)]\#b_{(2)}\right)
\end{eqnarray}
both hold for all $f\in(H/B^+H)^\ast$ and $b\in B$.

\item[(3)]
The inverse of the associator $\pd{\phi}$ is
\begin{equation}\label{eqn:phi^-1 2}
\pd{\phi}^{-1}=\sum_{i,j}\left(\e\#\zeta[\gamma(f_i^\ast)\gamma(f_j^\ast)_{(1)}]\right)
\otimes\left(f_i\#\zeta[\gamma(f_j^\ast)_{(2)}]\right)
\otimes(f_j\#1).
\end{equation}

\item[(4)]
The transformation $\pd{T}$ satisfies
$$\pd{T}(f\#b)=
\sum_i \left(\e\#\overline{\zeta}[(
  \overline{\gamma}(f_i^\ast)\leftharpoonup \pi^\ast(f))\iota(b)]\right)
  (f_i\#1)$$
for all $f\in(H/B^+H)^\ast$ and $b\in B$,
and the element $\pd{T}(\e\#1)$ equals
$$\pd{v}=\sum_i \left(\e\#\overline{\zeta}[\overline{\gamma}(f_i^\ast)]\right)
(f_i\#1).$$
\end{itemize}
\end{remark}

Another straightforward fact we would like to mention is:

\begin{remark}
The \textit{Heisenberg double} of $H$ is defined as the smash product algebra $H^\ast\#H$, where $H$ is regarded as a left $H^\ast$-module algebra via the hit action $\rightharpoonup$. It is easy to find that
$$\pi^\ast\otimes\iota:(H/B^+H)^\ast\#B\rightarrowtail H^\ast\#H$$
is an injection of algebras for each left partially dualized quasi-Hopf algebra $(H/B^+H)^\ast\#B$ of $H$.
\end{remark}

\subsection{Further notations for dual basis of the left coideal subalgebra and structures of left partial duals}

Since the structures in Theorem \ref{thm:partialdual} of the left partial dual $(H/B^+H)^\ast\#B$ are all given through dual basis, we porvide some notation in this subsection for convenience of subsequent calculations.

For the purpose, let us consider the $\k$-linear abelian category ${}_{(H/B^+H)^\ast}\mathfrak{M}^{B^\ast}$, whose objects are finite-dimensional vector spaces $V$ with both left $(H/B^+H)^\ast$-module and right $B^\ast$-comodule structures satisfying the compatibility relation:
\begin{equation}\label{eqn:DoiHopfmod}
\sum(fv)_{\la 0\ra}\otimes(fv)_{\la 1\ra}
=\sum f_{(1)}v_{\la 0\ra}\otimes (f_{(2)}\btr v_{\la 1\ra})
\;\;\;\;\;\;(\forall f\in(H/B^+H)^\ast,\;v\in V),
\end{equation}
where $v\mapsto\sum v_{\la 0\ra}\otimes v_{\la 1\ra}$ denotes the right $B^\ast$-comodule structure on $V$ by Sweedler notation with angle brackets.

We remark that the category ${}_{(H/B^+H)^\ast}\mathfrak{M}^{B^\ast}$ is a particular case of the category of \textit{Doi-Hopf modules} (\cite{CMZ97,CMIZ99}) in the literature,
and a similar process as \cite[Remark(1.3)(b)]{Doi92} would follow an isomorphism of categories
\begin{equation}\label{eqn:PsiPhi2}
{}_{(H/B^+H)^\ast}\mathfrak{M}^{B^\ast}\cong 
\Rep((H/B^+H)^\ast\#B),
\end{equation}
where the latter category consists of finite-dimensional left $(H/B^+H)^\ast\#B$-modules.
Specifically, for each object $V\in{}_{(H/B^+H)^\ast}\mathfrak{M}^{B^\ast}$, the left $(H/B^+H)^\ast\#B$-action defined by
\begin{equation}\label{eqn:Bhit}
(f\#b)v
:=\sum fv_{\la0\ra}\langle v_{\la1\ra},b\rangle
\;\;\;\;\;\;(\forall f\in(H/B^+H)^\ast,\; b\in B,\;v\in V)
\end{equation}
makes $V$ an object in $\Rep((H/B^+H)^\ast\#B)$.

Conversely, we regard $(H/B^+H)^\ast\#B$ as an object in the category ${}_{(H/B^+H)^\ast}\mathfrak{M}^{B^\ast}$ of Doi-Hopf modules, and the right $B^\ast$-comodule structure for the unit element $e:=\e\#1$ is supposed to be denoted by
$$\sum e_{\la0\ra}\otimes e_{\la1\ra}\in ((H/B^+H)^\ast\#B)\otimes B^\ast,$$
and accordingly
$\sum e_{\la0\ra}\otimes e_{\la1\ra}\otimes e_{\la2\ra}
:=\sum e_{\la0\ra}\otimes (e_{\la1\ra})_{(1)}\otimes (e_{\la1\ra})_{(2)}$,
etc.
In fact, concerning the isomorphism (\ref{eqn:PsiPhi2}) of categories, we could find that
\begin{equation}\label{eqn:e}
\sum e_{\la0\ra}\otimes e_{\la1\ra}=\sum_i (\e\#b_i)\otimes b_i^\ast,
\end{equation}
where $\{b_i\}$ is a basis of $B$ with dual basis $\{b_i^\ast\}$ of $B^\ast$. Equation (\ref{eqn:e}) would be used frequently in this paper. For example:

\begin{lemma}\label{lem:dualbases(e)}
For any $h^\ast\in H^\ast$,
\begin{eqnarray}\label{eqn:dualbases(e)}
\sum (e_{\la0\ra}\leftharpoonup h^\ast)\otimes e_{\la1\ra}
&=& \sum e_{\la0\ra}\otimes(h^\ast \btr e_{\la1\ra})   \\
&=& \sum e_{\la0\ra}\otimes\iota^\ast[h^\ast\zeta^\ast(e_{\la1\ra})]  \nonumber
\end{eqnarray}
holds, where $\sum e_{\la0\ra}\otimes e_{\la1\ra}$ is also regarded as an element in $B\otimes B^\ast$ without confusions.
\end{lemma}

More generally, we would simply write elements $f\#1$ and $\e\#b$ in the smash product $(H/B^+H)^\ast\#B$ by $f$ and $b$ in a number of cases, where equations
\begin{eqnarray}
bf &=& \sum(b_{(1)}\rightharpoonup f)b_{(2)}=\sum f_{(1)}(b\leftharpoonup f_{(2)})  \label{eqn:bf0}  \\
&=& \sum f_{(1)}\la f_{(2)},b_{(1)}\ra b_{(2)}  \label{eqn:bf}
\end{eqnarray}
could be written as a consequence.

\begin{lemma}
Denote $e=e'=\e\#1$. Then
\begin{itemize}
\item[(1)]
For any $f\in(H/B^+H)^\ast$,
\begin{eqnarray}\label{eqn:e<0>f}
\sum e_{\la0\ra}f\otimes e_{\la1\ra}
&=& \sum f_{(1)}e_{\la0\ra}\otimes(f_{(2)}\btr e_{\la1\ra})  \\
&=& \sum f_{(1)}e_{\la0\ra}\otimes\iota^\ast[f_{(2)}\zeta^\ast(e_{\la1\ra})].  \nonumber
\end{eqnarray}

\item[(2)]
We have
\begin{equation}\label{eqn:e<0>e'<0>}
\sum e_{\la0\ra}e'_{\la0\ra}\otimes e_{\la1\ra}\otimes e'_{\la1\ra}
=\sum e_{\la0\ra}\otimes e_{\la1\ra}\otimes e_{\la2\ra}.
\end{equation}
\end{itemize}
\end{lemma}

\begin{proof}
\begin{itemize}
\item[(1)]
For any $f\in(H/B^+H)^\ast$, it could be calculated that
$$\sum e_{\la0\ra}f\otimes e_{\la1\ra}
\overset{(\ref{eqn:bf0})}{=}
\sum f_{(1)}(e_{\la0\ra}\leftharpoonup f_{(2)})\otimes e_{\la1\ra}
\overset{(\ref{eqn:dualbases(e)})}{=}
\sum f_{(1)}e_{\la0\ra}\otimes (f_{(2)}\btr e_{\la1\ra}).$$

\item[(2)]
Suppose $\{b_i\}$ is a basis of $B$ with dual basis $\{b_i^\ast\}$ of $B^\ast$. It is not hard to verify that
\begin{eqnarray*}
\sum e_{\la0\ra}e'_{\la0\ra}\otimes e_{\la1\ra}\otimes e'_{\la1\ra}
&\overset{(\ref{eqn:e})}{=}&
\sum_{i,j} (\e\#b_ib_j)\otimes b_i^\ast\otimes b_j^\ast
~=~ \sum_{i,j} (\e\#b_j)\otimes b_i^\ast\otimes (b_j^\ast\leftharpoonup b_i)  \\
&=&
\sum_{i,j} (\e\#b_j)\otimes b_i^\ast\otimes \la {b_j^\ast}_{(1)},b_i\ra {b_j^\ast}_{(2)}  \\
&=& \sum_{j} (\e\#b_j)\otimes {b_j^\ast}_{(1)} \otimes {b_j^\ast}_{(2)}  \\
&\overset{(\ref{eqn:e})}{=}&
\sum e_{\la0\ra}\otimes e_{\la1\ra}\otimes e_{\la2\ra}.
\end{eqnarray*}
\end{itemize}
\end{proof}

\subsection{Proofs of Theorem \ref{thm:partialdual} (1) and (2) with additional formulas on {\pams}s}\label{subsection:Thm(1-2)Pf}

This subsection is devoted to showing
the compatibility of the operations in (1) and (2) of Theorem \ref{thm:partialdual}. Before that, in order to simplify the proofs, we would like to introduce more formulas on {\pams}s by making full use of Definition \ref{def:PAMS}(6).

Let $H$ be a finite-dimensional Hopf algebra. Suppose that $B$ is a left coideal subalgebra of $H$ with a {\pams} $(\zeta,\gamma^\ast)$. Then:

\begin{lemma}
For any $h^\ast\in H^\ast$ and $b^\ast\in B^\ast$,
\begin{equation}\label{eqn:zeta*(pi*gamma*zeta*)}
\sum \zeta^\ast(h^\ast_{(1)}\btr b^\ast_{(1)})
    \pi^\ast\left(\gamma^\ast[h^\ast_{(2)}\zeta^\ast(b^\ast_{(2)})_{(1)}]\right)
  \otimes \zeta^\ast(b^\ast_{(2)})_{(2)}
=\sum h^\ast\zeta^\ast(b^\ast)_{(1)}\otimes\zeta^\ast(b^\ast)_{(2)}.
\end{equation}
In particular:
\begin{itemize}
\item[(1)]
For any $b^\ast\in B^\ast$,
\begin{equation}\label{eqn:zeta*(pi*gamma*zeta*)1}
\sum \zeta^\ast(b^\ast_{(1)})\pi^\ast\left(\gamma^\ast[\zeta^\ast(b^\ast_{(2)})_{(1)}]\right)
\otimes\zeta^\ast(b^\ast_{(2)})_{(2)}
=\sum \zeta^\ast(b^\ast)_{(1)}\otimes\zeta^\ast(b^\ast)_{(2)};
\end{equation}

\item[(2)]
For any $h^\ast\in H^\ast$ and $b^\ast\in B^\ast$,
\begin{equation}\label{eqn:zeta*(pi*gamma*zeta*)2}
\sum \zeta^\ast(h^\ast_{(1)}\btr b^\ast_{(1)})
  \pi^\ast\left(\gamma^\ast[h^\ast_{(2)}\zeta^\ast(b^\ast_{(2)})]\right)
=h^\ast\zeta^\ast(b^\ast).
\end{equation}
\end{itemize}
\end{lemma}

\begin{proof}
Since $\zeta^\ast:B^\ast\rightarrow H^\ast$ is a left $B^\ast$-comodule map via $\iota^\ast$, we could know that
$$\sum b^\ast_{(1)}\otimes\zeta^\ast(b^\ast_{(2)})_{(1)}\otimes\zeta^\ast(b^\ast_{(2)})_{(2)}
\overset{(\ref{eqn:zeta*})}{=}
\sum \iota^\ast[\zeta^\ast(b^\ast)_{(1)}]\otimes\zeta^\ast(b^\ast)_{(2)}
  \otimes\zeta^\ast(b^\ast)_{(3)}.$$
Therefore,
\begin{eqnarray*}
&&
\sum \zeta^\ast(h^\ast_{(1)}\btr b^\ast_{(1)})
    \pi^\ast\left(\gamma^\ast[h^\ast_{(2)}\zeta^\ast(b^\ast_{(2)})_{(1)}]\right)
  \otimes \zeta^\ast(b^\ast_{(2)})_{(2)}  \\
&=&
\sum \zeta^\ast\left(h^\ast_{(1)}\btr \iota^\ast[\zeta^\ast(b^\ast)_{(1)}]\right)
    \pi^\ast\left(\gamma^\ast[h^\ast_{(2)}\zeta^\ast(b^\ast)_{(2)}]\right)
  \otimes \zeta^\ast(b^\ast)_{(3)}    \\
&\overset{(\ref{eqn:iota*})}{=}&
\sum \zeta^\ast\left(\iota^\ast[h^\ast_{(1)}\zeta^\ast(b^\ast)_{(1)}]\right)
    \pi^\ast\left(\gamma^\ast[h^\ast_{(2)}\zeta^\ast(b^\ast)_{(2)}]\right)
  \otimes \zeta^\ast(b^\ast)_{(3)}
  \\
&\overset{(\ref{eqn:convolution*})}{=}&
\sum h^\ast\zeta^\ast(b^\ast)_{(1)}\otimes\zeta^\ast(b^\ast)_{(2)},
\end{eqnarray*}
which is exactly Equation (\ref{eqn:zeta*(pi*gamma*zeta*)}).

\begin{itemize}
\item[(1)]
This is the case when we choose $h^\ast=\e$ in (\ref{eqn:zeta*(pi*gamma*zeta*)}).

\item[(2)]
This is the image of Equation (\ref{eqn:zeta*(pi*gamma*zeta*)}) under the linear map
$\id_{H^\ast}\otimes 1$.
\end{itemize}
\end{proof}

\begin{corollary}
For any $h^\ast,k^\ast\in H^\ast$ and $b^\ast\in B^\ast$,
\begin{equation}\label{eqn:gamma*gamma*1}
\sum \gamma^\ast[h^\ast\zeta^\ast(b^\ast_{(1)})]\gamma^\ast[\zeta^\ast(b^\ast_{(2)})_{(1)}]
\otimes\zeta^\ast(b^\ast_{(2)})_{(2)}
=\sum \gamma^\ast[h^\ast\zeta^\ast(b^\ast)_{(1)}]\otimes\zeta^\ast(b^\ast)_{(2)}
\end{equation}
and
\begin{equation}\label{eqn:gamma*gamma*2}
\sum \gamma^\ast[k^\ast\zeta^\ast(h^\ast_{(1)}\btr b^\ast_{(1)})]
\gamma^\ast[h^\ast_{(2)}\zeta^\ast(b^\ast_{(2)})]
=\gamma^\ast[k^\ast h^\ast\zeta^\ast(b^\ast)]
\end{equation}
hold.
\end{corollary}

\begin{proof}
As $\gamma^\ast:H^\ast\rightarrow(H/B^+H)^\ast$ is a right $(H/B^+H)^\ast$-module map via $\pi^\ast$, one could verify Equations (\ref{eqn:gamma*gamma*1}) and (\ref{eqn:gamma*gamma*2}) according to Equations (\ref{eqn:zeta*(pi*gamma*zeta*)1}) and (\ref{eqn:zeta*(pi*gamma*zeta*)2}). Specifically,
\begin{eqnarray*}
&& \sum \gamma^\ast[h^\ast\zeta^\ast(b^\ast_{(1)})]\gamma^\ast[\zeta^\ast(b^\ast_{(2)})_{(1)}]
\otimes\zeta^\ast(b^\ast_{(2)})_{(2)}  \\
&\overset{(\ref{eqn:gamma*})}{=}&
\sum \gamma^\ast\left[h^\ast\zeta^\ast(b^\ast_{(1)})
\pi^\ast\left(\gamma^\ast[\zeta^\ast(b^\ast_{(2)})_{(1)}]\right)\right]
\otimes\zeta^\ast(b^\ast_{(2)})_{(2)}  \\
&\overset{(\ref{eqn:zeta*(pi*gamma*zeta*)1})}{=}&
\sum \gamma^\ast[h^\ast\zeta^\ast(b^\ast)_{(1)}]\otimes\zeta^\ast(b^\ast)_{(2)},
\end{eqnarray*}
and
\begin{eqnarray*}
\sum \gamma^\ast[k^\ast\zeta^\ast(h^\ast_{(1)}\btr b^\ast_{(1)})]
\gamma^\ast[h^\ast_{(2)}\zeta^\ast(b^\ast_{(2)})]
&\overset{(\ref{eqn:gamma*})}{=}&
\sum \gamma^\ast\left[k^\ast\zeta^\ast(h^\ast_{(1)}\btr b^\ast_{(1)})
\pi^\ast\left(\gamma^\ast[h^\ast_{(2)}\zeta^\ast(b^\ast_{(2)})]\right)\right]  \\
&\overset{(\ref{eqn:zeta*(pi*gamma*zeta*)2})}{=}&
\gamma^\ast[k^\ast h^\ast\zeta^\ast(b^\ast)].
\end{eqnarray*}
\end{proof}

Now we could verify that the operations $\pd{\Delta}$ and $\pd{\e}$ in Theorem \ref{thm:partialdual} (1) and (2) are both algebra maps, and that (\ref{eqn:counitaxiom}) holds.
The proof is given by direct computations, with the help of formulas above as well as those in the previous subsection.

~

\begin{proof}
[Proofs of the compatibility for the structures in Theorem \ref{thm:partialdual} (1) and (2)]~

\begin{itemize}
\item[(1)]
Our goal is to prove that $\pd{\Delta}$ is a map of algebras.
We know at first by Proposition \ref{prop:PAMS-comptrival}(2) that $\gamma^\ast\circ\zeta^\ast$ is trivial, and hence Formula (\ref{eqn:Delta(f)}) provides that
\begin{eqnarray*}
\mathbf{\Delta}(\e\#1)
&\overset{(\ref{eqn:Delta(f)})}{=}&
\sum_{i}\left(\e\#b_i\right)\otimes\left(\gamma^\ast[\zeta^\ast(b_i^\ast)]\#1\right)
~\overset{(\ref{eqn:gamma*zeta*})}{=}~
\sum_{i}\left(\e\#b_i\right)\otimes\left(\langle b_i^\ast,1\rangle\e\#1\right)  \\
&=& (\e\#1)\otimes(\e\#1)
\end{eqnarray*}
holds.
Similarly, Formula (\ref{eqn:Delta(b)}) also follows $\pd{\Delta}(\e\#1)=(\e\#1)\otimes(\e\#1)$ because $\zeta\circ\gamma$ is trivial.

Secondly, as mentioned in Remark \ref{rmk:equivDelta}(1), it is assumed that
$$\pd{\Delta}(f\#b)=\pd{\Delta}(f\#1)\pd{\Delta}(\e\#b)$$
for all $f\in(H/B^+H)^\ast$ and $b\in B$, and we try to show
$\pd{\Delta}(f\#1)\pd{\Delta}(g\#1)=\pd{\Delta}(fg\#1)$ for any $f,g\in(H/B^+H)^\ast$.
In fact, recall that we could write
\begin{equation}\label{eqn:Delta(f):e}
\pd{\Delta}(f)=\sum f_{(1)}e_{\la0\ra}\otimes \gamma^\ast[f_{(2)}\zeta^\ast(e_{\la1\ra})]
\end{equation}
with the notation of Equation (\ref{eqn:e}).
Therefore, let $e=e'=\e\#1$ and calculate that
\begin{eqnarray*}
\pd{\Delta}(f)\pd{\Delta}(g)
&\overset{(\ref{eqn:Delta(f):e})}{=}&
\sum \left(f_{(1)}e_{\la0\ra}\otimes\gamma^\ast[f_{(2)}\zeta^\ast(e_{\la1\ra})]\right)
  \left(g_{(1)}e'_{\la0\ra}\otimes\gamma^\ast[g_{(2)}\zeta^\ast(e'_{\la1\ra})]\right)  \\
&=&
\sum f_{(1)}e_{\la0\ra}g_{(1)}e'_{\la0\ra}
  \otimes\gamma^\ast[f_{(2)}\zeta^\ast(e_{\la1\ra})]\gamma^\ast[g_{(2)}\zeta^\ast(e'_{\la1\ra})]  \\
&\overset{(\ref{eqn:e<0>f})}{=}&
\sum f_{(1)}g_{(1)}e_{\la0\ra}e'_{\la0\ra}
  \otimes\gamma^\ast[f_{(2)}\zeta^\ast(g_{(2)}\btr e_{\la1\ra})]
         \gamma^\ast[g_{(3)}\zeta^\ast(e'_{\la1\ra})]  \\
&\overset{(\ref{eqn:e<0>e'<0>})}{=}&
\sum f_{(1)}g_{(1)}e_{\la0\ra}
  \otimes\gamma^\ast[f_{(2)}\zeta^\ast(g_{(2)}\btr e_{\la1\ra})]
         \gamma^\ast[g_{(3)}\zeta^\ast(e_{\la2\ra})]  \\
&\overset{(\ref{eqn:gamma*gamma*2})}{=}&
\sum f_{(1)}g_{(1)}e_{\la0\ra}\otimes\gamma^\ast[f_{(2)}g_{(2)}\zeta^\ast(e_{\la1\ra})]  \\
&\overset{(\ref{eqn:Delta(f):e})}{=}&
\pd{\Delta}(fg),
\end{eqnarray*}
and one could also verify in a similar way that
$\mathbf{\Delta}(\e\#b)\mathbf{\Delta}(\e\#c)=\mathbf{\Delta}(\e\#bc)$
holds for any $b,c\in B$.

It might be slightly complicated to show that
\begin{equation}\label{eqn:Delta(bf)}
\mathbf{\Delta}((\e\#b)(f\#1))=\mathbf{\Delta}(\e\#b)\mathbf{\Delta}(f\#1)
\end{equation}
holds for any $b\in B$ and $f\in(H/B^+H)^\ast$.
Note by Equation (\ref{eqn:Delta(b)2}) that
\begin{equation}\label{eqn:Delta(b):e}
\pd{\Delta}(b)=\sum e_{\la0\ra}
\otimes\gamma^\ast[b_{(1)}\rightharpoonup\zeta^\ast(e_{\la1\ra})]b_{(2)}
\end{equation}
could be written with the notation of Equation (\ref{eqn:e}).

Now we let $e=e'=\e\#1$ and make following computations:
\begin{eqnarray*}
&& \pd{\Delta}(b)\pd{\Delta}(f)  \\
&\overset{(\ref{eqn:Delta(b):e}),\;(\ref{eqn:Delta(f):e})}{=}&
\sum \left(e_{\la0\ra}\otimes\gamma^\ast[b_{(1)}\rightharpoonup\zeta^\ast(e_{\la1\ra})]b_{(2)}\right)
  \left(f_{(1)}e'_{\la0\ra}\otimes\gamma^\ast[f_{(2)}\zeta^\ast(e'_{\la1\ra})]\right)  \\
&=&
\sum e_{\la0\ra}f_{(1)}e'_{\la0\ra}
  \otimes\gamma^\ast[b_{(1)}\rightharpoonup\zeta^\ast(e_{\la1\ra})]b_{(2)}
  \gamma^\ast[f_{(2)}\zeta^\ast(e'_{\la1\ra})]  \\
&\overset{(\ref{eqn:bf0})}{=}&
\sum e_{\la0\ra}f_{(1)}e'_{\la0\ra}
  \otimes\gamma^\ast[b_{(1)}\rightharpoonup\zeta^\ast(e_{\la1\ra})]
  \left(b_{(2)}\rightharpoonup\gamma^\ast[f_{(2)}\zeta^\ast(e'_{\la1\ra})]\right)b_{(3)}  \\
&=&
\sum e_{\la0\ra}f_{(1)}e'_{\la0\ra}
  \otimes\gamma^\ast[\zeta^\ast(e_{\la1\ra})_{(1)}]\gamma^\ast[f_{(2)}\zeta^\ast(e'_{\la1\ra})]_{(1)}
\\  &&  \;\;\;\;\;\;\;\;\;\;\;\;\;\;\;\;\;\;\;\;\;\;\;\;\;\;\; \left\la\zeta^\ast(e_{\la2\ra})_{(2)}\gamma^\ast[f_{(2)}\zeta^\ast(e'_{\la1\ra})]_{(2)},b_{(1)}\right\ra
  b_{(2)}  \\
&\overset{(\ref{eqn:gamma*})}{=}&
\sum e_{\la0\ra}f_{(1)}e'_{\la0\ra}
  \otimes\gamma^\ast\left[\zeta^\ast(e_{\la1\ra})_{(1)}
    \pi^\ast\left(\gamma^\ast[f_{(2)}\zeta^\ast(e'_{\la1\ra})]_{(1)}\right)\right]
\\  &&  \;\;\;\;\;\;\;\;\;\;\;\;\;\;\;\;\;\;\;\;\;\;\;\;\;\;\;
  \left\la\zeta^\ast(e_{\la2\ra})_{(2)}\gamma^\ast[f_{(2)}\zeta^\ast(e'_{\la1\ra})]_{(2)},b_{(1)}\right\ra
  b_{(2)}  \\
&\overset{(\ref{eqn:pi*})}{=}&
\sum e_{\la0\ra}f_{(1)}e'_{\la0\ra}
  \otimes\gamma^\ast\left[\zeta^\ast(e_{\la1\ra})_{(1)}
    \pi^\ast(\gamma^\ast[f_{(2)}\zeta^\ast(e'_{\la1\ra})])_{(1)}\right]
\\  &&  \;\;\;\;\;\;\;\;\;\;\;\;\;\;\;\;\;\;\;\;\;\;\;\;\;\;\;
  \left\la\zeta^\ast(e_{\la2\ra})_{(2)}
  \pi^\ast(\gamma^\ast[f_{(2)}\zeta^\ast(e'_{\la1\ra})])_{(2)},b_{(1)}\right\ra b_{(2)}  \\
&=&
\sum e_{\la0\ra}f_{(1)}e'_{\la0\ra}
  \otimes\gamma^\ast\left[b_{(1)}\rightharpoonup\zeta^\ast(e_{\la1\ra})
    \pi^\ast(\gamma^\ast[f_{(2)}\zeta^\ast(e'_{\la1\ra})])\right]b_{(2)}  \\
&\overset{(\ref{eqn:e<0>f})}{=}&
\sum f_{(1)}e_{\la0\ra}e'_{\la0\ra}
  \otimes\gamma^\ast\left[b_{(1)}\rightharpoonup\zeta^\ast(f_{(2)}\btr e_{\la1\ra})
    \pi^\ast(\gamma^\ast[f_{(3)}\zeta^\ast(e'_{\la1\ra})])\right]b_{(2)}  \\
&\overset{(\ref{eqn:e<0>e'<0>})}{=}&
\sum f_{(1)}e_{\la0\ra}
  \otimes\gamma^\ast\left[b_{(1)}\rightharpoonup\zeta^\ast(f_{(2)}\btr e_{\la1\ra})
    \pi^\ast(\gamma^\ast[f_{(3)}\zeta^\ast(e_{\la2\ra})])\right]b_{(2)}  \\
&\overset{(\ref{eqn:zeta*(pi*gamma*zeta*)2})}{=}&
\sum f_{(1)}e_{\la0\ra}
  \otimes\gamma^\ast[b_{(1)}\rightharpoonup f_{(2)}\zeta^\ast(e_{\la1\ra})]b_{(2)}.
\end{eqnarray*}
On the other hand,
\begin{eqnarray*}
&&  \pd{\Delta}(bf)  \\
&\overset{(\ref{eqn:bf})}{=}&
\sum\pd{\Delta}(f_{(1)}b_{(2)})\la f_{(2)},b_{(1)}\ra
~=~
\sum\pd{\Delta}(f_{(1)})\pd{\Delta}(b_{(2)})\la f_{(2)},b_{(1)}\ra  \\
&\overset{(\ref{eqn:Delta(b):e}),\;(\ref{eqn:Delta(f):e})}{=}&
\sum \left(f_{(1)}e_{\la0\ra}\otimes\gamma^\ast[f_{(2)}\zeta^\ast(e_{\la1\ra})]\right)
  \left(e'_{\la0\ra}\otimes\gamma^\ast[b_{(2)}\rightharpoonup\zeta^\ast(e'_{\la1\ra})]b_{(3)}\right)
  \la f_{(3)},b_{(1)}\ra  \\
&=&
\sum f_{(1)}e_{\la0\ra}e'_{\la0\ra}\otimes
  \gamma^\ast[f_{(2)}\zeta^\ast(e_{\la1\ra})]\gamma^\ast[b_{(2)}\rightharpoonup\zeta^\ast(e'_{\la1\ra})]b_{(3)}
  \la f_{(3)},b_{(1)}\ra  \\
&=&
\sum f_{(1)}e_{\la0\ra}e'_{\la0\ra}\otimes
  \gamma^\ast[f_{(2)}\zeta^\ast(e_{\la1\ra})]\gamma^\ast[\zeta^\ast(e'_{\la1\ra})_{(1)}]b_{(2)}
  \left\la f_{(3)}\zeta^\ast(e'_{\la1\ra})_{(2)},b_{(1)}\right\ra  \\
&\overset{(\ref{eqn:e<0>e'<0>})}{=}&
\sum f_{(1)}e_{\la0\ra}\otimes
  \gamma^\ast[f_{(2)}\zeta^\ast(e_{\la1\ra})]\gamma^\ast[\zeta^\ast(e_{\la2\ra})_{(1)}]b_{(2)}
  \left\la f_{(3)}\zeta^\ast(e_{\la2\ra})_{(2)},b_{(1)}\right\ra  \\
&\overset{(\ref{eqn:gamma*gamma*1})}{=}&
\sum f_{(1)}e_{\la0\ra}\otimes
  \gamma^\ast[f_{(2)}\zeta^\ast(e_{\la1\ra})_{(1)}]b_{(2)}
  \left\la f_{(3)}\zeta^\ast(e_{\la1\ra})_{(2)},b_{(1)}\right\ra  \\
&=&
\sum f_{(1)}e_{\la0\ra}\otimes
  \gamma^\ast[b_{(1)}\rightharpoonup f_{(2)}\zeta^\ast(e_{\la1\ra})]b_{(2)},
\end{eqnarray*}
which coincides with $\pd{\Delta}(b)\pd{\Delta}(f)$ computed above.
Finally as a conclusion, $\pd{\Delta}$ is an algebra map.

\item[(2)]
Let us verify that $\pd{\e}$ is also a map of algebras. Indeed,
for any $f,g\in(H/B^+H)^\ast$ and any $b,c\in B$,
\begin{eqnarray*}
\pd{\e}\left((f\#b)(g\#c)\right)
&\overset{(\ref{eqn:smashprod})}{=}&
\sum \pd{\e}\left(f(b_{(1)}\rightharpoonup g)\#b_{(2)}c\right)  \\
&\overset{(\ref{eqn:epsilon})}{=}&
\sum \la f(b_{(1)}\rightharpoonup g),1\ra\la\e,b_{(2)}c\ra  \\
&=& \la f,1\ra\la\iota(b)\rightharpoonup g,1\ra\la\e,c\ra
~=~ \la f,1\ra\la\pi^\ast(g),\iota(b)\ra\la\e,c\ra  \\
&=& \la f,1\ra\la g,\pi[\iota(b)]\ra\la\e,c\ra
~\overset{(\ref{eqn:piiota})}{=}~
\la f,1\ra\la\e,b\ra\la g,1\ra\la\e,c\ra  \\
&\overset{(\ref{eqn:epsilon})}{=}&
\pd{\e}(f\#b)\pd{\e}(g\#c)
\end{eqnarray*}
hold.

The equations (\ref{eqn:counitaxiom})
could be verified directly as well, and here we calculate with Formula (\ref{eqn:Deltacomplete}) on $\pd{\Delta}$: For any $f\in(H/B^+H)^\ast$ and any $b\in B$,
\begin{eqnarray*}
(\pd{\e}\otimes\id)\circ\pd{\Delta}(f\#b)
&\overset{(\ref{eqn:Deltacomplete})}{=}&
\sum_{i,j}\left\la\pd{\e},f_{(1)}\#b_i\zeta[\gamma(f_j^\ast)b_{(1)}]\right\ra
  \left(\gamma^\ast[f_{(2)}\zeta^\ast(b_i^\ast)]f_j\#b_{(2)}\right)  \\
&\overset{(\ref{eqn:epsilon})}{=}&
\sum_{i,j}\la f_{(1)},1\ra\la\e,b_i\ra\la\e,\zeta[\gamma(f_j^\ast)b_{(1)}]\ra
  \left(\gamma^\ast[f_{(2)}\zeta^\ast(b_i^\ast)]f_j\#b_{(2)}\right)  \\
&=&
\sum_{j}\la\e,\zeta[\gamma(f_j^\ast)b_{(1)}]\ra
  \left(\gamma^\ast[\pi^\ast(f)\zeta^\ast(\e)]f_j\#b_{(2)}\right)  \\
&\overset{(\ref{eqn:zetabiunitary})}{=}&
\sum_{j}\la\e,\gamma(f_j^\ast)b_{(1)}\ra
  \left(\gamma^\ast[\pi^\ast(f)]f_j\#b_{(2)}\right)  \\
&\overset{(\ref{eqn:gammabiunitary})}{=}&
\sum_{j}\la\e,f_j^\ast\ra\la\e,b_{(1)}\ra
  \left(\gamma^\ast[\pi^\ast(f)]f_j\#b_{(2)}\right)  \\
&=&
\gamma^\ast[\pi^\ast(f)]\#b
~\overset{(\ref{eqn:gamma*pi*})}{=}~
f\#b,
\end{eqnarray*}
where $\{b_i\}$ is a basis of $B$ with dual basis $\{b_i^\ast\}$ of $\{B^\ast$, and $\{f_i\}$ is a basis of $(H/B^+H)^\ast$ with dual basis $\{f_i^\ast\}$ of $H/B^+H$.

The other equation
$(\id\otimes\pd{\e})\circ\pd{\Delta}=\id$ holds due to a similar argument.
\end{itemize}
\end{proof}

\section{Dual tensor categories to finite-dimensional Hopf algebras}\label{section:4}

We refer to \cite[Chapters 1 to 4]{EGNO15} for the definitions and basic properties about tensor categories. For the purpose, $\k$ is assumed to be algebraically closed in this section.

\subsection{Module categories over finite-dimensional Hopf algebras, and dual tensor categories}

The definition of (left) module categories over monoidal categories could be found in \cite[Sections 7.1 to 7.3]{EGNO15}.
In this subsection, we recall some elementary properties of module categories over a finite tensor category $\C$ over $\k$, as well as the corresponding dual categories in the literature.

\begin{lemma}(\cite[Lemma 3.4]{EO04})
Suppose $\M$ is an exact module category over a finite tensor category $\C$. Then $\M$ is finite as a $\k$-linear abelian category.
\end{lemma}

Suppose the left $\C$-module category $\M$ is exact.
We define the \textit{dual category} of $\C$ with respect to $\M$ as
\begin{equation*}
\C_\M^\ast:=\Rex_\C(\M)^\mathrm{rev},
\end{equation*}
the category of $\k$-linear right exact $\C$-module endofunctors of $\M$.
Note that the tensor product on $\C_\M^\ast$ in this paper is chosen to be \textit{opposite} to the composition of $\C$-module endofunctors. In other words, it has the reserve tensor products with the notion ($\C_\M^\ast=\Rex_\C(\M)$) defined in \cite{EO04} and \cite{EGNO15}, etc. However, it is clear that both of them are multitensor categories with almost the same properties.

\begin{remark}
In fact, since $\M$ is an exact $\C$-module category, the dual category $\C_\M^\ast$ consists of all the additive $\C$-module endofunctors (or equivalently, all the exact $\C$-module endofunctors) of $\M$ according to \cite[Proposition 3.11]{EO04}.
\end{remark}

In this paper, we only pay attention to dealing with exact $\C$-module categories $\M$ which are furthermore indecomposable. It is known that $\M$ is equivalent to $\C_{A'}$, the category of right $A'$-modules in $\C$, for some algebra $A'\in\C$. Then the dual category $\C_\M^\ast$ could be also described in this situation:

\begin{lemma}(\cite[Theorem 3.17 and Section 3.3]{EO04})\label{lem:mod&dualcat}
Suppose $\M$ is an indecomposable exact left module category over a finite tensor category $\C$. If $Y\in\M$ generates $\M$, and let $A':=\intHom(Y,Y)\in\C$ be the algebra defined by the internal Hom functor, then:
\begin{itemize}
\item[(1)]
There is an equivalence of $\C$-module categories:
$$\M\approx \C_{A'},\;\;\;\;Z\mapsto\intHom(Y,Z);$$
\item[(2)]
$\C_\M^\ast$ is also a finite tensor category, which is equivalent to the category ${}_{A'}\C_{A'}$ consisting of $A'$-$A'$-bimodules in $\C$ via the functor:
\begin{equation}\label{eqn:dualcatequiv}
{}_{A'}\C_{A'}\approx \C_\M^\ast,\;\;\;\;M\mapsto -\otimes_{A'}M,
\end{equation}
where $\M$ is identified with $\C_{A'}$ according to the equivalence in (1).
\end{itemize}
\end{lemma}

Now let $H$ be a finite-dimensional Hopf algebra, and we focus on the case when $\C=\Rep(H)$, the finite tensor category of finite-dimensional left $H$-modules.

Firstly it is known in \cite{AM07} that indecomposable exact left $\Rep(H)$-module categories are classified by left $H$-comodule algebras $B$ which are $H$-simple from the right and with trivial $H$-coinvariants, up to equivariant Morita equivalences. This result is recalled for our later use:

\begin{lemma}(\cite[Theorem 3.3]{AM07})\label{lem:modcat}
For any indecomposable exact left $\Rep(H)$-module category $\M$, there exists a left $H$-comodule algebra $B$ satisfying:
\begin{itemize}
\item[(1)] $B$ has no non-trivial right ideal which is also an $H$-comodule, and
\item[(2)] $B$ has trivial $H$-coinvariants,
\end{itemize}
such that $\M\approx\Rep(B)$ as left $\Rep(H)$-module categories.
\end{lemma}

Conversely, each left $H$-comodule algebra $B$ with properties (1) and (2) of Lemma \ref{lem:modcat} would certainly determine an indecomposable exact left $\Rep(H)$-module category $\Rep(B)$. A particular example is when $B$ is a left coideal subalgebra of $H$ according to \cite[Theorem 6.1(2)]{Skr07} or \cite[Proposition 1.6]{AM07}. However, we do not know the answer of the following question:

\begin{question}
When would an indecomposable exact left $\Rep(H)$-module category $\M$ be equivalent to $\Rep(B)$ for some left coideal subalgebra $B$ of $H$?
\end{question}


Anyway in this paper, we focus on those left $\Rep(H)$-module categories $\M$ which are equivalent to $\Rep(B)$ for left coideal subalgebras $B$, as well as the corresponding dual tensor categories $\Rep(H)_{\Rep(B)}^\ast$.
Recall that the left $\Rep(H)$-module structure on $\Rep(B)$ is introduced in \cite[Section 1.5]{AM07} as follows:
\begin{equation*}
\Rep(H)\times\Rep(B)\rightarrow\Rep(B),\;\;\;\;(X,Y)\mapsto X\otimes_\k Y,
\end{equation*}
where the left $B$-actions on $X\otimes_\k Y$ is diagonal via
$$b\cdot(x\otimes m):=\sum b_{(1)}x\otimes b_{(2)}y$$
for any $b\in B$, $x\in X$ and $y\in Y$.

The properties in Lemma \ref{lem:mod&dualcat} when $\M=\Rep(B)$ should be mentioned in particular:

\begin{lemma}\label{lem:mod&dualcatoverH}
Let $B$ be a left coideal subalgebra of a finite-dimensional Hopf algebra $H$. Then:
\begin{itemize}
\item[(1)]
$\Rep(B)$ is an indecomposable exact left $\Rep(H)$-module category;
\item[(2)]
There is an equivalence of left $\Rep(H)$-module categories:
\begin{equation*}\label{eqn:modcatequiv}
\Rep(B)\approx \Rep(H)_{(H/B^+H)^\ast},\;\;\;\;
Y\mapsto \underline{\mathrm{Hom}}(\k,Y),
\end{equation*}
where $(H/B^+H)^\ast$ is a left $H$-module algebra induced by the structure of the right coideal subalgebra $\pi^\ast:(H/B^+H)^\ast\rightarrowtail H^\ast$ defined in (\ref{eqn:coidealsubalgs}).
\end{itemize}
\end{lemma}

\begin{proof}
\begin{itemize}
\item[(1)]
We have mentioned this after Lemma \ref{lem:modcat}, and it is also a combination of \cite[Propositions 1.6, 1.18 and 1.20(ii)]{AM07}.

\item[(2)]
Clearly, the $\Rep(H)$-module category $\Rep(B)$ is generated by the object $\k$ (with trivial $B$-action). It follows from \cite[Corollary 3.2]{AM07} that the left $\Rep(H)$-module category $\Rep(B)$ is equivalent to the category ${}_H\mathfrak{M}_{A'}$ of finite-dimensional $H$-$A'$-bimodules, where $A':=(H/S^{-1}(B^+)H)^\ast$ is a left $H$-module algebra due to \cite[Example 2.19]{AM07}.

However, we could know that $S^{-1}(B^+)H=B^+H$ holds by applying \cite[Lemma 3.1]{Kop93} to the Hopf algebra $H^\cop$ with antipode $S^{-1}$. Consequently, the left $H$-module algebra
$$A':=(H/S^{-1}(B^+)H)^\ast=(H/B^+H)^\ast$$
has the structure induced by the injection $\pi^\ast:(H/B^+H)^\ast\rightarrowtail H^\ast$ of right $H^\ast$-comodule algebras according to \cite[Example 2.19]{AM07}, as $\pi$ is the quotient map to $H/S^{-1}(B^+)H=H/B^+H$.

Finally, it could be obtained by the isomorphism ${}_H\mathfrak{M}_{A'}\cong\Rep(H)_{A'}$ that there is an equivalence $\Rep(B)\approx \Rep(H)_{A'}$ sending every object $Y$ to $\intHom(\k,Y)$ as desired.
\end{itemize}
\end{proof}

It follows from Lemmas \ref{lem:mod&dualcat} and \ref{lem:mod&dualcatoverH} that the corresponding dual tensor category $\Rep(H)_{\Rep(B)}^\ast$ could be identified with
${}_{(H/B^+H)^\ast}\Rep(H)_{(H/B^+H)^\ast}$,
the category of $(H/B^+H)^\ast$-$(H/B^+H)^\ast$-bimodules in $\Rep(H)$,
according to the equivalence (\ref{eqn:dualcatequiv}). More additional descriptions would be introduced in the next subsection.

\subsection{Identifications of the dual tensor category to finite-dimensional Hopf algebras}\label{subsection:relativeHopfmods}

We continue describing the dual tensor category $\Rep(H)_{\Rep(B)}^\ast$, where $B$ is a left coideal subalgebra of a finite-dimensional Hopf algebra $H$.

Firstly, note that there is an isomorphism of $\k$-linear abelian categories
\begin{equation}\label{eqn:dualcatiso}
{}_{(H/B^+H)^\ast}\Rep(H)_{(H/B^+H)^\ast}\cong
{}_{(H/B^+H)^\ast}\mathfrak{M}_{(H/B^+H)^\ast}^{H^\ast},
\;\;\;\;M\mapsto M,
\end{equation}
where the latter one consists of finite-dimensional $(H/B^+H)^\ast$-$(H/B^+H)^\ast$-bimodules equipped with right $H^\ast$-comodule structure preserving both left and right $(H/B^+H)^\ast$-actions.

It is evident that the isomorphism (\ref{eqn:dualcatiso}) would make ${}_{(H/B^+H)^\ast}\mathfrak{M}_{(H/B^+H)^\ast}^{H^\ast}$ also a finite tensor category,
whose detailed structures would be provided in the following proposition.
We always denote the right $H^\ast$-coaction on $M\in{}_{(H/B^+H)^\ast}\mathfrak{M}_{(H/B^+H)^\ast}^{H^\ast}$ with Sweedler notation by $m\mapsto\sum m_{(0)}\otimes m_{(1)}$.

\begin{proposition}\label{prop:monoidalstru0}
Let $B$ be a left coideal subalgebra of a finite-dimensional Hopf algebra $H$. Then ${}_{(H/B^+H)^\ast}\mathfrak{M}_{(H/B^+H)^\ast}^{H^\ast}$ is a finite tensor category, where
\begin{itemize}
\item[(1)]
The tensor product bifunctor is defined as:
\begin{equation}\label{eqn:tensorprod0}
(M,N)\mapsto M\otimes_{(H/B^+H)^\ast}N
\end{equation}
whose right $H^\ast$-coaction is diagonal:
$m\otimes_{(H/B^+H)^\ast}n\mapsto
\sum(m_{(0)}\otimes_{(H/B^+H)^\ast}n_{(0)})\otimes m_{(1)}n_{(1)}$
for any $m\in M$ and $n\in N$;

\item[(2)]
The unit object is $(H/B^+H)^\ast$, and the associativity and unit constraints are canonical;

\item[(3)]
For each $M\in{}_{(H/B^+H)^\ast}\mathfrak{M}_{(H/B^+H)^\ast}^{H^\ast}$, its left dual object is
$$M^\vee:=\Hom_{(H/B^+H)^\ast}(M_{(H/B^+H)^\ast},(H/B^+H)^\ast)$$
consisting of right $(H/B^+H)^\ast$-maps, with right $H^\ast$-comodule structure defined through
\begin{eqnarray}\label{eqn:Mdualcomod}
\sum m^\vee_{(0)}(m)\otimes m^\vee_{(1)}
&=& \sum [m^\vee(m_{(0)})]_{(1)}\otimes [m^\vee(m_{(0)})]_{(2)}S(m_{(1)})  \\
&\in& (H/B^+H)^\ast\otimes H^\ast  \nonumber
\end{eqnarray}
for any $m^\vee\in M^\vee$ and $m\in M$;

\item[(4)]
There is an equivalence of tensor categories:
\begin{equation}
{}_{(H/B^+H)^\ast}\mathfrak{M}^{H^\ast}_{(H/B^+H)^\ast}\approx\Rep(H)_{\Rep(B)}^\ast,\;\;\;\;
M\mapsto -\otimes_{(H/B^+H)^\ast}M.
\end{equation}
\end{itemize}
\end{proposition}

\begin{proof}
Clearly, the isomorphism (\ref{eqn:dualcatiso}) implies that the $\k$-linear abelian category ${}_{(H/B^+H)^\ast}\mathfrak{M}_{(H/B^+H)^\ast}^{H^\ast}$ is finite as well, and our goal is to verify that structures in (1) to (3) are admissible for it being a tensor category.
\begin{itemize}
\item[(1)]
Due to the coopposite version of Lemma \ref{lem:cocleftness}(1), one could find that $(H/B^+H)^\ast$ is Frobenius as a right coideal subalgebra of $H^\ast$, which implies that every $M\in{}_{(H/B^+H)^\ast}\mathfrak{M}_{(H/B^+H)^\ast}^{H^\ast}$ is left and right $(H/B^+H)^\ast$-free according to \cite[Theorem 2.1(4)]{Mas92}. Therefore, the tensor product bifunctor $-\otimes_{(H/B^+H)^\ast}-$ defined as (\ref{eqn:tensorprod0}) is biexact.

It remains to show that the right $H^\ast$-comodule structure of $M\otimes_{(H/B^+H)^\ast}N$ preserves both left and right $(H/B^+H)^\ast$-actions: Indeed, for any $f\in(H/B^+H)^\ast$ and $m\in M$, $n\in N$,
\begin{eqnarray*}
&& \sum\left(f(m\otimes_{(H/B^+H)^\ast}n)\right){}_{(0)}
  \otimes\left(f(m\otimes_{(H/B^+H)^\ast}n)\right){}_{(1)} \\
&=&
\sum(fm\otimes_{(H/B^+H)^\ast}n)_{(0)}\otimes(fm\otimes_{(H/B^+H)^\ast}n)_{(1)}  \\
&=&
\sum\left((fm)_{(0)}\otimes_{(H/B^+H)^\ast}n_{(0)}\right)\otimes(fm)_{(1)}n_{(1)}  \\
&=&
\sum (f_{(1)}m_{(0)}\otimes_{(H/B^+H)^\ast}n_{(0)})\otimes f_{(2)}m_{(1)}n_{(1)}  \\
&=&
\sum f_{(1)}(m_{(0)}\otimes_{(H/B^+H)^\ast}n_{(0)})\otimes f_{(2)}(m_{(1)}n_{(1)})
\end{eqnarray*}
holds in $(M\otimes_{(H/B^+H)^\ast}N)\otimes H^\ast$.
The other compatibility equation holds similarly.

\item[(2)]
These could be shown by straightforward computations, as direct consequences of canonical natural isomorphisms on tensor products over the algebra $(H/B^+H)^\ast$.

\item[(3)]
The existence of left and right dual objects is a result of \cite[Exercise 2.10.16]{EGNO15} for example, and here we show that every object $M\in{}_{(H/B^+H)^\ast}\mathfrak{M}_{(H/B^+H)^\ast}^{H^\ast}$ has the left dual $M^\vee$ with structures as desired.

At first, the bimodule structure of $$M^\vee=\Hom_{(H/B^+H)^\ast}(M_{(H/B^+H)^\ast},(H/B^+H)^\ast)$$ is chosen classical. Namely,
for any $f\in(H/B^+H)^\ast$, $m^\vee\in M^\vee$ and $m\in M$,
\begin{equation}\label{eqn:M^veebimod}
(fm^\vee)(m)=fm^\vee(m)\;\;\;\;\text{and}\;\;\;\;(m^\vee f)(m)=m^\vee(fm)
\end{equation}
both hold in $(H/B^+H)^\ast$.

Moreover, one could calculate to know that the right $H^\ast$-comodule structure (\ref{eqn:Mdualcomod}) preserves left and right $(H/B^+H)^\ast$-actions. Specifically, for any $f\in(H/B^+H)^\ast$ and $m^\vee\in M^\vee$, we have equations
\begin{eqnarray*}
\sum(fm^\vee)_{(0)}(m)\otimes(fm^\vee)_{(1)}
&\overset{(\ref{eqn:Mdualcomod})}{=}&
\sum[fm^\vee(m_{(0)})]_{(1)}\otimes[fm^\vee(m_{(0)})]_{(2)}S(m_{(1)})  \\
&=&
\sum f_{(1)}[m^\vee(m_{(0)})]_{(1)}\otimes f_{(2)}[m^\vee(m_{(0)})]_{(2)}S(m_{(1)})  \\
&\overset{(\ref{eqn:Mdualcomod})}{=}&
\sum f_{(1)}m^\vee_{(0)}(m)\otimes f_{(2)}m^\vee_{(1)}
\;\;\;\;\;\;(\forall m\in M),
\end{eqnarray*}
which are concluded as
$$\sum(fm^\vee)_{(0)}\otimes(fm^\vee)_{(1)}
=\sum f_{(1)}m^\vee_{(0)}\otimes f_{(2)}m^\vee_{(1)}.$$
On the other hand, the equation
$\sum(m^\vee f)_{(0)}\otimes(m^\vee f)_{(1)}
=\sum m^\vee_{(0)}f_{(1)}\otimes m^\vee_{(1)}f_{(2)}$
is a conclusion of similar computations:
\begin{eqnarray*}
&& \sum(m^\vee f)_{(0)}(m)\otimes(m^\vee f)_{(1)}  \\
&\overset{(\ref{eqn:Mdualcomod})}{=}&
\sum[m^\vee f(m_{(0)})]_{(1)}\otimes[m^\vee f(m_{(0)})]_{(2)}S(m_{(1)})  \\
&\overset{(\ref{eqn:M^veebimod})}{=}&
\sum[m^\vee(fm_{(0)})]_{(1)}\otimes[m^\vee(fm_{(0)})]_{(2)}S(m_{(1)})  \\
&=&
\sum[m^\vee(f_{(1)}m_{(0)})]_{(1)}
    \otimes[m^\vee(f_{(1)}m_{(0)})]_{(2)}S(f_{(2)}m_{(1)})f_{(3)}  \\
&\overset{(\ref{eqn:Mdualcomod})}{=}&
\sum m^\vee_{(0)}(f_{(1)}m)\otimes m^\vee_{(1)}f_{(2)}  \\
&\overset{(\ref{eqn:M^veebimod})}{=}&
\sum m^\vee_{(0)}f_{(1)}(m)\otimes m^\vee_{(1)}f_{(2)}
\;\;\;\;\;\;\;\;(\forall m\in M).
\end{eqnarray*}

Now let us point out that the corresponding evaluation and coevaluation are as follows:
\begin{equation}\label{eqn:Mevcoev}
\left\{
\begin{array}{cl}
\ev_M:M^\vee\otimes_{(H/B^+H)^\ast}M\rightarrow (H/B^+H)^\ast, &
m^\vee\otimes_{(H/B^+H)^\ast}m\mapsto m^\vee(m);  \\[6pt]
\coev_M:(H/B^+H)^\ast\rightarrow M\otimes_{(H/B^+H)^\ast}M^\vee, &
\e\mapsto \sum_i m_i\otimes_{(H/B^+H)^\ast}m_i^\vee,
\end{array}
\right.
\end{equation}
where $\{m_i\}$ is a finite right $(H/B^+H)^\ast$-basis of the free module $M$, with dual left $(H/B^+H)^\ast$-basis $\{m_i^\vee\}$ of $M^\vee$. We know that (\ref{eqn:Mevcoev}) satisfy the axioms for $M^\vee$ being a left dual of $M$ as an object in ${}_{(H/B^+H)^\ast}\mathfrak{M}_{(H/B^+H)^\ast}$, and hence it suffices to show that $\ev_M$ and $\coev_M$ preserve right $H^\ast$-coactions. This is due to following computations:

\begin{eqnarray*}
\sum m^\vee_{(0)}(m_{(0)})\otimes m^\vee_{(1)}m_{(1)}
&\overset{(\ref{eqn:Mdualcomod})}{=}&
\sum[m^\vee(m_{(0)})]_{(1)}\otimes[m^\vee(m_{(0)})]_{(2)}S(m_{(1)})m_{(2)}  \\
&=&
\sum[m^\vee(m)]_{(1)}\otimes[m^\vee(m)]_{(2)},
\end{eqnarray*}
and
\begin{eqnarray*}
&&
\sum_i \left({m_i}_{(0)}\otimes_{(H/B^+H)^\ast} {m_i^\vee}_{(0)}(m)\right)
    \otimes {m_i}_{(1)}{m_i^\vee}_{(1)}  \\
&\overset{(\ref{eqn:Mdualcomod})}{=}&
\sum_i \left({m_i}_{(0)}\otimes_{(H/B^+H)^\ast} [m_i^\vee(m_{(0)})]_{(1)}\right)
    \otimes {m_i}_{(1)}[m_i^\vee(m_{(0)})]_{(2)}S(m_{(1)})  \\
&=&
\sum_i \left({m_i}_{(0)}[m_i^\vee(m_{(0)})]_{(1)}\otimes_{(H/B^+H)^\ast}\e \right)
    \otimes {m_i}_{(1)}[m_i^\vee(m_{(0)})]_{(2)}S(m_{(1)})  \\
&=&
\sum_i \left([m_im_i^\vee(m_{(0)})]_{(0)}\otimes_{(H/B^+H)^\ast}\e \right)
    \otimes[m_im_i^\vee(m_{(0)})]_{(1)}S(m_{(1)})  \\
&=&
\sum \left(m_{(0)}\otimes_{(H/B^+H)^\ast}\e \right)\otimes m_{(1)}S(m_{(2)})  \\
&=&
\sum \left(m\otimes_{(H/B^+H)^\ast}\e \right)\otimes \e  \\
&=&
\sum_i \left(m_i\otimes_{(H/B^+H)^\ast}m_i^\vee(m) \right)\otimes \e
\;\;\;\;\;\;(\forall m\in M).
\end{eqnarray*}

\item[(4)]
Recall by the equivalence (\ref{eqn:dualcatequiv}) that
$${}_{(H/B^+H)^\ast}\Rep(H)_{(H/B^+H)^\ast}\approx\Rep(H)_{\Rep(B)}^\ast,\;\;\;\;
M\mapsto-\otimes_{(H/B^+H)^\ast}M$$
as tensor categories.
Thus it is sufficient to show that the tensor products are preserved by the inverse of the isomorphism (\ref{eqn:dualcatiso}), which is:
$${}_{(H/B^+H)^\ast}\mathfrak{M}_{(H/B^+H)^\ast}^{H^\ast}
\cong{}_{(H/B^+H)^\ast}\Rep(H)_{(H/B^+H)^\ast},
\;\;\;\;M\mapsto M.$$
It sends the right $H^\ast$-comodule structure of $M$ to the left hit $H$-action. Specifically,
\begin{equation}\label{eqn:hm}
hm=\sum m_{(0)}\la m_{(1)},h\ra
\end{equation}
holds for any $h\in H$ and $m\in M$.

In fact, since the left $H$-action of
$M\otimes_{(H/B^+H)^\ast}N\in{}_{(H/B^+H)^\ast}\Rep(H)_{(H/B^+H)^\ast}$
is diagonal, we could directly calculate to find that: For any $h\in H$, $m\in M$ and $n\in N$,
\begin{eqnarray*}
h(m\otimes_{(H/B^+H)^\ast}n)
&=&
\sum h_{(1)}m\otimes_{(H/B^+H)^\ast}h_{(2)}n  \\
&\overset{(\ref{eqn:hm})}{=}&
\sum m_{(0)}\la m_{(1)},h_{(1)}\ra \otimes_{(H/B^+H)^\ast} n_{(0)}\la n_{(1)},h_{(2)}\ra  \\
&=&
\sum (m_{(0)}\otimes_{(H/B^+H)^\ast} n_{(0)})\la m_{(1)}n_{(1)},h\ra  \\
&=&
\sum (m\otimes_{(H/B^+H)^\ast} n)_{(0)}
    \left\la (m\otimes_{(H/B^+H)^\ast} n)_{(1)},h\right\ra,
\end{eqnarray*}
while the left and right $(H/B^+H)^\ast$-actions remain unchanged on the respective tensorands of $M\otimes_{(H/B^+H)^\ast}N$.
\end{itemize}
\end{proof}

\begin{remark}
Suppose $\{m_i\}$ is a finite right $(H/B^+H)^\ast$-basis of the free module $M$ with dual left
$(H/B^+H)^\ast$-basis $\{m_i^\vee\}$ of $M^\vee$, which means that
\begin{equation}
m=\sum_i m_im_i^\vee(m)\;\;\;\;\text{and}\;\;\;\;
m^\vee=\sum_i m^\vee(m_i)m_i^\vee
\end{equation}
hold for all $m\in M$ and $m\in M^\vee$.
Then an equivalent formulation for the right $H^\ast$-comodule structure (\ref{eqn:Mdualcomod}) of $M^\vee$ is
\begin{equation}\label{eqn:Mdualcomod2}
\sum m^\vee_{(0)}\otimes m^\vee_{(1)}
=\sum_i [m^\vee({m_i}_{(0)})]_{(1)}m_i^\vee\otimes [m^\vee({m_i}_{(0)})]_{(2)}S({m_i}_{(1)})
\in M^\vee\otimes H^\ast.
\end{equation}
This is because the both sides of (\ref{eqn:Mdualcomod2}) maps $m\otimes\id$ to the sides of (\ref{eqn:Mdualcomod}) respectively.
\end{remark}

Next, let us identify the dual category $\Rep(H)_{\Rep(B)}^\ast$, or  ${}_{(H/B^+H)^\ast}\mathfrak{M}_{(H/B^+H)^\ast}^{H^\ast}$, in another way.
For the purpose, the notion of the \textit{cotensor product} $-\square_C-$ over a coalgebra $C$ would be used, and one might refer to \cite[Section 0]{Tak77(b)} for the definition and basic properties of cotensor products.

The following lemma is a direct consequence of \cite[Lemma 1.8]{Mas94}, where the $\k$-linear abelian equivalence are provided via the functors $\Phi$ and $\Psi$ defined in \cite[Section 1]{Tak79}:

\begin{lemma}\label{lem:abelequiv}
Let $H$ be finite-dimensional Hopf algebra with left coideal subalgebra $B$. Then
\begin{equation}\label{eqn:PsiPhi0}
M\mapsto \overline{M}:=M/M\left((H/B^+H)^\ast\right)^+\;\;\;\;\;\;\text{and}\;\;\;\;\;\;
V\square_{B^\ast}H^\ast\mapsfrom V
\end{equation}
gives an equivalence of $\k$-linear abelian categories
\begin{equation}\label{eqn:PsiPhi1}
{}_{(H/B^+H)^\ast}\mathfrak{M}_{(H/B^+H)^\ast}^{H^\ast}\approx {}_{(H/B^+H)^\ast}\mathfrak{M}^{B^\ast},
\end{equation}
where the latter one is the category of Doi-Hopf modules introduced at Equation (\ref{eqn:DoiHopfmod}).
\end{lemma}

\begin{proof}
We remark that the left $B^\ast$-comodule structure of $H^\ast$ is $(\iota^\ast\otimes\id)\circ\Delta$, and consequently the cotensor product becomes
$$V\square_{B^\ast}H^\ast:=
\Big\{\sum_i v_i\otimes h_i^\ast\in V\otimes H^\ast \mid
  \sum_i {v_i}_{\la0\ra}\otimes {v_i}_{\la1\ra}\otimes h_i^\ast
  =\sum_i v_i\otimes \iota^\ast({h_i^\ast}_{(1)})\otimes {h_i^\ast}_{(2)}\Big\}$$
for each right $B^\ast$-comodule $V$.

As mentioned in Lemma \ref{lem:cocleftness}, the right coideal subalgebra $(H/B^+H)^\ast$ must be Frobenius in this case, and thus by \cite[Lemma 1.8]{Mas94}, the correspondence (\ref{eqn:PsiPhi0}) gives an equivalence $\mathfrak{M}_{(H/B^+H)^\ast}^{H^\ast}\approx \mathfrak{M}^{B^\ast}$. It remains to prove that the functors in (\ref{eqn:PsiPhi0}) are compatible with left $(H/B^+H)^\ast$-module structures which we desire as follows:

For each $V\in{}_{(H/B^+H)^\ast}\mathfrak{M}^{B^\ast}$, define the left $(H/B^+H)^\ast$-action on $V\square_{B^\ast}H^\ast$ to be diagonal (cf. \cite[Lemma 2.9]{Mas94}), namely:
\begin{equation}\label{eqn:Psi(V)mod}
f\cdot(\sum_i v_i\otimes h^\ast_i):=\sum_i f_{(1)}v_i\otimes f_{(2)}h^\ast_i
\end{equation}
for any $f\in(H/B^+H)^\ast$ and $\sum_i v_i\otimes h^\ast_i\in V\square_{B^\ast}H^\ast$,
which is evidently well-defined. Recall that the right $(H/B^+H)^\ast$-action and $H^\ast$-coaction on $V\square_{B^\ast}H^\ast$ is completely defined through the second (co)tensorand $H^\ast$:
\begin{equation}\label{eqn:right(co)actions}
(\sum_i v_i\otimes h^\ast_i)\cdot f=\sum_i v_i\otimes h^\ast_i\pi^\ast(f)\;\;\;\;\text{and}
\;\;\;\;\sum_i v_i\otimes h^\ast_i\mapsto\sum_i v_i\otimes {h^\ast_i}_{(1)}\otimes{h^\ast_i}_{(2)},
\end{equation}
which are both left $(H/B^+H)^\ast$-module maps. Thus $V\square_{B^\ast}H^\ast\in {}_{(H/B^+H)^\ast}\mathfrak{M}_{(H/B^+H)^\ast}^{H^\ast}$.

On the other hand, it is clear that $\overline{M}:=M/M\left((H/B^+H)^\ast\right)^+$ has the quotient left $(H/B^+H)^\ast$-module structure. It makes the right $B^\ast$-comodule structure
\begin{equation}\label{eqn:Phi(M)comod}
\overline{M}\rightarrow \overline{M}\otimes B^\ast,\;\;\;\;
\overline{m}\mapsto\sum\overline{m_{(0)}}\otimes \iota^\ast(m_{(1)})
\end{equation}
preserve left $(H/B^+H)^\ast$-actions as in (\ref{eqn:DoiHopfmod}), which implies that $\overline{M}\in{}_{(H/B^+H)^\ast}\mathfrak{M}^{B^\ast}$. Consequently, the adjunction isomorphism
\begin{equation}\label{eqn:tau0}
M\cong \overline{M}\square_{B^\ast}H^\ast,\;\;\;\;
m\mapsto \sum \overline{m_{(0)}}\otimes m_{(1)}
\end{equation}
described in the proof of \cite[Theorem 1]{Tak79} also preserves the left $(H/B^+H)^\ast$-actions.

Furthermore, one could verify that the other adjunction isomorphism
\begin{equation}\label{eqn:sigma0}
\overline{V\square_{B^\ast}H^\ast}\cong V,\;\;\;\;
\overline{\sum_i v_i\otimes h^\ast_i}\mapsto\sum_i v_i\langle h^\ast_i,1\rangle
\end{equation}
in the proof of \cite[Theorem 1]{Tak79} preserves the left $(H/B^+H)^\ast$-actions as well.
\end{proof}

Now we are able to combine the equivalence (\ref{eqn:PsiPhi1}) and the isomorphism (\ref{eqn:PsiPhi2}) to provide that
\begin{equation}\label{eqn:PsiPhi3}
{}_{(H/B^+H)^\ast}\mathfrak{M}_{(H/B^+H)^\ast}^{H^\ast}\approx\Rep((H/B^+H)^\ast\#B),
\end{equation}
as $\k$-linear abelian categories.
However, note that ${}_{(H/B^+H)^\ast}\mathfrak{M}_{(H/B^+H)^\ast}^{H^\ast}$ is furthermore a tensor category with structures defined in Proposition \ref{prop:monoidalstru0}.
As mentioned in the paragraph before \cite[Theorem 3.3.5]{Sch02}, one could conclude that $\Rep((H/B^+H)^\ast\#B)$ is also a tensor category such that (\ref{eqn:PsiPhi3}) becomes a tensor equivalence.
Consequently,
it is suggested by similar arguments to \cite[Definition XV.1.1]{Kas95} or \cite[Proposition 13.2]{ES02} that the smash product algebra $(H/B^+H)^\ast\#B$ would become a quasi-Hopf algebra reconstructed. The remaining of this subsection is devoted to establishing a monoidal structure of the functor (\ref{eqn:PsiPhi3}).

\begin{notation}\label{not:Phi&Psi}
For convenience, the equivalence (\ref{eqn:PsiPhi3}) of $\k$-linear abelian categories and its quasi-inverse are denoted respectively by
\begin{equation}\label{eqn:Phi}
\begin{array}{crcl}
\Phi: & {}_{(H/B^+H)^\ast}\mathfrak{M}_{(H/B^+H)^\ast}^{H^\ast} &\rightarrow& \Rep((H/B^+H)^\ast\#B), \\
& M &\mapsto& \overline{M}:=M/M\left((H/B^+H)^\ast\right)^+
\end{array}
\end{equation}
and
\begin{equation}\label{eqn:Psi}
\begin{array}{crcl}
\Psi: & \Rep((H/B^+H)^\ast\#B) &\rightarrow& {}_{(H/B^+H)^\ast}\mathfrak{M}_{(H/B^+H)^\ast}^{H^\ast},  \\
& V &\mapsto& V\square_{B^\ast}H^\ast
\end{array}
\end{equation}
in this paper.
Recall in (\ref{eqn:Bhit}) that the right $B^\ast$-comodule structure
$v\mapsto\sum v_{\la0\ra}\otimes v_{\la1\ra}$ on $V\in\Rep((H/B^+H)^\ast\#B)$ is defined through
\begin{equation}\label{eqn:bv}
\sum v_{\la0\ra}\la v_{\la1\ra},b\ra=(\e\#b)v\;\;\;\;\;\;(\forall b\in B),
\end{equation}
and the structures defined by (\ref{eqn:Phi(M)comod}) and (\ref{eqn:Psi(V)mod}) in the proof of Lemma \ref{lem:abelequiv} make the functors $\Phi$ and $\Psi$ well-defined.
\end{notation}


We should remark that the definitions of $\Phi$ and $\Psi$ are independent of {\pams}s for $\iota:B\hookrightarrow H$. However, suitable monoidal structures of $\Phi$ would be provided with a {\pams} $(\zeta,\gamma^\ast)$ in the following lemma.

\begin{lemma}\label{lem:monoidalstruJ}
Let $H$ be a finite-dimensional Hopf algebra. Suppose that $B$ is a left coideal subalgebra of $H$ with a {\pams} $(\zeta,\gamma^\ast)$.
Then:
\begin{itemize}
\item[(1)]
There is a bifunctor $-\otimes-$ on $\Rep((H/B^+H)^\ast\#B)$: For any objects $V$ and $W$,
define $V\otimes W:=V\otimes_\k W$ as a vector space with left $(H/B^+H)^\ast\#B$-module structure:
\begin{eqnarray}\label{eqn:tensorprod}
\begin{array}{rcl}
((H/B^+H)^\ast\#B)\otimes(V\otimes W) &\rightarrow& V\otimes W  \\
(f\#b)\otimes(v\otimes w) &\mapsto& \mathbf{\Delta}(f\#b)(v\otimes w);
\end{array}
\end{eqnarray}

\item[(2)]
There is a natural isomorphism $J:\Phi(-)\otimes\Phi(-)\cong\Phi(-\otimes_{(H/B^+H)^\ast}-)$ in $\Rep((H/B^+H)^\ast\#B)$ defined as follows:
\begin{eqnarray}\label{eqn:monoidalstruJ}
\begin{array}{rcl}
J_{M,N}\;:\;
\overline{M}\otimes\overline{N} &\cong& \overline{M\otimes_{(H/B^+H)^\ast}N},  \\
\overline{m}\otimes\overline{n}
&\mapsto& \sum\overline{m_{(0)}\overline{\gamma}^\ast(m_{(1)})\otimes_{(H/B^+H)^\ast}n}
\end{array}
\end{eqnarray}
for all $M,N\in{}_{(H/B^+H)^\ast}\mathfrak{M}_{(H/B^+H)^\ast}^{H^\ast}$, and its inverse would be
\begin{eqnarray}\label{eqn:monoidalstruJ^-1}
\begin{array}{rcl}
J_{M,N}^{-1}\;:\;
\overline{M\otimes_{(H/B^+H)^\ast}N} &\cong& \overline{M}\otimes\overline{N},  \\
\overline{m\otimes_{(H/B^+H)^\ast}n}
&\mapsto& \sum\overline{m_{(0)}}\otimes \overline{\gamma^\ast(m_{(1)})n};
\end{array}
\end{eqnarray}

\item[(3)]
The equivalence $\Phi$ sends the unit object $(H/B^+H)^\ast$ to the trivial representation $\k$ (via the algebra map $\pd{\e}$) in the sense:
\begin{equation}\label{eqn:Phiunit}
\begin{array}{rccc}
\Phi((H/B^+H)^\ast)=\overline{(H/B^+H)^\ast}&\cong&\k&\in\Rep((H/B^+H)^\ast\#B)  \\
\overline{f}\;\; &\mapsto& \la f,1\ra.&
\end{array}
\end{equation}
\end{itemize}
\end{lemma}

\begin{proof}
\begin{itemize}
\item[(1)]
It is clear that (\ref{eqn:tensorprod}) makes $V\otimes W$ a left $(H/B^+H)^\ast\#B$-module, since $\pd{\Delta}$ is an algebra map by Theorem \ref{thm:partialdual} (1) and (2), which has been proved in Subsection \ref{subsection:Thm(1-2)Pf}.

\item[(2)]
At first, it should be verified that $J_{M,N}$ is a well-defined map for any $M,N\in{}_{(H/B^+H)^\ast}\mathfrak{M}_{(H/B^+H)^\ast}^{H^\ast}$, which is due to following calculations: For each $f\in(H/B^+H)^\ast$, 
\begin{eqnarray*}
J_{M,N}\left(\overline{mf}\otimes\overline{n}\right)
&\overset{(\ref{eqn:monoidalstruJ})}{=}&
\sum\overline{m_{(0)}f_{(1)}\overline{\gamma}^\ast(m_{(1)}f_{(2)})\otimes_{(H/B^+H)^\ast}n}  \\
&\overset{(\ref{eqn:gamma*bar1})}{=}&
\sum\overline{m_{(0)}\la f,1\ra\overline{\gamma}^\ast(m_{(1)})\otimes_{(H/B^+H)^\ast}n}
~\overset{(\ref{eqn:monoidalstruJ})}{=}~
J_{M,N}\left(\overline{m}\la f,1\ra\otimes\overline{n}\right)
\end{eqnarray*}
as well as
\begin{eqnarray*}
J_{M,N}\left(\overline{m}\otimes\overline{nf}\right)
&\overset{(\ref{eqn:monoidalstruJ})}{=}&
\sum\overline{m_{(0)}\overline{\gamma}^\ast(m_{(1)})\otimes_{(H/B^+H)^\ast}nf}  \\
&=&
\sum\overline{m_{(0)}\overline{\gamma}^\ast(m_{(1)})\otimes_{(H/B^+H)^\ast}n}\la f,1\ra
~\overset{(\ref{eqn:monoidalstruJ})}{=}~
J_{M,N}\left(\overline{m}\otimes\overline{n}\la f,1\ra\right),
\end{eqnarray*}
where the penultimate equality is because
\begin{equation}\label{eqn:mfbar}
\overline{mf}=\overline{m}\la f,1\ra\in \overline{M}:=M/M\left((H/B^+H)^\ast\right)^+
\;\;\;\;\;\;(\forall m\in M).
\end{equation}
The naturality of $J$ is evident.

On the other hand, $J^{-1}$ defined as (\ref{eqn:monoidalstruJ}) is also a well-defined map, since
\begin{eqnarray*}
J_{M,N}^{-1}\left(\overline{mf\otimes_{(H/B^+H)^\ast} n}\right)
&\overset{(\ref{eqn:monoidalstruJ^-1})}{=}&
\sum \overline{m_{(0)}f_{(1)}}\otimes\overline{\gamma^\ast(m_{(1)}f_{(2)})n}  \\
&=&
\sum \overline{m_{(0)}}\la f_{(1)},1\ra\otimes\overline{\gamma^\ast(m_{(1)}f_{(2)})n}  \\
&=&
\sum \overline{m_{(0)}}\otimes\overline{\gamma^\ast[m_{(1)}\pi^\ast(f)]n}
~\overset{(\ref{eqn:gamma*})}{=}~
\sum \overline{m_{(0)}}\otimes\overline{\gamma^\ast(m_{(1)})fn}  \\
&\overset{(\ref{eqn:monoidalstruJ^-1})}{=}&
J_{M,N}^{-1}\left(\overline{m\otimes_{(H/B^+H)^\ast}fn}\right),
\end{eqnarray*}
and
\begin{eqnarray*}
J_{M,N}^{-1}\left(\overline{m\otimes_{(H/B^+H)^\ast}nf}\right)
&\overset{(\ref{eqn:monoidalstruJ^-1})}{=}&
\sum \overline{m_{(0)}}\otimes\overline{\gamma^\ast(m_{(1)})nf}
~=~
\sum \overline{m_{(0)}}\otimes\overline{\gamma^\ast(m_{(1)})n}\la f,1\ra  \\
&\overset{(\ref{eqn:monoidalstruJ^-1})}{=}&
J_{M,N}^{-1}\left(\overline{m\otimes_{(H/B^+H)^\ast}n}\la f,1\ra\right).
\end{eqnarray*}
both hold for any $f\in(H/B^+H)^\ast$.

Moreover, one could directly find that $J_{M,N}^{-1}$ and $J_{M,N}$ are mutually inverse, by noting that
$$J_{M,N}(\overline{m}\otimes\overline{n})
=\sum\overline{m_{(0)}\otimes_{(H/B^+H)^\ast}\overline{\gamma}^\ast(m_{(1)})n}$$
as well as the assumption that $\overline{\gamma}^\ast$ is the convolution inverse of $\gamma^\ast$.

Finally, we aim to show that
$J_{M,N}^{-1}$ is a morphism in $\Rep((H/B^+H)^\ast\#B)$. Let $\{b_i\}$ be a linear basis of $B$ with dual basis $\{b^\ast_i\}$ of $B^\ast$, and let $\{f_i\}$ be a linear basis of $(H/B^+H)^\ast$ with dual basis $\{f^\ast_i\}$ of $H/B^+H$ as usual. Then compute for any $f\in(H/B^+H)^\ast$ that
\begin{eqnarray*}
&& (f\#1)\cdot J_{M,N}^{-1}\left(\overline{m\otimes_{(H/B^+H)^\ast}n}\right) \\
&\overset{(\ref{eqn:monoidalstruJ^-1}),\;(\ref{eqn:tensorprod})}{=}&
\pd{\Delta}(f\#1)\left(\sum\overline{m_{(0)}}\otimes\overline{\gamma^\ast(m_{(1)})n}\right) \\
&\overset{(\ref{eqn:Delta(f)})}{=}&
\sum_i (f_{(1)}\#b_i)\overline{m_{(0)}}
  \otimes(\gamma^\ast[f_{(2)}\zeta^\ast(b^\ast_i)]\#1)\overline{\gamma^\ast(m_{(1)})n}  \\
&\overset{(\ref{eqn:bv}),\;(\ref{eqn:Phi(M)comod})}{=}&
\sum_i \overline{f_{(1)}m_{(0)}}\la\iota^\ast(m_{(1)}),b_i\ra
  \otimes\overline{\gamma^\ast[f_{(2)}\zeta^\ast(b^\ast_i)]\gamma^\ast(m_{(2)})n}  \\
&\overset{(\ref{eqn:gamma*})}{=}&
\sum \overline{f_{(1)}m_{(0)}}
  \otimes\overline{\gamma^\ast\left(f_{(2)}
  \zeta^\ast[\iota^\ast(m_{(1)})]\pi^\ast[\gamma^\ast(m_{(2)})]\right)n}  \\
&\overset{(\ref{eqn:convolution*})}{=}&
\sum \overline{f_{(1)}m_{(0)}}\otimes\overline{\gamma^\ast(f_{(2)}m_{(1)})n}  \\
&\overset{(\ref{eqn:monoidalstruJ^-1})}{=}&
J_{M,N}^{-1}\left((f\#1)\overline{m\otimes_{(H/B^+H)^\ast}n}\right),
\end{eqnarray*}
and for $b\in B$ that
\begin{eqnarray*}
&& (\e\#b)\cdot J_{M,N}^{-1}\left(\overline{m\otimes_{(H/B^+H)^\ast}n}\right)  \\
&\overset{(\ref{eqn:monoidalstruJ^-1}),\;(\ref{eqn:tensorprod})}{=}&
\pd{\Delta}(\e\#b)
\left(\sum\overline{m_{(0)}}\otimes\overline{\gamma^\ast(m_{(1)})n}\right)  \\
&\overset{(\ref{eqn:Delta(b)})}{=}&
\sum_i (\e\#\zeta[\gamma(f^\ast_i)b_{(1)}])\overline{m_{(0)}}
  \otimes(f_i\#b_{(2)})\overline{\gamma^\ast(m_{(1)})n}  \\
&\overset{(\ref{eqn:bv}),\;(\ref{eqn:Phi(M)comod})}{=}&
\sum_i \overline{m_{(0)}}\la\iota^\ast(m_{(1)}),\zeta[\gamma(f^\ast_i)b_{(1)}]\ra
  \otimes\overline{f_i\gamma^\ast(m_{(2)})_{(1)}n_{(0)}}  \\
&&\;\;\;\;\;\; \la\iota^\ast[\gamma^\ast(m_{(2)})_{(2)}n_{(1)}],b_{(2)}\ra  \\
&=&
\sum_i \overline{m_{(0)}}\la\gamma^\ast(\zeta^\ast[\iota^\ast(m_{(1)})]_{(1)}),f^\ast_i\ra
    \la \zeta^\ast[\iota^\ast(m_{(1)})]_{(2)},b_{(1)}\ra  \\
&& \;\;\;\;\;\; \otimes \overline{f_i\gamma^\ast(m_{(2)})_{(1)}n_{(0)}}
  \la\gamma^\ast(m_{(2)})_{(2)}n_{(1)},\iota(b_{(2)})\ra  \\
&\overset{(\ref{eqn:iotapi})}{=}&
\sum \overline{m_{(0)}}
  \otimes\overline{\gamma^\ast\left(\zeta^\ast[\iota^\ast(m_{(1)})]_{(1)}\right)
  \gamma^\ast(m_{(2)})_{(1)}n_{(0)}}  \\
&& \;\;\;\;\;\;
  \la\zeta^\ast[\iota^\ast(m_{(1)})]_{(2)}\gamma^\ast(m_{(2)})_{(2)}n_{(1)},\iota(b)\ra  \\
&\overset{(\ref{eqn:pi*}),\;(\ref{eqn:gamma*})}{=}&
\sum \overline{m_{(0)}}  \otimes\,\overline{\gamma^\ast\left(\zeta^\ast[\iota^\ast(m_{(1)})]_{(1)}
  \pi^\ast[\gamma^\ast(m_{(2)})]_{(1)}\right)n_{(0)}} \\
&& \;\;\;\;\;\;
  \la\zeta^\ast[\iota^\ast(m_{(1)})]_{(2)}\pi^\ast[\gamma^\ast(m_{(2)})]_{(2)}n_{(1)},
     \iota(b)\ra  \\
&\overset{(\ref{eqn:convolution*})}{=}&
\sum \overline{m_{(0)}}
  \otimes\overline{\gamma^\ast(m_{(1)})n_{(0)}}\la \iota^\ast(m_{(2)}n_{(1)}),b\ra
  \\
&\overset{(\ref{eqn:monoidalstruJ^-1})}{=}&
J_{M,N}^{-1}\left((\e\#b)\overline{m\otimes_{(H/B^+H)^\ast}n}\right).
\end{eqnarray*}

\item[(3)]
It is straightforward to find that (\ref{eqn:Phiunit}) is an isomorphism in $\Rep((H/B^+H)^\ast\#B)$.
\end{itemize}
\end{proof}

\begin{remark}
Lemma \ref{lem:monoidalstruJ}(2) is an analogue of \cite[Lemma 3.3.4]{Sch02}. A similar treatment on such tensor functors is in the proof of \cite[Theorem 3.8]{AGM14}.
\end{remark}

Continue the idea in the paragraph before \cite[Theorem 3.3.5]{Sch02}, which has been introduced after (\ref{eqn:PsiPhi3}). In order to make $\Phi$ (as well as $\Psi$) an equivalence of tensor categories, the tensor product of $V,W\in\Rep((H/B^+H)^\ast\#B)$ is supposed to be identified with
$$V\otimes W\cong\Phi(\Psi(V)\otimes_{(H/B^+H)^\ast}\Psi(W))
\overset{(\ref{eqn:Phi}),\;(\ref{eqn:Psi})}{=}
\overline{(V\square_{B^\ast}H^\ast)\otimes_{(H/B^+H)^\ast}(W\square_{B^\ast}H^\ast)}.$$

In fact, we know by the results in Lemma \ref{lem:monoidalstruJ} that there is an associator $\pd{\phi}$ (which would be computed in the next subsection) of $(H/B^+H)^\ast\#B$ such that $J$ (\ref{eqn:monoidalstruJ}) satisfies the hexagon diagram, and hence the equivalence $\Phi$ is a tensor functor with monoidal structure $J$. Consequently, it would follow from \cite[Remark 2.4.10]{EGNO15} that $\Phi$ has a quasi-inverse tensor functor, which could be also chosen as $\Psi$ (with some monoidal structure). Moreover, the adjunction (\ref{eqn:sigma0}) could be realized as a natural isomorphism of tensor functors, denoted by
$$\sigma:\Phi\circ\Psi:=\overline{-\square_{B^\ast}H^\ast}\cong\Id.$$
Specifically, for each $V\in\Rep((H/B^+H)^\ast\#B)$ and all $\sum_i v_i\otimes h^\ast_i\in V\square_{B^\ast}H^\ast$, $v\in V$, one could directly verify that
\begin{equation}\label{eqn:sigmaV}
\sigma_V\Big(\overline{\sum_i v_i\otimes h^\ast_i}\Big)=\sum_i v_i\la h^\ast_i,1\ra\;\;\;\;\text{and}\;\;\;\;
\sigma^{-1}_V(v)=\overline{\sum v_{\la0\ra}\otimes \zeta^\ast(v_{\la1\ra})}
\end{equation}
are well-defined in $\Rep((H/B^+H)^\ast\#B)$ and mutually inverse.

In a word, our conclusion would be:

\begin{corollary}\label{cor:tensorequiv}
$\Rep((H/B^+H)^\ast\#B)$ is a finite tensor category with tensor product bifunctor (\ref{eqn:tensorprod}) and unit object $\k$. Moreover, $\Phi$ is a tensor equivalence with monoidal structure $J$ defined in (\ref{eqn:monoidalstruJ}).
\end{corollary}



With the usage of Corollary \ref{cor:tensorequiv}, our goal for the remaining of this subsection is to prove that structures in 
Theorem \ref{thm:partialdual} (4) and (5) make $(H/B^+H)^\ast\#B$ a quasi-Hopf algebra.
In fact, the process is essentially by noting that the composition
\begin{equation}\label{eqn:forgetful}
{}_{(H/B^+H)^\ast}\mathfrak{M}_{(H/B^+H)^\ast}^{H^\ast}\xrightarrow{\Phi}
\Rep((H/B^+H)^\ast\#B)\xrightarrow{\text{forgetful}}\Vec
\end{equation}
is a quasi-fiber functor to the category $\Vec$ of finite-dimensional vector spaces, and applying the reconstruction theorem of quasi-Hopf algebras (cf. \cite[Theorem 5.13.7]{EGNO15}) to obtain the structures desired.

\subsection{Proofs of Theorem \ref{thm:partialdual}(3) and Remark \ref{rmk:phi^-1} - the associator and its inverse}\label{subsection:Thm(3)Pf}

Before the proofs are given, let us denote the regular left module over $(H/B^+H)^\ast\#B$ by
$$R:=(H/B^+H)^\ast\#B\in\Rep((H/B^+H)^\ast\#B)$$
for simplicity, and then write
$\sigma_R:\overline{R\square_{B^\ast}H^\ast}\cong R$ (\ref{eqn:sigmaV}) with inverse
\begin{equation}\label{eqn:sigma^-1}
\sigma_R^{-1}:R\cong\overline{R\square_{B^\ast}H^\ast},\;\;\;\;
r\mapsto \overline{\sum r_{\la0\ra}\otimes\zeta^\ast(r_{\la1\ra})}.
\end{equation}
Other formulas we would use frequently are:
\begin{lemma}
Suppose that $M$ is an object in ${}_{(H/B^+H)^\ast}\mathfrak{M}_{(H/B^+H)^\ast}^{H^\ast}$. Then for any $m\in M$ and $r\in R$, the following correspondences hold:
\begin{equation}\label{eqn:Jbar1}
\begin{array}{rcl}
J_{R\square_{B^\ast}H^\ast,M}\;:\;\;\;\;\;\;\;\;
\overline{R\square_{B^\ast}H^\ast}\otimes\overline{M}
&\rightarrow& \overline{(R\square_{B^\ast}H^\ast)\otimes_{(H/B^+H)^\ast}M}  \\
 \overline{\sum r_{\la0\ra}\otimes\zeta^\ast(r_{\la1\ra})}\otimes\overline{m}
&\mapsto&
\overline{[\sum r_{\la0\ra}\otimes\zeta^\ast(r_{\la1\ra})]
\otimes_{(H/B^+H)^\ast}m}
\end{array}
\end{equation}
and
\begin{equation}\label{eqn:Jbar2}
\begin{array}{rcl}
J_{M,R\square_{B^\ast}H^\ast}^{-1}\;:\;\;\;\;\;\;
\overline{M\otimes_{(H/B^+H)^\ast}(R\square_{B^\ast}H^\ast)}
&\rightarrow& \overline{M}\otimes\overline{R\square_{B^\ast}H^\ast}  \\
 \overline{m\otimes_{(H/B^+H)^\ast}[\sum r_{\la0\ra}\otimes\zeta^\ast(r_{\la1\ra})]}
&\mapsto&
\sum \overline{m_{(0)}}\otimes
  \overline{e_{\la0\ra}\gamma^\ast(m_{(1)})r\otimes\zeta^\ast(e_{\la1\ra})}.
\end{array}
\end{equation}
\end{lemma}

\begin{proof}
These formulas are easy to obtain by straightforward calculations:
\begin{eqnarray*}
&& J_{R\square_{B^\ast}H^\ast,M}
\left(\overline{\sum r_{\la0\ra}\otimes\zeta^\ast(r_{\la1\ra})}\otimes\overline{m}\right) \\
&\overset{(\ref{eqn:monoidalstruJ}),\;(\ref{eqn:right(co)actions})}{=}&
\overline{\left[\sum r_{\la0\ra}\otimes\zeta^\ast(r_{\la1\ra})_{(1)}
  \pi^\ast\left(\overline{\gamma}^\ast[\zeta^\ast(r_{\la1\ra})_{(2)}]\right)\right]
  \otimes_{(H/B^+H)^\ast} m}  \\
&\overset{(\ref{eqn:convolution*1})}{=}&
\overline{\left[\sum r_{\la0\ra}\otimes
  \zeta^\ast\left(\iota^\ast[\zeta^\ast(r_{\la1\ra})]\right)\right]
  \otimes_{(H/B^+H)^\ast} m}  \\
&\overset{(\ref{eqn:iota*zeta*})}{=}&
\overline{\left[\sum r_{\la0\ra}\otimes\zeta^\ast(r_{\la1\ra})\right]
\otimes_{(H/B^+H)^\ast}m}.
\end{eqnarray*}

On the other hand, denote $e:=\e\#1$ as usual. It follows from (\ref{eqn:e}) that
\begin{equation}\label{eqn:e<0>r}
\sum r_{\la0\ra}\otimes r_{\la1\ra}=\sum e_{\la0\ra}r\otimes e_{\la1\ra}
\in R\otimes B^\ast\;\;\;\;\;\;(\forall r\in R),
\end{equation}
and we calculate:
\begin{eqnarray*}
&& J_{M,R\square_{B^\ast}H^\ast}^{-1}
  \left(\overline{m\otimes_{(H/B^+H)^\ast}
  \left[\sum r_{\la0\ra}\otimes\zeta^\ast(r_{\la1\ra})\right]}\right)  \\
&\overset{(\ref{eqn:monoidalstruJ})}{=}&
\sum\overline{m_{(0)}}\otimes
  \overline{\gamma^\ast(m_{(1)})\cdot
  [r_{\la0\ra}\otimes\zeta^\ast(r_{\la1\ra})]}  \\
&\overset{(\ref{eqn:Psi(V)mod})}{=}&
\sum\overline{m_{(0)}}\otimes
  \overline{\gamma^\ast(m_{(1)})_{(1)}r_{\la0\ra}
  \otimes \gamma^\ast(m_{(1)})_{(2)}\zeta^\ast(r_{\la1\ra})}  \\
&\overset{(\ref{eqn:e<0>r})}{=}&
\sum\overline{m_{(0)}}\otimes
  \overline{\gamma^\ast(m_{(1)})_{(1)}e_{\la0\ra}r
  \otimes \gamma^\ast(m_{(1)})_{(2)}\zeta^\ast(e_{\la1\ra})}  \\
&\overset{(\ref{eqn:convolution*})}{=}&
\sum\overline{m_{(0)}}  \\
&& \;\;\;\;\;\; \otimes\,
  \overline{\gamma^\ast(m_{(1)})_{(1)}e_{\la0\ra}r\otimes
  \zeta^\ast\left(\iota^\ast[\gamma^\ast(m_{(1)})_{(2)}\zeta^\ast(e_{\la1\ra})_{(1)}]\right)
  \gamma^\ast\left(\pi^\ast[\gamma^\ast(m_{(1)})_{(3)}\zeta^\ast(e_{\la1\ra})_{(2)}]\right)}  \\
&\overset{(\ref{eqn:mfbar})}{=}&
\sum\overline{m_{(0)}}\otimes
  \overline{\gamma^\ast(m_{(1)})_{(1)}e_{\la0\ra}r\otimes
  \zeta^\ast\left(\iota^\ast[\gamma^\ast(m_{(1)})_{(2)}\zeta^\ast(e_{\la1\ra})]\right)} \\
&\overset{(\ref{eqn:btr})}{=}&
\sum\overline{m_{(0)}}\otimes
  \overline{\gamma^\ast(m_{(1)})_{(1)}e_{\la0\ra}r\otimes
  \zeta^\ast\left(\gamma^\ast(m_{(1)})_{(2)}\btr e_{\la1\ra}\right)}  \\
&\overset{(\ref{eqn:e<0>f})}{=}&
\sum\overline{m_{(0)}}\otimes
  \overline{e_{\la0\ra}\gamma^\ast(m_{(1)})r\otimes\zeta^\ast(e_{\la1\ra})},
\end{eqnarray*}
where we have omitted notation $\sum$ of sum in some equations.
\end{proof}

Now we aim to find the associator $\pd{\phi}$ of $(H/B^+H)^\ast\#H$ and its inverse, by describing in details the hexagon diagram of the tensor functor $\Phi$:

\begin{proof}
[Proofs of Remark \ref{rmk:phi^-1} and Theorem \ref{thm:partialdual}(3)]~

We conclude in Corollary \ref{cor:tensorequiv} that $\Phi=\overline{(-)}$ is a tensor functor with monoidal structure $J$. This is equivalent to say that there is an associativity constraint $\phi$ such that
the following hexagon diagram in $\Rep((H/B^+H)^\ast\#B)$ commutes:
\begin{equation}\label{eqn:hexagon}
\xymatrix{
(\overline{M}\otimes\overline{N})\otimes\overline{P}
  \ar[r]^{\phi} \ar[d]_{J_{M,N}\otimes\id}
& \overline{M}\otimes(\overline{N}\otimes\overline{P})
  \ar[d]^{\id\otimes J_{N,P}}  \\
\overline{M\otimes_{(H/B^+H)^\ast}N}\otimes\overline{P}
  \ar[d]_{J_{M\otimes_{(H/B^+H)^\ast}N,P}}
& \overline{M}\otimes\overline{N\otimes_{(H/B^+H)^\ast}P}
  \ar[d]^{J_{M,N\otimes_{(H/B^+H)^\ast}P}}  \\
\overline{M\otimes_{(H/B^+H)^\ast}N\otimes_{(H/B^+H)^\ast}P}
  \ar[r]^{=}
& \overline{M\otimes_{(H/B^+H)^\ast}N\otimes_{(H/B^+H)^\ast}P}
}
\end{equation}
for all $M,N,P\in{}_{(H/B^+H)^\ast}\mathfrak{M}_{(H/B^+H)^\ast}^{H^\ast}$.

However, instead of the associator $\pd{\phi}$, it is easier to determine its inverse $\pd{\phi}^{-1}$, which should have the form (\ref{eqn:phi^-1}) in Remark \ref{rmk:phi^-1}. For the purpose, we still denote by $e=\e\#1\in R$ the unit element, and the inverse of the associator is supposed to be
$\pd{\phi}^{-1}=\phi^{-1}(e\otimes e\otimes e)\in R^{\otimes3}$.

Now choose
$M=N=P=R\square_{B^\ast}H^\ast$ and change the diagram (\ref{eqn:hexagon}) above into
\begin{equation}\label{eqn:hexagonphi^-1}
\xymatrix{
\overline{R\square_{B^\ast}H^\ast}^{\otimes3}
&& \overline{R\square_{B^\ast}H^\ast}^{\otimes3}
  \ar@{-->}[ll]_{\phi^{-1}}
  \ar[d]^{\id\otimes J_{M,M}}  \\
\overline{(R\square_{B^\ast}H^\ast)
  {\hspace{-10pt}\underset{(H/B^+H)^\ast}{\otimes}\hspace{-10pt}}(R\square_{B^\ast}H^\ast)}
  \otimes\overline{R\square_{B^\ast}H^\ast}
  \ar[u]_{J_{M,M}^{-1}\otimes\id}
&& \overline{R\square_{B^\ast}H^\ast}\otimes
  \overline{(R\square_{B^\ast}H^\ast)
  {\hspace{-10pt}\underset{(H/B^+H)^\ast}{\otimes}\hspace{-10pt}}(R\square_{B^\ast}H^\ast)}
  \ar@<-2ex>[dl]^{\;\;\;\;\;\;\;\;J_{M,M\otimes_{(H/B^+H)^\ast}M}}  \\
& \hspace{-600pt}
  \overline{(R\square_{B^\ast}H^\ast)
  {\hspace{-10pt}\underset{(H/B^+H)^\ast}{\otimes}\hspace{-10pt}}(R\square_{B^\ast}H^\ast)
  {\hspace{-10pt}\underset{(H/B^+H)^\ast}{\otimes}\hspace{-10pt}}(R\square_{B^\ast}H^\ast)}.
  \hspace{-600pt}
  \ar@<-2ex>[ul]^{J_{M\otimes_{(H/B^+H)^\ast}M,M}^{-1}\;\;\;\;\;\;\;\;}
}
\end{equation}
This is because the regular $(H/B^+H)^\ast\#B$-module $R$ is isomorphic to $\overline{R\square_{B^\ast}H^\ast}$ via $\sigma$.
As a consequence of the commutativity of the diagram (\ref{eqn:hexagonphi^-1}),
it follows by
$\sigma_R^{-1}(e)=\overline{\sum e_{\la0\ra}\otimes\zeta^\ast(e_{\la1\ra})}$ that
\begin{eqnarray}\label{eqn:computephi^-1}
&& \pd{\phi}^{-1}  \nonumber  \\
&=&
{\sigma_R}^{\otimes3}\circ\phi^{-1}
\left(\sigma_R^{-1}(e)\otimes\sigma_R^{-1}(e)\otimes\sigma_R^{-1}(e)\right)
\nonumber  \\
&\overset{(\ref{eqn:hexagonphi^-1})}{=}&
{\sigma_R}^{\otimes3}\circ(J_{M,M}^{-1}\otimes\id)\circ J_{M\otimes_{(H/B^+H)^\ast}M,M}^{-1}
  \circ J_{M,M\otimes_{(H/B^+H)^\ast}M}\circ(\id\otimes J_{M,M})
  \left(\sigma_R^{-1}(e)^{\otimes3}\right).  \nonumber  \\
\end{eqnarray}

Now we denote $e=e'=e''=\e\#1$ as usual, and calculate the preimage of (\ref{eqn:computephi^-1}) under ${\sigma_R}^{\otimes3}$, where the notation of sum $\sum$ is omitted for simplicity:
\begin{eqnarray*}
&&
\phi^{-1}\left(\overline{e_{\la0\ra}\otimes\zeta^\ast(e_{\la1\ra})}
\otimes\overline{e'_{\la0\ra}\otimes\zeta^\ast(e'_{\la1\ra})}
\otimes\overline{e''_{\la0\ra}\otimes\zeta^\ast(e''_{\la1\ra})}\right)  \\
&=&
(J^{-1}\otimes\id)\circ J^{-1}\circ J\circ(\id\otimes J)
\left(\overline{e_{\la0\ra}\otimes\zeta^\ast(e_{\la1\ra})}
\otimes\overline{e'_{\la0\ra}\otimes\zeta^\ast(e'_{\la1\ra})}
\otimes\overline{e''_{\la0\ra}\otimes\zeta^\ast(e''_{\la1\ra})}\right)  \\
&\overset{(\ref{eqn:Jbar1})}{=}&
(J^{-1}\otimes\id)\circ J^{-1}\circ J
\left(\overline{e_{\la0\ra}\otimes\zeta^\ast(e_{\la1\ra})}
\otimes\overline{[e'_{\la0\ra}\otimes\zeta^\ast(e'_{\la1\ra})]
\otimes_{(H/B^+H)^\ast}[e''_{\la0\ra}\otimes\zeta^\ast(e''_{\la1\ra})]}\right)  \\
&\overset{(\ref{eqn:Jbar1})}{=}&
(J^{-1}\otimes\id)\circ J^{-1}
\left(\overline{[e_{\la0\ra}\otimes\zeta^\ast(e_{\la1\ra})]
\otimes_{(H/B^+H)^\ast}[e'_{\la0\ra}\otimes\zeta^\ast(e'_{\la1\ra})]
\otimes_{(H/B^+H)^\ast}[e''_{\la0\ra}\otimes\zeta^\ast(e''_{\la1\ra})]}\right)  \\
&\overset{(\ref{eqn:Jbar2})}{=}&
(J^{-1}\otimes\id)  \\
&&
\left(\overline{[e_{\la0\ra}\otimes\zeta^\ast(e_{\la1\ra})_{(1)}]
\otimes_{(H/B^+H)^\ast}[e'_{\la0\ra}\otimes\zeta^\ast(e'_{\la1\ra})_{(1)}]}
\otimes
\overline{e''_{\la0\ra}
\gamma^\ast[\zeta^\ast(e_{\la1\ra})_{(2)}\zeta^\ast(e'_{\la1\ra})_{(2)}]
\otimes\zeta^\ast(e''_{\la1\ra})}\right)  \\
&\overset{(\ref{eqn:iota*})}{=}&
(J^{-1}\otimes\id)  \\
&&
\left(\overline{[e_{\la0\ra}\otimes\zeta^\ast(e_{\la1\ra})_{(1)}]
\otimes_{(H/B^+H)^\ast}[e'_{\la0\ra}\otimes\zeta^\ast(e'_{\la1\ra})]}
\otimes
\overline{e''_{\la0\ra}
\gamma^\ast[\zeta^\ast(e_{\la1\ra})_{(2)}\zeta^\ast(e'_{\la2\ra})]
\otimes\zeta^\ast(e''_{\la1\ra})}\right)  \\
&\overset{(\ref{eqn:Jbar2})}{=}&
\overline{e_{\la0\ra}\otimes\zeta^\ast(e_{\la1\ra})_{(1)}}
\otimes
\overline{E_{\la0\ra}\gamma^\ast[\zeta^\ast(e_{\la1\ra})_{(2)}]e'_{\la0\ra}
\otimes\zeta^\ast(E_{\la1\ra})}
\otimes
\overline{e''_{\la0\ra}
\gamma^\ast[\zeta^\ast(e_{\la1\ra})_{(2)}\zeta^\ast(e'_{\la1\ra})]
\otimes\zeta^\ast(e''_{\la1\ra})},
\end{eqnarray*}
where $E:=\e\#1$ as well. Its image under ${\sigma_R}^{\otimes3}$ then becomes
$$\pd{\phi}^{-1}=\sum
e_{\la0\ra}\otimes\gamma^\ast[\zeta^\ast(e_{\la1\ra})_{(1)}]e'_{\la0\ra}
\otimes\gamma^\ast[\zeta^\ast(e_{\la1\ra})_{(2)}\zeta^\ast(e'_{\la1\ra})],$$
which is the same as the form in Remark \ref{rmk:phi^-1}, according to Formula (\ref{eqn:e}).

Of course, there is a similar argument implying that the associator
$\pd{\phi}=\phi(e\otimes e\otimes e)$ equals to (\ref{eqn:associator}) in Theorem \ref{thm:partialdual}(3), but which includes much more complicated calculations. Therefore,
we aim to compute to check that $\pd{\phi}^{-1}\pd{\phi}=e\otimes e\otimes e$ holds for the element
$$\pd{\phi}
=\sum E_{\la0\ra}\otimes
  E'_{\la0\ra}\overline{\gamma}^\ast[S^{-1}(\overline{\zeta}^\ast(E_{\la1\ra})_{(1)})]
  \otimes \overline{\gamma}^\ast[S^{-1}(\overline{\zeta}^\ast(E'_{\la1\ra})
    \overline{\zeta}^\ast(E_{\la1\ra})_{(2)})]$$
as in (\ref{eqn:associator}), where $E=E'=\e\#1$. It follows that $\pd{\phi}$ and $\pd{\phi}^{-1}$ are mutually inverse since the algebra $(H/B^+H)^\ast\#B$ is finite-dimensional.

For the purpose, an additional formula should be noted: For all $b^\ast\in B^\ast$ and $h^\ast\in H^\ast$,
\begin{eqnarray}\label{eqn:zeta*pi*gamma*bar}
&& \sum\zeta^\ast(b^\ast_{(1)})
  \pi^\ast\left(
  \overline{\gamma}^\ast[S^{-1}(\overline{\zeta}^\ast(b^\ast_{(2)})h^\ast)]\right) \nonumber  \\
&\overset{(\ref{eqn:convolution*2})}{=}&
\sum\zeta^\ast(b^\ast_{(1)})\overline{\zeta}^\ast(b^\ast_{(2)})_{(2)}h^\ast_{(2)}
  \zeta^\ast\left(\iota^\ast
  [S^{-1}(\overline{\zeta}^\ast(b^\ast_{(2)})_{(1)}h^\ast_{(1)})]\right)  \nonumber  \\
&\overset{(\ref{eqn:iota*})}{=}&
\sum\zeta^\ast(b^\ast_{(1)})\overline{\zeta}^\ast(b^\ast_{(2)})_{(2)}h^\ast_{(2)}
  \zeta^\ast\left(S^{-1}(h^\ast_{(1)})\btr
  \iota^\ast[S^{-1}(\overline{\zeta}^\ast(b^\ast_{(2)})_{(1)})]\right)  \nonumber  \\
&\overset{(\ref{eqn:zeta*bar2})}{=}&
\sum\zeta^\ast(b^\ast_{(1)})\overline{\zeta}^\ast(b^\ast_{(2)})h^\ast_{(2)}
  \zeta^\ast(S^{-1}(h^\ast_{(1)})\btr b^\ast_{(3)})  \nonumber  \\
&=&
\sum h^\ast_{(2)}\zeta^\ast(S^{-1}(h^\ast_{(1)})\btr b^\ast).
\end{eqnarray}
Then we make desired calculations for $\pd{\phi}^{-1}$ and $\pd{\phi}$ with forms above:
\begin{eqnarray*}
&&  \pd{\phi}^{-1}\pd{\phi}  \\
&=&
\sum e_{\la0\ra}E_{\la0\ra}
  \otimes \gamma^\ast[\zeta^\ast(e_{\la1\ra})_{(1)}]e'_{\la0\ra}
    E'_{\la0\ra}\overline{\gamma}^\ast[S^{-1}(\overline{\zeta}^\ast(E_{\la1\ra})_{(1)})]  \\
&& \;\;\;\;\;\;
  \otimes \gamma^\ast[\zeta^\ast(e_{\la1\ra})_{(2)}\zeta^\ast(e'_{\la1\ra})]
    \overline{\gamma}^\ast[S^{-1}(\overline{\zeta}^\ast(E'_{\la1\ra})
    \overline{\zeta}^\ast(E_{\la1\ra})_{(2)})]  \\
&\overset{(\ref{eqn:e<0>e'<0>})}{=}&
\sum e_{\la0\ra}
  \otimes \gamma^\ast[\zeta^\ast(e_{\la1\ra})_{(1)}]e'_{\la0\ra}
    \overline{\gamma}^\ast[S^{-1}(\overline{\zeta}^\ast(e_{\la2\ra})_{(1)})]  \\
&& \;\;\;\;\;\;
  \otimes \gamma^\ast[\zeta^\ast(e_{\la1\ra})_{(2)}\zeta^\ast(e'_{\la1\ra})]
    \overline{\gamma}^\ast[S^{-1}(\overline{\zeta}^\ast(e'_{\la2\ra})
    \overline{\zeta}^\ast(e_{\la2\ra})_{(2)})]  \\
&\overset{(\ref{eqn:gamma*})}{=}&
\sum e_{\la0\ra}
  \otimes \gamma^\ast[\zeta^\ast(e_{\la1\ra})_{(1)}]e'_{\la0\ra}
    \overline{\gamma}^\ast[S^{-1}(\overline{\zeta}^\ast(e_{\la2\ra})_{(1)})]  \\
&& \;\;\;\;\;\;
  \otimes \gamma^\ast\left[\zeta^\ast(e_{\la1\ra})_{(2)}\zeta^\ast(e'_{\la1\ra})
    \pi^\ast\left(\overline{\gamma}^\ast[S^{-1}(\overline{\zeta}^\ast(e'_{\la2\ra})
    \overline{\zeta}^\ast(e_{\la2\ra})_{(2)})]\right)\right]  \\
&\overset{(\ref{eqn:zeta*pi*gamma*bar})}{=}&
\sum e_{\la0\ra}
  \otimes \gamma^\ast[\zeta^\ast(e_{\la1\ra})_{(1)}]e'_{\la0\ra}
    \overline{\gamma}^\ast[S^{-1}(\overline{\zeta}^\ast(e_{\la2\ra})_{(1)})]  \\
&& \;\;\;\;\;\;
  \otimes \gamma^\ast\left[\zeta^\ast(e_{\la1\ra})_{(2)}
    \overline{\zeta}^\ast(e_{\la2\ra})_{(3)}
    \zeta^\ast\left(S^{-1}(\overline{\zeta}^\ast(e_{\la2\ra})_{(2)})\btr
    e'_{\la1\ra}\right)\right]  \\
&\overset{(\ref{eqn:e<0>f})}{=}&
\sum e_{\la0\ra}
  \otimes \gamma^\ast[\zeta^\ast(e_{\la1\ra})_{(1)}]
    \overline{\gamma}^\ast[S^{-1}(\overline{\zeta}^\ast(e_{\la2\ra})_{(1)})]_{(1)}
    e'_{\la0\ra}  \\
&& \;\;\;\;\;\;
  \otimes \gamma^\ast\left[\zeta^\ast(e_{\la1\ra})_{(2)}
    \overline{\zeta}^\ast(e_{\la2\ra})_{(3)}
    \zeta^\ast\left(S^{-1}(\overline{\zeta}^\ast(e_{\la2\ra})_{(2)})
    \overline{\gamma}^\ast[S^{-1}(\overline{\zeta}^\ast(e_{\la2\ra})_{(1)})]_{(2)}\btr
    e'_{\la1\ra}\right)\right]  \\
&\overset{(\ref{eqn:gamma*})}{=}&
\sum e_{\la0\ra}
  \otimes \gamma^\ast\left[\zeta^\ast(e_{\la1\ra})_{(1)}\pi^\ast\left(
    \overline{\gamma}^\ast[S^{-1}(\overline{\zeta}^\ast(e_{\la2\ra})_{(1)})]{}_{(1)}\right)\right]
    e'_{\la0\ra}  \\
&& \;\;\;\;\;\;
  \otimes \gamma^\ast\left[\zeta^\ast(e_{\la1\ra})_{(2)}
    \overline{\zeta}^\ast(e_{\la2\ra})_{(3)}
    \zeta^\ast\left(S^{-1}(\overline{\zeta}^\ast(e_{\la2\ra})_{(2)})
    \overline{\gamma}^\ast[S^{-1}(\overline{\zeta}^\ast(e_{\la2\ra})_{(1)})]_{(2)}\btr
    e'_{\la1\ra}\right)\right]  \\
&\overset{(\ref{eqn:pi*})}{=}&
\sum e_{\la0\ra}
  \otimes \gamma^\ast\left[\zeta^\ast(e_{\la1\ra})_{(1)}\pi^\ast\left(
    \overline{\gamma}^\ast[S^{-1}(\overline{\zeta}^\ast(e_{\la2\ra})_{(1)})]\right){}_{(1)}\right]
    e'_{\la0\ra}  \\
&& \;\;\;\;\;\;
  \otimes \gamma^\ast\left[\zeta^\ast(e_{\la1\ra})_{(2)}
    \overline{\zeta}^\ast(e_{\la2\ra})_{(3)}
    \zeta^\ast\left(S^{-1}(\overline{\zeta}^\ast(e_{\la2\ra})_{(2)})
    \pi^\ast\left(
    \overline{\gamma}^\ast[S^{-1}(\overline{\zeta}^\ast(e_{\la2\ra})_{(1)})]\right){}_{(2)}
    \btr e'_{\la1\ra}\right)\right]  \\
&=&
\sum e_{\la0\ra}
  \otimes \gamma^\ast\left[\zeta^\ast(e_{\la1\ra})_{(1)}
    \overline{\zeta}^\ast(e_{\la2\ra})_{(4)}S^{-1}(\overline{\zeta}^\ast(e_{\la2\ra})_{(3)})
    \pi^\ast\left(
    \overline{\gamma}^\ast[S^{-1}(\overline{\zeta}^\ast(e_{\la2\ra})_{(1)})]\right){}_{(1)}\right]
    e'_{\la0\ra}  \\
&& \;\;\;\;\;\;
  \otimes \gamma^\ast\left[\zeta^\ast(e_{\la1\ra})_{(2)}
    \overline{\zeta}^\ast(e_{\la2\ra})_{(3)}
    \zeta^\ast\left(S^{-1}(\overline{\zeta}^\ast(e_{\la2\ra})_{(2)})
    \pi^\ast\left(
    \overline{\gamma}^\ast[S^{-1}(\overline{\zeta}^\ast(e_{\la2\ra})_{(1)})]\right){}_{(2)}
    \btr e'_{\la1\ra}\right)\right]  \\
&\overset{(\ref{eqn:convolution*1})}{=}&
\sum e_{\la0\ra}
  \otimes \gamma^\ast\left[\zeta^\ast(e_{\la1\ra})_{(1)}
    \overline{\zeta}^\ast(e_{\la2\ra})_{(2)}
    \zeta^\ast\left(
    \iota^\ast[S^{-1}(\overline{\zeta}^\ast(e_{\la2\ra})_{(1)})]\right){}_{(1)}\right]
    e'_{\la0\ra}  \\
&& \;\;\;\;\;\;
  \otimes \gamma^\ast\left[\zeta^\ast(e_{\la1\ra})_{(2)}
    \overline{\zeta}^\ast(e_{\la2\ra})_{(3)}
    \zeta^\ast\left(\zeta^\ast\left(
    \iota^\ast[S^{-1}(\overline{\zeta}^\ast(e_{\la2\ra})_{(1)})]\right){}_{(2)}
    \btr e'_{\la1\ra}\right)\right]  \\
&\overset{(\ref{eqn:zeta*bar2})}{=}&
\sum e_{\la0\ra}
  \otimes \gamma^\ast\left[\zeta^\ast(e_{\la1\ra})_{(1)}
    \overline{\zeta}^\ast(e_{\la2\ra})_{(1)}
    \zeta^\ast(e_{\la3\ra})_{(1)}\right]
    e'_{\la0\ra}  \\
&& \;\;\;\;\;\;
  \otimes \gamma^\ast\left[\zeta^\ast(e_{\la1\ra})_{(2)}
    \overline{\zeta}^\ast(e_{\la2\ra})_{(2)}
    \zeta^\ast\left(\zeta^\ast(e_{\la3\ra})_{(2)}
    \btr e'_{\la1\ra}\right)\right]  \\
&=&
\sum e_{\la0\ra}
  \otimes \gamma^\ast[\zeta^\ast(e_{\la1\ra})_{(1)}]e'_{\la0\ra}
  \otimes \gamma^\ast[\zeta^\ast(\zeta^\ast(e_{\la1\ra})_{(2)}\btr e'_{\la1\ra})]  \\
&\overset{(\ref{eqn:gamma*zeta*})}{=}&
\sum e_{\la0\ra}\otimes \gamma^\ast[\zeta^\ast(e_{\la1\ra})]e'\otimes e''
~\overset{(\ref{eqn:gamma*zeta*})}{=}~
e\otimes e\otimes e.
\end{eqnarray*}
\end{proof}

\subsection{Dual objects, and the proof of Theorem \ref{thm:partialdual}(4) - antipodes}\label{subsection:Thm(4)Pf}

Similarly to the proofs in the previous subsection, we still consider the particular objects $M=R\square_{B^\ast}H^\ast$ in the tensor category ${}_{(H/B^+H)^\ast}\mathfrak{M}_{(H/B^+H)^\ast}^{H^\ast}$. Moreover, its left dual object $M^\vee=(R\square_{B^\ast}H^\ast)^\vee$ should be also dealt with in order to determine the antipodes 
of the quasi-Hopf algebra $(H/B^+H)^\ast\#B$. Let us try to answer this question for an arbitrary representation $V$:

\begin{lemma}\label{lem:Mdualbases}
Suppose $\{v_i\}$ is a linear basis of $V\in\Rep((H/B^+H)^\ast\#B)$ with dual linear basis $\{v_i^\ast\}$ of $V^\ast$. If we denote
\begin{equation}\label{eqn:Mbasis}
m_i=\sum {v_i}_{\la0\ra}\otimes \zeta^\ast({v_i}_{\la1\ra})\in V\square_{B^\ast}H^\ast
\;\;\;\;(\forall i),
\end{equation}
then $\{m_i\}$ is a $(H/B^+H)^\ast$-basis of the free right $(H/B^+H)^\ast$-module $M=V\square_{B^\ast}H^\ast$, with the dual $(H/B^+H)^\ast$-basis $\{m_i^\vee\}$ defined by
\begin{equation}\label{eqn:Mdualbasis}
m_i^\vee=v_i^\ast\otimes\gamma^\ast:V\square_{B^\ast}H^\ast\rightarrow(H/B^+H)^\ast
\;\;\;\;(\forall i).
\end{equation}
\end{lemma}

\begin{proof}
It has been mentioned in the proof of Proposition \ref{prop:monoidalstru0}(1) that $M=V\square_{B^\ast}H^\ast$ is free as a right $(H/B^+H)^\ast$-module according to \cite[Theorem 2.1(4)]{Mas92}. In fact, recall in \cite[Theorem 2.3(ii)(b)]{MD92} that there is an isomorphism
\begin{equation}\label{eqn:H*iso}
H^\ast\cong B^\ast\otimes (H/B^+H)^\ast,\;\;\;\;
h^\ast\mapsto\sum\iota^\ast(h^\ast_{(1)})\otimes\gamma^\ast(h_{(2)})
\end{equation}
of left $B^\ast$-comodules as well as right $(H/B^+H)^\ast$-modules.
Therefore, the following composition of canonical isomorphisms
\begin{eqnarray*}
V\square_{B^\ast}H^\ast
&\overset{\id\otimes(\ref{eqn:H*iso})}{\cong}&
V\square_{B^\ast}(B^\ast\otimes (H/B^+H)^\ast)  \\
&=& (V\square_{B^\ast}B^\ast)\otimes (H/B^+H)^\ast
\overset{\id\otimes\la-,1\ra\otimes\id}{\cong}
V\otimes (H/B^+H)^\ast
\end{eqnarray*}
preserves right $(H/B^+H)^\ast$-actions (where the equality is due to a canonical isomorphism in the second paragraph in \cite[Page 632]{Tak77(b)}), by which the element
$\sum v_{\la0\ra}\otimes\zeta^\ast(v_{\la1\ra})$
is mapped to
\begin{eqnarray*}
\sum v_{\la0\ra}\left\langle\iota^\ast[\zeta^\ast(v_{\la1\ra})_{(1)}],1\right\rangle
  \otimes\gamma^\ast[\zeta^\ast(v_{\la1\ra})_{(2)}]
&=&
\sum v_{\la0\ra}\otimes\gamma^\ast[\zeta^\ast(v_{\la1\ra})]  \\
&\overset{(\ref{eqn:gamma*zeta*})}{=}&
v\otimes\e
\end{eqnarray*}
for each $v\in V$.
This implies that $\{m_i\}$ defined in (\ref{eqn:Mbasis}) is a free $(H/B^+H)^\ast$-basis as desired, since their images $\{v_i\otimes\e\}$ is evidently a basis of the right free module $V\otimes (H/B^+H)^\ast$.

On the other hand, note by (\ref{eqn:gamma*}) that $\gamma^\ast$ is a right $(H/B^+H)^\ast$-module map, and so is $m_i^\vee=v_i^\ast\otimes\gamma^\ast$ defined in (\ref{eqn:Mdualbasis}). Thus it suffices to show that
$m_i^\vee(m_j)=\delta_{ij}\e$
holds for all $i,j$, where $\delta$ is the Kronecker notation, and this is due to the following computations:
\begin{eqnarray*}
m_i^\vee(m_j)
&\overset{(\ref{eqn:Mdualbasis}),\;(\ref{eqn:Mbasis})}{=}&
(v_i^\ast\otimes\gamma^\ast)
\left(\sum {v_j}_{\la0\ra}\otimes \zeta^\ast({v_j}_{\la1\ra})\right)
~=~
\sum \la v_i^\ast,{v_j}_{\la0\ra}\ra\gamma^\ast[\zeta^\ast({v_j}_{\la1\ra})]  \\
&\overset{(\ref{eqn:gamma*zeta*})}{=}&
\sum \la v_i^\ast,{v_j}_{\la0\ra}\ra\la{v_j}_{\la1\ra},1\ra\e
~=~
\la v_i^\ast,v_j\ra\e
~=~
\delta_{ij}\e.
\end{eqnarray*}
\end{proof}

There are additional useful formulas on more general elements in $M=V\square_{B^\ast}H^\ast$ and $M^\vee=(V\square_{B^\ast}H^\ast)^\vee$, which would be slightly simplified by an evident equation: Denote $e:=\e\#1$ as usual, then it is similar to (\ref{eqn:e<0>r}) that
\begin{equation}\label{eqn:e<0>v}
\sum v_{\la0\ra}\otimes v_{\la1\ra}=\sum e_{\la0\ra}v\otimes e_{\la1\ra}
\in V\otimes B^\ast
\end{equation}
holds for each element $v$ in an arbitrary $(H/B^+H)^\ast\#B$-module $V$.

\begin{lemma}
Let $R:=(H/B^+H)^\ast\#B$ be the regular module in $\Rep((H/B^+H)^\ast\#B)$. Suppose $r\in R$ and denote $m:=\sum r_{\la0\ra}\otimes\zeta^\ast(r_{\la1\ra})\in R\square_{B^\ast}H^\ast$.
\begin{itemize}
\item[(1)]
We have
\begin{equation}\label{eqn:m0m1m2}
\sum m_{(0)}\overline{\gamma}^\ast(m_{(1)})\otimes m_{(2)}
=\sum\left(e_{\la0\ra}r\otimes\zeta^\ast(e_{\la1\ra})\right)\otimes\zeta^\ast(e_{\la2\ra})
\in(R\square_{B^\ast}H^\ast)\otimes H^\ast;
\end{equation}

\item[(2)]
For any $r^\ast\in R^\ast$, if we denote
$m^\vee:=r^\ast\otimes\gamma^\ast\in(R\square_{B^\ast}H^\ast)^\vee$, then
\begin{equation}\label{eqn:pi*(m0)S(m1)}
\sum \pi^\ast[m^\vee(m_{(0)})]S(m_{(1)})
=\sum \la r^\ast,e_{\la0\ra}r\ra\overline{\zeta}^\ast(e_{\la1\ra})\in H^\ast.
\end{equation}
\end{itemize}
\end{lemma}

\begin{proof}
\begin{itemize}
\item[(1)]
Recall in (\ref{eqn:right(co)actions}) that the right $H^\ast$-coaction on $m\in R\square_{B^\ast}H^\ast$ provides
\begin{eqnarray}\label{eqn:m0m1}
\sum m_{(0)}\otimes m_{(1)}
&\overset{(\ref{eqn:right(co)actions})}{=}&
\sum\left(r_{\la0\ra}\otimes\zeta^\ast(r_{\la1\ra})_{(1)}\right)
  \otimes\zeta^\ast(r_{\la1\ra})_{(2)}  \nonumber  \\
&\overset{(\ref{eqn:e<0>v})}{=}&
\sum\left(e_{\la0\ra}r\otimes\zeta^\ast(e_{\la1\ra})_{(1)}\right)
  \otimes\zeta^\ast(e_{\la1\ra})_{(2)}  \\
&& \in\;(R\square_{B^\ast}H^\ast)\otimes H^\ast,  \nonumber
\end{eqnarray}
and hence
$$\sum m_{(0)}\otimes m_{(1)}\otimes m_{(2)}
=\sum\left(e_{\la0\ra}r\otimes\zeta^\ast(e_{\la1\ra})_{(1)}\right)
  \otimes\zeta^\ast(e_{\la1\ra})_{(2)}\otimes\zeta^\ast(e_{\la1\ra})_{(3)}.$$
Consequently, we could calculate that
\begin{eqnarray*}
\sum m_{(0)}\overline{\gamma}^\ast(m_{(1)})\otimes m_{(2)}
&\overset{(\ref{eqn:right(co)actions})}{=}&
\sum e_{\la0\ra}r\otimes\zeta^\ast(e_{\la1\ra})_{(1)}
  \pi^\ast(\overline{\gamma}^\ast[\zeta^\ast(e_{\la1\ra})_{(2)}])
  \otimes\zeta^\ast(r_{\la1\ra})_{(3)}  \\
&\overset{(\ref{eqn:convolution*1})}{=}&
\sum e_{\la0\ra}r\otimes\zeta^\ast(\iota^\ast[\zeta^\ast(e_{\la1\ra})_{(1)}])
  \otimes\zeta^\ast(e_{\la1\ra})_{(2)}  \\
&\overset{(\ref{eqn:iota*})}{=}&
\sum e_{\la0\ra}r\otimes\zeta^\ast(e_{\la1\ra})
  \otimes\zeta^\ast(e_{\la2\ra}).
\end{eqnarray*}

\item[(2)]
This is also due to direct calculations:
\begin{eqnarray*}
\sum \pi^\ast[m^\vee(m_{(0)})]S(m_{(1)})
&\overset{(\ref{eqn:convolution*2})}{=}&
\sum \pi^\ast[m^\vee(m_{(0)})]\pi^\ast[\overline{\gamma}^\ast(m_{(1)})]
  \overline{\zeta}^\ast[\iota^\ast(m_{(2)})]  \\
&=&
\sum \pi^\ast[m^\vee\left(m_{(0)}\overline{\gamma}^\ast(m_{(1)})\right)]
  \overline{\zeta}^\ast[\iota^\ast(m_{(2)})]  \\
&\overset{(\ref{eqn:m0m1m2})}{=}&
\sum \pi^\ast[(r^\ast\otimes\gamma^\ast)
  \left(e_{\la0\ra}r\otimes\zeta^\ast(e_{\la1\ra})\right)]
  \overline{\zeta}^\ast(\iota^\ast[\zeta^\ast(e_{\la2\ra})])  \\
&=&
\sum \la r^\ast,e_{\la0\ra}r\ra\pi^\ast(\gamma^\ast[\zeta^\ast(e_{\la1\ra})])
  \overline{\zeta}^\ast(e_{\la2\ra})  \\
&\overset{(\ref{eqn:gamma*zeta*})}{=}&
\sum \la r^\ast,e_{\la0\ra}r\ra\overline{\zeta}^\ast(e_{\la1\ra}).
\end{eqnarray*}
\end{itemize}
\end{proof}

Afterwards, we might specify some further morphisms involving the evaluation $\ev_M$ and coevaluation $\coev_M$ (\ref{eqn:Mevcoev}) for the object
$M=R\square_{B^\ast}H^\ast\in{}_{(H/B^+H)^\ast}\mathfrak{M}_{(H/B^+H)^\ast}^{H^\ast}$:

\begin{corollary}\label{cor:Mbarevcoev}
Let $\{r_i\}$ be a linear basis of $R:=(H/B^+H)^\ast\#B$ with dual basis $\{r_i^\ast\}$ of $R$. Denote $M:=R\square_{B^\ast}H^\ast$ and $m_i:=\sum {r_i}_{\la0\ra}\otimes \zeta^\ast({r_i}_{\la1\ra})$, and suppose $\{m_i^\vee\}$ is the $(H/B^+H)^\ast$-basis of $M^\vee$ dual to $\{m_i\}$ of $M$ as usual.
\begin{itemize}
\item[(1)]
Suppose $r^\ast\in R^\ast$ and $m^\vee:=r^\ast\otimes\gamma^\ast\in M^\vee$. The following composition map
$$\overline{M^\vee}\otimes\overline{M}
\xrightarrow{J_{M^\vee,M}} \overline{M^\vee\otimes_{(H/B^+H)^\ast}M}
\xrightarrow{\overline{\ev_M}}\overline{(H/B^+H)^\ast}\cong\k$$
satisfies that: For any $f\in (H/B^+H)^\ast$ and $b\in B$,
\begin{equation}\label{eqn:Mbarev}
\overline{\ev_M}\circ J\Big([(f\#b)\overline{m^\vee}]\otimes\overline{m}\Big)
=\sum_j \la r^\ast,e_{\la0\ra}{r_j}\ra
\overline{m_j^\vee\left(\overline{\gamma}^\ast\big[\pi^\ast(f)\big(
\iota(b)\rightharpoonup
\overline{\zeta}^\ast(e_{\la1\ra})\big)\big]m\right)}.
\end{equation}

\item[(2)]
The following composition map
$$\k\cong\overline{(H/B^+H)^\ast}
\xrightarrow{\overline{\coev_M}} \overline{M\otimes_{(H/B^+H)^\ast}M^\vee}
\xrightarrow{J_{M,M^\vee}^{-1}} \overline{M}\otimes\overline{M^\vee}$$
satisfies that:
\begin{equation}\label{eqn:Mbarcoev}
J_{M,M^\vee}^{-1}\circ \overline{\coev_M}(1)
=\sum_i \overline{e_{\la0\ra}r_i\otimes\zeta^\ast(e_{\la1\ra})}
  \otimes\overline{m_i^\vee},
\end{equation}
where $e=\e\#1$.
\end{itemize}
\end{corollary}

\begin{proof}
\begin{itemize}
\item[(1)]
By recalling the monoidal structures introduced in Proposition \ref{prop:monoidalstru0} and Lemma \ref{lem:monoidalstruJ}:
\begin{eqnarray*}
&&
\overline{\ev_M}\circ J\Big([(f\#b)\overline{m^\vee}]\otimes\overline{m}\Big)  \\
&\overset{(\ref{eqn:Bhit}),\;(\ref{eqn:Phi(M)comod})}{=}&
\overline{\ev_M}\circ J\Big(\sum\overline{fm^\vee_{(0)}}
  \la \iota^\ast(m^\vee_{(1)}),b\ra\otimes\overline{m}\Big)  \\
&\overset{(\ref{eqn:monoidalstruJ})}{=}&
\overline{\ev_M}\Big(\sum\overline{f_{(1)}m^\vee_{(0)}
  \overline{\gamma}^\ast(f_{(2)}m^\vee_{(1)})\la m^\vee_{(2)},\iota(b)\ra
  \otimes_{(H/B^+H)^\ast} m}\Big)  \\
&\overset{(\ref{eqn:Mevcoev})}{=}&
\sum\overline{f_{(1)}m^\vee_{(0)}
  \Big(\overline{\gamma}^\ast[f_{(2)}(\iota(b)\rightharpoonup m^\vee_{(1)})]m\Big)}  \\
&\overset{(\ref{eqn:Mdualcomod2})}{=}&
\sum_j \overline{f_{(1)}[m^\vee({m_j}_{(0)})]_{(1)}m_j^\vee
  \Big(\overline{\gamma}^\ast[f_{(2)}(\iota(b)\rightharpoonup
  [m^\vee({m_j}_{(0)})]_{(2)}S({m_j}_{(1)}))]m\Big)}  \\
&\overset{(\ref{eqn:mfbar})}{=}&
\Big\langle \sum_j f_{(1)}[m^\vee({m_j}_{(0)})]_{(1)}m_j^\vee
  \Big(\overline{\gamma}^\ast[f_{(2)}(\iota(b)\rightharpoonup
  [m^\vee({m_j}_{(0)})]_{(2)}S({m_j}_{(1)}))]m\Big),\;1\Big\rangle  \\
&=&
\Big\langle \sum_j m_j^\vee\Big(\overline{\gamma}^\ast[\pi^\ast(f)(\iota(b)\rightharpoonup
  \pi^\ast[m^\vee({m_j}_{(0)})]S({m_j}_{(1)})]m\Big),\;1\Big\rangle  \\
&\overset{(\ref{eqn:mfbar})}{=}&
\sum_j \overline{m_j^\vee\Big(\overline{\gamma}^\ast[\pi^\ast(f)(\iota(b)\rightharpoonup
  \pi^\ast[m^\vee({m_j}_{(0)})]S({m_j}_{(1)}))]m\Big)}  \\
&\overset{(\ref{eqn:pi*(m0)S(m1)})}{=}&
\sum_j \la r^\ast,e_{\la0\ra}r_j\ra
  \overline{m_j^\vee\Big(\overline{\gamma}^\ast[\pi^\ast(f)(\iota(b)\rightharpoonup
  \overline{\zeta}^\ast(e_{\la1\ra}))]m\Big)}.
\end{eqnarray*}

\item[(2)]
Similarly, we calculate:
\begin{eqnarray*}
&& J_{M,M^\vee}^{-1}\circ \overline{\coev_M}(1)
~\overset{(\ref{eqn:Mevcoev})}{=}~
J_{M,M^\vee}^{-1}\Big(\sum_i\overline{m_i\otimes_{(H/B^+H)^\ast}m_i^\vee}\Big)  \\
&\overset{(\ref{eqn:monoidalstruJ^-1})}{=}&
\sum_i \overline{{m_i}_{(0)}}\otimes
  \overline{\gamma^\ast({m_i}_{(1)})m_i^\vee}
~\overset{(\ref{eqn:m0m1})}{=}~
\sum_i \overline{e_{\la0\ra}r_i\otimes\zeta^\ast(e_{\la1\ra})_{(1)}}
  \otimes\overline{\gamma^\ast[\zeta^\ast(e_{\la1\ra})_{(2)}]m_i^\vee}  \\
&\overset{(\ref{eqn:convolution*})}{=}&
\sum_i \overline{e_{\la0\ra}r_i\otimes
  \zeta^\ast(\iota^\ast[\zeta^\ast(e_{\la1\ra})_{(1)}])
  \pi^\ast(\gamma^\ast[\zeta^\ast(e_{\la1\ra})_{(2)}])}
  \otimes\overline{\gamma^\ast[\zeta^\ast(e_{\la1\ra})_{(3)}]m_i^\vee}  \\
&\overset{(\ref{eqn:mfbar})}{=}&
\sum_i \overline{e_{\la0\ra}r_i\otimes
  \zeta^\ast(\iota^\ast[\zeta^\ast(e_{\la1\ra})_{(1)}])}
  \la\gamma^\ast[\zeta^\ast(e_{\la1\ra})_{(2)}],1\ra
  \otimes\overline{\gamma^\ast[\zeta^\ast(e_{\la1\ra})_{(3)}]m_i^\vee}  \\
&\overset{(\ref{eqn:gammabiunitary}),\;(\ref{eqn:zetabiunitary})}{=}&
\sum_i \overline{e_{\la0\ra}r_i\otimes
  \zeta^\ast(\iota^\ast[\zeta^\ast(e_{\la1\ra})_{(1)}])}
  \otimes\overline{\gamma^\ast[\zeta^\ast(e_{\la1\ra})_{(2)}]m_i^\vee}  \\
&\overset{(\ref{eqn:zeta*})}{=}&
\sum_i \overline{e_{\la0\ra}r_i\otimes\zeta^\ast(e_{\la1\ra})}
  \otimes\overline{\gamma^\ast[\zeta^\ast(e_{\la2\ra})]m_i^\vee}
~\overset{(\ref{eqn:gamma*zeta*})}{=}~
\sum_i \overline{e_{\la0\ra}r_i\otimes\zeta^\ast(e_{\la1\ra})}
  \otimes\overline{m_i^\vee}.
\end{eqnarray*}
\end{itemize}
\end{proof}

\begin{remark}
There is one additional formula related to the evaluations: Suppose with notations in Corollary \ref{cor:Mbarevcoev} that $m:=\sum r_{\la0\ra}\otimes\zeta^\ast(r_{\la1\ra})\in M$ and $m^\vee:=r^\ast\otimes\gamma^\ast\in M^\vee$ for some $r\in R$ and $r^\ast\in R^\ast$. Then
\begin{eqnarray}\label{eqn:m^vee(fm)bar}
\overline{m^\vee(fm)}
&=&
\sum\la r^\ast,f_{(1)}r_{\la0\ra}\ra\overline{ \gamma^\ast[f_{(2)}\zeta^\ast(r_{\la1\ra})]}  \nonumber  \\
&\overset{(\ref{eqn:mfbar})}{=}&
\sum\la r^\ast,f_{(1)}r_{\la0\ra}\ra\la\gamma^\ast[f_{(2)}\zeta^\ast(r_{\la1\ra})],1\ra
\nonumber  \\
&\overset{(\ref{eqn:gammabiunitary}),\;(\ref{eqn:zetabiunitary})}{=}&
\la r^\ast,fr\ra
\end{eqnarray}
holds for each $f\in(H/B^+H)^\ast$,
where the last equality is because $\gamma^\ast$ and $\zeta^\ast$ are both counitary.
\end{remark}

For simplicity, we also let $(\pd{S},\pd{\alpha},\pd{\beta})$ denote an antipode $\pd{S}$ with distinguished elements $\pd{\alpha},\pd{\beta}$ of a quasi-Hopf algebra.
Now we have obtained enough formulas to show that $(\pd{S}_1,\pd{\alpha}_1,\pd{\beta}_1)$ and $(\pd{S}_2,\pd{\alpha}_2,\pd{\beta}_2)$ in Theorem \ref{thm:partialdual}(4) are both antipodes of $(H/B^+H)^\ast\#B$, where the constructions in \cite[Section 9.4]{Maj95} are applied to the quasi-fiber functor (\ref{eqn:forgetful}) from ${}_{(H/B^+H)^\ast}\mathfrak{M}_{(H/B^+H)^\ast}^{H^\ast}$ to $\Vec$:

\begin{lemma}\label{lem:preantipode}
Consider the quasi-bialgebra $(H/B^+H)^\ast\#B$ with associator $\pd{\phi}$. Suppose $\Phi=\overline{(-)}$ (\ref{eqn:Phi}) is a tensor equivalence with monoidal structure $J$. Then:
\begin{itemize}
\item[(1)]
There exists a linear transformation $\pd{T}$ on $(H/B^+H)^\ast\#B$ such that:
For any $f\in(H/B^+H)^\ast$ and $b\in B$, the following diagram in $\Vec$ commutes for each object $M\in{}_{(H/B^+H)^\ast}\mathfrak{M}_{(H/B^+H)^\ast}^{H^\ast}$:
\begin{equation}\label{eqn:preantipodediagram}
\xymatrix{
\overline{M}
  \ar[rr]^{\overline{\coev_M}\otimes\id\;\;\;\;\;\;\;\;\;\;\;\;\;\;\;\;\;\;\;\;}
  \ar@{-->}[d]_{\pd{T}(f\#b)}
&& \overline{M\otimes_{(H/B^+H)^\ast}M^\vee}\otimes\overline{M}
  \ar[rr]^{\;\;\;\;\;\;\;\;J^{-1}_{M,M^\vee}\otimes\id}
&& \overline{M}\otimes\overline{M^\vee}\otimes\overline{M}
   \ar[d]^{\id\otimes(f\#b)\otimes\id}  \\
\overline{M}
&& \overline{M}\otimes\overline{M^\vee\otimes_{(H/B^+H)^\ast}M}
  \ar[ll]^{\id\otimes\overline{\ev_M}\;\;\;\;\;\;\;\;\;\;\;\;\;\;\;\;\;\;\;\;}
&& \overline{M}\otimes\overline{M^\vee}\otimes\overline{M}
   \ar[ll]^{\;\;\;\;\;\;\;\;\id\otimes J_{M^\vee,M}}.
}
\end{equation}
Moreover, the equations
\begin{equation}\label{eqn:preantipode1}
\sum\pd{T}(p_{\pd{(1)}}q)p_{\pd{(2)}}=\pd{\e}(p)\pd{T}(q)
=\sum p_{\pd{(1)}}\pd{T}(qp_{\pd{(2)}})
\end{equation}
and
\begin{equation}\label{eqn:preantipode2}
\sum \pd{\phi}^1\pd{T}(\pd{\phi}^2)\pd{\phi}^3
=\e\#1
=\sum\pd{T}(\overline{\pd{\phi}}^1)\overline{\pd{\phi}}^2\pd{T}(\overline{\pd{\phi}}^3)
\end{equation}
hold for all $p,q\in(H/B^+H)^\ast\#B$, where we denote that
$$\pd{\Delta}(p)=\sum p_{\pd{(1)}}\otimes p_{\pd{(2)}}
\;\;\;\;(\forall x\in(H/B^+H)^\ast\#B)$$
with bold subscripts in parentheses, and that $$\pd{\phi}=\sum\pd{\phi}^1\otimes\pd{\phi}^2\otimes\pd{\phi}^3,\;\;\;\;
\pd{\phi}^{-1}=\sum\overline{\pd{\phi}}^1\otimes\overline{\pd{\phi}}^2
  \otimes\overline{\pd{\phi}}^3;$$

\item[(2)]
If the element $\pd{\upsilon}:=\pd{T}(\e\#1)\in(H/B^+H)^\ast\#B$ is invertible,
then the quasi-bialgebra $(H/B^+H)^\ast\#B$ have two antipodes with their distinguished elements as follows:
$$\pd{S}_1:=\pd{T}(-)\pd{\upsilon}^{-1}\;\;\;\;\text{with}\;\;\;\;
\pd{\alpha}_1:=\pd{\upsilon},\;\;\pd{\beta}_1:=\e\#1$$
and
$$\pd{S}_2:=\pd{\upsilon}^{-1}\pd{T}(-)\;\;\;\;\text{with}\;\;\;\;
\pd{\alpha}_2:=\e\#1,\;\;\pd{\beta}_2:=\pd{\upsilon}.$$
\end{itemize}
\end{lemma}

\begin{proof}
\begin{itemize}
\item[(1)]
It is clear by the reconstruction theorem of quasi-bialgebras that: $(H/B^+H)^\ast\#B$ could be identified with the algebra $\End(\Phi)$ of linear natural transformations of $\Phi$ (composed with the forgetful functor to $\Vec$). Thus $\pd{T}$ is a well-defined map, and we aim to prove the former equation in (\ref{eqn:preantipode1}), as the latter one could be shown similarly.

Recall in Lemma \ref{lem:monoidalstruJ}(1) that $\pd{\Delta}$ also induces the tensor product bifunctor $-\otimes-$ on $\Rep((H/B^+H)^\ast\#B)$. It follows that the diagram
\begin{equation*}
\xymatrix{
\overline{M^\vee}\otimes\overline{M}
  \ar[rr]^{J_{M^\vee,M}\;\;\;\;\;\;\;\;\;\;}
  \ar[d]_{\sum p_{\pd{(1)}}\otimes p_{\pd{(2)}}}
&& \overline{M^\vee\otimes_{(H/B^+H)^\ast}M}
  \ar[rr]^{\;\;\;\;\;\;\;\;\;\;\;\;\;\;\overline{\ev_M}}  \ar[d]^{p}
&& \k  \ar[d]^{\pd{\e}(p)}  \\
\overline{M^\vee}\otimes\overline{M}
  \ar[rr]_{J_{M^\vee,M}\;\;\;\;\;\;\;\;\;\;}
&& \overline{M^\vee\otimes_{(H/B^+H)^\ast}M}
  \ar[rr]_{\;\;\;\;\;\;\;\;\;\;\;\;\;\;\overline{\ev_M}}
&& \k
}
\end{equation*}
commutes for all $p\in(H/B^+H)^\ast\#B$, which could also be written as an equation:
\begin{equation}\label{eqn:x(1)x(2)}
\overline{\ev_M}\circ J_{M^\vee,M}\circ(\sum p_{\pd{(1)}}\otimes p_{\pd{(2)}})
=\pd{\e}(p)\cdot\overline{\ev_M}\circ J_{M^\vee,M}.
\end{equation}

Now consider the following diagram for any $p,q\in(H/B^+H)^\ast\#B$:
\begin{equation*}
\xymatrix{
\overline{M}
  \ar[rr]^{\overline{\coev_M}\otimes\id\;\;\;\;\;\;\;\;\;\;\;\;\;\;\;\;\;\;\;\;}
  \ar[d]_{p_{\pd{(2)}}}
&& \overline{M\otimes_{(H/B^+H)^\ast}M^\vee}\otimes\overline{M}
  \ar[rr]^{\;\;\;\;\;\;\;\;J^{-1}_{M,M^\vee}\otimes\id}
  \ar[d]^{\id\otimes p_{\pd{(2)}}}
&& \overline{M}\otimes\overline{M^\vee}\otimes\overline{M}
   \ar[d]^{\id\otimes\id\otimes p_{\pd{(2)}}}  \\
\overline{M}
  \ar[rr]^{\overline{\coev_M}\otimes\id\;\;\;\;\;\;\;\;\;\;\;\;\;\;\;\;\;\;\;\;}
  \ar@{-->}[d]_{\pd{T}(p_{\pd{(1)}}q)}
&& \overline{M\otimes_{(H/B^+H)^\ast}M^\vee}\otimes\overline{M}
  \ar[rr]^{\;\;\;\;\;\;\;\;J^{-1}_{M,M^\vee}\otimes\id}
&& \overline{M}\otimes\overline{M^\vee}\otimes\overline{M}
   \ar[d]^{\id\otimes p_{\pd{(1)}}q\otimes\id}  \\
\overline{M}
&& \overline{M}\otimes\overline{M^\vee\otimes_{(H/B^+H)^\ast}M}
  \ar[ll]^{\id\otimes\overline{\ev_M}\;\;\;\;\;\;\;\;\;\;\;\;\;\;\;\;\;\;\;\;}
&& \overline{M}\otimes\overline{M^\vee}\otimes\overline{M}
   \ar[ll]^{\;\;\;\;\;\;\;\;\id\otimes J_{M^\vee,M}}.
}
\end{equation*}
The bottom hexagon is the definition (\ref{eqn:preantipodediagram}) of $\pd{T}$ and hence commutes, while the commutativity of other two squares is because of the bifunctor $-\otimes-$. As a conclusion, we find that
\begin{eqnarray*}
&& \sum \pd{T}(p_{\pd{(1)}}q)p_{\pd{(2)}}  \\
&=&
\sum (\id\otimes\overline{\ev_M})\circ(\id\otimes J_{M^\vee,M})
  \circ(\id\otimes p_{\pd{(1)}}q\otimes\id)\circ(\id\otimes\id\otimes p_{\pd{(2)}})  \\
&& \;\;\;\;\;\; \circ\,(J^{-1}_{M,M^\vee}\otimes\id)\circ(\overline{\coev_M}\otimes\id)  \\
&=&
\sum (\id\otimes\overline{\ev_M})\circ(\id\otimes J_{M^\vee,M})
  \circ(\id\otimes p_{\pd{(1)}}\otimes p_{\pd{(2)}})\circ(\id\otimes q\otimes \id)  \\
&& \;\;\;\;\;\; \circ\,(J^{-1}_{M,M^\vee}\otimes\id)
  \circ(\overline{\coev_M}\otimes\id)  \\
&\overset{(\ref{eqn:x(1)x(2)})}{=}&
\e(p) (\id\otimes\overline{\ev_M})\circ(\id\otimes J_{M^\vee,M})
  \circ(\id\otimes q\otimes \id)\circ(J^{-1}_{M,M^\vee}\otimes\id)
  \circ(\overline{\coev_M}\otimes\id)  \\
&\overset{(\ref{eqn:preantipodediagram})}{=}&
\e(p)\pd{T}(q).
\end{eqnarray*}

On the other hand, a similar argument provides the commuting diagram:
\begin{equation}
\xymatrix{
\overline{M}
  \ar[rr]^{\overline{\coev_M}\otimes\id\;\;\;\;\;\;\;\;\;\;\;\;\;\;\;\;\;\;\;\;}
  \ar[d]_{\pd{\phi}^3}
&& \overline{M\otimes_{(H/B^+H)^\ast}M^\vee}\otimes\overline{M}
  \ar[rr]^{\;\;\;\;\;\;\;\;J^{-1}_{M,M^\vee}\otimes\id}
  \ar[d]^{\id\otimes \pd{\phi}^3}
&& \overline{M}\otimes\overline{M^\vee}\otimes\overline{M}
   \ar[d]^{\id\otimes\id\otimes \pd{\phi}^3}  \\
\overline{M}
  \ar[rr]^{\overline{\coev_M}\otimes\id\;\;\;\;\;\;\;\;\;\;\;\;\;\;\;\;\;\;\;\;}
  \ar[d]_{\pd{T}(\pd{\phi}^2)}
&& \overline{M\otimes_{(H/B^+H)^\ast}M^\vee}\otimes\overline{M}
  \ar[rr]^{\;\;\;\;\;\;\;\;J^{-1}_{M,M^\vee}\otimes\id}
&& \overline{M}\otimes\overline{M^\vee}\otimes\overline{M}
   \ar[d]^{\id\otimes\pd{\phi}^2\otimes\id}  \\
\overline{M}
  \ar[d]_{\pd{\phi}^1}
&& \overline{M}\otimes\overline{M^\vee\otimes_{(H/B^+H)^\ast}M}
  \ar[ll]^{\id\otimes\overline{\ev_M}\;\;\;\;\;\;\;\;\;\;\;\;\;\;\;\;\;\;\;\;}
  \ar[d]^{\pd{\phi}^1\otimes\id}
&& \overline{M}\otimes\overline{M^\vee}\otimes\overline{M}
   \ar[ll]^{\;\;\;\;\;\;\;\;\id\otimes J_{M^\vee,M}}
   \ar[d]^{\pd{\phi}^1\otimes\id\otimes\id}  \\
\overline{M}
&& \overline{M}\otimes\overline{M^\vee\otimes_{(H/B^+H)^\ast}M}
  \ar[ll]^{\id\otimes\overline{\ev_M}\;\;\;\;\;\;\;\;\;\;\;\;\;\;\;\;\;\;\;\;}
&& \overline{M}\otimes\overline{M^\vee}\otimes\overline{M}
   \ar[ll]^{\;\;\;\;\;\;\;\;\id\otimes J_{M^\vee,M}}.
}
\end{equation}
It could be concluded as
\begin{eqnarray*}
&& \sum\pd{\phi}^1\pd{T}(\pd{\phi}^2)\pd{\phi}^3  \\
&=&
\sum (\id\otimes\overline{\ev_M})\circ(\id\otimes J_{M^\vee,M})
  \circ(\pd{\phi}^1\otimes\pd{\phi}^2\otimes\pd{\phi}^3)
  \circ(J^{-1}_{M,M^\vee}\otimes\id)\circ(\overline{\coev_M}\otimes\id)  \\
&=&
\e\#1,
\end{eqnarray*}
where the last equation is due to the canonical way to regard $\overline{M^\vee}$ as a left dual object of $\overline{M}\in\Rep((H/B^+H)^\ast\#B)$ via the equivalence $\Phi$.
The other equation in (\ref{eqn:preantipode2}) holds analogously.

\item[(2)]
The axioms of antipodes could be verified directly. For instance, we might compute for any $p\in(H/B^+H)^\ast\#B$ that
\begin{eqnarray*}
\sum\pd{S}_1(p_{\pd{(1)}})\pd{\alpha}_1p_{\pd{(2)}}
&=&
\sum\left(\pd{T}(p_{\pd{(1)}})\pd{\upsilon}^{-1}\right)\pd{\upsilon}p_{\pd{(2)}}
=\sum\pd{T}(p_{\pd{(1)}})p_{\pd{(2)}}  \\
&\overset{(\ref{eqn:preantipode1})}{=}&
\pd{\e}(p)\pd{T}(\e\#1)=\pd{\e}(p)\pd{\alpha}_1
\end{eqnarray*}
and
$$\sum\pd{\phi}^1\pd{\beta}_2\pd{S}_2(\pd{\phi}^2)\pd{\alpha}_2\pd{\phi}^3
=\sum\pd{\phi}^1\pd{\upsilon}\left(\pd{\upsilon}^{-1}\pd{T}(\pd{\phi}^2)\right)\pd{\phi}^3
=\sum\pd{\phi}^1\pd{T}(\pd{\phi}^2)\pd{\phi}^3
\overset{(\ref{eqn:preantipode2})}{=}\e\#1,$$
etc.
\end{itemize}
\end{proof}

The linear transformation $\pd{T}$ satisfying (\ref{eqn:preantipode1}) and the first equation in (\ref{eqn:preantipode2}) is referred to be the \textit{preantipode} of the quasi-bialgebra $(H/B^+H)^\ast\#B$. See \cite[Definition 1]{Sar17} for details. It is known by \cite[Theorem 5]{Sar17} that the preantipode of a quasi-bialgebra is unique when it exists.

In fact, our desired results in Theorem \ref{thm:partialdual}(4) on the antipodes of $(H/B^+H)^\ast\#B$ could also be provided as properties for the preantipode $\pd{T}$, which would complete the proof:

~

\begin{proof}
[Proof of Theorem \ref{thm:partialdual}(4)]~

At first, we know by \cite[Theorem 6]{Sar17} that every quasi-Hopf algebra must have a (unique) preantipode $\pd{T}=\pd{\beta}\pd{S}(-)\pd{\alpha}$, which is also independent of the choice of antipode $(\pd{S},\pd{\alpha},\pd{\beta})$. Conversely, it is concluded in \cite[Proposition 1.2]{Sar21} that a \textit{finite-dimensional} quasi-bialgebra with preantipode is furthermore a quasi-Hopf algebra (similar explanations could also be found in \cite[Remark 3.12]{AP12} and \cite[Theorem 3.1]{Sch04}).

Therefore according to the statements in Lemma \ref{lem:preantipode}, it suffices to formulate the detailed expression of $\pd{T}$ defined by the commuting diagram (\ref{eqn:preantipodediagram}), with a similar process as the proof in the previous subsection as follows:
When $M=R\square_{B^\ast}H^\ast$, we aim to compute the image of the element
$$\sigma_R^{-1}(e)=\overline{\sum e_{\la0\ra}\otimes\zeta^\ast(e_{\la1\ra})}
\in \overline{M}=\overline{R\square_{B^\ast}H^\ast}$$
under the transformation
$$\sigma_R\circ\pd{T}(f\#b)
\overset{(\ref{eqn:preantipodediagram})}{=}
\big(\sigma_R\otimes(\overline{\ev_M}\circ J_{M^\vee,M})\big)
 \circ\big(\id\otimes(f\#b)\otimes\id\big)
 \circ\big((J^{-1}_{M,M^\vee}\circ\overline{\coev_M})\otimes\id\big)$$
for any $f\in(H/B^+H)^\ast$ and $b\in B$, where the definitions of $\sigma$ and $\sigma^{-1}$ are provided in (\ref{eqn:sigmaV}) and (\ref{eqn:sigma^-1}).

Specifically, we still let $\{r_i\}$ be a linear basis of $R:=(H/B^+H)^\ast\#B$ with dual basis $\{r_i^\ast\}$ of $R$, and denote $m_i:=\sum {r_i}_{\la0\ra}\otimes \zeta^\ast({r_i}_{\la1\ra})$ and $m_i^\vee:=\sum r^\ast_i\otimes\gamma^\ast$. Now with the usage of Corollary \ref{cor:Mbarevcoev}, one has following calculations by omitting some of $\sum$'s without confusions:
\begin{eqnarray*}
&& \sigma_R\circ\pd{T}(f\#b)
  \left(\overline{e_{\la0\ra}\otimes\zeta^\ast(e_{\la1\ra})}\right)  \\
&=&
\big(\sigma_R\otimes(\overline{\ev_M}\circ J_{M^\vee,M})\big)
  \circ\big(\id\otimes(f\#b)\otimes\id\big)
  \circ\big((J^{-1}_{M,M^\vee}\circ\overline{\coev_M})\otimes\id\big)  \\
&&\;\;\;\;\left(\overline{e_{\la0\ra}\otimes\zeta^\ast(e_{\la1\ra})}\right)  \\
&\overset{(\ref{eqn:Mbarcoev})}{=}&
\big(\sigma_R\otimes(\overline{\ev_M}\circ J_{M^\vee,M})\big)
  \circ\big(\id\otimes(f\#b)\otimes\id\big)  \\
&&\;\;\;\;\left(\sum_i \overline{e'_{\la0\ra}r_i\otimes\zeta^\ast(e'_{\la1\ra})}
    \otimes\overline{m_i^\vee}
    \otimes\overline{e_{\la0\ra}\otimes\zeta^\ast(e_{\la1\ra})}\right)  \\
&=&
\big(\sigma_R\otimes(\overline{\ev_M}\circ J_{M^\vee,M})\big)
  \left(\sum_i \overline{e'_{\la0\ra}r_i\otimes\zeta^\ast(e'_{\la1\ra})}
    \otimes(f\#b)\overline{m_i^\vee}
    \otimes\overline{e_{\la0\ra}\otimes\zeta^\ast(e_{\la1\ra})}\right)  \\
&\overset{(\ref{eqn:Mbarev})}{=}&
\sum_{i,j} \sigma_R\left(\overline{e'_{\la0\ra}r_i\otimes\zeta^\ast(e'_{\la1\ra})}\right)
    \left\la r^\ast_i,e''_{\la0\ra}r_j\right\ra  \\
&&\;\;\;\;\;\;
    \overline{m_j^\vee\left(\overline{\gamma}^\ast\big[\pi^\ast(f)\big(
    \iota(b)\rightharpoonup\overline{\zeta}^\ast(e''_{\la1\ra})\big)\big]
    \cdot(e_{\la0\ra}\otimes\zeta^\ast(e_{\la1\ra})\right)}  \\
&\overset{(\ref{eqn:m^vee(fm)bar})}{=}& \sum_{i,j} e'_{\la0\ra}r_i\left\langle\zeta^\ast(e'_{\la1\ra}),1\right\rangle
    \left\la r^\ast_i,e''_{\la0\ra}r_j\right\ra  \\
&&\;\;\;\;\;\;
    \overline{m_j^\vee\left(\overline{\gamma}^\ast\big[\pi^\ast(f)\big(
    \iota(b)\rightharpoonup\overline{\zeta}^\ast(e''_{\la1\ra})\big)\big]
    \cdot(e_{\la0\ra}\otimes\zeta^\ast(e_{\la1\ra})\right)}  \\
&\overset{(\ref{eqn:sigmaV})}{=}&
\sum_j e''_{\la0\ra}r_j
    \overline{m_j^\vee\left(\overline{\gamma}^\ast\big[\pi^\ast(f)\big(
    \iota(b)\rightharpoonup\overline{\zeta}^\ast(e''_{\la1\ra})\big)\big]
    \cdot(e_{\la0\ra}\otimes\zeta^\ast(e_{\la1\ra})\right)}  \\
&=&
\sum_j e''_{\la0\ra}r_j
    \left\la r^\ast_j,\overline{\gamma}^\ast\big[\pi^\ast(f)\big(
    \iota(b)\rightharpoonup\overline{\zeta}^\ast(e''_{\la1\ra})\big)\big]e\right\ra  \\
&=&
\sum_j e_{\la0\ra}\overline{\gamma}^\ast\big[\pi^\ast(f)\big(
    \iota(b)\rightharpoonup\overline{\zeta}^\ast(e_{\la1\ra})\big)\big],
\end{eqnarray*}
where $e=e'=e''=\e\#1$ as usual. This is exactly the desired element $\pd{T}(f\#b)\in(H/B^+H)^\ast\#B$ in Theorem \ref{thm:partialdual}(4), and the proof is completed.
\end{proof}

Up to now, we have accomplished the entire proof of Theorem \ref{thm:partialdual} by making $\Phi$ (\ref{eqn:Phi}) a tensor equivalence. Please note that the base field is assumed to be algebraically closed in this section, but in fact Theorem \ref{thm:partialdual} holds over an arbitrary field $\k$. This is because operations of (quasi-)Hopf algebras are invariant under base field extensions.

For example, let $\overline{\k}$ be the algebraically closed closure of $\k$. Then via the injection $\iota\otimes\id_{\overline{\k}}$, we could regard $B\otimes\overline{\k}$ as a left coideal subalgebra of the Hopf algebra $H\otimes\overline{\k}$ over $\overline{\k}$, and formulate its {\pams} $(\zeta\otimes\id_{\overline{\k}},\gamma^\ast\otimes\id_{\overline{\k}})$. Then there is an isomorphism of quasi-Hopf algebras over $\overline{\k}$:
\begin{eqnarray*}
&&
\begin{array}{ccc}
((H/B^+H)^\ast\#B)\otimes\overline{\k}
&\cong&
\Hom_{\overline{\k}}(H/B^+H\otimes\overline{\k},\;\overline{\k})
  \;\#_{\overline{\k}}\;(B\otimes\overline{\k}),  \\
(f\#b)\otimes a
&\mapsto&
(f\otimes1)\#_{\overline{\k}}(b\otimes a)
\end{array} \\
&& \hspace{160pt}
\left(\forall f\in(H/B^+H)^\ast,\;b\in B,\;a\in\overline{\k}\right).
\end{eqnarray*}
It helps us generalize the structures of Theorem \ref{thm:partialdual} from the case over the algebraically closed field $\overline{\k}$ to an arbitrary one $\k$.

We end this subsection by remarking on the invertibility of the element $\pd{\upsilon}=\pd{T}(\e\#1)$:

\begin{remark}\label{rmk:upsiloninvertible}
Suppose that $\pd{T}$ is the preantipode of a quasi-Hopf algebra. Then the followings are equivalent:
\begin{itemize}
\item[(1)]
The image $\pd{\upsilon}$ of the unit element under $\pd{T}$ is invertible;
\item[(2)]
Some antipode $\pd{S}$ has both invertible distinguished elements $\pd{\alpha}$ and $\pd{\beta}$;
\item[(3)]
Every antipode $\pd{S}$ has both invertible distinguished elements $\pd{\alpha}$ and $\pd{\beta}$.
\end{itemize}
\end{remark}

\begin{proof}
The equivalence of (1) and (2) is due to the fact that $\pd{\upsilon}=\pd{\beta}\pd{\alpha}$ since $\pd{S}$ preserves the unit element.
On the other hand, (2) implies (3) according to \cite[Proposition 1.1]{Dri89}.
\end{proof}

Although the author is not able to decide whether $\pd{\upsilon}$ is invertible or not, we hope to conjecture by the results in the next subsection that:

\begin{conjecture}
The equivalent properties in Remark \ref{rmk:upsiloninvertible} always hold for the left partially dualized quasi-Hopf algebra $(H/B^+H)^\ast\#B$.
\end{conjecture}


\subsection{Reconstruction theorem for left partial duals, and consequences}

Let us rewrite Corollary \ref{cor:tensorequiv} with more details as follows.

\begin{theorem}\label{thm:reconstruction}
Let $H$ be a finite-dimensional Hopf algebra. Suppose that $B$ is a left coideal subalgebra of $H$ with a {\pams} $(\zeta,\gamma^\ast)$. Then there is a tensor equivalence $\Phi$ between
\begin{itemize}
\item
The category ${}_{(H/B^+H)^\ast}\mathfrak{M}_{(H/B^+H)^\ast}^{H^\ast}$ of finite-dimensional relative Doi-Hopf modules, and
\item
the category of finite-dimensional representations of the left partial dual $(H/B^+H)^\ast\#B$ determined by $(\zeta,\gamma^\ast)$,
\end{itemize}
defined as
$$\begin{array}{cccc}
\Phi:&
{}_{(H/B^+H)^\ast}\mathfrak{M}_{(H/B^+H)^\ast}^{H^\ast} &\approx&
\Rep((H/B^+H)^\ast\#B),  \\
& M &\mapsto&
\overline{M}:=M/M\left((H/B^+H)^\ast\right)^+,
\end{array}$$
with monoidal structure
$$\begin{array}{rcl}
J_{M,N}\;:\;
\overline{M}\otimes\overline{N} &\cong& \overline{M\otimes_{(H/B^+H)^\ast}N},  \\
\overline{m}\otimes\overline{n}
&\mapsto& \sum\overline{m_{(0)}\overline{\gamma}^\ast(m_{(1)})\otimes_{(H/B^+H)^\ast}n}
\end{array}.$$
\end{theorem}


As we have formulated in Proposition \ref{prop:monoidalstru0}(4) that ${}_{(H/B^+H)^\ast}\mathfrak{M}_{(H/B^+H)^\ast}^{H^\ast}$ is equivalent to
$$\Rep(H)_{\Rep(B)}^\ast:=\Rex_{\Rep(H)}\left(\Rep(B)\right)^\mathrm{rev},$$
the dual tensor category of $\Rep(H)$ with respect to its left module category $\Rep(B)$,
it is immediate that:

\begin{corollary}\label{cor:catMoritaequiv}
Under the assumptions in Theorem \ref{thm:reconstruction}, there is an equivalence between finite tensor categories:
$$\Rep(H)_{\Rep(B)}^\ast\approx\Rep((H/B^+H)^\ast\#B).$$
In particular, the tensor category $\Rep((H/B^+H)^\ast\#B)$ is categorically Morita equivalent to $\Rep(H)$.
\end{corollary}

We should remark that two tensor categories $\C$ and $\D$ are said to be \textit{categorically Morita equivalent} in the sense of \cite[Definition 7.12.17]{EGNO15}, if there exists an exact left $\C$-module category $\M$ and a tensor equivalence $\D\approx\C_\M^\ast$.
Moreover, the categorical Morita equivalence is indeed an equivalence relation according to \cite[Proposition 4.6]{Mug03}.

An important property on relations between categorically Morita equivalent tensor categories is Schauenburg's equivalence (\cite[Theorem 3.3]{Sch01}) between centers:
\begin{equation}\label{eqn:Schauenburgequiv}
\Z(\C)\approx\Z(\C_\M^\ast),
\end{equation}
where $\Z(\C)$ denotes the \textit{left center} of $\C$ (see \cite{JS91} e.g.).
When $\C$ is finite and $\M$ is an indecomposable $\C$-module category, (\ref{eqn:Schauenburgequiv}) is an equivalence of braided finite tensor categories (\cite[Section 2]{JS93}). Here we would not recall the details of the functor (\ref{eqn:Schauenburgequiv}), but note that a description might be found in \cite[Section 3.7]{Shi20}.

In particular, if $K$ is a quasi-Hopf algebra, then the objects in the left center to the left $K$-modules are furthermore identified with Yetter-Drinfeld modules over $K$.
The definitions of \textit{Yetter-Drinfeld modules over quasi-Hopf algebras} are referred to \cite[Section 2]{BCP06}, which generalize the notions of Yetter-Drinfeld modules over Hopf algebras.
Specifically, it is concluded before and in \cite[Theorem 2.10]{BCP06} that:

\begin{lemma}\label{lem:centerisoYDmods}
Suppose $K$ is a finite-dimensional quasi-Hopf algebra. Let ${}^{K}_{K}\YD$ (resp. ${}_K\YD{}^K$, ${}^K\YD{}_K$ and $\YD{}^{K}_{K}$) be the category of finite-dimensional left-left (resp. left-right, right-left and right-right) Yetter-Drinfeld modules over $K$. Then:
\begin{itemize}
\item[(1)](\cite[Section 2]{BCP06})
There exist braided tensor isomorphisms:
\begin{equation}\label{eqn:centerisoYDmods}
\Z(\Rep(K))\cong{}^{K}_{K}\YD\cong\left({}_K\YD{}^K\right)^\mathrm{in}
\cong\left({}^K\YD{}_K\right)^{\mathrm{rev}\,\mathrm{in}}
\cong\left(\YD{}^{K}_{K}\right)^\mathrm{rev},
\end{equation}
where $(-)^\mathrm{rev}$ denotes the monoidal category with reverse tensor products, and $(-)^\mathrm{in}$ denotes the braided category with reverse braiding;

\item[(2)](\cite[Section 8.5]{BCPV19})
Let $D(K)$ be the quantum double of $K$, and let $\Rep(D(K))$ be the category of finite-dimensional left $D(K)$-modules. Then there is a tensor isomorphism:
\begin{equation}\label{eqn:YDmodsisodouble}
{}_K\YD{}^K\cong\Rep(D(K)).
\end{equation}
\end{itemize}
\end{lemma}

\begin{proof}
\begin{itemize}
\item[(1)]
In fact, the original statements in \cite{BCP06} are isomorphisms between categories of infinite-dimensional objects. However, since the category isomorphisms in the proof of \cite[Theorems 2.4 and 2.10]{BCP06} (and their inverses) clearly preserve finite-dimensional objects, we could know that (\ref{eqn:centerisoYDmods}) hold as well.

\item[(2)]
The isomorphism (\ref{eqn:YDmodsisodouble}) is mentioned in \cite[Page 330]{BCPV19}.
\end{itemize}
\end{proof}

\begin{remark}
Here the definition of the quantum double $D(K)$ as a (quasitriangular) quasi-Hopf algebras is referred to as \cite[Theorem 3.9]{HN99}. We also remark that $D(K)$ coincides with the notion of \textit{Drinfeld double} when $K$ is in particular a finite-dimensional Hopf algebra.
\end{remark}

As a conclusion, one could obtain some consequent relations between $H$ and its left partial dual $(H/B^+H)^\ast\#B$ as finite-dimensional quasi-Hopf algebras:

\begin{proposition}\label{prop:YDmodsequiv}
Let $H$ be a finite-dimensional Hopf algebra. Suppose that $(H/B^+H)^\ast\#B$ is a left partially dualized quasi-Hopf algebra of $H$. Then:
\begin{itemize}
\item[(1)]
There is an equivalence between left centers:
$$\Z\left(\Rep((H/B^+H)^\ast\#B)\right)\approx\Z\left(\Rep(H)\right)$$
as braided finite tensor categories;

\item[(2)]
There exist four braided tensor equivalences between the categories of finite-dimensional Yetter-Drinfeld modules over $(H/B^+H)^\ast\#B$ and $H$:
\begin{align*}
& {}^{(H/B^+H)^\ast\#B}_{(H/B^+H)^\ast\#B}\YD\approx{}^{H}_{H}\YD,
  \;\;\;\;\;\;\;\;\;\;\;\;\;\;\;\;
  {}_{(H/B^+H)^\ast\#B}\YD{}^{(H/B^+H)^\ast\#B}\approx{}_{H}\YD{}^{H},  \\
& {}^{(H/B^+H)^\ast\#B}\YD{}_{(H/B^+H)^\ast\#B}\approx{}^{H}\YD{}_{H}
  \;\;\;\;\;\;\text{and}\;\;\;\;\;\;
  \YD{}^{(H/B^+H)^\ast\#B}_{(H/B^+H)^\ast\#B}\approx\YD{}^{H}_{H}
\end{align*}
as braided finite tensor categories;

\item[(3)]
There is a tensor equivalence between the categories of finite-dimensional representations of quantum doubles:
$$\Rep\left(D((H/B^+H)^\ast\#B)\right)\approx\Rep\left(D(H)\right).$$
\end{itemize}
\end{proposition}

\begin{proof}
(1)  is an immediate conclusion of Corollary \ref{cor:catMoritaequiv} and the equivalence (\ref{eqn:Schauenburgequiv}) for the case when $\C=\Rep(H)$ and $\M=\Rep(B)$.

The equivalences desired in (2) and (3) are obtained by combining (1) and Lemma \ref{lem:centerisoYDmods}.
\end{proof}

There are supposed to be more tensor equivalences between categories of finite-dimensional modules over $H$ or partial duals. For example, one might establish a tensor isomorphism
$$\YD^{(H/B^+H)^\ast\#B}_{(H/B^+H)^\ast\#B}
\cong\left({}^{B^\biop\#(H/B^+H)^{\ast\,\biop}}_{B^\biop\#(H/B^+H)^{\ast\,\biop}}\YD\right)
{}^\mathrm{rev}$$
according to \cite[Proposition 2.7]{BCP06} and Proposition \ref{prop:partialdualbiop} below, but we would not attempt to gather them in this paper.

Finally, suppose that $B$ is a given left coideal subalgebra of $H$. It is clear that up to equivalences, the left $\Rep(H)$-module category $\Rep(B)$, as well as the corresponding dual category $\Rep(H)_{\Rep(B)}^\ast$, does not depend on the choices of {\pams}s $(\zeta,\gamma^\ast)$.
As a conclusion, we could find by \cite[Theorem 2.2]{NS08} that:

\begin{proposition}\label{prop:partialdualsequiv}
Let $H$ be a finite-dimensional Hopf algebra with is a left coideal subalgebra $B$.
Then all the left partially dualized quasi-Hopf algebras of $H$ determined by all the {\pams}s for $B\subseteq H$ are gauge equivalent to each others.
\end{proposition}

\begin{remark}\label{rmk:gaugeequiv(MAMS)}
The notion of \textit{gauge transformation} and \textit{gauge equivalence} of quasi-Hopf algebras is introduced in \cite[Chapter XV]{Kas95}, and one could find the details in \cite[Section 1]{NS08} for example.

It is known by \cite[Theoreom 6.1]{EG02} and \cite[Theorem 2.2]{NS08} that two finite-dimensional quasi-Hopf algebra $K$ and $K'$ are gauge equivalent if and only if $\Rep(K)$ and $\Rep(K')$ are tensor equivalent.
Thus, Proposition \ref{prop:partialdualsequiv} follows.
\end{remark}

\section{Opposite and cooposite structures, and right partial dualization}\label{section:5}

This section consists of some descriptions of the opposite, coopposite and dual structures of left partially dualized quasi-Hopf algebras $(H/B^+H)^\ast\#B$.

\subsection{Biopposite structures of left partially dualized quasi-Hopf algebras}\label{subsection:5.1}

We still let $H$ be a finite-dimensional Hopf algebra over $\k$. Suppose $\pi^\ast:A'\rightarrowtail H^\ast$ is a right coideal subalgebra and $\iota:B\rightarrowtail H$ is a left coideal subalgebra. It is known in \cite[Remark 1.3(b)]{Doi92} that $A'\#B$ is an algebra with the smash product structure, which is in fact a general case of (\ref{eqn:smashprod}). One could directly verify that
\begin{equation}\label{eqn:smashproduct-op}
(A'\#B)^\op\cong B^\biop\#{A'}^\biop,\;\;a'\#b\mapsto b\#a'
\end{equation}
is an isomorphism of algebras. Thus in particular, the opposite of the left partially dualized quasi-Hopf algebra $(H/B^+H)^\ast\#B$ is isomorphic to
\begin{equation}\label{eqn:antismashprod}
B^\biop\#(H/B^+H)^{\ast\,\biop}
\end{equation}
as algebras.

However, the latter algebra (\ref{eqn:antismashprod}) could also be a partial dual of $H^{\ast\,\biop}$ determined by the {\pams} $(\gamma^\ast,\zeta)$ according to Proposition \ref{prop:BCP-selfdual}. Specifically,
there is a biunitary linear isomorphism $\vartheta$ anti-preserving the multiplication and ``comultiplication'' between quasi-Hopf algebras $(H/B^+H)^\ast\#B$ and $B^\biop\#(H/B^+H)^{\ast\,\biop}$:

\begin{proposition}\label{prop:partialdualbiop}
There is an isomorphism
$$
\begin{array}{crcl}
\vartheta: & ((H/B^+H)^\ast\#B)^\biop & \cong & B^\biop\#(H/B^+H)^{\ast\,\biop}  \\
& f\#b & \mapsto & b\#f
\end{array}
$$
of quasi-bialgebras.
Furthermore, if the element $\pd{\upsilon}$ (\ref{eqn:upsilon}) is invertible, then:
\begin{itemize}
\item[(1)]
The antipode $(\pd{S}_1,\e\#1,\pd{\upsilon})$ of the quasi-Hopf algebra $((H/B^+H)^\ast\#B)^\biop$ is preserved by $\vartheta$ to be an antipode
$$(\vartheta\circ\pd{S}_1\circ\vartheta^{-1},\; 1\#\e,\; \vartheta(\pd{\upsilon}))$$
of $B^\biop\#(H/B^+H)^{\ast\,\biop}$;

\item[(2)]
The antipode $(\pd{S}_2,\pd{\upsilon},\e\#1)$ of the quasi-Hopf algebra $((H/B^+H)^\ast\#B)^\biop$ is preserved by $\vartheta$ to be an antipode
$$(\vartheta\circ\pd{S}_2\circ\vartheta^{-1},\; \vartheta(\pd{\upsilon}),\; 1\#\e)$$
of $B^\biop\#(H/B^+H)^{\ast\,\biop}$.
\end{itemize}
\end{proposition}

\begin{proof}
Recall in \cite{Dri89} that the biopposite structures of a quasi-Hopf algebra form again a quasi-Hopf algebra (or see \cite[Example 2.13 and Remark 3.16(5)]{BCPV19} e.g.). For example in our situations, if we use the same notations in Theorem \ref{thm:partialdual}, then the quasi-bialgebra $((H/B^+H)^\ast\#B)^\biop$ would have the opposite multiplication with the algebra map $\pd{\Delta}^\cop$, and the associator $\pd{\phi}^\biop$ as the inverse of
\begin{equation}\label{eqn:biopassociator^-1}
(\pd{\phi}^\biop)^{-1}=\sum_{i,j}
\left(\gamma^\ast[\zeta^\ast(b_i^\ast)_{(2)}\zeta^\ast(b_j^\ast)]\#1\right)
\otimes\left(\gamma^\ast[\zeta^\ast(b_i^\ast)_{(1)}]\#b_j\right)
\otimes\left(\e\#b_i\right)
\in((H/B^+H)^\ast\#B)^{\otimes3}
\end{equation}
which is obtained by flipping the first and third tensorands of (\ref{eqn:phi^-1}) in Remark \ref{rmk:phi^-1}, where $\{b_i\}$ is a linear basis of $B$ with dual basis $\{b_i^\ast\}$ of $B^\ast$.

Now we write out the quasi-bialgebra structure of the partial dual $B^\biop\#(H/B^+H)^{\ast\,\biop}$ with the language of the {\pams} $(\gamma^\ast,\zeta)$ and dual bases, and check that they are preserved by the isomorphism $\vartheta$ of algebras:
\begin{itemize}
\item
Referred to Remark \ref{rmk:equivDelta}(1),  the ``comultiplication'' of $B^\biop\#(H/B^+H)^{\ast\,\biop}$ maps $b\#f$ to
$$\sum_{i,j}\left(b_{(2)}\#\gamma^\ast[f_{(2)}\zeta^\ast(b_j^\ast)]f_i\right)
\otimes\left(b_j\zeta[\gamma(f_i^\ast)b_{(1)}]\#f_{(1)}\right)$$
for any $f\in(H/B^+H)^\ast$ and $b\in B$. This is exactly the same as $(\vartheta\otimes\vartheta)\circ\pd{\Delta}^\cop(f\#b)$, namely, the image of (\ref{eqn:Deltacomplete}) under $\vartheta\otimes\vartheta$ composite with the flip map.

\item
Clearly, the ``counit'' of $B^\biop\#(H/B^+H)^{\ast\,\biop}$ also maps $b\#f$ to $\la f,1\ra\la\e,b\ra$ for any $f\in(H/B^+H)^\ast$ and $b\in B$.

\item
Referred to Remark \ref{rmk:equivDelta}(3), the inverse of the associator of $B^\biop\#(H/B^+H)^{\ast\,\biop}$ would be
$$\sum_{i,j}
\left(1\#\gamma^\ast[\zeta^\ast(b_j^\ast)_{(2)}\zeta^\ast(b_i^\ast)]\right)
\otimes\left(b_i\#\gamma^\ast[\zeta^\ast(b_j^\ast)_{(1)}]\right)
\otimes\left(b_j\#\e\right),$$
which is exactly the same as
$(\vartheta\otimes\vartheta\otimes\vartheta)\left((\pd{\phi}^\biop)^{-1}\right)$,
namely, the image of (\ref{eqn:biopassociator^-1}) under $\vartheta^{\otimes3}$.
\end{itemize}
As a conclusion, $\vartheta$ is an isomorphism of quasi-bialgebras.

When $\pd{\upsilon}$ has an inverse $\pd{\upsilon}^{-1}\in(H/B^+H)^\ast\#B$, we only explain that (1) holds, as (2) is completely similar. In fact, since the antipodes of the quasi-Hopf algebra $(H/B^+H)^\ast\#B$ and its biopposite structure share the same transformation $\pd{S}_1$ but switched distinguished elements, it is sufficient to notice that $\vartheta(\e\#1)=1\#\e$ holds.
\end{proof}

\subsection{Opposite and coopposite structures of left partial dual}\label{subsection:5.2}

As the left partially dualized quasi-Hopf algebra $(H/B^+H)^\ast\#B$ is finite-dimensional, it follows from \cite[Theorem 2.5(1)]{BC03} that each of its antipodes is bijective. Of course, this could also be known due to the fact that $\Rep(H)$ is a finite tensor category. Anyway, there must exist an opposite quasi-Hopf algebra $((H/B^+H)^\ast\#B)^\op$, which would be described up to isomorphisms in this subsection.

However, the antipodes of $(H/B^+H)^\ast\#B$ are not unique, among which we have not declared a canonical one from the view of Theorem \ref{thm:partialdual}(4). Thus we would only provide quasi-bialgebra isomorphisms, since they automatically preserve the corresponding antipodes as well. In addition, these results are similar to those in Proposition \ref{prop:partialdualbiop}, but determined by the systems in Corollary \ref{cor:PAMSopcop}.

\begin{proposition}
Let $H$ be a finite-dimensional Hopf algebra. Suppose that $B$ is a left coideal subalgebra of $H$ with a {\pams} $(\zeta,\gamma^\ast)$. Then:
\begin{itemize}
\item[(1)]
As quasi-bialgebras, the left partially dualized quasi-Hopf algebra of $H^\op$ determined by $(\overline{\zeta}\circ S^{-1},\overline{\gamma}^\ast)$ is the same as the left partial dual of $H^\cop$ determined by $(\overline{\zeta},\overline{\gamma}^\ast\circ S^{-1})$;

\item[(2)]
Denote the quasi-Hopf bialgebra in (1) by $(H/B^+H)^{\ast\,\op}\#B^\op$. Then there is an isomorphism of quasi-bialgebras:
\begin{equation}\label{eqn:varphi}
\begin{array}{crcl}
\varphi: & ((H/B^+H)^\ast\#B)^\op & \cong & (H/B^+H)^{\ast\,\op}\#B^\op  \\
& f\#b & \mapsto & (\e\#b)(f\#1)=\sum f_{(1)}\#(b\leftharpoonup S^{-1}(f_{(2)})),
\end{array}
\end{equation}
whose inverse satisfies that
\begin{equation}\label{eqn:varphiinv}
\varphi^{-1}(f\#b)=\sum f_{(1)}\#(b\leftharpoonup f_{(2)})
\;\;\;\;\;\;(\forall f\in(H/B^+H)^\ast,b\in B).
\end{equation}
\end{itemize}
\end{proposition}

\begin{proof}
\begin{itemize}
\item[(1)]
In order to specify the structures of the former left partial dual $K_1$ of $H^\op$ determined by $(\overline{\zeta}\circ S^{-1},\overline{\gamma}^\ast)$, note at first that the right $H^{\ast\,\cop}$-coaction on the algebra $(H/B^+H)^{\ast\,\op}$ should be:
$$(H/B^+H)^{\ast\,\op}\rightarrow (H/B^+H)^{\ast\,\op}\otimes H^{\ast\,\cop},\;\;
f\mapsto\sum f_{(1)}\otimes S^{-1}(f_{(2)}),$$
which is induced by dualizing (\ref{eqn:rightH^opmod}) linearly, or equivalently, induced such that the injection $S^{-1}\circ\pi^\ast$ of right $H^{\ast\,\cop}$-module algebras. Then we could know that the algebra structure of $K_1$ is the smash product of $(H/B^+H)^{\ast\,\op}$ and $B^\op$ with multiplication:
\begin{eqnarray}\label{eqn:K_1product}
(f\#b)\otimes(g\#c) &\mapsto&
\sum f\cdot^\op g_{(1)}\#(b\leftharpoonup S^{-1}(g_{(2)}))\cdot^\op c  \nonumber  \\
&=& \sum g_{(1)}f\#\la S^{-1}(g_{(2)}),b_{(1)}\ra cb_{(2)}
\end{eqnarray}
for all $f,g\in(H/B^+H)^\ast$ and $b,c\in B$. On the other hand, the ``comultiplication'' of $K_1$ becomes:
\begin{eqnarray}\label{eqn:K_1coproduct}
&&f\#b  \nonumber  \\
&\overset{(\ref{eqn:Deltacomplete})}{\mapsto}&
\sum_{i,j}
\left(f_{(1)}\#b_i
  \cdot^\op\overline{\zeta}[S^{-1}(b_{(1)}\cdot^\op\overline{\gamma}(f_j^\ast))]\right)
\otimes
\left(\overline{\gamma}^\ast[S^{-1}(f_{(2)})S^{-1}(\overline{\zeta}(b_i^\ast))]
  \cdot^\op f_j\#b_{(2)}\right)  \nonumber  \\
&=&
\sum_{i,j}
\left(f_{(1)}\#\overline{\zeta}[S^{-1}(\overline{\gamma}(f_j^\ast)b_{(1)})]b_i\right)
\otimes
\left(f_j\overline{\gamma}^\ast[S^{-1}(\overline{\zeta}(b_i^\ast)f_{(2)})]\#b_{(2)}\right),
\end{eqnarray}
where $\{b_i\}$ is a linear basis of $B$ with dual basis $\{b^\ast_i\}$ of $B^\ast$, and $\{f_i\}$ is a linear basis of $(H/B^+H)^\ast$ with dual basis $\{f^\ast_i\}$ of $H/B^+H$ as usual. Moreover, the associator of $K_1$ would have inverse
\begin{eqnarray}\label{eqn:K_1associator^-1}
&& \sum_{i,j}\left(\e\#b_i\right)
\otimes\left(\overline{\gamma}^\ast[S^{-1}(\overline{\zeta}^\ast(b^\ast_i))_{(2)}]\#b_j\right)
\otimes
\left(\overline{\gamma}^\ast[S^{-1}(\overline{\zeta}^\ast(b^\ast_i))_{(1)}
  S^{-1}(\zeta^\ast(b^\ast_j))]\#1\right)  \nonumber  \\
&=&
\sum_{i,j}\left(\e\#b_i\right)
\otimes\left(\overline{\gamma}^\ast[S^{-1}(\overline{\zeta}^\ast(b^\ast_i)_{(1)})]\#b_j\right)
\otimes
\left(\overline{\gamma}^\ast[S^{-1}(\overline{\zeta}^\ast(b^\ast_j)
  \overline{\zeta}^\ast(b^\ast_i)_{(2)})]\#1\right).
\end{eqnarray}

As for the latter left partial dual $K_2$ of $H^\cop$ determined by $(\overline{\zeta},\overline{\gamma}^\ast\circ S^{-1})$, One could write its structures in the similar way, by noting that the left $H^\cop$-coaction on $B^\op$ becomes (\ref{eqn:leftH^copcomod}). For example, the multiplication of $K_2$ would be:
\begin{eqnarray*}
&& (f\#b)\otimes(g\#c) \\
&\mapsto&
\sum f\cdot^\op (S^{-1}(b_{(1)})\rightharpoonup g)\#b_{(2)}\cdot^\op c  \nonumber  \\
&=& \sum g_{(1)}f\la g_{(2)},S^{-1}(b_{(1)})\ra\# cb_{(2)}
\;\;\;\;\;\;\;\;(\forall f,g\in(H/B^+H)^\ast,\;\forall b,c\in B),
\end{eqnarray*}
which is exactly the same as (\ref{eqn:K_1product}). It is also straightforward to compute and verify that $K_2$ has both the same ``comultiplication'' (\ref{eqn:K_1coproduct}) and the inverse (\ref{eqn:K_1associator^-1}) of associator with $K_1$.

The equalities for the units and ``counits'' of $K_1$ and $K_2$ are clear.

\item[(2)]
It is direct to find that $\varphi$ is an isomorphism of algebras, and has inverse defined through (\ref{eqn:varphiinv}). Now we prove that $\varphi$ preserves the ``comultiplications'': Denote $e=\e\#1$, and recall in (\ref{eqn:Delta(f):e}) that
$$\pd{\Delta}(f\#1)=\sum \left(f_{(1)}\#e_{\la0\ra}\right)\otimes \left(\gamma^\ast[f_{(2)}\zeta^\ast(e_{\la1\ra})]\#1\right)\;\;\;\;(\forall f\in(H/B^+H)^\ast)$$
holds with some of our usual notations. Then we compute that:
\begin{eqnarray*}
&& (\varphi\otimes\varphi)\circ\pd{\Delta}(f\#1)  \\
&\overset{(\ref{eqn:varphi})}{=}&
\sum\left(f_{(1)}\#(e_{\la0\ra}\leftharpoonup S^{-1}(f_{(2)}))\right)
  \otimes\left(\gamma^\ast[f_{(3)}\zeta^\ast(e_{\la1\ra})]\#1\right)  \\
&\overset{(\ref{eqn:dualbases(e)})}{=}&
\sum\left(f_{(1)}\#e_{\la0\ra}\right)
  \otimes\left(\gamma^\ast[f_{(3)}\zeta^\ast(S^{-1}(f_{(2)})\btr e_{\la1\ra})]\#1\right)  \\
&\overset{(\ref{eqn:btr})}{=}&
\sum\left(f_{(1)}\#e_{\la0\ra}\right)
  \otimes\left(\gamma^\ast\left[f_{(3)}
  \zeta^\ast(\iota^\ast[S^{-1}(f_{(2)})\zeta^\ast(e_{\la1\ra})])\right]\#1\right)  \\
&\overset{(\ref{eqn:convolution*1})}{=}&
\sum\left(f_{(1)}\#e_{\la0\ra}\right)
  \otimes\left(\gamma^\ast\left[f_{(4)}S^{-1}(f_{(3)})\zeta^\ast(e_{\la1\ra})_{(1)}
  \pi^\ast(\overline{\gamma}^\ast[S^{-1}(f_{(2)})\zeta^\ast(e_{\la1\ra})_{(2)}])\right]
  \#1\right)  \\
&\overset{(\ref{eqn:gamma*})}{=}&
\sum\left(f_{(1)}\#e_{\la0\ra}\right)
  \otimes\left(\gamma^\ast[\zeta^\ast(e_{\la1\ra})_{(1)}]
  \overline{\gamma}^\ast[S^{-1}(f_{(2)})\zeta^\ast(e_{\la1\ra})_{(2)}])]\#1\right)  \\
&\overset{(\ref{eqn:gamma*bar2})}{=}&
\sum\left(f_{(1)}\#e_{\la0\ra}\right)
  \otimes\left(\overline{\gamma}^\ast\left[S^{-1}(f_{(2)})\zeta^\ast(e_{\la1\ra})_{(2)}
  S^{-1}(\pi^\ast(\gamma^\ast[\zeta^\ast(e_{\la1\ra})_{(1)}])\right]\#1\right)  \\
&=&
\sum\left(f_{(1)}\#e_{\la0\ra}\right)
  \otimes\left(\overline{\gamma}^\ast\left[
  S^{-1}\left(\pi^\ast(\gamma^\ast[\zeta^\ast(e_{\la1\ra})_{(1)}])
  S[\zeta^\ast(e_{\la1\ra})_{(2)}]f_{(2)}\right)\right]\#1\right)  \\
&\overset{(\ref{eqn:convolution*2})}{=}&
\sum\left(f_{(1)}\#e_{\la0\ra}\right)
  \otimes\left(\overline{\gamma}^\ast[
  S^{-1}(\overline{\zeta}^\ast(\iota^\ast[\zeta^\ast(e_{\la1\ra})]))f_{(2)}]\#1\right)  \\
&\overset{(\ref{eqn:iota*zeta*})}{=}&
\sum\left(f_{(1)}\#e_{\la0\ra}\right)
  \otimes\left(\overline{\gamma}^\ast[
  S^{-1}(\overline{\zeta}^\ast(e_{\la1\ra}))f_{(2)}]\#1\right),  \\
\end{eqnarray*}
and this equals to the ``coproduct'' of $\varphi(f\#1)$ in $(H/B^+H)^{\ast\,\op}\#B^\op$, which is exactly (\ref{eqn:K_1coproduct}) when $b=1$. On the other hand, it is evident that a similar statement would hold on the elements of form $\e\#b$ for $b\in B$.

It remains to show that $\varphi$ preserves the associators. Due to the construction of the opposite quasi-bialgebra, we are supposed to verify that
$(\varphi^{-1})^{\otimes3}$ maps the inverse (\ref{eqn:K_1associator^-1}) of the associator of $(H/B^+H)^{\ast\,\op}\#B^\op$ to the element $\pd{\phi}\in ((H/B^+H)^\ast\#B)^{\otimes3}$ in (\ref{eqn:associator}).
In fact, note that $\varphi(f\#b)$ is actually the product $(\e\#f)(b\#1)$ in the left partial dual $(H/B^+H)^\ast\#B$, and hence
\begin{eqnarray*}
&& (\varphi^{-1})^{\otimes3}\left(\sum_{i,j}\left(\e\#b_i\right)
\otimes\left(\overline{\gamma}^\ast[S^{-1}(\overline{\zeta}^\ast(b^\ast_i)_{(1)})]\#b_j\right)
\otimes
\left(\overline{\gamma}^\ast[S^{-1}(\overline{\zeta}^\ast(b^\ast_j)
  \overline{\zeta}^\ast(b^\ast_i)_{(2)})]\#1\right)\right)  \\
&=&
\sum_{i,j}\left(\e\#b_i\right)
\otimes\left(\e\#b_j\right)
\left(\overline{\gamma}^\ast[S^{-1}(\overline{\zeta}^\ast(b^\ast_i)_{(1)})]\#1\right)
\otimes
\left(\overline{\gamma}^\ast[S^{-1}(\overline{\zeta}^\ast(b^\ast_j)
  \overline{\zeta}^\ast(b^\ast_i)_{(2)})]\#1\right)  \\
&=&
\pd{\phi}.
\end{eqnarray*}
\end{itemize}
\end{proof}

\begin{remark}
Of course, the coopposite quasi-Hopf algebra $((H/B^+H)^\ast\#B)^\cop$ might be also described analogously. It should be isomorphic to the left partial dual determined by a {\pams} for an injection from $B^\cop$ to $H^\op$ or $H^\cop$ of left comodule algebras, but we would not discuss this situation here.
\end{remark}

\subsection{Right partially dualized coquasi-Hopf algebras}\label{subsection:5.3}

Up to now, the definitions and a number of properties of {\pams}s and left partially dualized quasi-Hopf algebras are introduced. However, we need the dual versions of them, which would be applied more directly in Sections \ref{section:particularcases} and \ref{section:genuinequasi-Hopf}.

Let us begin by recalling the construction of a generalized smash coproduct introduced in \cite[Section 4]{CMZ97} in our situations: Suppose that $B$ is a left coideal subalgebra of a finite-dimensional Hopf algebra $H$ with a {\pams} $(\zeta,\gamma^\ast)$:
\begin{equation*}
\begin{array}{ccc}
\xymatrix{
B \ar@<.5ex>[r]^{\iota} & H \ar@<.5ex>@{-->}[l]^{\zeta} \ar@<.5ex>[r]^{\pi\;\;\;\;\;\;}
& H/B^+H \ar@<.5ex>@{-->}[l]^{\gamma\;\;\;\;\;\;}  }
&\;\;\text{and}\;\;&
\xymatrix{
(H/B^+H)^\ast \ar@<.5ex>[r]^{\;\;\;\;\;\;\pi^\ast}
& H^\ast \ar@<.5ex>@{-->}[l]^{\;\;\;\;\;\;\gamma^\ast} \ar@<.5ex>[r]^{\iota^\ast}
& B^\ast \ar@<.5ex>@{-->}[l]^{\zeta^\ast}.  }
\end{array}
\end{equation*}
Note that $B^\ast$ is a right $H$-comodule coalgebra with the coaction
$$B^\ast\rightarrow B^\ast\otimes H,\;\;\;\;
b^\ast\mapsto\sum_i (h_i^\ast\btr b^\ast)\otimes h_i$$
where $\{h_i\}$ is a linear basis of $H$ with dual basis $\{h_i^\ast\}$ of $H^\ast$,
and $H/B^+H$ is a right $H$-module coalgebra via $\btl$.
Here we recall that the actions $\btl$ and $\btr$ satisfy the Equations (\ref{eqn:btl}) and (\ref{eqn:btr}) as follows:
\begin{equation}\label{eqn:btlbtrnew(MAMS)}
x\btl h=\pi[\gamma(x)h]\;\;\;\;
\;\;\;\;\text{and}\;\;\;\;\;\;\;\;
h^\ast\btr b^\ast=\iota^\ast[h^\ast\zeta^\ast(b^\ast)]
\end{equation}
for all $h\in H$, $x\in H/B^+H$, $h^\ast\in H^\ast$, $b\in B^\ast$.

As a conclusion, we could formulate the \textit{smash coproduct} coalgebra $H/B^+H\btd B^\ast$ of $H/B^+H$ and $B^\ast$ as follows:
As a vector space, $H/B^+H\btd B^\ast:=H/B^+H\otimes B^\ast$. The comultiplication and counit are given by:
\begin{equation}\label{eqn:rightpdcomultiplication}
x\btd b^\ast\mapsto
\sum_i \left[x_{(1)}\btd(h_i^\ast\btr b^\ast_{(1)})\right]
  \otimes\left[(x_{(2)}\btl h_i)\btd b^\ast_{(2)}\right]
\end{equation}
and $x\btd b^\ast\mapsto\la\e_C,x\ra\la b^\ast,1_B\ra$ respectively for all $x\in H/B^+H$ and $b^\ast\in B^\ast$.

Now with the help of the notion of left partial duals, we introduce the definition of \textit{right partially dualized coquasi-Hopf algebras}, whose coalgebra structure is in fact the smash coproduct above:

\begin{definition}\label{def:rightpartialdual}
Suppose $(H/B^+H)^\ast\#B$ is the left partial dual of $H$ determined by the {\pams} $(\zeta,\gamma^\ast)$.
Then $H/B^+H\btd B^\ast$ is called the right partially dualized coquasi-Hopf algebra (or right partial dual for simplicity) of $H$ determined by $(\zeta,\gamma^\ast)$, if
\begin{equation}\label{eqn:PAMS(right)}
(H/B^+H)^\ast\#B\cong(H/B^+H\btd B^\ast)^\ast,\;\;\;\;
f\#b\mapsto\la f\otimes b,-\ra
\end{equation}
is an isomorphism of quasi-Hopf algebras.
\end{definition}

Here we omit the whole structures of the right partial dual $H/B^+H\btd B^\ast$ as a coquasi-Hopf algebra, which are completely dual to Theorem \ref{thm:partialdual}, but remark that it has ``unit element'' and ``multiplication'' as follows:

\begin{proposition}\label{prop:rightpdmultiplication}
Let $H$ be a finite-dimensional Hopf algebra. Suppose that $B$ is a left coideal subalgebra of $H$ with a {\pams} $(\zeta,\gamma^\ast)$. Then the right partially dualized coquasi-Hopf algebra $H/B^+H\btd B^\ast$ determined by $(\zeta,\gamma^\ast)$ has
the ``unit element'' $\pi(1)\btd\iota^\ast(\e)$, and the ``multiplication'' given by:
\begin{equation}\label{eqn:rightpdmultiplication}
(x\btd b^\ast)(y\btd c^\ast)
=\sum\left(x\btl[\zeta^\ast(b^\ast_{(1)})\rightharpoonup\gamma(y_{(1)})]\right)
  \btd\left([\zeta^\ast(b^\ast_{(2)})\leftharpoonup\gamma(y_{(2)})]\btr c^\ast\right)
\end{equation}
for all $x,y\in H/B^+H$ and $b^\ast,c^\ast\in B^\ast$.
\end{proposition}

\begin{proof}
Since the linear isomorphism dual to (\ref{eqn:PAMS(right)}) maps $\pi(1)\btd\iota^\ast(\e)$ to the ``counit'' (\ref{eqn:epsilon}) of the left partial dual $(H/B^+H)^\ast\#B$, the element $\pi(1)\btd\iota^\ast(\e)$ is in fact the ``unit element'' of the right partial dual $H/B^+H\btd B^\ast$.

On the other hand, note that the definitions of $\btl$ and $\btr$ imply that
\begin{equation}\label{eqn:btlbtr}
\sum \la h^\ast,b_{(1)}\ra\la b^\ast,b_{(2)}\ra
=\la h^\ast\btr b^\ast,b\ra
\;\;\;\;\text{and}\;\;\;\;
\sum \la f_{(1)},x\ra\la f_{(2)},h\ra
=\la f,x\btl h\ra
\end{equation}
hold for all $h^\ast\in H^\ast$, $b^\ast\in B^\ast$, $b\in B$, $f\in(H/B^+H)^\ast$, $x\in H/B^+H$ and $h\in H$.
Then according to the duality between
$(H/B^+H)^\ast\#B$ and $H/B^+H\btd B^\ast$ in the sense of (\ref{eqn:PAMS(right)}), we calculate for any $f\in(H/B^+H)^\ast$ and $b\in B$ that
\begin{eqnarray*}
&& \la f\#b,(x\btd b^\ast)(y\btd c^\ast)\ra  \\
&=&
\la \pd{\Delta}(f\#b),(x\btd b^\ast)\otimes(y\btd c^\ast)\ra  \\
&\overset{(\ref{eqn:Deltacomplete})}=&
\sum_{i,j}\la f_{(1)}\#b_i\zeta[\gamma(f_j^\ast)b_{(1)}],x\btd b^\ast\ra
  \la \gamma^\ast[f_{(2)}\zeta^\ast(b^\ast_i)]f_j\#b_{(2)},y\btd c^\ast\ra  \\
&=&
\sum_{i,j}\la f_{(1)},x\ra\la b^\ast,b_i\zeta[\gamma(f_j^\ast)b_{(1)}]\ra
  \la \gamma^\ast[f_{(2)}\zeta^\ast(b^\ast_i)]f_j,y\ra\la c^\ast,b_{(2)}\ra  \\
&=&
\sum\la f_{(1)},x\ra\la b^\ast_{(2)},\zeta[\gamma(y_{(2)})b_{(1)}]\ra
  \la \gamma^\ast[f_{(2)}\zeta^\ast(b^\ast_{(1)})],y_{(1)}\ra\la c^\ast,b_{(2)}\ra  \\
&=&
\sum\la f_{(1)},x\ra\la \zeta^\ast(b^\ast_{(2)}),\gamma(y_{(2)})b_{(1)}\ra
  \la f_{(2)}\zeta^\ast(b^\ast_{(1)}),\gamma(y_{(1)})\ra\la c^\ast,b_{(2)}\ra  \\
&=&
\sum\la f_{(1)},x\ra\la \zeta^\ast(b^\ast_{(2)})\leftharpoonup\gamma(y_{(2)}),b_{(1)}\ra
  \la f_{(2)},\zeta^\ast(b^\ast_{(1)})\rightharpoonup\gamma(y_{(1)})\ra\la c^\ast,b_{(2)}\ra
\\
&\overset{(\ref{eqn:btlbtr})}{=}&
\sum\left\la f,x\btl[\zeta^\ast(b^\ast_{(1)})\rightharpoonup\gamma(y_{(1)})]\right\ra
  \left\la[\zeta^\ast(b^\ast_{(2)})\leftharpoonup\gamma(y_{(2)})]\btr c^\ast,b\right\ra  \\
&=&
\left\la f\#b,\sum\left(x\btl[\zeta^\ast(b^\ast_{(1)})\rightharpoonup\gamma(y_{(1)})]\right)
  \btd
  \left([\zeta^\ast(b^\ast_{(2)})\leftharpoonup\gamma(y_{(2)})]\btr c^\ast\right)\right\ra,
\end{eqnarray*}
where $\{b_i\}$ is a linear basis of $B$ with dual basis $\{b_i^\ast\}$ of $B^\ast$, and $\{f_i\}$ is a linear basis of $(H/B^+H)^\ast$ with dual basis $\{f_i^\ast\}$ of $H/B^+H$ as usual. Thus (\ref{eqn:rightpdmultiplication}) is obtained as a conclusion.
\end{proof}

One could see in subsequent sections that the notions of left and right partial duals are both useful for constructions.
As mentioned in the following lemma,
although the ``multiplication'' of the coquasi-Hopf algebra $H/B^+H\btd B^\ast$ might not be associative, it has some similar formulas with smash product algebras.

For simplicity, let us denote
\begin{equation}\label{eqn:1CeB}
1_C=\pi(1)\in C\;\;\;\;\;\;\;\;\text{and}\;\;\;\;\;\;\;\;
\e_B=\iota^\ast(\e)\in B^\ast.
\end{equation}
Then since $\pi$ and $\iota^\ast$ are coalgebra maps, we could write
$$\Delta_C(1_C)=1_C\otimes 1_C
\;\;\;\;\;\;\;\;\text{and}\;\;\;\;\;\;\;\;
\Delta_{B^\ast}(\e_B)=\e_B\otimes\e_B.$$
It follows also that
$\gamma(1_C)=1$ and $\zeta^\ast(\e_B)=\e$ due to the assumption that $\zeta$ and $\gamma$ are biunitary.

\begin{lemma}
In the right partially dualized coquasi-Hopf algebra $H/B^+H\btd B^\ast$, the following equations
\begin{equation}\label{eqn:rightPDmultiplicationORIGINAL(MAMS)}
(x\btd \e_B)(y\btd c^\ast)
=\big(x\btl\gamma(y)\big)\btd c^\ast\;\;\;\;\text{and}\;\;\;\;
(x\btd b^\ast)(1_B\btd c^\ast)
=x \btd\big(\zeta^\ast(b^\ast)\btr c^\ast)
\end{equation}
hold for all $x,y\in C$ and $b^\ast,c^\ast\in B^\ast$
\end{lemma}

\begin{proof}
The first equation could be obtained by direct calculations:
Since $\zeta$ and $\gamma$ are both counitary, we find
for any $x,y\in H/B^+H$ and $c^\ast\in B^\ast$ that
\begin{eqnarray*}
(x\btd \e_B)(y\btd c^\ast)
&\overset{(\ref{eqn:rightpdmultiplication})}=&
  \sum\big(x\btl[\zeta^\ast(\e_B)\rightharpoonup\gamma(y_{(1)})]\big)
  \btd\big([\zeta^\ast(\e_B)\leftharpoonup\gamma(y_{(2)})]\btr c^\ast\big)  \\
&\overset{(\ref{eqn:zetabiunitary})}=&
  \sum\big(x\btl[\e\rightharpoonup\gamma(y_{(1)})]\big)
  \btd\big([\e\leftharpoonup\gamma(y_{(2)})]\btr c^\ast\big)  \\
&=&
\sum\big(x\btl\gamma(y_{(1)})\big)
  \btd\big(\la\e,\gamma(y_{(2)})\ra\e\btr c^\ast\big)  \\
&\overset{(\ref{eqn:gammabiunitary})}=&
\big(x\btl\gamma(y)\big)\btd c^\ast.
\end{eqnarray*}
It is straightforward to prove the other equation in (\ref{eqn:rightPDmultiplicationORIGINAL(MAMS)}) in a similar way.
\end{proof}

\section{Conditions and examples when the associator is trivial}\label{section:particularcases}

This section is devoted to discussing the situation when a left partially dualized quasi-Hopf algebra has trivial associator. Specifically,
we first provide some sufficient or necessary conditions for it. Then in this case, the orders of (left or right) partial dualizations would be studied.

Later in Subsections \ref{subsection:matchedpairofgrps},  \ref{subsection:partialdualbiproduct} and \ref{subsection:Taftalg},
let us attempt to formulate several classical structures in the literature as examples of the partially dualized Hopf algebras. We also remark that the quantum double Hopf algebra (\cite{DT94,Dri86}) could also be formulated as left partial duals, which is considered in another manuscript.

To begin with, let us rewrite a {\pams} as
\begin{equation*}
\begin{array}{ccc}
\xymatrix{
B \ar@<.5ex>[r]^{\iota} & H \ar@<.5ex>@{-->}[l]^{\zeta} \ar@<.5ex>[r]^{\pi}
& C \ar@<.5ex>@{-->}[l]^{\gamma}  }
&\;\;\text{and}\;\;&
\xymatrix{
C^\ast \ar@<.5ex>[r]^{\pi^\ast}
& H^\ast \ar@<.5ex>@{-->}[l]^{\gamma^\ast} \ar@<.5ex>[r]^{\iota^\ast}
& B^\ast \ar@<.5ex>@{-->}[l]^{\zeta^\ast}  }
\end{array}
\end{equation*}
of the form in Definition \ref{def:PAMS} (without identifying $B\subseteq H$ and $C=H/B^+H$), for the purpose of studying general examples. Meanwhile, the left partial dual and right partial dual of $H$ determined by $(\zeta,\gamma^\ast)$ above are denoted by
$$C^\ast\# B\;\;\;\;\;\;\;\;\text{and}\;\;\;\;\;\;\;\;C\btd B^\ast$$
respectively.

\subsection{Admissible mapping systems (determining partially dualized Hopf algebras)}\label{subsection:6.1}

Firstly by the structures in Theorem \ref{thm:partialdual}, it is easy to determine the left partial duals for two trivial choices of left coideal subalgebras $B\subseteq H$ as follows:

\begin{corollary}\label{cor:trivialcases}
Let $H$ be a finite-dimensional Hopf algebra with a left coideal subalgebra $B$.
Denote by $u:\k\rightarrow H$ the unit of $H$.
\begin{itemize}
\item[(1)]
If $B=H$, then there is a unique {\pams} $(\id_H,u^\ast)$ for $\id_H$, and it determines the left partial dual
$$C^\ast\#B=\e\#H\cong H$$
as Hopf algebras;

\item[(2)]
If $B=\k1$, then there is a unique {\pams} $(\e,\id_H^\ast)$ for $u$, and it determines the left partial dual
$$C^\ast\#B=H^\ast\#1\cong H^\ast$$
as Hopf algebras.
\end{itemize}
\end{corollary}

\begin{remark}
The isomorphisms in (1) and (2) Corollary \ref{cor:trivialcases} could be regarded as algebraic versions of the tensor equivalences in \cite[Example 7.12.3]{EGNO15} and \cite[Theorem 4.2]{Ost03} respectively:
$$\C_\C^\ast\approx\C\;\;\;\;\;\;{and}\;\;\;\;\;\;\Rep(H)_\Vec^\ast\approx \Rep(H^\ast),$$
where $\C$ is a finite tensor category, and $H$ is a finite-dimensional Hopf algebra.
\end{remark}

Similar statements hold on right partial duals $C\btd B^\ast$ for the left coideal subalgebra $B\subseteq H$ being $H$ or $\k1$.

In conclusion,
note that the assumptions of Corollary \ref{cor:trivialcases} imply that the left and right partial duals both become Hopf algebras.
More generally, let us recall by Remark \ref{rmk:equivDelta}(3) that the associator $\pd{\phi}$ of the left partial dual $C^\ast\#B$ determined by $(\zeta,\gamma^\ast)$ has inverse
\begin{equation}\label{eqn:phi^-1(MAMS)}
\pd{\phi}^{-1}=\sum_{i,j}\left(\e\#\zeta[\gamma(f_i^\ast)\gamma(f_j^\ast)_{(1)}]\right)
\otimes\left(f_i\#\zeta[\gamma(f_j^\ast)_{(2)}]\right)
\otimes(f_j\#1)
\in(C^\ast\#B)^{\otimes 3},
\end{equation}
where $\{f_i\}$ is a linear basis of $C^\ast$ with dual basis $\{f_i^\ast\}$ of $C$.
Analogous to the notion introduced in \cite[Section 2.8]{Rad85}, we would use in this paper the term of \textit{\ams} for the case when $\pd{\phi}$ is trivial:

\begin{definition}\label{def:ams}
We say that a {\pams} $(\zeta,\gamma^\ast)$ is admissible, if the associator $\pd{\phi}$ of its left partial dual (or namely, the reassociator of the right partial dual) is trivial.

In this situation, we also call $(\zeta,\gamma^\ast)$ an {\ams}, and call $C^\ast\#B$ (resp. $C\btd B^\ast$) a left (resp. right) partially dualized Hopf algebra of $H$.
\end{definition}

\begin{remark}\label{rmk:associatortrivial(MAMS)}
By calculating the image of (\ref{eqn:phi^-1(MAMS)}) under
$(1\otimes \id_B)\otimes(x\otimes\id_B)\otimes(y\otimes\e) $
for arbitrary elements $x,y\in C$,
one could find that
$(\zeta,\gamma^\ast)$ is admissible
if and only if the equations
\begin{equation}\label{eqn:associatortrivial(MAMS)}
\sum\zeta[\gamma(x)\gamma(y)_{(1)}]\otimes\zeta[\gamma(y)_{(2)}]
=\la \e_C,x\ra\la \e_C,y\ra (1\otimes1)\;\;\;\;\;\;\;\;(\forall x,y\in C)
\end{equation}
holds.
\end{remark}

Now we provide more applicable sufficient and necessary conditions when a {\pams} is admissible.

\begin{proposition}\label{lem:PAMSdetermineHopfalgs(MAMS)}
Suppose \begin{equation*}
\begin{array}{ccc}
\xymatrix{
B \ar@<.5ex>[r]^{\iota} & H \ar@<.5ex>@{-->}[l]^{\zeta} \ar@<.5ex>[r]^{\pi}
& C \ar@<.5ex>@{-->}[l]^{\gamma}  }
&\;\;\text{and}\;\;&
\xymatrix{
C^\ast \ar@<.5ex>[r]^{\pi^\ast}
& H^\ast \ar@<.5ex>@{-->}[l]^{\gamma^\ast} \ar@<.5ex>[r]^{\iota^\ast}
& B^\ast \ar@<.5ex>@{-->}[l]^{\zeta^\ast}  },
\end{array}
\end{equation*}
is a {\pams}. Then the followings are equivalent:
\begin{itemize}
\item[(1)]
$(\zeta,\gamma^\ast)$ is admissible;

\item[(2)]
$\gamma(C)$ is a right coideal subalgebra of $H$;

\item[(3)]
$\zeta^\ast(B^\ast)$ is a left coideal subalgebra of $H^\ast$.
\end{itemize}
\end{proposition}

\begin{proof}
Firstly, let us show that (2) implies (1):
Suppose $\gamma(C)$ is a right coideal subalgebra of $H$. Then
$\sum\gamma(x)\gamma(y)_{(1)}\otimes\gamma(y)_{(2)}\in \gamma(C)\otimes H$
holds for any $x,y\in C$.
However, since $\zeta\circ\gamma$ is trivial and $\gamma$ is counitary, we could find as a consequence that
\begin{eqnarray*}
\sum\zeta[\gamma(x)\gamma(y)_{(1)}]\otimes\zeta[\gamma(y)_{(2)}]
&=& \sum\la\e,\gamma(x)\gamma(y)_{(1)}\ra 1\otimes\zeta[\gamma(y)_{(2)}]  \\
&=& \la \e,x\ra (1\otimes\zeta[\gamma(y)])  \\
&=& \la \e,x\ra\la \e,y\ra(1\otimes1)\;\;\;\;\;\;\;\;(\forall x,y\in C);
\end{eqnarray*}
This is equivalent to that the associator $\pd{\phi}$ is trivial due to Remark \ref{rmk:associatortrivial(MAMS)}.

Then we aim to prove that (1) implies (2):
Note in Definition \ref{def:PAMS}(4)
that $\gamma$ is a right $C$-comodule map, which means that
\begin{equation}\label{eqn:gamma(y)12(MAMS)}
\sum \gamma(x)_{(1)}\otimes\pi[\gamma(x)_{(2)}]
=\sum \gamma(x_{(1)})\otimes x_{(2)}
\end{equation}
and hence
\begin{equation}\label{eqn:gamma(y)123(MAMS)}
\sum \gamma(y)_{(1)}\otimes\gamma(y)_{(2)}\otimes\pi[\gamma(y)_{(3)}]
=\sum \gamma(y_{(1)})_{(1)}\otimes\gamma(y_{(1)})_{(2)}\otimes y_{(2)}
\end{equation}
hold for all $x,y\in C$.
Furthermore,
according to the requirement
\begin{equation}\label{eqn:convolutionprod(MAMS)}
(\iota\circ\zeta)\ast(\gamma\circ\pi)=\id_H
\end{equation}
in Definition \ref{def:PAMS}(6), we have direct calculations that
\begin{eqnarray}\label{eqn:associatortrivial-result1}
\sum\zeta[\gamma(x)\gamma(y)_{(1)}]\otimes\gamma(y)_{(2)}
&\overset{(\ref{eqn:convolutionprod(MAMS)})}{=}&
\sum \zeta[\gamma(x)\gamma(y)_{(1)}]\otimes
  \iota\left(\zeta[\gamma(y)_{(2)}]\right)\gamma\left(\pi[\gamma(y)_{(3)}]\right)  \nonumber  \\
&\overset{(\ref{eqn:gamma(y)123(MAMS)})}{=}&
\sum \zeta[\gamma(x)\gamma(y_{(1)})_{(1)}]\otimes
  \iota\left(\zeta[\gamma(y_{(1)})_{(2)}]\right)\gamma(y_{(2)})  \nonumber  \\
&\overset{(\ref{eqn:associatortrivial(MAMS)})}=&
  \la\e,x\ra (1\otimes\gamma(y))
\end{eqnarray}
hold for all $x,y\in C$.
Thus
\begin{eqnarray}\label{eqn:associatortrivial-result2}
\gamma(x)\gamma(y)
&=& \sum \iota\left(\zeta[\gamma(x)_{(1)}\gamma(y)_{(1)}]\right)
  \gamma\left(\pi[\gamma(x)_{(2)}\gamma(y)_{(2)}]\right)  \nonumber  \\
&\overset{(\ref{eqn:btl})}{=}&
\sum \iota\left(\zeta[\gamma(x)_{(1)}\gamma(y)_{(1)}]\right)
  \gamma\left(\pi[\gamma(x)_{(2)}]\btl\gamma(y)_{(2)}\right)  \nonumber  \\
&\overset{(\ref{eqn:gamma(y)12(MAMS)})}{=}&
\sum \iota\left(\zeta[\gamma(x_{(1)})\gamma(y)_{(1)}]\right)
  \gamma\left(x_{(2)}\btl\gamma(y)_{(2)}\right)  \nonumber  \\
&\overset{(\ref{eqn:associatortrivial-result1})}=&
  \gamma[x\btl\gamma(y)]\;\;\in\;\gamma(C)
\end{eqnarray}
for all $x,y\in C$. In other words, $\gamma(C)$ is a subalgebra of $H$, since $\gamma$ is unitary as well.
Meanwhile, for any $x\in C$, it follows from Equation (\ref{eqn:associatortrivial-result1}) that
\begin{equation}\label{eqn:zetagamma123(MAMS)}
\sum \zeta[\gamma(x)_{(1)}]\otimes \gamma(x)_{(2)}\otimes \gamma(x)_{(3)}
=\sum 1\otimes \gamma(x)_{(1)}\otimes \gamma(x)_{(2)},
\end{equation}
and consequently
\begin{eqnarray}\label{eqn:associatortrivial-result3}
\Delta(\gamma(x)) &=& \sum\gamma(x)_{(1)}\otimes\gamma(x)_{(2)}
~\overset{(\ref{eqn:convolutionprod(MAMS)})}{=}~
\sum \iota\left(\zeta[\gamma(x)_{(1)}]\right)
  \gamma\left(\pi[\gamma(x)_{(2)}]\right)\otimes\gamma(x)_{(3)}  \nonumber  \\
&\overset{(\ref{eqn:zetagamma123(MAMS)})}{=}&
\sum \gamma\left(\pi[\gamma(x)_{(1)}]\right)\otimes\gamma(x)_{(2)}
  \;\;\in\;\gamma(C)\otimes H.
\end{eqnarray}
As a conclusion, $\gamma(C)$ is a right coideal subalgebra of $H$.

Finally, with the help of the equivalence between (1) and (2), we try to explain that (1) and (3) are equivalent as well: Recall in Propositions \ref{prop:BCP-selfdual} and \ref{prop:partialdualbiop} that $(\gamma^\ast,\zeta)$ would be also a {\pams}
$$\begin{array}{ccc}
\xymatrix{
C^{\ast\,\biop} \ar@<.5ex>[r]^{\pi^\ast}
& H^{\ast\,\biop}
\ar@<.5ex>[l]^{\gamma^\ast}
\ar@<.5ex>[r]^{\iota^\ast}
& B^{\ast\,\biop} \ar@<.5ex>[l]^{\zeta^\ast}  }
&\text{and}&
\xymatrix{
B^\biop \ar@<.5ex>[r]^{\iota}
& H^\biop
\ar@<.5ex>[l]^{\zeta}
\ar@<.5ex>[r]^{\pi}
& C^\biop \ar@<.5ex>[l]^{\gamma}  },
\end{array}$$
and it determines left partial dual $B^\biop\#C^{\ast\,\biop}$ which is isomorphic to $(C^\ast\#B)^\biop$ as a quasi-bialgebra.
Therefore, (1) is also equivalent to the claim that $(\gamma^\ast,\zeta)$ is an {\ams}, which holds if and only if $\zeta^\ast(B^{\ast\,\biop})$ is a right coideal subalgebra of $H^{\ast\,\biop}$, or equivalently,
$\zeta^\ast(B^\ast)$ is a left coideal subalgebra of $H^\ast$.
\end{proof}

\begin{remark}
One might find the following fact by Proposition \ref{lem:PAMSdetermineHopfalgs(MAMS)}:
If $(\zeta,\gamma^\ast)$ is an {\ams}, then
\begin{equation*}
\begin{array}{ccc}
\xymatrix{
B^\ast \ar@<.5ex>[r]^{\zeta^\ast}
& H^\ast \ar@<.5ex>@{-->}[l]^{\iota^\ast} \ar@<.5ex>[r]^{\gamma^\ast}
& C^\ast \ar@<.5ex>@{-->}[l]^{\pi^\ast}  }
&\;\;\text{and}\;\;&
\xymatrix{
C \ar@<.5ex>[r]^{\gamma} & H \ar@<.5ex>@{-->}[l]^{\pi} \ar@<.5ex>[r]^{\zeta}
& B \ar@<.5ex>@{-->}[l]^{\iota}  }
\end{array}
\end{equation*}
is also an {\ams}.
\end{remark}

As direct consequences of Proposition \ref{lem:PAMSdetermineHopfalgs(MAMS)}, some additional sufficient conditions are obtained:

\begin{corollary}\label{cor:Hopfalgconditions}
Suppose \begin{equation*}
\begin{array}{ccc}
\xymatrix{
B \ar@<.5ex>[r]^{\iota} & H \ar@<.5ex>@{-->}[l]^{\zeta} \ar@<.5ex>[r]^{\pi}
& C \ar@<.5ex>@{-->}[l]^{\gamma}  }
&\;\;\text{and}\;\;&
\xymatrix{
C^\ast \ar@<.5ex>[r]^{\pi^\ast}
& H^\ast \ar@<.5ex>@{-->}[l]^{\gamma^\ast} \ar@<.5ex>[r]^{\iota^\ast}
& B^\ast \ar@<.5ex>@{-->}[l]^{\zeta^\ast}  },
\end{array}
\end{equation*}
is a {\pams}. Then it is admissible, if one of the following conditions hold:
\begin{itemize}
\item[(1)]
$B$ is a bialgebra, and $\zeta:H\rightarrow B$ is a bialgebra map;

\item[(2)]
$C$ is a bialgebra, and $\gamma:C\rightarrow H$ is a bialgebra map;

\item[(3)]
$\zeta$ is an algebra map, and $\gamma$ is a coalgebra map.
\end{itemize}
\end{corollary}

\begin{proof}
\begin{itemize}
\item[(1)]
Under the assumptions, it is evident that $\zeta^\ast:B^\ast\rightarrow H^\ast$ is also a bialgebra map. Thus $\zeta^\ast(B^\ast)$ would be a subbialgebra and hence a left coideal subalgebra of $H^\ast$. It follows from Proposition \ref{lem:PAMSdetermineHopfalgs(MAMS)} that $(\zeta,\gamma^\ast)$ is an {\ams}.

\item[(2)]
Similarly to the argument with (1), it is because $\gamma(C)$ would be a subbialgebra and hence a right coideal subalgebra of $H$.

\item[(3)]
Let us verify Equation (\ref{eqn:associatortrivial(MAMS)}) under the assumptions:: For any $x,y\in C$,
\begin{eqnarray*}
\sum\zeta[\gamma(x)\gamma(y)_{(1)}]\otimes\zeta[\gamma(y)_{(2)}]
&=&
\sum\zeta[\gamma(x)\gamma(y_{(1)})]\otimes\zeta[\gamma(y_{(2)})]  \\
&=&
\sum\zeta[\gamma(x)]\zeta[\gamma(y_{(1)})]\otimes\zeta[\gamma(y_{(2)})]  \\
&=& \la \e,x\ra\la \e,y\ra (1\otimes1),
\end{eqnarray*}
where the last equality is due to Proposition \ref{prop:PAMS-comptrival}(2) that $\zeta\circ\gamma$ is trivial.
\end{itemize}
\end{proof}

\begin{remark}
As long as the left partial dualized quasi-Hopf algebra $C^\ast\#B$ has trivial associator $\pd{\phi}$, the element $\pd{\upsilon}$ (\ref{eqn:upsilon}) would become the unit element $\e\#1$. This is due to \cite[Remark 2)]{Dri89} and the fact that $\pd{\upsilon}=\pd{\beta}\pd{\alpha}$.
\end{remark}


\subsection{Partial self-duality property and multiple partial dualizations}\label{subsection:6.2}

Throughout this subsection, let $H$ be a finite-dimensional Hopf algebra, and let
\begin{equation}\label{eqn:ams6.2}
\begin{array}{ccc}
\xymatrix{
B \ar@<.5ex>[r]^{\iota} & H \ar@<.5ex>@{-->}[l]^{\zeta} \ar@<.5ex>[r]^{\pi}
& C \ar@<.5ex>@{-->}[l]^{\gamma}  }
&\;\;\text{and}\;\;&
\xymatrix{
C^\ast \ar@<.5ex>[r]^{\pi^\ast}
& H^\ast \ar@<.5ex>@{-->}[l]^{\gamma^\ast} \ar@<.5ex>[r]^{\iota^\ast}
& B^\ast \ar@<.5ex>@{-->}[l]^{\zeta^\ast}  }
\end{array}
\end{equation}
be an \textit{admissible} mapping system.
Recall by Definition \ref{def:ams} that $(\zeta,\gamma^\ast)$ would determine (left and right) partially dualized Hopf algebras $C^\ast\#B$ and $C\btd B^\ast$ respectively.

In simple terms, the aim of this subsection is to describe the following properties of partially dualized Hopf algebras:
\begin{itemize}
\item
$H$ is the canonical right (resp. left) partial dual of the left (resp. right) partial dual of $H$;

\item
$H^\ast$ is the canonical left (resp. right) partial dual of the left (resp. right) partial dual of $H$;

\item
$H$ is obtained by quadruple canonical left (resp. right) partial dualizations from $H$ itself.
\end{itemize}
These claims might be concluded visually as a diagram:
\begin{equation}\label{eqn:multiplePD(MAMS)}
\xymatrix{
\;\;\;\;H\;\;\;\;
  \ar@<0.75ex>@{|->}[rrr]^{\text{left partial dualization}}
  \ar@<0.75ex>@{|-->}[ddd]^{\text{\shortstack{right partial\\dualization}}}
&&& C^\ast\#B
  \ar@<0.75ex>@{|->}[ddd]^{\text{\shortstack{left partial\\dualization}}}
  \ar@<0.75ex>@{|-->}[lll]^{\text{right partial dualization}}
\\  \\  \\  C\btd B^\ast
  \ar@<0.75ex>@{|->}[uuu]^{\text{\shortstack{left partial\\dualization}}}
  \ar@<0.75ex>@{|-->}[rrr]^{\text{right partial dualization}}
&&& \;\;\;H^\ast\;\;.
  \ar@<0.75ex>@{|->}[lll]^{\text{left partial dualization}}
  \ar@<0.75ex>@{|-->}[uuu]^{\text{\shortstack{right partial\\dualization}}}
}
\end{equation}

For this purpose, our first goal is to establish a canonical {\ams} for the left coideal subalgebra $C\cong C\btd\iota^\ast(\e)$ of the right partially dualized Hopf algebra $C\btd B^\ast$.
Here we still use the notations
\begin{equation}\label{eqn:1CeB(MAMS)}
1_C=\pi(1)\in C\;\;\;\;\;\;\;\;\text{and}\;\;\;\;\;\;\;\;
\e_B=\iota^\ast(\e)\in B^\ast
\end{equation}
in (\ref{eqn:1CeB}) for simplicity. In addition, recall in (\ref{eqn:btlbtrnew(MAMS)}) that the actions $\btl$ and $\btr$ have been defined such that: For all $h\in H$, $x\in C$, $h^\ast\in H^\ast$, $b\in B^\ast$, we could write
\begin{equation}\label{eqn:btlbtrnew2(MAMS)}
x\btl h=\pi[\gamma(x)h]\;\;\;\;
\;\;\;\;\text{and}\;\;\;\;\;\;\;\;
h^\ast\btr b^\ast=\iota^\ast[h^\ast\zeta^\ast(b^\ast)],
\end{equation}
or equivalently,
\begin{equation}\label{eqn:btlbtrnew3(MAMS)}
\la f,x\btl h\ra=\sum\la f_{(1)},x\ra\la f_{(2)},h\ra
\;\;\;\;\text{and}\;\;\;\;
\la h^\ast\btr b^\ast,b\ra=\sum\la h^\ast,b_{(1)}\ra\la b^\ast,b_{(2)}\ra
\end{equation}
with notations (\ref{eqn:coidealsubalgsnotation(MAMS)}) for all $f\in C^\ast$ and $b\in B$.

\begin{lemma}\label{lem:Cstru&B*stru}
Suppose (\ref{eqn:ams6.2}) is an {\ams}. Then:
\begin{itemize}
\item[(1)]
$C$ is a left $C\btd B^\ast$-comodule algebra, with multiplication given by
\begin{equation}\label{eqn:Cmultiplication(MAMS)}
\cdot:\;C\otimes C\rightarrow C,\;\;x\otimes y\mapsto x\btl \gamma(y)
\end{equation}
and unit element $1_C$, and the left $C\btd B^\ast$-comodule structure on $C$ is
$$\rho_C:\;C\rightarrow (C\btd B^\ast)\otimes C,\;\;
x\mapsto\sum_i [x_{(1)}\btd\iota^\ast(h_i^\ast)]\otimes(x_{(2)}\btl h_i),$$
where $\{h_i\}$ is a linear basis of $H$ with dual basis $\{h_i^\ast\}$ of $H^\ast$.

\item[(2)]
The right $C^\ast\# B^\ast$-comodule structure
$$\rho_B:\;B\rightarrow B\otimes (C^\ast\# B),\;\;
b \mapsto \sum_j\zeta[\gamma(f_j^\ast)b_{(1)}]\otimes(f_j\# b_{(2)})$$
makes the algebra $B$ a right $C^\ast\# B$-comodule algebra,
where $\{f_j\}$ is a linear basis of $C^\ast$ with dual basis $\{f_j^\ast\}$ of $C$.

\item[(3)]
The right $C\btd B^\ast$-module structure
\begin{eqnarray*}
B^\ast\otimes(C\btd B^\ast) &\rightarrow& B^\ast,\;\;  \\
b^\ast\otimes(y\btd c^\ast) &\mapsto& [\zeta^\ast(b^\ast)\leftharpoonup\gamma(y)]\btr c^\ast
\end{eqnarray*}
makes the coalgebra $B^\ast$ a right $C\btd B^\ast$-module coalgebra.
\end{itemize}
\end{lemma}

\begin{proof}
\begin{itemize}
\item[(1)]
Firstly, it is direct to check that $C\btd \e_B$ is a left coideal subalgebra of $C\btd B^\ast$. Indeed, for any $x,y\in C$, we find in the Hopf algebra $C\btd B^\ast$ that
\begin{equation*}
(x\btd\e_B)(y\btd\e_B)=(x\btl\gamma(y)\big)\btd\e_B
\end{equation*}
by Equation (\ref{eqn:rightPDmultiplicationORIGINAL(MAMS)}), as well as that
\begin{eqnarray*}
\pd{\Delta}(x\btd \e_B)
&\overset{(\ref{eqn:rightpdcomultiplication})}=&
\sum_i [x_{(1)}\btd(h_i^\ast\btr \e_B)]\otimes[(x_{(2)}\btl h_i)\btd \e_B]  \\
&=& \sum_i [x_{(1)}\btd\iota^\ast(h_i)]\otimes[(x_{(2)}\btl h_i)\btd \e_B].
\end{eqnarray*}
Therefore, the desired claim could be obtained via the linear isomorphism
$$C\cong C\btd\e_B,\;\;x\mapsto x\btd\e_B.$$

\item[(2)]
This due to similar reasons to (1): Since we know by (\ref{eqn:Delta(b)}) that
\begin{equation*}
\pd{\Delta}(\e\#b)
=\sum_{j}\left(\e\#\zeta[\gamma(f_j^\ast)b_{(1)}]\right)
  \otimes\left(f_j\#b_{(2)}\right)
\end{equation*}
for any $b\in B$ in the Hopf algebra $C\#B^\ast$,
the desired claim holds via the algebra isomorphism
$$B\cong\e\#B,\;\;b\mapsto \e\#b.$$

\item[(3)]
This is the dual claim of (2), as one could check the following equations
for all $b\in B$, $b^\ast,c^\ast\in B^\ast$ and $y\in C$:
\begin{eqnarray*}
\la b^\ast\otimes(y\btd c^\ast),\rho_B(b)\ra
&=&
\sum_j\la b^\ast\otimes(y\btd c^\ast),
  \zeta[\gamma(f_j^\ast)b_{(1)}]\otimes(f_j\# b_{(2)})\ra  \\
&=&
\sum_j\la b^\ast,\zeta[\gamma(f_j^\ast)b_{(1)}]\ra\la f_j,y\ra
  \la c^\ast,b_{(2)}\ra  \\
&=&
\sum\la b^\ast,\zeta[\gamma(y)b_{(1)}]\ra\la c^\ast,b_{(2)}\ra  \\
&=&
\sum\la \zeta^\ast(b^\ast)_{(1)},\gamma(y)\ra
  \la\zeta^\ast(b^\ast)_{(2)},b_{(1)}\ra\la c^\ast,b_{(2)}\ra  \\
&\overset{(\ref{eqn:btlbtrnew3(MAMS)})}=&
\sum\la \zeta^\ast(b^\ast)_{(1)},\gamma(y)\ra
  \la\zeta^\ast(b^\ast)_{(2)}\btr c^\ast,b\ra  \\
&=&
\sum\la [\zeta^\ast(b^\ast)\leftharpoonup\gamma(y)]\btr c^\ast,b\ra.
\end{eqnarray*}
\end{itemize}
\end{proof}

\begin{lemma}\label{lem:rightPDams(MAMS)}
With notations above as well as the structures given in Lemma \ref{lem:Cstru&B*stru} (1) and (3), there is an {\ams}
\begin{equation}\label{eqn:PAMSrightPD}
\begin{array}{ccc}
\xymatrix{
C \ar@<.5ex>[r]^{\iota'\;\;\;\;}
& C\btd B^\ast \ar@<.5ex>[l]^{\zeta'\;\;\;\;} \ar@<.5ex>[r]^{\;\;\;\;\pi'}
& B^\ast \ar@<.5ex>[l]^{\;\;\;\;\gamma'}  }
&\;\;\text{and}\;\;&
\xymatrix{
B \ar@<.5ex>[r]^{{\pi'}^\ast\;\;\;\;}
& C^\ast\# B \ar@<.5ex>[l]^{{\gamma'}^\ast\;\;\;\;}
  \ar@<.5ex>[r]^{\;\;\;\;{\iota'}^\ast}
& C^\ast \ar@<.5ex>[l]^{\;\;\;\;{\zeta'}^\ast}  },
\end{array}
\end{equation}
where:
\begin{equation}\label{eqn:iota'pi'zeta'gamma'(MAMS)}
\iota'=\id_C\btd \e_B,\;\;\;\;\pi'=\e_C\otimes\id_{B^\ast},\;\;\;\;
\zeta'=\id_C\otimes 1_B,\;\;\;\;\gamma'=1_C\btd\id_{B^\ast}.
\end{equation}
\end{lemma}

\begin{proof}
Note that
$$
\begin{array}{cccc}
\iota':& C &\rightarrow& C\btd B^\ast  \\
& x &\mapsto& x\btd\e_B
\end{array}
\;\;\;\;\text{and}\;\;\;\;
\begin{array}{cccc}
\pi': & C\btd B^\ast &\rightarrow& B^\ast  \\
& x\btd b^\ast &\mapsto& \la\e_C,x\ra b^\ast
\end{array}
$$
induce respectively the isomorphisms $C\cong C\btd\e$ and $B^\ast\cong (C\btd B^\ast)/(C^+\btd B^\ast)$ appearing in the proof of Lemma \ref{lem:Cstru&B*stru}. Thus $\iota'$ and $\pi'$ fit the requirements as consequences.

On the other hand, it is straightforward to know by Equation (\ref{eqn:rightPDmultiplicationORIGINAL(MAMS)}) that $\zeta'$ is a biunitary left $C$-module map. Besides, suppose $\{h_i\}$ is a linear basis of $H$ with dual basis $\{h_i^\ast\}$ of $H^\ast$. Then by the following calculation
\begin{eqnarray*}
(\id\otimes\pi')\circ\pd{\Delta}\circ\gamma'(b^\ast)
&=& (\id\otimes\pi')\circ\pd{\Delta}(1_C\btd b^\ast)  \\
&\overset{(\ref{eqn:rightpdcomultiplication})}=&
\sum_i [1_C\btd(h_i^\ast\btr b^\ast_{(1)})]
      \otimes\pi'[(1_C\btl h_i)\btd b^\ast_{(2)}]  \\
&\overset{(\ref{eqn:btlbtrnew3(MAMS)})}=&
\sum_i [1_C\btd(h_i^\ast\btr b^\ast_{(1)})]
      \otimes \pi'[\pi(h_i)\btd b^\ast_{(2)}]  \\
&=& \sum_i [1_C\btd(h_i^\ast\btr b^\ast_{(1)})]
      \otimes \la\e_C,\pi(h_i)\ra b^\ast_{(2)}  \\
&=& \sum_i [1_C\btd(\e\btr b^\ast_{(1)})]
      \otimes b^\ast_{(2)}  \\
&=& \sum (1_C\btd b^\ast_{(1)})\otimes b^\ast_{(2)}  \\
&=& \sum \gamma'(b^\ast_{(1)})\otimes b^\ast_{(2)},
\end{eqnarray*}
we know that $\gamma'$ is a right $B^\ast$-comodule map, which is clearly biunitary.

Then due to Remark \ref{rmk:pams}(2), in order to complete proving that (\ref{eqn:PAMSrightPD}) is a {\pams},
it suffices to verify that $(\iota'\circ\zeta')\ast(\gamma'\circ\pi')=\id_{C\btd B^\ast}$: Specifically, one could calculate for any $x\in C$ and $b^\ast\in B^\ast$ that
\begin{eqnarray*}
&& [(\iota'\circ\zeta')\ast(\gamma'\circ\pi')](x\btd b^\ast)  \\
&\overset{(\ref{eqn:rightpdcomultiplication})}=&
\sum_i (\iota'\circ\zeta')[x_{(1)}\btd(h_i^\ast\btr b^\ast_{(1)})]
  \cdot(\gamma'\circ\pi')[(x_{(2)}\btl h_i)\btd b^\ast_{(2)}]  \\
&=&
\sum_i (x_{(1)}\la h_i^\ast\btr b^\ast_{(1)},1_B\ra\btd \e_B)
  \cdot(1_C\btd\la\e_C,x_{(2)}\btl h_i\ra b^\ast_{(2)})  \\
&\overset{(\ref{eqn:btlbtrnew2(MAMS)})}=&
\sum_i (x_{(1)}\la h_i^\ast,1\ra\la b^\ast_{(1)},1_B\ra\btd \e_B)
  \cdot(1_C\btd\la\e_C,x_{(2)}\ra\la\e, h_i\ra b^\ast_{(2)})  \\
&=& (x\btd \e_B)(1_C\btd b^\ast)  \\
&\overset{(\ref{eqn:rightPDmultiplicationORIGINAL(MAMS)})}=&
x\btd b^\ast.
\end{eqnarray*}

Finally, we claim that the image $\gamma'(B^\ast)$ is a right coideal subalgebra, and hence
(\ref{eqn:PAMSrightPD}) is an {\ams} due to Proposition \ref{lem:PAMSdetermineHopfalgs(MAMS)}.
Indeed,
it follows from a similar argument to the proof of Lemma \ref{lem:Cstru&B*stru}(1) that $1_C\btd B^\ast$ is a right coideal subalgebra of $C\btd B^\ast$.
\end{proof}

In order to show our ``partial self-property'' stated below,
note that if $(\zeta,\gamma^\ast)$ is a {\ams},
then there are two formulas (\ref{eqn:associatortrivial-result2}) and (\ref{eqn:associatortrivial-result3}) appearing in the proof of Proposition \ref{lem:PAMSdetermineHopfalgs(MAMS)}:
For any $x,y\in C$,
\begin{equation}\label{eqn:gamma(x)gamma(y)(MAMS)}
\gamma(x)\gamma(y)=\gamma[x\btl\gamma(y)]
\end{equation}
and
\begin{equation}\label{eqn:gamma(x)1gamma(x)2(MAMS)}
\sum \gamma(x)_{(1)}\otimes\gamma(x)_{(2)}=
\sum \gamma\left(\pi[\gamma(x)_{(1)}]\right)\otimes\gamma(x)_{(2)}
\end{equation}
hold.

\begin{proposition}\label{prop:selfPD}
Denoted by $B\times C$ the left partially dualized Hopf algebra determined by the {\ams} (\ref{eqn:PAMSrightPD}) given in Lemma \ref{lem:rightPDams(MAMS)}. Then:
\begin{itemize}
\item[(1)]
The structures of $B\times C$ are given as follows: For any $b,c\in B$ and $x,y\in C$,
\begin{itemize}
\item
Multiplication:
\begin{equation}\label{eqn:crossprodmult(MAMS)}
(b\times x)(c\times y)
= \sum \zeta[\iota(b)\gamma(x)_{(1)}\iota(c)_{(1)}]\times \pi[\gamma(x)_{(2)}\iota(c)_{(2)}\gamma(y)]
\end{equation}

\item
Unit element: $1_B\times 1_C$;

\item
Comultiplication:
\begin{equation}\label{eqn:crossprodcomult(MAMS)}
\pd{\Delta}(b\times x)
= \sum \Big(\zeta(b_{(1)})\times \pi[\iota(b_{(2)})_{(1)}\gamma(x_{(1)})_{(1)}]\Big)
  \otimes\Big(\zeta[\iota(b_{(2)})_{(2)}\gamma(x_{(1)})_{(2)}]\times x_{(2)}\Big)
\end{equation}

\item
Counit: $\e_B\times\e_C$;

\end{itemize}

\item[(2)]
Furthermore, the Hopf algebra $B\times C$ is isomorphic to $H$ through
\begin{equation}\label{eqn:vartheta(MAMS)}
\vartheta:B\times C\cong H,\;\;b\times x\mapsto \iota(b)\gamma(x).
\end{equation}
\end{itemize}
\end{proposition}

\begin{proof}
\begin{itemize}
\item[(1)]
In order to determine the left partial dualized Hopf algebra $B\times C$ determined by the {\ams} (\ref{eqn:PAMSrightPD}), recall in Lemma \ref{lem:Cstru&B*stru} (1) and (2) that the right $C^\ast\# B$-coaction on $B$ and left $C\btd B^\ast$-coaction on $C$ (induced by ${\pi'}^\ast$ and $\iota'$ respectively) are:
\begin{equation}\label{eqn:Cstru&B*stru(MAMS)}
\begin{array}{ccc}
B &\rightarrow& B\otimes(C^\ast\# B)  \\
b &\mapsto& \sum\limits_j\zeta[\gamma(f_j^\ast)b_{(1)}]\otimes(f_j\# b_{(2)})
\end{array}
\;\;\text{and}\;\;
\begin{array}{ccc}
C &\rightarrow& (C\btd B^\ast)\otimes C  \\
x &\mapsto& \sum\limits_i[x_{(1)}\btd\iota^\ast(h^\ast_i)]\otimes(x_{(2)}\btl h_i)
\end{array},
\end{equation}
where $\{f_j\}$ is a basis of $C^\ast$ with dual basis $\{f^\ast_j\}$ of $C$, and $\{h_i\}$ is a basis of $H$ with dual basis $\{h^\ast_i\}$ of $H^\ast$.


Then one could calculate to find that the multiplication (\ref{eqn:smashprod}) on $B\times C$ as the smash product algebra is determined by:
\begin{eqnarray}
(b\times x)(c\times y)
&\overset{(\ref{eqn:smashprod})}=&
\sum_{i,j} b\zeta[\gamma(f^\ast_j)c_{(1)}]
  \times \la x_{(1)}\btd\iota^\ast(h^\ast_i),f_j\#c_{(2)}\ra [(x_{(2)}\btl h_i)\cdot y]  \nonumber  \\
&\overset{(\ref{eqn:Cmultiplication(MAMS)})}=&
\sum_{i,j} b\zeta[\gamma(f^\ast_j)c_{(1)}]
  \times \la f_j,x_{(1)}\ra\la\iota^\ast(h^\ast_i),c_{(2)}\ra
  [(x_{(2)}\btl h_i)\btl \gamma(y)]  \nonumber  \\
&=&
\sum b\zeta[\gamma(x_{(1)})c_{(1)}]
  \times [x_{(2)}\btl\iota(c_{(2)})\gamma(y)]
\label{eqn:crossprodmult0(MAMS)}   \\
&\overset{(\ref{eqn:iotapi})}=&
\sum b\zeta[\gamma(x_{(1)})\iota(c)_{(1)}]
  \times [x_{(2)}\btl\iota(c)_{(2)}\gamma(y)]
\label{eqn:crossprodmult2(MAMS)}    \\
&=&
\sum \zeta[\iota(b)\gamma(x)_{(1)}\iota(c)_{(1)}]
  \times \left(\pi[\gamma(x)_{(2)}]\btl\iota(c)_{(2)}\gamma(y)\right) \nonumber \\
&\overset{(\ref{eqn:iotapi})}=&
\sum \zeta[\iota(b)\gamma(x)_{(1)}\iota(c)_{(1)}]
  \times\pi[\gamma(x)_{(2)}\iota(c)_{(2)}\gamma(y)],  \nonumber
\end{eqnarray}
where the penultimate equality follows from Definition \ref{def:PAMS}(4) that $\zeta$ preserves left $B$-actions and $\gamma$ preserves right $C$-actions.

On the other hand, we should use the formulas in Theorem \ref{thm:partialdual}(1)
to determine the comultiplication on $B\times C$ as the left partial dual. Before that, note in $C\btd B^\ast$ that
\begin{eqnarray}\label{eqn:Delta(b*)(MAMS)}
\pd{\Delta}(1_C\btd b^\ast)
&=& \sum_l [1_C\btd(h^\ast_l\btr b^\ast_{(1)})]
  \otimes[(1_C\btl h_l)\btd b^\ast_{(2)}]  \nonumber  \\
  &=& \sum_l [1_C\btd(h^\ast_l\btr b^\ast_{(1)})]
  \otimes[\pi(h_l)\btd b^\ast_{(2)}]
\;\;\;\;(\forall b^\ast\in B^\ast),
\end{eqnarray}
which is due to the formula (\ref{eqn:rightpdcomultiplication}) of the comultiplication on the right partially dualized Hopf algebra.

Now let us focus on the {\ams} $(\zeta',{\gamma'}^\ast)$ given in (\ref{eqn:PAMSrightPD}). As the coactions of $B$ and $C$ are chosen as
(\ref{eqn:Cstru&B*stru(MAMS)}), we calculate the coproduct in the left partially dualized Hopf algebra $B\times C$ that:
\begin{eqnarray*}
\pd{\Delta}\left(b\times 1_C\right)
&\overset{(\ref{eqn:Delta(f)})}=&
\sum_{i,j} \left(\zeta[\gamma(f_j^\ast)b_{(1)}]
  \times \zeta'[\gamma'(b_i^\ast)\leftharpoonup(f_j\# b_{(2)})]\right)
  \otimes\left(b_i\times 1_C\right)  \\
&\overset{(\ref{eqn:iota'pi'zeta'gamma'(MAMS)})}=&
\sum_{i,j} \left(\zeta[\gamma(f_j^\ast)b_{(1)}]
  \times \zeta'\left[(1_C\btd b_i^\ast)\leftharpoonup(f_j\# b_{(2)})\right]\right)
  \otimes\left(b_i\times 1_C\right)  \\
&\overset{(\ref{eqn:Delta(b*)(MAMS)})}=&
\sum_{i,j,l} \left(\zeta[\gamma(f_j^\ast)b_{(1)}]
  \times \left\la f_j\# b_{(2)},1_C\btd(h^\ast_l\btr b^\ast_i{}_{(1)})\right\ra
  \zeta'[\pi(h_l)\btd b_i^\ast{}_{(2)}]\right)  \\
&& \;\;\;\;\;\; \otimes\left(b_i\times 1_C\right)  \\
&\overset{(\ref{eqn:iota'pi'zeta'gamma'(MAMS)})}=&
\sum_{i,j,l} \left(\zeta[\gamma(f_j^\ast)b_{(1)}]
  \times \la f_j,1_C\ra\la h^\ast_l\btr b^\ast_i{}_{(1)},b_{(2)}\ra
  \,\pi(h_l)\,\la b_i^\ast{}_{(2)},1_B\ra\right)  \\
&& \;\;\;\;\;\; \otimes\left(b_i\times 1_C\right)  \\
&=&
\sum_{i,l} \left(\zeta[\gamma(1_C)b_{(1)}]
  \times \la h^\ast_l\btr b^\ast_i,b_{(2)}\ra
  \,\pi(h_l)\right) \otimes\left(b_i\times 1_C\right)  \\
&\overset{(\ref{eqn:btlbtrnew3(MAMS)})}=&
\sum_{i,l} \left(\zeta(b_{(1)})
  \times \la h^\ast_l,b_{(2)}\ra\la b^\ast_i,b_{(3)}\ra
  \pi(h_l)\right)\otimes\left(b_i\times 1_C\right)  \\
&=&
\sum \left(\zeta(b_{(1)})\times\pi(b_{(2)})\right)\otimes\left(b_{(3)}\times 1_C\right),
\end{eqnarray*}
as well as
\begin{eqnarray*}
\pd{\Delta}(1_B\times x)
&\overset{(\ref{eqn:Delta(b)})}=&
\sum_{i,j}
  \left(1_B\times \zeta'\left[\gamma'(b^\ast_j)\big(x_{(1)}\btd\iota^\ast(h^\ast_i)\big)\right]\right)
  \otimes \left(b_j\times (x_{(2)}\btl h_i)\right)  \\
&\overset{(\ref{eqn:iota'pi'zeta'gamma'(MAMS)})}=&
\sum_{i,j} \left(1_B\times \zeta'\left[\big(1_C\btd b^\ast_j\big)
  \big(x_{(1)}\btd\iota^\ast(h^\ast_i)\big)\right]\right)
  \otimes \left(b_j\times (x_{(2)}\btl h_i)\right)  \\
&\overset{(\ref{eqn:rightpdmultiplication})}=&
\sum_{i,j} \left(1_B\times (1_C\btl[\zeta^\ast(b^\ast_j{}_{(1)})\rightharpoonup\gamma(x_{(1)})])\right)
   \\
&& \;\;\;\;\;\;
\otimes \left\la [\zeta^\ast(b^\ast_j{}_{(2)})\leftharpoonup\gamma(x_{(2)})]\btr\iota^\ast(h^\ast_i),
  1_B\right\ra\left(b_j\times (x_{(3)}\btl h_i)\right)  \\
&\overset{(\ref{eqn:btlbtrnew2(MAMS)}),\;(\ref{eqn:btlbtrnew3(MAMS)})}=&
\sum_{i,j} \left(1\times \pi[\zeta^\ast(b^\ast_j{}_{(1)})\rightharpoonup\gamma(x_{(1)})]\right)
  \la \zeta^\ast(b^\ast_j{}_{(2)})\leftharpoonup\gamma(x_{(2)}),1\ra  \\
&&  \;\;\;\;\;\;
\otimes \la \iota^\ast(h^\ast_i),1_B\ra\left(b_j\times (x_{(3)}\btl h_i)\right)  \\
&=&
\sum_{i} \left(1\times \pi[\zeta^\ast(b^\ast_j{}_{(1)})\rightharpoonup\gamma(x_{(1)})]\right)
  \la \zeta^\ast(b^\ast_j{}_{(2)}),\gamma(x_{(2)})\ra  \\
&&  \;\;\;\;\;\;
\otimes \left(b_j\times (x_{(3)}\btl 1)\right)  \\
&=&
\sum_{i} \left(1\times \pi[\zeta^\ast(b^\ast_j{}_{(1)})\rightharpoonup\gamma(x_{(1)})]\right)
  \la b^\ast_j{}_{(2)},\zeta[\gamma(x_{(2)})]\ra
  \otimes \left(b_j\times x_{(3)}\right),
\end{eqnarray*}
where $\{b_i\}$ is a basis of $B$ with dual basis $\{b_i^\ast\}$ of $B^\ast$.
However, since $\zeta\circ\gamma$ is trivial according to Proposition \ref{prop:PAMS-comptrival}(2), we could continue above calculations:
\begin{eqnarray*}
\pd{\Delta}(1_B\times x)
&=&
\sum_{i} \left(1\times \pi[\zeta^\ast(b^\ast_j)\rightharpoonup\gamma(x_{(1)})]\right)
  \otimes \left(b_j\times  x_{(2)}\right)  \\
&=&
\sum_{i} \left(1\times \pi[\gamma(x_{(1)})_{(1)}]\right)\la \zeta^\ast(b^\ast_j),\gamma(x_{(1)})_{(2)}\ra
  \otimes \left(b_j\times  x_{(2)}\right)  \\
&=&
\sum \left(1\times \pi[\gamma(x_{(1)})_{(1)}]\right)
  \otimes \left(\zeta[\gamma(x_{(1)})_{(2)}]\times  x_{(2)}\right).
\end{eqnarray*}

At final, as the left partially dualized Hopf algebra $B\times C$ has the smash product algebra structure, one could combine the previous two results to obtain our desired claim:
\begin{eqnarray}
&& \pd{\Delta}(b\times x)  \nonumber \\
&=& \pd{\Delta}(b\times 1_C)\pd{\Delta}(1_B\times x)  \nonumber \\
&=&
\sum \left(\zeta(b_{(1)})\times\pi(b_{(2)})\right)
     \left(1\times \pi[\gamma(x_{(1)})_{(1)}]\right)
     \otimes
     \left(b_{(3)}\times 1_C\right)
     \left(\zeta[\gamma(x_{(1)})_{(2)}]\times  x_{(2)}\right)  \nonumber \\
&=&
\sum \left(\zeta(b_{(1)})\times\pi(b_{(2)})\cdot\pi[\gamma(x_{(1)})_{(1)}]\right)
\otimes
     \left(b_{(3)}\zeta[\gamma(x_{(1)})_{(2)}]\times  x_{(2)}\right)  \nonumber \\
&\overset{(\ref{eqn:Cmultiplication(MAMS)})}=&
\sum \left(\zeta(b_{(1)})
  \times\big[\pi(b_{(2)})\btl\gamma\big(\pi[\gamma(x_{(1)})_{(1)}]\big)\big]\right)
  \otimes
     \left(b_{(3)}\zeta[\gamma(x_{(1)})_{(2)}]\times  x_{(2)}\right)  \nonumber \\
&\overset{(\ref{eqn:gamma(x)1gamma(x)2(MAMS)})}=&
\sum \left(\zeta(b_{(1)})
  \times\big[\pi(b_{(2)})\btl\gamma(x_{(1)})_{(1)}\big]\right)
  \otimes
     \left(b_{(3)}\zeta[\gamma(x_{(1)})_{(2)}]\times  x_{(2)}\right)
\label{eqn:crossprodcomult0(MAMS)}  \\
&\overset{(\ref{eqn:iotapi})}=&
\sum \left(\zeta(b_{(1)})
  \times \pi\big[b_{(2)}\gamma(x_{(1)})_{(1)}\big]\right)
  \otimes
     \left(\zeta[\iota(b_{(3)})\gamma(x_{(1)})_{(2)}]\times  x_{(2)}\right)
\label{eqn:crossprodcomult2(MAMS)}  \\
&\overset{(\ref{eqn:iotapi})}=&
\sum \left(\zeta(b_{(1)})
  \times \pi\big[\iota(b_{(2)})_{(1)}\gamma(x_{(1)})_{(1)}\big]\right)
  \otimes
     \left(\zeta[\iota(b_{(2)})_{(2)}\gamma(x_{(1)})_{(2)}]\times  x_{(2)}\right).
  \nonumber
\end{eqnarray}

As for the unit, note again that the algebra structure of the left partial dual $B\times C$ is defined as the smash product. Consequently, as $1_B$ and $1_C$ are the unit elements of $B$ and $C$ respectively, we know that $1_B\times 1_C$ is the unite element of $B\times C$.

Finally, it follows from Theorem \ref{thm:partialdual}(2) that the counit is $\e_B\times\e_C$, since $B$ is the linear dual space of the right module coalgebra $B^\ast$ whose counit is $1_B$.

\item[(2)]
It is evident to find that $\vartheta(1_B\times 1_C)=\iota(1_B)\gamma(1_C)=1$ as well as
$$\la\e,\vartheta(b\times x)\ra
=\la\e,\iota(b)\ra\la\e,\gamma(x)\ra
\overset{(\ref{eqn:1CeB(MAMS)})}=\la\e_B,b\ra\la\e_C,x\ra$$
for all $b\in B$ and $x\in C$.
This is because $\iota$ is an algebra map and $\gamma$ is biunitary according to Definition \ref{def:PAMS} (1) and (5).
Therefore, in order to show that $\vartheta$ (\ref{eqn:vartheta(MAMS)}) is a Hopf algebra map, it remains to check that $\vartheta$ preserves the multiplications and comultiplications.

Now we suppose $b,c\in B$ and $x,y\in C$. Recall again by Definition \ref{def:PAMS} (1) and (4) that $\zeta$ preserves left $B$-actions and $\pi$ preserves right $H$-actions.
Then it follows from the equation (\ref{eqn:crossprodmult2(MAMS)}) in the proof of (1) that
\begin{eqnarray*}
(b\times x)(c\times y)
&\overset{(\ref{eqn:crossprodmult2(MAMS)})}=&
\sum \zeta[\iota(b)\gamma(x)_{(1)}\iota(c)_{(1)}]
  \times\pi[\gamma(x)_{(2)}\iota(c)_{(2)}\gamma(y)]  \\
&=&
\sum b\zeta[\gamma(x)_{(1)}\iota(c)_{(1)}]
  \times\big(\pi[\gamma(x)_{(2)}\iota(c)_{(2)}]\btl\gamma(y)\big),
\end{eqnarray*}
and hence
\begin{eqnarray*}
\vartheta((b\times x)(c\times y))
&=&
\sum\vartheta\left(b\zeta[\gamma(x)_{(1)}\iota(c)_{(1)}]
  \times\big(\pi[\gamma(x)_{(2)}\iota(c)_{(2)}]\btl\gamma(y)\big)\right)  \\
&\overset{(\ref{eqn:vartheta(MAMS)})}=&
\sum\iota\left(b\zeta[\gamma(x)_{(1)}\iota(c)_{(1)}]\right)
  \gamma\left(\pi[\gamma(x)_{(2)}\iota(c)_{(2)}]\btl\gamma(y)\right)  \\
&\overset{(\ref{eqn:gamma(x)gamma(y)(MAMS)})}=&
\sum \iota(b)\iota\left(\zeta[\gamma(x)_{(1)}\iota(c)_{(1)}]\right)
  \gamma\left(\pi[\gamma(x)_{(2)}\iota(c)_{(2)}]\right)\gamma(y)  \\
&=& \iota(b)\gamma(x)\iota(c)\gamma(y)  \\
&\overset{(\ref{eqn:vartheta(MAMS)})}=&
\vartheta(b\times x)\vartheta(c\times y),
\end{eqnarray*}
where the penultimate equality follows from Definition \ref{def:PAMS}(6) that
\begin{equation}\label{eqn:convprod(MAMS)}
(\iota\circ\zeta)\ast(\gamma\ast\pi)=\id_H.
\end{equation}

On the other hand,
it follows from the equation (\ref{eqn:crossprodcomult2(MAMS)}) in the proof of (1) that
\begin{eqnarray*}
&& (\vartheta\otimes\vartheta)\circ\pd{\Delta}\left(b\times x\right)  \\
&\overset{(\ref{eqn:crossprodcomult2(MAMS)})}=&
\sum \vartheta\left(\zeta(b_{(1)})
  \times \pi\big[b_{(2)}\gamma(x_{(1)})_{(1)}\big]\right)
  \otimes
     \vartheta\left(\zeta[\iota(b_{(3)})\gamma(x_{(1)})_{(2)}]\times  x_{(2)}\right)  \\
&\overset{(\ref{eqn:vartheta(MAMS)})}=&
\sum \iota[\zeta(b_{(1)})]\gamma\big(\pi\big[b_{(2)}\gamma(x_{(1)})_{(1)}\big]\big)
  \otimes \iota\big(b_{(3)}\zeta[\gamma(x_{(1)})_{(2)}]\big)\gamma(x_{(2)})  \\
&\overset{(\ref{eqn:iotapi})}=&
\sum \iota[\zeta(b_{(1)})]\gamma[\pi(b_{(2)})\btl \gamma(x_{(1)})_{(1)}]
  \otimes \iota(b_{(3)})\iota\big(\zeta[\gamma(x_{(1)})_{(2)}]\big)\gamma(x_{(2)})  \\
&\overset{(\ref{eqn:gamma(x)gamma(y)(MAMS)})}=&
\sum \iota[\zeta(b_{(1)})]\gamma[\pi(b_{(2)})]\gamma(x_{(1)})_{(1)}
  \otimes \iota(b_{(3)})\iota\zeta[\gamma(x_{(1)})_{(2)}]\big)\gamma(x_{(2)})  \\
&\overset{(\ref{eqn:convprod(MAMS)})}=&
\sum b_{(1)}\gamma(x_{(1)})_{(1)}
  \otimes \iota(b_{(2)})\iota\big(\zeta[\gamma(x_{(1)})_{(2)}]\big)\gamma(x_{(2)})  \\
&=&
\sum b_{(1)}\gamma(x)_{(1)}
  \otimes \iota(b_{(2)})\iota\big(\zeta[\gamma(x)_{(2)}]\big)\gamma\big(\pi[\gamma(x)_{(3)}]\big)  \\
&\overset{(\ref{eqn:convprod(MAMS)})}=&
\sum b_{(1)}\gamma(x)_{(1)}\otimes \iota(b_{(2)})\gamma(x)_{(2)}   \\
&\overset{(\ref{eqn:iotapi})}=&
\sum \iota(b)_{(1)}\gamma(x)_{(1)}\otimes \iota(b)_{(2)}\gamma(x)_{(2)}   \\
&\overset{(\ref{eqn:vartheta(MAMS)})}=&
\Delta\circ\vartheta\left(b\times x\right),
\end{eqnarray*}
where the sixth equality is also due to Definition \ref{def:PAMS}(4) that $\gamma$ is a right $C$-comodule map.
\end{itemize}
\end{proof}

\begin{remark}
According to the equations (\ref{eqn:crossprodmult0(MAMS)}) and (\ref{eqn:crossprodcomult0(MAMS)}) in the proof of Proposition \ref{prop:selfPD},
we remark that the multiplication and comultiplication on $B\times C$ could be also expressed as follows: For any $b,c\in B$ and $x,y\in C$,
\begin{equation*}
(b\times x)(c\times y)
= \sum b\zeta[\gamma(x_{(1)})c_{(1)}]\times [x_{(2)}\btl\iota(c_{(2)})\gamma(y)]
\end{equation*}
as well as
\begin{equation*}
\pd{\Delta}(b\times x)
=\sum \Big(\zeta(b_{(1)})\times [\pi(b_{(2)})\btl \gamma(x_{(1)})_{(1)}]\Big)
  \otimes\Big(b_{(3)}\zeta[\gamma(x_{(1)})_{(2)}]\times x_{(2)}\Big).
\end{equation*}
\end{remark}

Moreover, with the notations in Proposition \ref{prop:selfPD}, it is straightforward to find that $\vartheta^{-1}=(\zeta\otimes\pi)\circ\Delta$ by using Definition \ref{def:PAMS}(6) that
$(\iota\circ\zeta)\ast(\gamma\ast\pi)=\id_H$.
Thus, as the right partial dual is defined in Definition \ref{def:rightpartialdual} to be the linear dual of the left partial dual, we could obtain the following consequence of Proposition \ref{prop:selfPD}(2):

\begin{corollary}\label{cor:doublerightPD(MAMS)}
The right partially dualized Hopf algebra
determined by the {\ams} (\ref{eqn:PAMSrightPD}) given in Lemma \ref{lem:rightPDams(MAMS)} is isomorphic to $H^\ast$ through
\begin{equation*}
b^\ast\otimes f\mapsto \zeta^\ast(b^\ast)\pi^\ast(f).
\end{equation*}
\end{corollary}

Roughly speaking, Corollary \ref{cor:doublerightPD(MAMS)} means that the
double right partial dualization could send $H$ to its dual Hopf algebra $H^\ast$.
As consequences, we might see that:
\begin{itemize}
\item
The triple right partial dualization (in canonical ways) could send $H$ to the dual of $C\btd B^\ast$, which becomes $C^\ast\#B$ by definitions;

\item
The quadruple right partial dualization (in canonical ways) could become double dualization, sending $H$ to it self.
\end{itemize}
Moreover, these two properties imply an opposite version of Proposition \ref{prop:selfPD}(2):
\begin{itemize}
\item
The (canonical) right partial dual of $C^\ast\#B$ could also be obtained by the quadruple right partial dualization from $H$, and hence it becomes $H$ itself (up to a canonical isomorphism).
\end{itemize}
Let us states the complete result as follows, without providing the detailed but straightforward proof:

\begin{proposition}\label{prop:selfPD2}
The right partially dualized Hopf algebra determined by the {\ams}
\begin{equation*}
\begin{array}{ccc}
\xymatrix{
C^\ast \ar@<.5ex>[r]^{{\zeta'}^\ast\;\;\;\;}
& C^\ast\# B \ar@<.5ex>[l]^{{\iota'}^\ast\;\;\;\;}
  \ar@<.5ex>[r]^{\;\;\;\;{\gamma'}^\ast}
& B \ar@<.5ex>[l]^{\;\;\;\;{\pi'}^\ast}  }
&\;\;\text{and}\;\;&
\xymatrix{
B^\ast \ar@<.5ex>[r]^{\gamma'\;\;\;\;}
& C\btd B^\ast \ar@<.5ex>[l]^{\pi'\;\;\;\;} \ar@<.5ex>[r]^{\;\;\;\;\zeta'}
& C \ar@<.5ex>[l]^{\;\;\;\;\iota'}  }
\end{array}
\end{equation*}
is the same as $B\times C$ introduced in Proposition \ref{prop:selfPD}(1), which is isomorphic to $H$ as well.
\end{proposition}

Similarly, we could also write analogous properties for (double, triple and quadruple) left partial dualization which are analogous to Corollary \ref{cor:doublerightPD(MAMS)} and its consequences for right partial dualization.
Combining these results as well as Propositions \ref{prop:selfPD} and \ref{prop:selfPD2}, one might find a conclusion as the diagram (\ref{eqn:multiplePD(MAMS)}).

\subsection{Bismash products constructed from matched pair of groups}\label{subsection:matchedpairofgrps}

The first classical structure we considered is when $H$ is the group algebra $\k(F\bowtie G)$, where $(F,G)$ is a matched pair of finite groups (\cite[Definition 2.1]{Tak81}). Our claim is that the bismash product $\k^G\#\k F$ and $\k G\#\k^F$ are exactly left and right partially dualized Hopf algebras of $\k(F\bowtie G)$ respectively.

Let us begin by recalling the definitions stated in \cite{Mas02}:

\begin{definition}(\cite[Definition 1.1]{Mas02})
Let $F$ and $G$ be two groups. Suppose
$\triangleright:G\times F\rightarrow F$ and $\triangleleft:G\times F\rightarrow G$ are group actions satisfying
$$x\triangleright bc=(x\triangleright b)((x\triangleleft b)\triangleright c)
\;\;\;\;\;\;\text{and}\;\;\;\;\;\;
xy\triangleleft b=(x\triangleleft(y\triangleright b))(y\triangleleft b)$$
for all $b,c\in F$ and $x,y\in G$. Then
$(F,G,\triangleright,\triangleleft)$ is called a matched pair of groups.

In this situation, the cartesian product $F\times G$ forms a group under the product
\begin{equation}\label{eqn:matchedpairprod}
(b,x)(c,y)=(b(x\triangleright c),(x\triangleleft c)y).
\end{equation}
This group is denoted by $F\bowtie G$.
\end{definition}

When $F$ and $G$ are both finite, denote the dual Hopf algebras $(\k F)^\ast$, $(\k G)^\ast$ and $\k(F\bowtie G)^\ast$ by $\k^F$, $\k^G$ and $\k^{F\bowtie G}$, respectively.
If we define the following injections and projections between $F\times G$ and its factors:
$$\begin{array}{cccc}
\iota_F:F\rightarrow F\times G,& \pi_F:F\times G\rightarrow F,  &
\iota_G:G\rightarrow F\times G,& \pi_G:F\times G\rightarrow G,  \\
\;\;\;\;\;\;b\mapsto(b,1),&\;\;\;\;\;\;\;\;\;(b,x)\mapsto b,  &
\;\;\;\;\;\;x\mapsto(1,x),&\;\;\;\;\;\;\;\;\;(b,x)\mapsto x,
\end{array}$$
then a straightforward verification follows that $(\pi_F,\iota_G^\ast)$ is a {\pams}. Furthermore, $(\pi_F,\iota_G^\ast)$ is an {\ams}
according to Corollary \ref{cor:Hopfalgconditions}(2) as $\iota_G$ is a bialgebra map.

\begin{lemma}
Let $(F,G)$ be a matched pair of finite groups.
There is an {\ams} $(\pi_F,\iota_G^\ast)$ for the injection $\iota_F$ of Hopf algebras:
\begin{equation*}
\begin{array}{ccc}
\xymatrix{
\k F \ar@<.5ex>[r]^{\iota_F\;\;\;\;\;\;}
& \k(F\bowtie G) \ar@<.5ex>@{-->}[l]^{\pi_F\;\;\;\;\;\;} \ar@<.5ex>[r]^{\;\;\;\;\;\;\pi_G}
& \k G \ar@<.5ex>@{-->}[l]^{\;\;\;\;\;\;\iota_G}  }
&\;\;\text{and}\;\;&
\xymatrix{
\k^G \ar@<.5ex>[r]^{\pi_G^\ast\;\;\;\;}
& \k^{F\bowtie G} \ar@<.5ex>@{-->}[l]^{\iota_G^\ast\;\;\;\;}
  \ar@<.5ex>[r]^{\;\;\;\;\iota_F^\ast}
& \k^F \ar@<.5ex>@{-->}[l]^{\;\;\;\;\pi_F^\ast}  },
\end{array}
\end{equation*}
where the right $\k(F\bowtie G)$-module structure of $\k G$ is given by
\begin{equation}\label{eqn:kGrightmod}
x\otimes(b,y)\mapsto (x\triangleleft b)y
\;\;\;\;\;\;(\forall b\in F,\;\forall x,y\in G).
\end{equation}
\end{lemma}

\begin{remark}
Suppose $\{p_x\mid x\in G\}$ is the basis of $\k^G$ which is dual to $G$, and suppose $\{p_{(b,x)}\mid b\in F,\;x\in G\}$ is the basis of $\k^{F\bowtie G}$ dual to $F\bowtie G$. The right $\k^{F\bowtie G}$-comodule structure of $\k^G$ induced by (\ref{eqn:kGrightmod}) would be:
\begin{equation}\label{eqn:k^Grightcomod}
p_x\mapsto
\sum {p_x}_{(1)}\otimes{p_x}_{(2)}
:=\sum_{\substack{d\in F,\;z,w\in G \\ (z\triangleleft d)w=x}}
p_z\otimes p_{(d,w)}
\in \k^G\otimes\k^{F\bowtie G}
\;\;\;\;\;\;(\forall x\in G).
\end{equation}
\end{remark}

Therefore the {\ams} $(\pi_F,\iota_G^\ast)$ would determine a left partially dualized Hopf algebra $\k^G\#\k F$ with structures defined by Theorem \ref{thm:partialdual}. Specifically, for any $b,c\in F$ and $x,y\in G$:
\begin{itemize}
\item
The unit is $\e\#1$, and the multiplication is defined such that:
\begin{eqnarray*}
(p_x\#b)(p_y\#c)
&\overset{(\ref{eqn:smashprod})}{=}&
\sum p_x{p_y}_{(1)}\#(b\leftharpoonup {p_y}_{(2)})c
~\overset{(\ref{eqn:k^Grightcomod})}{=}~
\sum_{\substack{d\in F,\;z,w\in G \\ (z\triangleleft d)w=y}}
  p_xp_z\#(b\leftharpoonup p_{(d,w)})c  \\
&=&
\sum_{\substack{d\in F,\;z,w\in G \\ (z\triangleleft d)w=y}}
  p_xp_z\#\la p_{(d,w)},(b,1)\ra bc
~=~
\delta_{x\triangleleft b,y}p_x\#bc;
\end{eqnarray*}

\item
Since it could be verified that
\begin{equation}\label{eqn:injG*projF*}
\begin{array}{rcccrcc}
\iota_G^\ast:\k^{F\bowtie G}&\rightarrow&\k^G,
&& \pi_F^\ast:\k^F&\rightarrow&\k^{F\bowtie G},    \\
p_{(c,w)}&\mapsto&\delta_{c,1}p_w,
&& p_c&\mapsto&\sum_{w'\in G}p_{(c,w')}
\end{array}
\end{equation}
hold, the comultiplication and counit are defined respectively such that
\begin{eqnarray*}
\pd{\Delta}(p_x\#b)
&\overset{(\ref{eqn:Deltacomplete})}{=}&
\sum_{c\in F,\;y\in G}
  \left({p_x}_{(1)}\#c\pi_F[\iota_G(y)b_{(1)}]\right)
  \otimes\left(\iota_G^\ast[{p_x}_{(2)}\pi^\ast_F(p_c)]p_y\#b_{(2)}\right)  \\
&\overset{(\ref{eqn:k^Grightcomod})}{=}&
\sum_{\substack{c,d\in F,\;y,z,w\in G \\ (z\triangleleft d)w=x}}
  \left(p_z\#c\pi_F[\iota_G(y)(b,1)]\right)
  \otimes\left(\iota_G^\ast[p_{(d,w)}\pi^\ast_F(p_c)]p_y\#b\right)  \\
&\overset{(\ref{eqn:injG*projF*})}{=}&
\sum_{\substack{c,d\in F,\;y,z,w\in G \\ (z\triangleleft d)w=x}}
  \left(p_z\#c\pi_F[(1,y)(b,1)]\right)  \\
&& \;\;\;\;\;\;\;\;\;\;\;\;\;\;\;\;\;\;\;\;\;\;
  \otimes\left(\iota_G^\ast[p_{(d,w)}(\sum\nolimits_{w'\in G} p_{(c,w')})]p_y\#b\right)  \\
&\overset{(\ref{eqn:matchedpairprod})}{=}&
\sum_{\substack{c\in F,\;w,y,z\in G \\ (z\triangleleft c)w=x}}
  \left(p_z\#c\pi_F[(y\triangleright b,y\triangleleft b)]\right)
  \otimes\left(\iota_G^\ast[p_{(c,w)}]p_y\#b\right)  \\
&\overset{(\ref{eqn:injG*projF*})}{=}&
\sum_{\substack{c\in F,\;w,y,z\in G \\ (z\triangleleft c)w=x}}
  \left(p_z\#c(y\triangleright b)\right)
  \otimes\left(\delta_{c,1}p_wp_y\#b\right)  \\
&=&
\sum_{\substack{y,z\in G \\ zy=x}}
  \left(p_z\#(y\triangleright b)\right)\otimes\left(p_y\#b\right)
~=~
\sum_{y\in G}
  \left(p_{xy^{-1}}\#(y\triangleright b)\right)\otimes\left(p_y\#b\right)
\end{eqnarray*}
and
$$\pd{\e}(p_x\#b)=\la p_x,1\ra\la\e,b\ra=\delta_{x,1}.$$
\end{itemize}
Clearly, the above structures coincide completely with the \textit{bismash product} structure of $\k^G\#\k F$, which could be found in \cite[Section 3]{LMS06} for example (the dual version of the one in \cite[Preliminaries]{BGM96}):

\begin{proposition}\label{prop:matchedpairbismashprod}
Let $(F,G)$ be a matched pair of finite groups. Then the bismash product Hopf algebra $\k^G\#\k F$ is the left partially dualized Hopf algebra determined by the {\ams} $(\pi_F,\iota_G^\ast)$ for  $\iota_F$.
\end{proposition}

Then a result in \cite{BGM96} on the gauge equivalence of the Drinfeld doubles could be obtained:

\begin{corollary}(cf. \cite[Proposition 5.5]{BGM96})
The Hopf algebras $D(\k^G\#\k F)$ and $D(\k(G\bowtie F))$ are gauge equivalent.
\end{corollary}

\begin{proof}
Straightforward by Propositions \ref{prop:matchedpairbismashprod} and \ref{prop:YDmodsequiv}.
\end{proof}

On the other hand, the {\ams} $(\pi_F,\iota_G^\ast)$ determines a right partially dualized Hopf algebra $\k G\#\k^F$ as well. Then we could write (\ref{eqn:PAMS(right)}) as a self-duality of bismash products (\cite{BGM96}), without fulfilling the details and the proof:
\begin{corollary}(cf. \cite[Proposition 2.1]{BGM96})
There is an isomorphism of Hopf algebras:
$$(\k^G\#\k F)^\ast\cong\k G\#\k^F.$$
\end{corollary}

\begin{remark}
Moreover, with the help of the results in Subsection \ref{subsection:6.2}, one could establish more relations of the Hopf algebras $\k(G\bowtie F)$, $\k^{G\bowtie F}$, $\k F\#\k^G$ and $\k^F\#\k G$. Specifically, we have the following diagram due to (\ref{eqn:multiplePD(MAMS)}):
\begin{equation*}
\xymatrix{
\k(G\bowtie F)
  \ar@<0.75ex>@{|->}[rrr]^{\text{left partial dualization}}
  \ar@<0.75ex>@{|-->}[ddd]^{\text{\shortstack{right partial\\dualization}}}
&&& k^G\#\k F
  \ar@<0.75ex>@{|->}[ddd]^{\text{\shortstack{left partial\\dualization}}}
  \ar@<0.75ex>@{|-->}[lll]^{\text{right partial dualization}}
\\  \\  \\  \k G\#\k^F
  \ar@<0.75ex>@{|->}[uuu]^{\text{\shortstack{left partial\\dualization}}}
  \ar@<0.75ex>@{|-->}[rrr]^{\text{right partial dualization}}
&&& \k^{G\bowtie F}\;.
  \ar@<0.75ex>@{|->}[lll]^{\text{left partial dualization}}
  \ar@<0.75ex>@{|-->}[uuu]^{\text{\shortstack{right partial\\dualization}}}
}
\end{equation*}
\end{remark}

\subsection{Partial duals of bosonizations, and dually paired (braided) Hopf algebras}\label{subsection:partialdualbiproduct}

In \cite{HS13}, the bosonizations of dually paired Hopf algebras $(B',B)$ in the category ${}^A_A\YD$ of (left-left) Yetter-Drinfeld modules over a Hopf algebra $A$ with bijective antipode are studied. We remark that this construction could be regarded as a particular case in \cite{BLS15} when $B$ is a Hopf algebra in an arbitrary braided category instead of $\Vec$, and is referred as \textit{partially dualized Hopf algebras}.
In this subsection, we show that our construction for the partial dualization is a generalization of the structures introduced in \cite[Sections 1 and 2]{HS13} when $A$ and $B$ are finite-dimensional.

Let $A$ be a finite-dimensional Hopf algebra over $\k$, and let $B$ be a finite-dimensional Hopf algebra in the category ${}^A_A\YD$. It is known that we could formulate the \textit{bosonization} or \textit{Radford's biproduct} $H:=B\dotrtimes A$,
which is a Hopf algebra over $\k$ with the structures of smash product and smash coproduct.
Specifically, the multiplication and comultiplication in $H$ is given by
$$(b\dotrtimes x)(c\dotrtimes y):=\sum b(x_{(1)}c)\dotrtimes x_{(2)}y$$
and
$$b\dotrtimes x\mapsto
\sum (b^{(1)}\dotrtimes b^{(2)\la-1\ra}x_{(1)})\otimes(b^{(2)\la0\ra}\dotrtimes x_{(2)})$$
for any $b,c\in B$ and $x,y\in A$, where the coalgebra and left $A$-comodule structure on $B$ are denoted respectively by
$$b\mapsto\sum b^{(1)}\otimes b^{(2)}\in B\otimes B
\;\;\;\;\text{and}\;\;\;\;
b\mapsto\sum b^{\la-1\ra}\otimes b^{\la0\ra}\in A\otimes B.$$
As for the other side, the bosonization $A\dotltimes C$ is defined similarly if $C$ is a Hopf algebra in $\YD{^A_A}$.

However, we would consider an equivalent construction for $B\dotrtimes A$, which is formulated as follows:

\begin{lemma}(\cite[Section 2]{AS98})\label{lem:braidedHopfalg}
Suppose that $H$ and $A$ are finite-dimensional Hopf algebras with Hopf algebra maps
$$\pi:H\twoheadrightarrow A\;\;\;\;\text{and}\;\;\;\;\gamma:A\rightarrowtail H$$
satisfying $\pi\circ\gamma=\id_{A}$.
Then
\begin{equation}\label{eqn:braidedHopfalg}
B:=\{h\in H\mid \sum h_{(1)}\otimes\pi(h_{(2)})=h\otimes\pi(1)\}
\overset{\iota}{\hookrightarrow} H
\end{equation}
is a Hopf algebra in ${}^A_A\YD$, and $H\cong B\dotrtimes A$ as Hopf algebras.
\end{lemma}

Thus the existence of an {\ams} introduced in \cite[Theorem 3(c)]{Rad85} could be as the following lemma:

\begin{lemma}\label{lem:admissiblemapsys}
Under the assumptions in Lemma \ref{lem:braidedHopfalg}, there exists a unique left $B$-module map $\zeta:H\rightarrow B$ such that $(\zeta,\gamma^\ast)$ is an {\ams} for the inclusion $\iota:B\hookrightarrow H$:
$$\begin{array}{ccc}
\xymatrix{
B \ar@<.5ex>[r]^{\iota} & H \ar@<.5ex>@{-->}[l]^{\zeta} \ar@<.5ex>[r]^{\pi}
& A \ar@<.5ex>[l]^{\gamma}  }
&\;\;\text{and}\;\;&
\xymatrix{
A^\ast \ar@<.5ex>[r]^{\pi^\ast}
& H^\ast \ar@<.5ex>[l]^{\gamma^\ast} \ar@<.5ex>[r]^{\iota^\ast}
& B^\ast \ar@<.5ex>@{-->}[l]^{\zeta^\ast}  }.
\end{array}$$
\end{lemma}

\begin{proof}
It is evident that $B$ defined in (\ref{eqn:braidedHopfalg}) is a left coideal subalgebra of $H$ via $\iota$. The existence and uniqueness of $\gamma$ is due to Corollary \ref{cor:BCP-exist}(2). Finally, according to (\ref{eqn:braidedHopfalg}) and \cite[Theorem 3(c)]{Rad85}, it could be found that $(\zeta,\gamma^\ast)$ satisfies the requirements for in Definition \ref{def:PAMS}. Furthermore, the {\pams} $(\zeta,\gamma^\ast)$ is admissible due to Corollary \ref{cor:Hopfalgconditions}(2).
\end{proof}

Now we aim to study the right partially dualized Hopf algebra $A\btd B^\ast$ of $H$ determined by the {\ams} $(\zeta,\gamma^\ast)$ as in Lemma \ref{lem:admissiblemapsys}:

\begin{corollary}
Under the assumptions in Lemmas \ref{lem:braidedHopfalg} and \ref{lem:admissiblemapsys},
the right partially dualized Hopf algebra $A\btd B^\ast$ of $H$ determined by the {\ams} $(\zeta,\gamma^\ast)$ has following structures:
\begin{itemize}
\item
Unit element $\pi(1)\btd\iota^\ast(\e)$ and multiplication:
\begin{equation}\label{eqn:rightpdsmashprod}
(x\btd b^\ast)(y\btd c^\ast)
=\sum xy_{(1)}\btd\left([\zeta^\ast(b^\ast)\leftharpoonup\gamma(y_{(2)})]\btr c^\ast\right)
\end{equation}
for any $x,y\in A$ and $b^\ast,c^\ast\in B^\ast$;

\item
Counit $x\btd b^\ast\mapsto\la\e,x\ra\la b^\ast,1\ra$, and comultiplication $\pd{\Delta}$ satisfying
\begin{equation}\label{eqn:rightpdsmashcoprod}
\pd{\Delta}(x\btd b^\ast)
=\sum_i \left[x_{(1)}\btd(h_i^\ast\btr b^\ast_{(1)})\right]
  \otimes\left[x_{(2)}\pi(h_i)\btd b^\ast_{(2)}\right]
\end{equation}
for any $x\in A$ and $b^\ast\in B^\ast$,
where $\{h_i\}$ is a linear basis of $H$ with dual basis $\{h_i^\ast\}$ of $H$.

\end{itemize}
\end{corollary}

\begin{proof}
We aim to specialize the multiplication (\ref{eqn:rightpdmultiplication}) and comultiplication (\ref{eqn:rightpdcomultiplication}) on $A\btd B^\ast$ in this situation. Note that when $\gamma$ and $\pi$ are Hopf algebra maps, we have
$$x\btl h=\pi[\gamma(x)h]=x\pi(h)\;\;\;\;(\forall x\in A, h\in H).$$
Therefore:
\begin{itemize}
\item
For any $x,y\in A$ and $b^\ast,c^\ast\in B^\ast$,
\begin{eqnarray*}
&&  (x\btd b^\ast)(y\btd c^\ast)  \\
&\overset{(\ref{eqn:rightpdmultiplication})}{=}&
\sum\left(x\btl[\zeta^\ast(b^\ast_{(1)})\rightharpoonup\gamma(y_{(1)})]\right)
  \btd\left([\zeta^\ast(b^\ast_{(2)})\leftharpoonup\gamma(y_{(2)})]\btr c^\ast\right)  \\
&=&
\sum x\pi[\gamma(y_{(1)})_{(1)}]\la\zeta^\ast(b^\ast_{(1)}),\gamma(y_{(1)})_{(2)}\ra
  \btd\left([\zeta^\ast(b^\ast_{(2)})\leftharpoonup\gamma(y_{(2)})]\btr c^\ast\right)  \\
&=&
\sum x\pi[\gamma(y_{(1)})]\la\zeta^\ast(b^\ast_{(1)}),\gamma(y_{(2)})\ra
  \btd\left([\zeta^\ast(b^\ast_{(2)})\leftharpoonup\gamma(y_{(3)})]\btr c^\ast\right)  \\
&=&
\sum xy_{(1)}\la b^\ast_{(1)},\zeta[\gamma(y_{(2)})]\ra
  \btd\left([\zeta^\ast(b^\ast_{(2)})\leftharpoonup\gamma(y_{(3)})]\btr c^\ast\right)  \\
&=&
\sum xy_{(1)}\btd\left([\zeta^\ast(b^\ast)\leftharpoonup\gamma(y_{(2)})]\btr c^\ast\right);
\end{eqnarray*}

\item
For any $x\in A$ and $b^\ast\in B^\ast$,
\begin{eqnarray*}
\pd{\Delta}(x\btd b^\ast)
&\overset{(\ref{eqn:rightpdcomultiplication})}{=}&
\sum_i \left[x_{(1)}\btd(h_i^\ast\btr b^\ast_{(1)})\right]
  \otimes\left[(x_{(2)}\btl h_i)\btd b^\ast_{(2)}\right]  \\
&=&
\sum_i \left[x_{(1)}\btd(h_i^\ast\btr b^\ast_{(1)})\right]
  \otimes\left[x_{(2)}\pi(h_i)\btd b^\ast_{(2)}\right],
\end{eqnarray*}
where $\{h_i\}$ is a linear basis of $H$ with dual basis $\{h_i^\ast\}$ of $H$.
\end{itemize}
The unit and counit are evident.
\end{proof}

Furthermore, one could find that the right partial dual $A\btd B^\ast$ is also a bosonization $A\dotltimes B^\ast$, where $B^\ast$ is a Hopf algebra in $\YD^A_A$.
In fact, recall in the construction of the bosonization $B\dotrtimes A$ that the left $A$-action $\triangleright$ on $B$ is adjoint satisfying
$$x\triangleright b:=\sum\zeta[\gamma(x_{(1)})\iota(b)S(\gamma(x_{(2)}))]
\;\;\;\;\;\;(\forall x\in A,\;b\in B),$$
and the left $A$-coaction and comultiplication of $b\in B$ are respectively
\begin{align*}
& b\mapsto\sum b^{\la-1\ra}\otimes b^{\la0\ra}
:=\sum \pi(b_{(1)})\otimes b_{(2)}\in A\otimes B\;\;\;\;\text{and}\\
& b\mapsto\sum b^{(1)}\otimes b^{(2)}:=\sum \zeta(b_{(1)})\otimes b_{(2)}\in B\otimes B,
\end{align*}
These structures would induce a Hopf algebra structure on $B^\ast$ in $\YD^A_A$ as follows:

\begin{proposition}\label{prop:B*asbraidedHopfalg}
Under the assumptions in Lemmas \ref{lem:braidedHopfalg} and \ref{lem:admissiblemapsys}, suppose $\{b_i\}$ is a linear basis of $B$ with dual basis $\{b_i^\ast\}$ of $B^\ast$. Then $B^\ast$ is a Hopf algebra in $\YD^A_A$ with right $A$-action and coaction:
\begin{equation}\label{eqn:B*asYDmod}
b^\ast\otimes x\mapsto \iota^\ast[\zeta^\ast(b^\ast)\leftharpoonup\gamma(x)]
\;\;\;\;\text{and}\;\;\;\;
b^\ast\mapsto
\sum_i b_i^\ast\otimes\pi[\zeta^\ast(b^\ast)\rightharpoonup\iota(b_i)],
\end{equation}
as well as multiplication
\begin{equation}\label{eqn:B*asalg}
b^\ast\otimes c^\ast\mapsto\zeta^\ast(b^\ast)\btr c^\ast,
\end{equation}
where $b^\ast,c^\ast\in B^\ast$ and $x\in A$;
\end{proposition}

\begin{proof}
At first, note that \cite[Equation (1.35)]{HS13} becomes
\begin{equation}\label{eqn:zetaleftmodmap}
\zeta[\gamma(x)h]=x\triangleright \zeta(h)\;\;\;\;\;\;(\forall x\in A,\;h\in H)
\end{equation}
with our notations on the bosonization $H=B\dotrtimes A$. Let us show that structures (\ref{eqn:B*asYDmod}) on $B^\ast$ is induced via the contravariant functor $\Hom_\k(-,\k)$: Specifically, for any $b\in B$,
\begin{eqnarray*}
\la\iota^\ast[\zeta^\ast(b^\ast)\leftharpoonup\gamma(x)],b\ra
&=&
\la\zeta^\ast(b^\ast)\leftharpoonup\gamma(x),\iota(b)\ra
~=~
\la\zeta^\ast(b^\ast),\gamma(x)\iota(b)\ra  \\
&=&
\la b^\ast,\zeta[\gamma(x)\iota(b)]\ra
~\overset{(\ref{eqn:zetaleftmodmap})}{=}~
\la b^\ast,x\triangleright\zeta[\iota(b)]\ra
~=~
\la b^\ast,x\triangleright b\ra,
\end{eqnarray*}
and one could directly verify that
$$\sum_i b_i^\ast
  \left\la b^\ast,\pi[\zeta^\ast(b^\ast)\rightharpoonup\iota(b_i)]\right\ra
=\sum b^{\la-1\ra}\la b^\ast,b^{\la0\ra}\ra$$
holds in $A$ as well.

Finally, the multiplication (\ref{eqn:B*asalg}) is indeed dual to the comultiplication of $B$, namely:
\begin{eqnarray*}
\la \zeta^\ast(b^\ast)\btr c^\ast,b\ra
&\overset{(\ref{eqn:btr})}{=}&
\la \iota^\ast[\zeta^\ast(b^\ast)\zeta^\ast(c^\ast)],b\ra
~\overset{(\ref{eqn:iotapi})}{=}~
\sum\la \zeta^\ast(b^\ast),b_{(1)}\ra\la\zeta^\ast(c^\ast),\iota(b_{(2)})\ra  \\
&=&
\sum\la b^\ast,\zeta(b_{(1)})\ra\la c^\ast,b_{(2)}\ra
~=~
\sum\la b^\ast,b^{(1)}\ra\la c^\ast,b^{(2)}\ra
\end{eqnarray*}
holds for any $b\in B$.
\end{proof}

On the other hand, it is clear that
$A\cong A\dotltimes\iota^\ast(\e)\hookrightarrow A\dotltimes B^\ast$ has also an retraction as a Hopf algebra map.
Thus there must exists a Hopf algebra $B'$ in ${}^A_A\YD$ such that
$$A\dotltimes B^\ast\cong B'\dotrtimes A$$
as Hopf algebras over $\k$, and we would try to show that $(B,B')$ would be a \textit{dual pair of Hopf algebras} in ${}^A_A\YD$ in the sense of \cite[Definition 2.2]{HS13}. Note that $B'$ is actually unique up to isomorphisms determined by $B$.

Our method is to construct $B'$ from the (braided) Hopf algebra $B^\ast$ in $\YD^A_A$ with structures in Proposition \ref{prop:B*asbraidedHopfalg}. An additional lemma is needed for the purpose:

\begin{lemma}\label{lem:bosoiso}
Let $A$ be a finite-dimensional Hopf algebra over $\k$. Suppose $C$ is a Hopf algebra in $\YD{}^A_A$ with structures denoted by:
\begin{itemize}
\item
The right $A$-action $\triangleleft$, and the right $A$-coaction
$c\mapsto\sum c^{\la0\ra}\otimes c^{\la1\ra}$;
\item
The comultiplication $c\mapsto\sum c^{\la0\ra}\otimes c^{\la1\ra}$
\end{itemize}
for $c\in C$. Then:
\begin{itemize}
\item[(1)](cf. \cite[Section 2.2]{AG99})
If $C_{\mathbf{op\,cop}}$ is the opposite and coopposite (braided) Hopf algebra to $C$ in $\YD{}^A_A$, then $C_{\mathbf{op\,cop}}$ is a Hopf algebra in ${}^A_A\YD$ whose left $A$-action and coaction are given as
\begin{equation}\label{eqn:CbiopasYDmod}
x\otimes c\mapsto c\triangleleft S_A(x)
\;\;\;\;\;\;\text{and}\;\;\;\;\;\;
c\mapsto\sum S^{-1}(c^{\la1\ra})\otimes c^{\la0\ra}
\end{equation}
where $c\in C$ and $x\in A$;

\item[(2)]
There is an isomorphism between bosonizations:
\begin{equation}
C_{\mathbf{op\,cop}}\dotrtimes A\cong A\dotltimes C
\end{equation}
as Hopf algebras over $\k$.
\end{itemize}
\end{lemma}

\begin{proof}
\begin{itemize}
\item[(1)]
The fact that the structures (\ref{eqn:CbiopasYDmod}) make $C_{\mathbf{op\,cop}}$ a left-left Yetter-Drinfeld module over $A$ is due to the proof of \cite[Proposition 2.21]{AG99}, in which the equivalence $\mathfrak{R}:{}^A_A\YD\rightarrow \YD{}^A_A$ sends exactly the object $C_{\mathbf{op\,cop}}$ to $C$. It also follows by the paragraph after \cite[Proposition 2.21]{AG99} that $C_{\mathbf{op\,cop}}$ becomes a Hopf algebra in ${}^A_A\YD$.

\item[(2)]
This might be a known result, but
it is enough and straightforward to verify that
$$C_{\mathbf{op\,cop}}\dotrtimes A\rightarrow (A\dotltimes C)^\biop,\;\;
c\dotrtimes x\mapsto S(x)\dotltimes c$$
is an isomorphism of Hopf algebras, as their antipodes are both linear isomorphisms anti-preserving algebra and coalgebra structures.
\end{itemize}
\end{proof}

Now consider the case when $C$ in Lemma \ref{lem:bosoiso} is chosen as the braided Hopf algebra $B^\ast$ with structures in Proposition \ref{prop:B*asbraidedHopfalg}. The construction implies immediately that we could find that the evaluation $(B^\ast)_{\mathbf{op\,cop}}\otimes B\rightarrow\k$ satisfies the axioms in \cite[Definition 2.2]{HS13} of dually paired Hopf algebras:

\begin{corollary}\label{cor:B'asbraidedHopfalg}
Under the assumptions in Lemmas \ref{lem:braidedHopfalg} and \ref{lem:admissiblemapsys} as well as Proposition \ref{prop:B*asbraidedHopfalg}, denote $B':=(B^\ast)_{\mathbf{op\,cop}}$ with opposite algebra structure to (\ref{eqn:B*asalg}) and coopposite coalgebra structure to $\Delta_{B^\ast}$. Then:
\begin{itemize}
\item[(1)]
$B'$ is a Hopf algebra in ${}^A_A\YD$ whose left $A$-action and coaction are given as
\begin{equation*}
x\otimes b^\ast\mapsto \iota^\ast[\zeta^\ast(b^\ast)\leftharpoonup\gamma(S_A(x))]
\end{equation*}
and
\begin{equation*}
b^\ast\mapsto\sum b^{\ast\,\la-1\ra}\otimes b^{\ast\,\la0\ra}
=\sum_i S_A^{-1}(\pi[\zeta^\ast(b^\ast)\rightharpoonup\iota(b_i)])\otimes b_i^\ast,
\end{equation*}
where $b^\ast\in B^\ast$ and $x\in A$;

\item[(2)]
$(B,B')$ is a dual pair of Hopf algebras in ${}^A_A\YD$ with the evaluation.
\end{itemize}
\end{corollary}

It is known in \cite{HS13} that the categories of Yetter-Drinfeld modules over $B'\dotrtimes A$ and $B\dotrtimes A$ are equivalent as braided tensor categories.
Let us formulate the existence for the left-left case as an application,
from the point of view that $B'\dotrtimes A$ is isomorphic to a right partial dually dualized Hopf algebra of $B\dotrtimes A$:

\begin{proposition}(\cite[Theorem 7.1]{HS13})
Let $A$ be a finite-dimensional Hopf algebra. Suppose $(B,B')$ is a dual pair of Hopf algebras in ${}^A_A\YD$. Then there is a braided tensor equivalence
$${}^{B\dotrtimes A}_{B\dotrtimes A}\YD\approx{}^{B'\dotrtimes A}_{B'\dotrtimes A}\YD$$
between the categories of finite-dimensional left-left Yetter-Drinfeld modules.
\end{proposition}

\begin{proof}
Without the loss of generality, assume that $B'$ is defined as in Corollary \ref{cor:B'asbraidedHopfalg}.
Now we focus on the Hopf algebra $B'\dotrtimes A$, which is isomorphic to the right partial dual $A\dotltimes B^\ast$ of $H=B\dotrtimes A$ according to Lemma \ref{lem:bosoiso}.

Indeed, it is known by \cite[Theorem 4.2]{Ost03} that $\Rep(A\dotltimes B^\ast)$ is categorically Morita equivalent to the category $\Rep(A^\ast\#B)$, since the left partial dual $A^\ast\#B$ is the dual Hopf algebra of $A\dotltimes B^\ast$ by Definition \ref{def:rightpartialdual}. Consequently, (\ref{eqn:Schauenburgequiv}) provides the braided tensor equivalences
$$\Z(\Rep(A^\ast\#B))\approx\Z(\Rep(A\dotltimes B^\ast))
\cong\Z(\Rep(B'\dotrtimes A))$$
of left centers. It also implies by Lemma \ref{lem:centerisoYDmods}(1) that
${}^{A^\ast\#B}_{A^\ast\#B}\YD\approx{}^{B'\dotrtimes A}_{B'\dotrtimes A}\YD$, where the former category is braided tensor equivalent to
${}^{B\dotrtimes A}_{B\dotrtimes A}\YD$ according to Proposition \ref{prop:YDmodsequiv}(1).
\end{proof}

\subsection{Partially dualized Hopf algebras of $4$-dimensional Taft algebra}\label{subsection:Taftalg}

Of course, there exist finite-dimensional quasi-Hopf algebras which are not partial duals of any finite-dimensional Hopf algebra. The $2$-dimensional quasi-Hopf algebra (cf. \cite[Example 3.26]{BCPV19}) is an example, since each left coideal subalgebra $B$ of a two-dimensional Hopf algebra $H$ must be trivial (namely, $B=H$ or $B=\k1$), and the partial dual would become a Hopf algebra by Corollary \ref{cor:trivialcases}.

In this subsection, we hope to determine partially dualized quasi-Hopf algebras of the $4$-dimensional Taft algebra as examples, just in order to show the steps for determining partial duals for specific Hopf algebras.

Suppose the characteristic of the base field $\k$ is not $2$. Let $H$ be the $4$-dimensional Taft algebra (introduced by Sweedler), which is generated by elements $g$ and $x$ with relations
$$g^2=1,\;\;\;\;x^2=0,\;\;\;\;xg=-gx$$
as an algebra, and the comultiplication is given by
$$\Delta(g)=g\otimes g,\;\;\;\;\Delta(x)=x\otimes 1+g \otimes x.$$

It is known that all the indecomposable left coideals of $H$ are
$$\k1,\;\;\;\;\k\{1,x\},\;\;\;\;\k g\;\;\;\;\text{and}\;\;\;\;\k\{g,xg\},$$
and hence there is exactly one left coideal subalgebra
$B:=k\{1,x\}$ which is not a Hopf subalgebra.
Without tedious verifications, we determine all {\pams}s for $B\subseteq H$ and their left partially dualized quasi-Hopf algebras as follows:

\begin{example}\label{ex:4dimTaftalg(MAMS)}
Let $B$ be the non-trivial left coideal subalgebra of the $4$-dimensional Taft algebra $H$.
Then the quotient coalgebra $H/B^+H=\k\{\overline{1},\overline{g}\}$, and $(H/B^+H)^\ast$ is identified with the group algebra of the cyclic group with two elements
$$\e=p_1+p_g\;\;\;\;\text{and}\;\;\;\;f:=p_1-p_g,$$
where $p_1,p_g\in H^\ast$ are defined as
$$p_1:x^ig^j\mapsto \delta_{i,0}\delta_{j,0}
\;\;\;\;\text{and}\;\;\;\;
p_g:x^ig^j\mapsto \delta_{i,0}\delta_{j,1}
\;\;\;\;\;\;(0\leq i,j\leq 1)$$
Furthermore:
\begin{itemize}
\item[(1)]
The diagram
\begin{equation*}
\begin{array}{ccc}
\xymatrix{
B \ar@<.5ex>[r]^{\iota} & H \ar@<.5ex>@{-->}[l]^{\zeta} \ar@<.5ex>[r]^{\pi\;\;\;\;\;\;}
& H/B^+H \ar@<.5ex>@{-->}[l]^{\gamma\;\;\;\;\;\;}  }
&\;\;\text{and}\;\;&
\xymatrix{
(H/B^+H)^\ast \ar@<.5ex>[r]^{\;\;\;\;\;\;\pi^\ast}
& H^\ast \ar@<.5ex>@{-->}[l]^{\;\;\;\;\;\;\gamma^\ast} \ar@<.5ex>[r]^{\iota^\ast}
& B^\ast \ar@<.5ex>@{-->}[l]^{\zeta^\ast}  }
\end{array}
\end{equation*}
is a {\pams} $(\zeta,\gamma^\ast)$ for $\iota$, if and only if there exists a scalar $\lambda\in\k$ such that
\begin{align*}
& \zeta(1)=1,\;\;\zeta(g)=1+\lambda x,\;\;\zeta(x)=\zeta(xg)=x, \\
& \gamma(\overline{1})=1,\;\;\gamma(\overline{g})=(1-\lambda x)g.
\end{align*}

\item[(2)]
The left partially dualized quasi-Hopf algebra $(H/B^+H)^\ast\#B$ of $H$ determined by $(\zeta,\gamma^\ast)$ as in (1) has structures, where we denote
$$e:=\e\#1,\;\;\;\;f:=f\#1\;\;\;\;\text{and}\;\;\;\;x:=\e\#x$$
for simplicity:
\begin{itemize}
\item
As an algebra, it is generated by $f$ and $x$ with relations:
$$f^2=e,\;\;\;\;x^2=0\;\;\;\;\text{and}\;\;\;\;xf=-fx;$$

\item
The comultiplication $\pd{\Delta}$ satisfies that
$$\pd{\Delta}(f)=f\otimes f-\lambda[fx\otimes(e-f)],
\;\;\;\;\pd{\Delta}(x)=x\otimes f+e\otimes x+\lambda(x\otimes fx);$$

\item
The associator is trivial:
$\pd{\phi}=e\otimes e\otimes e$ (see the remark below for a reason);

\item
$\pd{\upsilon}=e$, and an antipode $\pd{S}$ with trivial distinguished elements satisfies that
$$\pd{S}(f)=f+\lambda(x+fx),\;\;\;\;\pd{S}(x)=fx.$$
\end{itemize}

\end{itemize}
\end{example}

\begin{remark}\label{rmk:4dimTaftalg(MAMS)}
According to Proposition \ref{lem:PAMSdetermineHopfalgs(MAMS)},
every {\pams} $(\zeta,\gamma^\ast)$ of $B\subseteq H$ described in Example \ref{ex:4dimTaftalg(MAMS)}(1) is admissible.
In fact, one could verify that
$$\gamma(\overline{g})^2=1\in\gamma(H/B^+H),\;\;\;\;
\Delta(\gamma(\overline{g}))
=\gamma(\overline{g})\otimes g-1\otimes \lambda xg
\in \gamma(H/B^+H)\otimes H,$$
and hence $\gamma(H/B^+H)$ is a right coideal subalgebra if $H$.

Furthermore,
note that the left partial dual described in Example \ref{ex:4dimTaftalg(MAMS)}(2) with coefficient $\lambda=0$ becomes exactly the $4$-dimensional Taft algebra. Thus it follows
due to Proposition \ref{prop:partialdualsequiv}, that
each left partially dualized Hopf algebra $(H/B^+H)^\ast\#B$ is gauge equivalent (and hence isomorphic) to $H$ itself.
\end{remark}

As a result of Remark \ref{rmk:4dimTaftalg(MAMS)}, it seems that partial dualization for the 4-dimensional Taft algebra $H$ does not produce other quasi-Hopf algebras than $H$ itself.
Besides, it would be more complicated to classify {\pams}s and left partial duals of Taft algebras with higher dimensions.

However, in the final section below, we would consider other constructive examples, which include some kind of non-semisimple cases as well as a specific genuine semisimple quasi-Hopf algebra.

\section{Examples of genuine quasi-Hopf algebras determined by abelian extension of Hopf algebras}\label{section:genuinequasi-Hopf}

In the end of this paper, we consider partial duals determined by those {\pams}s arising from extensions of Hopf algebras (\cite{Mas94,AD95}). Then,
we explain how our results are applied to construct new examples of genuine quasi-Hopf algebras.

\subsection{On Hopf algebra extensions and group-theoretical quasi-Hopf algebras}

Recall in \cite[Definitions 1.3 and 2.4]{Mas94} that if a diagram
\begin{equation}\label{eqn:Hopfalgext}
\xymatrix{B \ar[r]^{\iota} & H \ar[r]^{\pi} & C }
\end{equation}
satisfies that
\begin{itemize}
\item[(1)]
$\iota$ is an injection of Hopf algebras, and $\pi$ is a surjection of Hopf algebras;
\item[(2)]
The image of $\iota$ equals the space of the coinvariants of the right $C$-comodule $H$ with structure $(\id_H\otimes\pi)\circ \Delta$
(or other equivalent conditions stated in \cite[Lemma 1.2]{Mas94}) holds,
\end{itemize}
then (\ref{eqn:Hopfalgext}) is called an \textit{extension of Hopf algebras}.
In particular, the extension is said to be \textit{abelian} if $B$ is commutative and $C$ is cocommutative (cf. \cite{Mas02}).

It is clear that for an extension (\ref{eqn:Hopfalgext}) of Hopf algebras, there always exists a {\pams} of form
\begin{equation*}
\begin{array}{ccc}
\xymatrix{
B \ar@<.5ex>[r]^{\iota} & H \ar@<.5ex>@{-->}[l]^{\zeta} \ar@<.5ex>[r]^{\pi}
& C \ar@<.5ex>@{-->}[l]^{\gamma}  }
&\;\;\text{and}\;\;&
\xymatrix{
C^\ast \ar@<.5ex>[r]^{\pi^\ast}
& H^\ast \ar@<.5ex>@{-->}[l]^{\gamma^\ast} \ar@<.5ex>[r]^{\iota^\ast}
& B^\ast \ar@<.5ex>@{-->}[l]^{\zeta^\ast}  },
\end{array}
\end{equation*}
for $\iota$.
In this situation, the algebra structure of the left partial dual determined by $(\zeta,\gamma^\ast)$ is shown to be special:

\begin{lemma}(cf. \cite[Theorem 1.3]{Nat03})\label{lem:abelextleftPD}
Let $H$ be a finite-dimensional Hopf algebra fitting into an extension $B\xrightarrow{\iota}H\xrightarrow{\pi}C$
of Hopf algebras.
Then:
\begin{itemize}
\item[(1)]
Every left partially dualized quasi-Hopf algebra
$$C^\ast\#B=C^\ast\otimes B$$
as algebras. In particular, if the extension is abelian, then $C^\ast\#B$ is commutative.

\item[(2)]
Every right partially dualized coquasi-Hopf algebra
$$C\btd B^\ast=C\otimes B^\ast$$
as coalgebras.
Furthermore, if $x\in C$ and $b^\ast\in B$ are group-like elements, then $x\btd b^\ast$ is also group-like in $C\btd B^\ast$.
\end{itemize}
\end{lemma}

\begin{proof}
\begin{itemize}
\item[(1)]
Note here that $\iota:B\rightarrowtail H$ and $\pi:H\twoheadrightarrow C$ are Hopf algebra maps, and hence the right $H^\ast$-comodule structure of $C^\ast$ and the left $H$-comodule structure of $B$ are respectively:
$$\begin{array}{ccc}
\begin{array}{ccc}
C^\ast &\rightarrow& C^\ast\otimes H^\ast  \\
f &\mapsto& \sum f_{(1)}\otimes \pi^\ast(f_{(2)})
\end{array}
&\;\;\;\;\text{and}\;\;\;\;&
\begin{array}{ccc}
B &\rightarrow& H\otimes B  \\
b &\mapsto& \sum \iota(b_{(1)})\otimes b_{(2)},
\end{array}
\end{array}$$
where the Sweedler notations stand for the comultiplications of respective Hopf algebras $C^\ast$ and $B$.
Then we calculate in the left partially dualized quasi-Hopf algebra $C^\ast\#B$ with the formula (\ref{eqn:smashprod}) that: For all $f,g\in C^\ast$ and $b,c\in B$,
\begin{eqnarray*}
(f\#b)(g\#c)
&=&
\sum f(\iota(b_{(1)})\rightharpoonup g)\# b_{(2)}c
~=~
\sum fg_{(1)}\langle\pi^\ast(g_{(2)}),\iota(b_{(1)})\rangle\# b_{(2)}c  \\
&=&
\sum fg_{(1)}\langle g_{(2)},\pi[\iota(b_{(1)})]\rangle\# b_{(2)}c
\overset{(\ref{eqn:piiota})}{=}
\sum fg_{(1)}\langle g_{(2)},\e(b_{(1)})1\rangle\# b_{(2)}c  \\
&=& fg\#bc.
\end{eqnarray*}

Furthermore, if $B$ is commutative and $C$ is cocommutative, then the tensor product $C^\ast\otimes B$ of algebras becomes commutative, since $C^\ast$ is commutative as well.

\item[(2)]
Since
$C\btd B^\ast=(C^\ast\#B)^\ast$ as quasi-bialgebras,
we know that the first claim in (2) follows immediately from (1). Also, it is clear that the second claim on group-like elements hold as a direct consequence.
\end{itemize}
\end{proof}

\begin{remark}
Suppose $H$ is a semisimple Hopf algebra fitting into an abelian extension of Hopf algebras. Then we could infer by Lemma \ref{lem:abelextleftPD}(1) and Corollary \ref{cor:catMoritaequiv} that the fusion category $\Rep(H)$ is categorically Morita equivalent to a pointed one, and hence $\Rep(H)$ is \textit{group-theoretical} in the sense of \cite[Definition 8.40]{ENO05}. This fact was shown by Natale in \cite[Theorem 1.3]{Nat03}.
\end{remark}

Another particular case for extension of Hopf algebras is the split condition:

\begin{definition}(\cite[Definition 6.5.2]{Sch02})\label{def:splitHopfalgext}
An extension
$B\xrightarrow{\iota} H\xrightarrow{\pi} C$
of Hopf algebras is said to be split, if there exist a left $B$-module coalgebra map $\zeta:H\rightarrow B$ and a right $C$-comodule algebra map $\gamma:C\rightarrow H$, such that $(\iota\circ\zeta)\ast(\gamma\circ\pi)=\id_H$ holds.
\end{definition}

Under the conditions in Definition \ref{def:splitHopfalgext},
it is clear that $\zeta$ preserves the counits, and $\gamma$ preserves the units. Moreover, due to the arguments in the proof of \cite[Theorem 9]{DT86} and
\cite[Lemma 2.15]{Mas92}, we could assume that $\zeta$ and $\gamma$ both preserve the units and the counits since $\zeta(1)$ must be a group-like element in this case. Details are also found in the paragraphs before Lemma \ref{lem:cleftcocleft2}.

Consequently, a split extension (\ref{eqn:Hopfalgext}) of finite-dimensional Hopf algebras induces a {\pams} $(\zeta,\gamma^\ast)$ of $\iota$, such that $\zeta$ is a left $B$-module coalgebra map and $\gamma$ is a right $C$-comodule algebra map. In this situation, we could have an observation as in the following lemma:

\begin{lemma}\label{lem:splitextdetermineHopfalgs(MAMS)}
Let $H$ be a finite-dimensional Hopf algebra fitting into a split extension $B\xrightarrow{\iota}H\xrightarrow{\pi}C$
of Hopf algebras.
Suppose that $(\zeta,\gamma^\ast)$ is a {\pams} for $\iota$, where $\zeta$ is a left $B$-module coalgebra map and $\gamma$ is a right $C$-comodule algebra map.
Then
$(\zeta,\gamma^\ast)$ determines partially dualized Hopf algebras.
\end{lemma}

\begin{proof}
Our goal is to show that $\gamma(C)$ is a right coideal of $H$. Firstly, since $\zeta$ is a coalgebra map and $\gamma$ is a right $C$-comodule map, we could write
\begin{equation}\label{eqn:zetacoalgmap(MAMS)}
\sum\zeta(h_{(1)})\otimes\zeta(h_{(2)})=\sum\zeta(h)_{(1)}\otimes \zeta(h)_{(2)}
\;\;\;\;\;\;\;\;(\forall h\in H)
\end{equation}
and
\begin{equation}\label{eqn:gammarightC-comodmap(MAMS)}
\sum\gamma(x_{(1)})\otimes\pi[\gamma(x_{(2)})]=\sum\gamma(x_{(1)})\otimes x_{(2)}
\;\;\;\;\;\;\;\;(\forall x\in C).
\end{equation}
Now let us calculate for any $u\in C$ that:
\begin{eqnarray*}
\sum\zeta[\gamma(x)_{(1)}]\otimes\gamma(x)_{(2)}
&\overset{(\ref{eqn:convolutionprod(MAMS)})}=&
\sum\zeta[\gamma(x)_{(1)}]\otimes
  \iota\left(\zeta[\gamma(x)_{(2)}]\right)\gamma\left(\pi[\gamma(x)_{(3)}]\right)  \\
&\overset{(\ref{eqn:gammarightC-comodmap(MAMS)})}=&
\sum\zeta[\gamma(x_{(1)})_{(1)}]\otimes
  \iota\left(\zeta[\gamma(u_{(1)})_{(2)}]\right)\gamma(u_{(2)})  \\
&\overset{(\ref{eqn:zetacoalgmap(MAMS)})}=&
\sum\zeta[\gamma(x_{(1)})]_{(1)}\otimes
  \iota\left(\zeta[\gamma(x_{(1)})]_{(2)}\right)\gamma(x_{(2)})  \\
&=&
1_B\otimes \gamma(x),
\end{eqnarray*}
where the last equality is due to Proposition \ref{prop:PAMS-comptrival}(2) that $\zeta\circ\gamma$ is trivial.
Hence
\begin{eqnarray*}
\sum\gamma(x)_{(1)}\otimes\gamma(x)_{(2)}
&\overset{(\ref{eqn:convolutionprod(MAMS)})}=&
\sum\iota\left(\zeta[\gamma(x)_{(1)}]\right)\gamma\left(\pi[\gamma(x)_{(2)}]\right)
  \otimes\gamma(x)_{(3)}  \\
&=&
\sum\iota(1_B)\gamma\left(\pi[\gamma(x)_{(1)}]\right)\otimes\gamma(x)_{(2)}  \\
&=&
\sum\gamma\left(\pi[\gamma(x)_{(1)}]\right)\otimes\gamma(x)_{(2)}
\;\;\in \gamma(C)\otimes H.
\end{eqnarray*}

At final, we might apply Proposition \ref{lem:PAMSdetermineHopfalgs(MAMS)} to know that $(\pi,\iota^\ast)$ determines partially dualized Hopf algebras.
\end{proof}

\begin{corollary}(cf. \cite{Sch02})\label{cor:splitabelext}
Suppose $H$ is a finite-dimensional Hopf algebra fitting into a split abelian extension of Hopf algebras.
Then $\Rep(H)$ is categorically Morita equivalent to $\Rep(K)$ for some finite-dimensional cocommutative Hopf algebra $K$.
\end{corollary}

\begin{proof}
As it is explained in the paragraphs after Definition \ref{def:splitHopfalgext}, there must exist a {\pams} $(\zeta,\gamma^\ast)$ satisfying the assumptions of Lemma \ref{lem:splitextdetermineHopfalgs(MAMS)}. Consequently, we could conclude by Lemmas \ref{lem:abelextleftPD} and \ref{lem:splitextdetermineHopfalgs(MAMS)} that $H$ has a commutative left partially dualized Hopf algebra, whose dual Hopf algebra $K$ is then cocommutative. It follows that $\Rep(H)$ and $\Rep(K)$ are also categorically Morita equivalent.
\end{proof}

\subsection{On Hopf algebras with coradical fitting into abelian extensions}\label{subsection:7.2}

In this subsection, we attempt to show that our partial dual constructions are applied to a class of non-semisimple Hopf algebras,
by providing a generalization of Lemma \ref{lem:abelextleftPD}(1) as follows.

Recall that a finite-dimensional algebra $K$ is said to be \textit{basic}, if every irreducible representation of $K$ is 1-dimesnional (or equivalently, the dual coalgebra $K^\ast$ is pointed).

\begin{proposition}\label{prop:coradicalabelext(MAMS)}
Let $H$ be a finite-dimensional Hopf algebra whose coradical $H_0$ is a Hopf subalgebra fitting into an abelian extension. Then $H$ has a left partially dualized quasi-Hopf algebra which is basic.
\end{proposition}

The proof of Proposition \ref{prop:coradicalabelext(MAMS)} is based on the following observation:

\begin{lemma}\label{lem:PAMScoradicalabelext(MAMS)}
Under the assumptions of Proposition \ref{prop:coradicalabelext(MAMS)},
we denote by $\jmath:H_0\hookrightarrow H$ the inclusion of the coradical $H_0$ into $H$, and denote $\iota:=\jmath\circ\iota_0$.
Suppose $B\xrightarrow{\iota_0}H_0\xrightarrow{\pi_0}C$ is an abelian extension of Hopf algebras, and $(\zeta_0,\gamma_0^\ast)$ is a {\pams} of $\iota_0$.
Then there is a $\k$-linear diagram
$$\xymatrix{
B \ar@<.5ex>[rr]^{\iota_0}   \ar@{=}[dd]
&& H_0 \ar@<.5ex>@{-->}[ll]^{\zeta_0} \ar@<.5ex>[rr]^{\pi_0}
  \ar@<.5ex>[dd]^{\jmath}
&& C \ar@<.5ex>@{-->}[ll]^{\gamma_0}  \ar[dd]^{\ell}
\\  \\
B  \ar@<.5ex>[rr]^{\iota}
&& H  \ar@<.5ex>[rr]^{\pi}  \ar@<.5ex>@{-->}[ll]^{\zeta}
  \ar@<.5ex>@{-->}[uu]^{\zeta'}
&& D \ar@<.5ex>@{-->}[ll]^{\gamma}\;,
}$$
such that:
\begin{itemize}
\item[(1)]
$\zeta':H\rightarrow H_0$ is a left $H_0$-module coalgebra map satisfying $\zeta'\circ\jmath=\id_{H_0}$;

\item[(2)]
$\zeta:=\zeta_0\circ\zeta'$ denotes the composition, and $(\zeta,\gamma^\ast)$ is a {\pams} of $\iota:B\hookrightarrow H$;

\item[(3)]
$\ell$ is an injection of coalgebras whose image $\ell(C)$ is the coradical of $D$;

\item[(4)]
$\gamma[\ell(C)]\subseteq H_0$.
\end{itemize}
\end{lemma}

\begin{proof}
Let us show the existence of the desired maps by steps:
\begin{itemize}
\item[(1)]
According to \cite[Theorem 3.1]{Mas03} and its following sentence, there exists a left $H_0$-module coalgebra map $\zeta':H\rightarrow H_0$ with $\zeta'(1)=1_B$. Therefore, we find that $\zeta'(a)=a\zeta'(1)=a$ hold for each $a\in H_0$.

\item[(2)]
We aim to verify that the composition map $\zeta:H\rightarrow B$ satisfies the desired properties in Definition \ref{def:PAMS}, and then construct $\gamma$ as (\ref{eqn:gammabardef}) in order to obtain the {\pams} $(\zeta,\gamma^\ast)$.

Firstly, it is clear that $\zeta:H\rightarrow B$ is a left $B$-module map, since $\zeta'$ and $\zeta_0$ preserve left $H_0$-actions and $B$-actions respectively.
Moreover, note in (1) that $\zeta'$ is a coalgebra map. Thus one could check that $\zeta=\zeta_0\circ\zeta'$ has convolution inverse $\overline{\zeta_0}\circ\zeta'$. Namely,
\begin{equation}\label{eqn:zeta0zeta'convinv(MAMS)}
\overline{\zeta}=\overline{\zeta_0}\circ\zeta',
\end{equation}
where $\overline{\zeta_0}$ is the convolution inverse of the cointegral $\zeta_0$.

On the other hand, since $\zeta'$ and $\zeta_0$ are both unitary, we find that $\zeta$ is also unitary. Besides, note again that $\zeta'$ is a coalgebra map, and it follows that $\zeta'$ (as well as $\zeta_0$, of course) is counitary. Therefore, their composition $\zeta$ is also counitary.

Consequently, if we denote $D:=H/B^+H$ and the quotient map $\pi:H\twoheadrightarrow D$, then the map
\begin{equation}\label{eqn:gammadef7.2(MAMS)}
\gamma:D\rightarrow H,\;\pi(h)\mapsto\sum\iota[\overline{\zeta}(h_{(1)})]h_{(2)}
\end{equation}
would satisfy that $(\zeta,\gamma^\ast)$ is a {\pams} of $\iota$. This is
due to a similar argument with Lemmas \ref{lem:cleftcocleft} and \ref{lem:cleftcocleft2}.

\item[(3)]
Here we make identification $C=H_0/B^+H_0$ without the loss of generality, and claim that $$B^+H\cap H_0=B^+H_0$$ as subspaces of $H$. In fact, it follows from (1) that $\zeta'$ is a projection from $H$ to its Hopf subalgebra $H_0$ ( namely, $\zeta'\mid_{H_0}=\id_{H_0}$). Thus one has
$$B^+H\cap H_0=\zeta'(B^+H\cap H_0)\subseteq\zeta'(B^+H)
=B^+\zeta'(H)=B^+H_0,$$
where the penultimate equality is because $\zeta'$ preserves left $H_0$-actions (and hence preserves left $B$-actions).
Conversely, it is evident that $B^+H\cap H_0\supseteq B^+H_0$ holds as well.

As a result, we could define $\ell:C\rightarrow D$ to be the composition of the following injections:
$$\ell:H_0/B^+H_0=H_0/(B^+H\cap H_0)\cong (H_0+B^+H)/B^+H
\hookrightarrow H/B^+H.$$
in other words, this definition means that
\begin{equation}\label{eqn:ellpi0=pijmath(MAMS)}
\ell[\pi_0(a)]=\pi[\jmath(a)]\;\;\;\;(\forall a\in H_0),
\end{equation}
and hence $\ell$ is a injective coalgebra map.

Furthermore, since $H_0$ is a cosemisimple Hopf algebra, its quotient Hopf algebra would also be cosemisimple by the dual version of \cite[Proposition 10.3.4]{Rad12}. Thus $\ell(C)$ is a cosemisimple subcoalgebra of $D$. Meanwhile, note that $H$ is cogenerated by its coradical $H_0$ (via wedge products). It follows that the $D$ is cogenerated by $\ell(C)$, because
$$D=H/B^+H=\pi(H),\;\;\;\;\ell(C)=\ell(H_0/B^+H_0)=\pi[\jmath(H_0)]=\pi(H_0),$$
and $\pi$ is a coalgebra map. One could conclude that $\ell(C)$ is the coradical of $D$.

\item[(4)]
Recall in (\ref{eqn:ellpi0=pijmath(MAMS)}) that $\ell\circ\pi_0=\pi\circ\jmath$ holds,
which implies that
$$\gamma[\ell(C)]=\gamma\left(\ell[\pi_0(H_0)]\right)
=\gamma\left(\pi[\jmath(H_0)]\right)=\gamma[\pi(H_0)].$$
Thus it suffices to show $\gamma\left(\pi[\jmath(H_0)]\right)\subseteq H_0$.
In fact, according to the definition (\ref{eqn:gammadef7.2(MAMS)}) of $\gamma$, we calculate for any $a\in H_0$ that
$$\gamma[\pi(a)]
=\sum\iota[\overline{\zeta}(a_{(1)})]a_{(2)}
\overset{(\ref{eqn:zeta0zeta'convinv(MAMS)})}{=}
\sum\iota\left(\overline{\zeta_0}[\zeta'(a_{(1)})]\right)a_{(2)}
\in H_0,
$$
since the convolution inverse $\overline{\zeta_0}$ of $\zeta_0$ is a map into the Hopf subalgebra $H_0$.
\end{itemize}
\end{proof}

\begin{proof}
[Proof of Proposition \ref{prop:coradicalabelext(MAMS)}]~

Our goal is to prove the desired claim in the dual form. Specifically, we aim to show that: With the notations of Lemma \ref{lem:PAMScoradicalabelext(MAMS)}, the right partially dualized coquasi-Hopf algebra $D\btd B^\ast$ of $H$ determined by the {\pams} $(\zeta,\gamma^\ast)$ is pointed.

Note at first that $\iota=\jmath\circ\iota_0$ is a Hopf algebra map, and hence $\iota^\ast$ is a Hopf algebra map as well. Then according to the definitions of the actions $\btr$ and $\btl$ in (\ref{eqn:btlbtrnew(MAMS)}), as well as
the definition (\ref{eqn:rightpdcomultiplication}) of the comultiplication $\pd{\Delta}$ on the right partial dual $D\btd B^\ast$, we could write for any $x\in D$ and $b^\ast\in B^\ast$ that
\begin{eqnarray}\label{eqn:Delta(ub*)(MAMS)}
\pd{\Delta}(x\btd b^\ast)
&=&
\sum_i \left[x_{(1)}\btd(h_i^\ast\btr b^\ast_{(1)})\right]
  \otimes\left[(x_{(2)}\btl h_i)\btd b^\ast_{(2)}\right]  \nonumber  \\
&=&
\sum_i \left(x_{(1)}\btd\iota^\ast[h_i^\ast\zeta^\ast(b^\ast_{(1)})]\right)
  \otimes\left(\pi[\gamma(x_{(2)})h_i]\btd b^\ast_{(2)}\right)  \nonumber  \\
&=&
\sum_i \left(x_{(1)}\btd\iota^\ast(h_i^\ast)\iota^\ast[\zeta^\ast(b^\ast_{(1)})]\right)
  \otimes\left(\pi[\gamma(x_{(2)})h_i]\btd b^\ast_{(2)}\right)  \nonumber  \\
&\overset{(\ref{eqn:iota*zeta*})}=&
\sum_i \left(x_{(1)}\btd\iota^\ast(h_i^\ast)b^\ast_{(1)}\right)
  \otimes\left(\pi[\gamma(x_{(2)})h_i]\btd b^\ast_{(2)}\right),
\end{eqnarray}
where $\{h_i\}$ is a linear basis of $H$ with dual basis $\{h_i^\ast\}$ of $H^\ast$.

However, one could find that the element
\begin{equation}\label{eqn:inB*B(MAMS)}
\sum_i \iota^\ast(h_i^\ast)\otimes h_i\in B^\ast\otimes B,
\end{equation}
where $B^\ast\otimes B$ is regarded as a subspace of $B^\ast\otimes H$. Indeed, its image under $b\otimes\id_H$ for an arbitrary $b\in B$ would become
$$\sum_i \langle\iota^\ast(h_i^\ast),b\rangle h_i
=\sum_i \langle h_i^\ast,\iota(b)\rangle h_i=\iota(b)\in B.$$
Therefore, if we identify $C$ as a subcoalgebra of $D$ via the injection $\ell$ for convenience, then
\begin{equation}\label{eqn:inB*H0(MAMS)}
\sum_i \iota^\ast(h_i^\ast)\otimes \gamma(C)h_i
\subseteq B^\ast\otimes \gamma(C)B\subseteq B^\ast\otimes H_0,
\end{equation}
where the last inclusion is due to Lemma \ref{lem:PAMScoradicalabelext(MAMS)}(4)
that $\gamma(C)\subseteq H_0$.

As a consequence, since $\pi\mid_{H_0}=\pi_0$ is a Hopf algebra map when the injections $\jmath$ and $\ell$ are abbreviated, we could write by Equation (\ref{eqn:Delta(ub*)(MAMS)}) for any $y\in C$ and $b^\ast\in B^\ast$ that
\begin{eqnarray*}
\pd{\Delta}(y\btd b^\ast)
&=&
\sum_i \left(y_{(1)}\btd\iota^\ast(h_i^\ast)b^\ast_{(1)}\right)
  \otimes\left(\pi[\gamma(y_{(2)})h_i]\btd b^\ast_{(2)}\right)  \\
&=&
\sum_i \left(y_{(1)}\btd\iota^\ast(h_i^\ast)b^\ast_{(1)}\right)
  \otimes\left(\pi_0[\gamma(y_{(2)})h_i]\btd b^\ast_{(2)}\right)  \\
&\overset{(\ref{eqn:inB*B(MAMS)})}=&
\sum_i \left(y_{(1)}\btd\iota^\ast(h_i^\ast)b^\ast_{(1)}\right)
  \otimes\left(\pi_0[\gamma(y_{(2)})]\pi_0(h_i)\btd b^\ast_{(2)}\right)  \\
&=&
\sum_i \left(y_{(1)}\btd\iota^\ast(h_i^\ast)b^\ast_{(1)}\right)
  \otimes\left(\pi[\gamma(y_{(2)})]\pi(h_i)\btd b^\ast_{(2)}\right)  \\
&=&
\sum_i \left(y_{(1)}\btd\iota^\ast(h_i^\ast)b^\ast_{(1)}\right)
  \otimes\left(y_{(2)}\pi(h_i)\btd b^\ast_{(2)}\right),
\end{eqnarray*}
where the last equality is obtained from Lemma \ref{lem:cleftcocleft2}(2) that $\pi\circ\gamma=\id_C$. However, it is straightforward to find that
$\sum_i\iota^\ast(h_i^\ast)\otimes\pi(h_i)=\e\otimes1$ according to Equation (\ref{eqn:piiota}) that $\pi\circ\iota$ is trivial, and hence
$$\pd{\Delta}(y\btd b^\ast)
=\sum (y_{(1)}\btd b^\ast_{(1)})
  \otimes(y_{(2)}\btd b^\ast_{(2)})
\;\;\;\;(\forall y\in C,\;\forall b^\ast\in B^\ast).$$
Thus we conclude that $C\btd B^\ast$ is a subcoalgebra of $D\btd B^\ast$, and $C\btd B^\ast=C\otimes B^\ast$ as coalgebras. It also follows that $C\btd B^\ast$ is cosemisimple and cocommutative according to the definition of abelian extension $B\rightarrow H_0\rightarrow C$ of the cosemisimple Hopf algebras.

On the other hand, if we denote by $\{D_n\}$ the coradical filtration of the coalgebra $D$, then it is clear that $D_0=C$ by Lemma \ref{lem:PAMScoradicalabelext(MAMS)}(4) and that
$$\Delta(D_n)\subseteq D_{n-1}\otimes D+D\otimes C\;\;\;\;(\forall n\geq1).$$
Similarly, one could find with the help of Equation (\ref{eqn:Delta(ub*)(MAMS)}) that
\begin{eqnarray*}
\Delta(D_n\btd B^\ast)
&\subseteq&
(D_{n-1}\btd B^\ast)\otimes(D\btd B^\ast)
+ \sum_i\left(D\btd \iota(h_i^\ast)B^\ast\right)
\otimes\left(\pi[\gamma(C)h_i]\btd B^\ast\right)  \\
&\subseteq&
(D_{n-1}\btd B^\ast)\otimes(D\btd B^\ast)
+ \left(D\btd B^\ast\right)
\otimes\left(C\btd B^\ast\right)
\end{eqnarray*}
for all $n\geq1$, where the last inclusion is due to (\ref{eqn:inB*H0(MAMS)}) and the fact that $\pi(H_0)=\pi_0(H_0)=C$. As a conclusion, the right partial dual $D\btd B^\ast$ has the coradical filtration $\{D_n\btd B^\ast\}$, and hence its coradical should be $D_0\btd B^\ast$, which is exactly $C\btd B^\ast$, a cosemisimple and cocommutative subcoalgebra. This means that $D\btd B^\ast$ is a pointed coalgebra, or equivalently, the left partial dual $D^\ast\#B$ of $H$ is a basic algebra.
\end{proof}

\subsection{A method to determine examples of genuine quasi-Hopf algebras}

We have mentioned in
Remark \ref{rmk:gaugeequiv(MAMS)} about the notion of \textit{gauge equivalence} of quasi-Hopf algebras. Here, let us recall that a finite-dimensional quasi-Hopf algebra $K$ is said to be \textit{genuine}, if $K$ is not gauge equivalent to any Hopf algebra (or equivalently, $\Rep(K)$ is not tensor equivalent to $\Rep(K')$ for any finite-dimensional Hopf algebra $K'$).

In order to decide whether a quasi-Hopf algebra is genuine,
we would use the notion of the \textit{Grothendieck ring} $\mathsf{Gr}(\C)$ (\cite[Section 3]{Eti02} and \cite[Section 2.1]{EO04}) of a finite tensor category $\C$, and
one could see \cite[Sections 3.1 and 4.5]{EGNO15} for its detailed theory.
Clearly, the Grothendieck ring is invariant under tensor equivalences with the following properties for the case of quasi-Hopf algebras:

\begin{lemma}\label{lem:Groring(MAMS)}
Let $K$ and $K'$ be finite-dimensional quasi-Hopf algebras.
\begin{itemize}
\item[(1)]
(cf. \cite[Remark 4.5.6]{EGNO15})
If $K$ and $K'$ are gauge equivalent, then $\mathsf{Gr}(\Rep(K))\cong\mathsf{Gr}(\Rep(K'))$ as $\mathbb{Z}_+$-rings;

\item[(2)]
Suppose $K$ is isomorphic to the dual group algebra $(\k G)^\ast$ as an algebra and a coalgebra, where $G$ is a finite group. Then $\mathsf{Gr}(\Rep(K))\cong \mathbb{Z}G$ as $\mathbb{Z}_+$-rings.
\end{itemize}
\end{lemma}

\begin{remark}\label{rmk:Z+ringofgrps(MAMS)}
The definition of $\mathbb{Z}_+$-rings referred to as in \cite[Sections 3.1]{EGNO15} implies the following fact: For two finite groups $G$ and $G'$, one has
\begin{center}
$\mathbb{Z}G\cong\mathbb{Z}G'$ as $\mathbb{Z}_+$-rings
\;\;\;\;$\Longleftrightarrow$\;\;\;\;
$G\cong G'$ as groups.
\end{center}
\end{remark}

Another notion of gauge invariant called the \textit{exponent} of fusion categories and finite-dimensional semisimple quasi-Hopf algebras:

\begin{definition}(\cite[Section 6]{Eti02})
Let $\C$ be a fusion category over $\mathbb{C}$. The exponent of $\C$ is defined as the order of the square of the braiding $\sigma$ for its left center $\Z(\C)$. Namely,
\begin{equation}
\exp(\C):=\min\{n\geq1\mid \forall V,W\in\Z(\C),\;\;
(\sigma_{W,V}\circ\sigma_{V,W})^n=\id_V\otimes\id_W\}.
\end{equation}

When $\C=\Rep(K)$ is the category of finite-dimensional modules over a semisimple quasi-Hopf algebra $K$ over $\mathbb{C}$, we also denote
$\exp(K):=\exp(\Rep(K)).$
\end{definition}

\begin{remark}\label{rmk:fusioncatexp(MAMS)}
\begin{itemize}
\item[(1)]
The exponent $\exp(\C)$ of a fusion category $\C$ is a positive integer;

\item[(2)]
Note by \cite[Theorem 2.15]{ENO05} that the dual tensor category $\C_\M^\ast$ is also fusion if $\M$ is a finite semisimple indecomposable $\C$-module category.
Furthermore, it follows from \cite[Proposition 6.3(2)]{Eti02} that $\exp(\C_\M^\ast)=\exp(\C)$ for any dual fusion category $\C_\M^\ast$;

\item[(3)]
In particular, if two semisimple quasi-Hopf algebra $K$ and $K'$ over $\mathbb{C}$ are gauge equivalent, then $\exp(K)=\exp(K')$.

\item[(4)]
Suppose $H$ is a semisimple Hopf algebra over $\mathbb{C}$. Then it is known that
$$\exp(H)=\min\big\{n\geq1\mid \forall h\in H,\;\;
  \sum h_{(1)}S^{-2}(h_{(2)})\cdots S^{-2n+2}(h_{(n)})=\e(h)1_H\big\},$$
which is also called in \cite{Kas99,Kas00,EG99} the \textit{exponent} of $H$ (with involutory antipode).
\end{itemize}
\end{remark}

However, due to the reconstruction $\Rep(H)_{\Rep(B)}^\ast\approx\Rep(C^\ast\# B)$ in Corollary \ref{cor:catMoritaequiv}, one could know by Remark \ref{rmk:fusioncatexp(MAMS)} that:

\begin{corollary}\label{cor:leftpdsemisimple(MAMS)}
Let $H$ be a semisimple Hopf algebra over $\mathbb{C}$ with left partially dualized quasi-Hopf algebra $C^\ast\# B$. Then:
\begin{itemize}
\item[(1)]
$C^\ast\# B$ is a semisimple quasi-Hopf algebra;
\item[(2)]
$\exp(H)=\exp(C^\ast\# B)$ holds.
\end{itemize}
\end{corollary}

For the remaining of this paper, we focus on our construction of genuine quasi-Hopf algebras:

Let $H$ be the $8$-dimensional Kac Hopf algebra over $\mathbb{C}$ found by Kac in 1960s. For convenience, we would study it with the structure described in \cite[Remark 2.14(2)]{Mas95}: Suppose $q$ is a primitive $8$th root of 1, and write $\sqrt{-1}:=q^2$. As an algebra, $H$ is generated by elements $b,c,w$ with relations
$$b^2=c^2=w^2=1,\;\;\;\;
cb=bc,\;\;\;\;wb=cw,\;\;\;\;wc=bw.$$
The coalgebra structure and the antipode are given by
$$\left\{\begin{array}{l}
\Delta(b)=b\otimes b,\;\;\;\; \Delta(c)=c\otimes c,\\
\Delta(w)
=\left(\frac{1}{2}(1+bc)\otimes 1+\frac{1+\sqrt{-1}}{4}(1-bc)\otimes b +\frac{1-\sqrt{-1}}{4}(1-bc)\otimes c\right)(w\otimes w),\\
\e(b)=\e(c)=\e(w)=1,  \\
S(b)=b,\;\;\;\;S(c)=c,\;\;\;\;
S(w)=\left(\frac{1+\sqrt{-1}}{2}b+\frac{1-\sqrt{-1}}{2}c\right)w.
\end{array}\right.$$
It is known that $H$ is a semisimple Hopf algebra, and its exponent is 8 due to \cite[Theorem 12]{Kas99}.

In fact, $H$ is group-theoretical, since it fits into an abelian extension given in the following lemma, from which a particular {\pams} arises:

\begin{lemma}\label{lem:K8pams}
With the notations above, consider the abelian extension
$B\xrightarrow{\iota}H\xrightarrow{\pi}C$,
where:
\begin{itemize}
\item
$B=\mathbb{C}\{1,b,c,bc\}$ is the Klein four-group algebra;
\item
$C:=H/B^+H=\mathbb{C}\{\overline{1},\overline{w}\}$
is the algebra of the cyclic group generated by $\overline{w}$;
\item
$\iota$ is the (normal) inclusion of Hopf algebras, and $\pi$ is the quotient Hopf algebra map.
\end{itemize}
Then
\begin{equation*}
\begin{array}{ccc}
\xymatrix{
B \ar@<.5ex>[r]^{\iota} & H \ar@<.5ex>@{-->}[l]^{\zeta} \ar@<.5ex>[r]^{\pi}
& C \ar@<.5ex>@{-->}[l]^{\gamma}  }
&\;\;\text{and}\;\;&
\xymatrix{
C^\ast \ar@<.5ex>[r]^{\pi^\ast}
& H^\ast \ar@<.5ex>@{-->}[l]^{\gamma^\ast} \ar@<.5ex>[r]^{\iota^\ast}
& B^\ast \ar@<.5ex>@{-->}[l]^{\zeta^\ast}  },
\end{array}
\end{equation*}
is a {\pams}, where:
\begin{itemize}
\item
$\zeta$ is the left $B$-module map satisfying
$\zeta(w)=1$ and $\zeta\mid_B=\id_B$,
and
\item
$\gamma$ satisfies $\gamma(\overline{1})=1$ and
$\gamma(\overline{w})=w$.
\end{itemize}
\end{lemma}

\begin{proof}
Note that the axioms (1) and (2) in Definition \ref{def:PAMS} hold since $B\xrightarrow{\iota}H\xrightarrow{\pi}C$ is an extension of Hopf algebras, and it remains to check the axioms (3) to (6).

\textbf{Axioms (3) and (4):}
As $H$ is clearly a free left $B$-module with basis $\{1,w\}$, we know that $\zeta$ is a well-defined map preserving left $B$-actions. Moreover, the convolution inverse $\overline{\zeta}$ of $\zeta$ is the linear map satisfying
\begin{equation}\label{eqn:K8zetabar(MAMS)}
\overline{\zeta}(1)=1,\;\;\;\;
\overline{\zeta}(w)=\frac{1}{2}(1-\sqrt{-1}b+\sqrt{-1}c+bc),\;\;\;\;
\overline{\zeta}(ah)=\overline{\zeta}(h)S(a)\;\;\;\;
(\forall a\in B,\;\forall h\in H).
\end{equation}
In fact, 
we could find that
\begin{eqnarray*}
&& (\overline{\zeta}\ast\zeta)(w)  \\
&=& \frac{1}{2}\overline{\zeta}[(1+bc)w]\zeta(w)
+\frac{1+\sqrt{-1}}{4}\overline{\zeta}[(1-bc)w]\zeta(bw) +\frac{1-\sqrt{-1}}{4}\overline{\zeta}[(1-bc)w]\zeta(cw)  \\
&=& \frac{1}{2}\overline{\zeta}(w)S(1+bc)\zeta(w)
+\frac{1+\sqrt{-1}}{4}\overline{\zeta}(w)S(1-bc)b\zeta(w) +\frac{1-\sqrt{-1}}{4}\overline{\zeta}(w)S(1-bc)c\zeta(w)  \\
&=& \frac{1}{2}\overline{\zeta}(w)(1+bc)
+\frac{1+\sqrt{-1}}{4}\overline{\zeta}(w)(b-c) +\frac{1-\sqrt{-1}}{4}\overline{\zeta}(w)(c-b)  \\
&=& \overline{\zeta}(w)\cdot\frac{1}{2}(1+\sqrt{-1}b-\sqrt{-1}c+bc)  \\
&=& \frac{1}{4}(1-\sqrt{-1}b+\sqrt{-1}c+bc)(1+\sqrt{-1}b-\sqrt{-1}c+bc)  \\
&=& 1
~=~ \e(w)1,
\end{eqnarray*}
and $(\zeta\ast\overline{\zeta})(w)=1$ holds similarly. Therefore, for any $a\in B$,
we have
\begin{eqnarray*}
&& (\overline{\zeta}\ast\zeta)(aw)
=\sum \overline{\zeta}(a_{(1)}w_{(1)})\zeta(a_{(2)}w_{(2)})
=\sum \overline{\zeta}(w_{(1)})S(a_{(1)})a_{(2)}\zeta(w_{(2)})=\e(a)1,  \\
&& (\zeta\ast\overline{\zeta})(aw)
=\sum \zeta(a_{(1)}w_{(1)})\overline{\zeta}(a_{(2)}w_{(2)})
=\sum a_{(1)}\zeta(w_{(1)})\overline{\zeta}(w_{(2)})S(a_{(2)})=\e(a)1.
\end{eqnarray*}

On the other hand, the linear map $\gamma$ is well-defined, as $\{\overline{1},\overline{w}\}$ is a linear basis of $C$.
In addition,
it preserves right $C$-coactions, because
$$(\id\otimes\pi)\circ\Delta\circ\gamma(\overline{1})
=1=(\gamma\otimes\id)\circ\Delta(\overline{1})$$
and
\begin{eqnarray*}
&& (\id\otimes\pi)\circ\Delta\circ\gamma(\overline{w})
~=~ \sum w_{(1)}\otimes\pi(w_{(2)})  \\
&=& \frac{1}{2}(1+bc)w\otimes\pi(w)+\frac{1+\sqrt{-1}}{4}(1-bc)w\otimes\pi(bw) +\frac{1-\sqrt{-1}}{4}(1-bc)w\otimes\pi(cw)  \\
&=& \left(\frac{1}{2}(1+bc)+\frac{1+\sqrt{-1}}{4}(1-bc)
+\frac{1-\sqrt{-1}}{4}(1-bc)\right)w\otimes\pi(w)   \\
&=& w\otimes\overline{w}
~=~ (\gamma\otimes\id)\circ\Delta(\overline{w}).
\end{eqnarray*}
According to the proof of Lemma \ref{lem:cleftcocleft}, one knows that $\gamma$ should have convolution inverse, which is indeed $\gamma$ itself.

\textbf{Axiom (5):}
This is straightforward to verify.

\textbf{Axiom (6):}
Firstly, note by direct calculations that
\begin{eqnarray*}
&& [(\iota\circ\zeta)\ast(\gamma\circ\pi)](w)  \\
&=& \sum \iota[\zeta(w_{(1)})]\gamma[\pi(w_{(2)})]  \\
&=&
\frac{1}{2}\zeta[(1+bc)w]\gamma[\pi(w)]
+\frac{1+\sqrt{-1}}{4}\zeta[(1-bc)w]\gamma[\pi(bw)]  \\ && +\frac{1-\sqrt{-1}}{4}\zeta[(1-bc)w]\gamma[\pi(cw)]  \\
&=&
\frac{1}{2}(1+bc)\zeta(w)\gamma(\overline{w})
+\frac{1+\sqrt{-1}}{4}(1-bc)\zeta(w)\gamma(\overline{w}) +\frac{1-\sqrt{-1}}{4}(1-bc)\zeta(w)\gamma(\overline{w})   \\
&=&
\left(\frac{1}{2}(1+bc)+\frac{1+\sqrt{-1}}{4}(1-bc)
+\frac{1-\sqrt{-1}}{4}(1-bc)\right)w  \\
&=& w.
\end{eqnarray*}
Therefore, for any $a\in B$, we also have
\begin{eqnarray*}
[(\iota\circ\zeta)\ast(\gamma\circ\pi)](aw)
&=& \sum \iota[\zeta(a_{(1)}w_{(1)})]\gamma[\pi(a_{(2)}w_{(2)})]  \\
&=& \sum \iota[a_{(1)}\zeta(w_{(1)})]\gamma[\e(a_{(2)})\pi(w_{(2)})]  \\
&=& \sum a\iota[\zeta(w_{(1)})]\gamma[\pi(w_{(2)})]
~=~
aw.
\end{eqnarray*}
\end{proof}

Consequently, it follows from Corollary \ref{cor:leftpdsemisimple(MAMS)} and Lemma \ref{lem:abelextleftPD}(1) that the left partially dualized quasi-Hopf algebra $C^\ast\#B$ determined by $(\zeta,\gamma^\ast)$ in Lemma \ref{lem:K8pams} is semisimple and commutative. Thus $C^\ast\#B\cong(\mathbb{C}G)^\ast$ as algebras and coalgebras for some group $G$ of order 8.

However, in order to decide which group $G$ is, we should provide additional formulas in the right partial dual $C\btd B^\ast\cong \mathbb{C}G$. According to the axioms in Definition \ref{def:PAMS}, it is not hard to prove the following lemma:

\begin{lemma}\label{lem:rightPDformulas}
Let $B\xrightarrow{\iota}H\xrightarrow{\pi}C$ be an extension of finite-dimensional Hopf algebras.
Suppose $C\btd B^\ast$ is a right partially dualized coquasi-Hopf algebra of $H$. Then
for all $x,y\in C$ and $b^\ast,c^\ast\in B^\ast$, we have
\begin{equation}\label{eqn:rightPDmultiplication1}
(x\btd \e)(y\btd c^\ast)=xy\btd c^\ast,
\;\;\;\;\;\;\;\;
(x\btd b^\ast)(1\btd c^\ast)=x\btd b^\ast c^\ast
\end{equation}
and
\begin{equation}\label{eqn:rightPDmultiplication2}
(1\btd b^\ast)(y\btd \e)
=\sum\pi\left[\zeta^\ast(b^\ast_{(1)})\rightharpoonup\gamma(y_{(1)})\right]
  \btd\iota^\ast\left[\zeta^\ast(b^\ast_{(2)})\leftharpoonup\gamma(y_{(2)})\right].
\end{equation}
\end{lemma}

\begin{proof}
Note in Proposition \ref{prop:PAMS-comptrival}(1) that $\zeta\circ\iota=\id_B$ and $\pi\circ\gamma=\id_C$ hold.
Therefore, since $\iota$ and $\pi$ are Hopf algebra maps, one could find by the definitions of $\btl$ (\ref{eqn:btl}) and $\btr$ (\ref{eqn:btr}) that
\begin{equation}\label{eqn:btlHopfext}
x\btl h=\pi[\gamma(x)h]
=\pi[\gamma(x)]\pi(h)=x\pi(h)
\;\;\;\;\;\;\;\;(\forall h\in H,\;\forall x\in C\cong H/B^+H)
\end{equation}
and
\begin{equation}\label{eqn:btrHopfext}
h^\ast\btr b^\ast=\iota^\ast[h^\ast\zeta^\ast(b^\ast)]
=\iota^\ast(h^\ast)\iota^\ast[\zeta^\ast(b^\ast)]
=\iota^\ast(h^\ast)b^\ast
\;\;\;\;\;\;\;\;(\forall h^\ast\in H^\ast,\;\forall b\in B^\ast).
\end{equation}

Suppose $x,y\in C$ and $b^\ast,c^\ast\in B^\ast$.
Recall in Proposition \ref{prop:rightpdmultiplication} that
\begin{equation}\label{eqn:rightpdmultiplication(MAMS)}
(x\btd b^\ast)(y\btd c^\ast)
=\sum\left(x\btl[\zeta^\ast(b^\ast_{(1)})\rightharpoonup\gamma(y_{(1)})]\right)
  \btd\left([\zeta^\ast(b^\ast_{(2)})\leftharpoonup\gamma(y_{(2)})]\btr c^\ast\right)
\end{equation}
holds in the right partial dual $C\btd B^\ast$, and
note in Definition \ref{def:PAMS}(5) that $\zeta$ and $\gamma$ preserve both units and counits. Therefore, we could calculate that:
\begin{eqnarray*}
(x\btd \e)(y\btd c^\ast)
&\overset{(\ref{eqn:rightPDmultiplicationORIGINAL(MAMS)})}{=}&
(x\btl\gamma(y))\btd c^\ast
\overset{(\ref{eqn:btlHopfext})}{=}
x\pi[\gamma(y)]\btd c^\ast  \\
&=&
xy\btd c^\ast.
\end{eqnarray*}
The other equation in (\ref{eqn:rightPDmultiplication1}) holds due to similar calculations.

Moreover, we could also have
\begin{eqnarray*}
(1\btd b^\ast)(y\btd \e)
&\overset{(\ref{eqn:rightpdmultiplication(MAMS)})}{=}&
\sum\left(1\btl[\zeta^\ast(b^\ast_{(1)})\rightharpoonup\gamma(y_{(1)})]\right)
  \btd\left([\zeta^\ast(b^\ast_{(2)})\leftharpoonup\gamma(y_{(2)})]\btr \e\right)  \\
&\overset{(\ref{eqn:btlHopfext}),\;(\ref{eqn:btrHopfext})}{=}&
\sum\pi\left[\zeta^\ast(b^\ast_{(1)})\rightharpoonup\gamma(y_{(1)})\right]
  \btd\iota^\ast\left[\zeta^\ast(b^\ast_{(2)})\leftharpoonup\gamma(y_{(2)})\right].
\end{eqnarray*}
\end{proof}

Now we could describe the detailed structures of the desired right partial dual $C\btd B^\ast$ of the 8-dimensional Kac algebra $H$, with the language of generators and relations of an algebra and a coalgebra:

\begin{lemma}
With the notations in Lemma \ref{lem:K8pams}, the linear functions
\begin{equation}\label{eqn:b*andc*}
\varphi=p_1+p_b-p_c-p_{bc}\;\;\;\;\;\;\;\;\text{and}\;\;\;\;\;\;\;\;
\psi=p_1-p_b+p_c-p_{bc}
\end{equation}
are group-like elements in $B^\ast$, where $\{p_1,p_b,p_c,p_{bc}\}$ is the basis of $B^\ast$ dual to $\{1,b,c,bc\}$.
Then
the elements
\begin{equation}\label{eqn:3grplikes(MAMS)}
\overline{w}\btd\e,\;\;\;\;\;\;\;\;
\overline{1}\btd \varphi\;\;\;\;\;\;\;\;
\text{and}\;\;\;\;\;\;\;\;
\overline{1}\btd \psi
\end{equation}
are group-like in $C\btd B^\ast$, and they satisfy
$$\left\{\begin{array}{ll}
(\overline{w}\btd\e)^2
=(\overline{1}\btd{\varphi})^2
=(\overline{1}\btd{\psi})^2=1\btd\e, \;\;\;\; &
(\overline{1}\btd \psi)(\overline{1}\btd \varphi)
=(\overline{1}\btd \varphi)(\overline{1}\btd \psi), \\
(\overline{1}\btd \varphi)(\overline{w}\btd\e)
=(\overline{w}\btd\e)(\overline{1}\btd \psi), &
(\overline{1}\btd \psi)(\overline{w}\btd\e)
=(\overline{w}\btd\e)(\overline{1}\btd \varphi).
\end{array}\right.$$
\end{lemma}

\begin{proof}
At first, it is evident that $\varphi$ and $\psi$ defined in (\ref{eqn:b*andc*}) are both group-like in $B^\ast$, as one could directly check that they are algebra maps.
Therefore, it follows from Lemma \ref{lem:abelextleftPD}(2) that the three elements given in (\ref{eqn:3grplikes(MAMS)}) are all group-like.

Next, we focus on showing the equations
$$(\overline{1}\btd \varphi)(\overline{w}\btd\e)
=(\overline{w}\btd\e)(\overline{1}\btd \psi)
\;\;\;\;\text{and}\;\;\;\;
(\overline{1}\btd \psi)(\overline{w}\btd\e)
=(\overline{w}\btd\e)(\overline{1}\btd \varphi),$$
while others are easy to verify by Lemma \ref{lem:rightPDformulas}. In fact, note that
\begin{eqnarray*}
&& \sum \pi(w_{(1)})\otimes w_{(2)}  \\
&=& \frac{1}{2}\pi[(1+bc)w]\otimes w+\frac{1+\sqrt{-1}}{4}\pi[(1-bc)w]\otimes bw +\frac{1-\sqrt{-1}}{4}\pi[(1-bc)w]\otimes cw  \\
&=& \frac{1}{2}\cdot 2\overline{w}\otimes w
~=~ \overline{w}\otimes w,
\end{eqnarray*}
as well as for any $a\in\{1,b,c,bc\}$ that
\begin{eqnarray*}
&& \sum\langle\zeta^\ast(\varphi)_{(1)},w\rangle
\langle\iota^\ast[\zeta^\ast(\varphi)_{(2)}],a\rangle
~=~ \langle\zeta^\ast(\varphi),wa\rangle
~=~ \langle \varphi,\zeta(wa)\rangle  \\
&=& \left\{\begin{array}{ll}
    \langle \varphi,\zeta(w)\rangle, & a=1, \\
    \langle \varphi,y\zeta(w)\rangle, & a=b,  \\
    \langle \varphi,x\zeta(w)\rangle, & a=c, \\
    \langle \varphi,xy\zeta(w)\rangle, & a=bc,
\end{array}\right.
\overset{(\ref{eqn:b*andc*})}{=} \left\{\begin{array}{ll}
    1, & a=1, \\
    -1, & a=b,  \\
    1, & a=c, \\
    -1, & a=bc,
\end{array}\right.
\overset{(\ref{eqn:b*andc*})}{=} \langle \psi,a\rangle.
\end{eqnarray*}
Then we could find that
\begin{eqnarray*}
(\overline{1}\btd \varphi)(\overline{w}\btd\e)
&\overset{(\ref{eqn:rightPDmultiplication2})}{=}&
\sum\pi\left[\zeta^\ast(\varphi)\rightharpoonup\gamma(\overline{w})\right]
  \btd\iota^\ast\left[\zeta^\ast(\varphi)\leftharpoonup\gamma(\overline{w})\right] \\
&=&
\sum\pi\left[\zeta^\ast(\varphi)\rightharpoonup w\right]
  \btd\iota^\ast\left[\zeta^\ast(\varphi)\leftharpoonup w\right] \\
&=&
\sum\pi(w_{(1)})\langle \varphi,\zeta(w_{(2)})\rangle
  \btd \langle \zeta^\ast(\varphi)_{(1)},w\rangle\iota^\ast[\zeta^\ast(\varphi)_{(2)}] \\
&=&
\overline{w}\langle \varphi,\zeta(w)\rangle\btd \psi
~=~
\overline{w}\langle \varphi,1\rangle\btd \psi  \\
&\overset{(\ref{eqn:b*andc*})}{=}&
\overline{w}\btd \psi
\overset{(\ref{eqn:rightPDmultiplication1})}{=}
(\overline{w}\btd \e)(\overline{1}\btd \psi),
\end{eqnarray*}
and the other equation
$(\overline{1}\btd \psi)(\overline{w}\btd\e)
=(\overline{w}\btd \e)(\overline{1}\btd \varphi)$
holds due to similar calculations.
\end{proof}

For convenience, let us denote by
\begin{equation}\label{eqn:D8(MAMS)}
D_8=\langle\, r,s,t\mid r^2=s^2=t^2=1,\;sr=rt,\;tr=rs,\;ts=st\,\rangle
\end{equation}
the dihedral group of order 8.
Then
\begin{equation}\label{eqn:D8iso}
r\mapsto \overline{w}\btd\e,\;\;\;\;s\mapsto \overline{1}\btd \varphi,\;\;\;\;
t\mapsto \overline{1}\btd \psi
\end{equation}
provide an isomorphism
$\mathbb{C}D_8\cong C\btd B^\ast$ as algebras and coalgebras.

\begin{proposition}\label{prop:PDofK8genuine(MAMS)}
With notations in Lemma \ref{lem:K8pams}, the left partial dual $C^\ast\#B$ determined by the {\pams} $(\zeta,\gamma^\ast)$ is
a commutative semisimple quasi-Hopf algebra which is genuine.
\end{proposition}

\begin{proof}
As mentioned before Lemma \ref{lem:rightPDformulas}, one concludes by Lemma \ref{lem:abelextleftPD}(1) and Corollary \ref{cor:leftpdsemisimple(MAMS)} that
each left partially dualized quasi-Hopf algebra $C^\ast\#B$ of $H$ would be commutative and semisimple with
$$\exp(C^\ast\#B)=\exp(H)=8.$$
Thus by Remark \ref{rmk:fusioncatexp(MAMS)}(3), any Hopf algebra gauge equivalent to $C^\ast\#B$ would be also commutative and semisimple with exponent 8, but the only Hopf algebra satisfying these properties is $\mathbb{C}C_8$ (where $C_8$ denotes the cyclic group of order 8), because of the classification result \cite[Theorem 2.13]{Mas95}.

On the other hand, we know according to the paragraph before this proposition that $C^\ast\#B\cong(\mathbb{C}D_8)^\ast$ as algebras and coalgebras. As a consequence, it follows from Lemma \ref{lem:Groring(MAMS)}(2) the Grothendieck ring of the fusion category $\Rep(C^\ast\#B)$ should be $\mathbb{Z}D_8$. However by Remark \ref{rmk:Z+ringofgrps(MAMS)}, as a $\mathbb{Z}_+$-ring, $\mathbb{Z}D_8$ is not isomorphic to $\mathbb{Z}C_8$, which is exactly the Grothendieck ring of the fusion category $\Rep((\mathbb{C}C_8)^\ast)\cong\Rep(\mathbb{C}C_8)$. Thus
$$\Rep(C^\ast\#B)\;\;\;\;\;\;\;\;\text{and}\;\;\;\;\;\;\;\;\Rep(\mathbb{C}C_8)$$
are not tensor equivalent, and hence $C^\ast\#B$ is not gauge equivalent to the Hopf algebra $\mathbb{C}C_8$.
\end{proof}

\begin{remark}
Moreover, according to Proposition \ref{prop:partialdualsequiv},
every left partial dual $C^\ast\#B$ determined by an arbitrary {\pams} for $\iota:B\hookrightarrow H$ would be
a commutative semisimple quasi-Hopf algebra which is genuine.
\end{remark}

Therefore, the left partially dualized quasi-Hopf algebra $C^\ast\#B$ is isomorphic to $(\mathbb{C}D_8)^\ast$ with some non-trivial associator, which should be identified with a normalized 3-cocycle on $D_8$. Also, its antipode could be also described.

\begin{example}\label{ex:D8(MAMS)}
There is a normalized 3-cocycle $\pd{\phi}$ on $D_8$ such that
\begin{eqnarray}\label{eqn:3cocycle}
&&\pd{\phi}\Big(r^{i}s^{j}t^{k},r^{i'}s^{j'}t^{k'},r^{i''}s^{j''}t^{k''}\Big)^{-1}
\nonumber \\
&=&\delta_{i'',0}+\delta_{i'',1}\Big(\delta_{j,k}+\delta_{j+k,1}
\Big[\frac{1-\sqrt{-1}}{2}(-1)^{j'}+\frac{1+\sqrt{-1}}{2}(-1)^{k'}\Big]\Big)
\end{eqnarray}
hold for all $i,i',i'',j,j',j'',k,k',k''\in\{0,1\}$. Furthermore,
$(\mathbb{C}D_8)^\ast$ is a genuine quasi-Hopf algebra with associator $\pd{\phi}$ and antipode $(\pd{S},\pd{\alpha},\pd{\beta})$ given as follows:
\begin{eqnarray}\label{eqn:D8antipode(MAMS)}
& \pd{S}(r^{i}s^{j}t^{k})=\delta_{i,0}rs^{1-k}t^{1-j}+\delta_{i,0}s^{1-j}t^{1-k}
,\;\;\;\;\;\;\;\;
\pd{\alpha}=\pd{\e},  \nonumber  \\
& \pd{\beta}(r^{i}s^{j}t^{k})
=\delta_{i,0}+\delta_{i,1}
 \big(\delta_{j,k}+\sqrt{-1}\delta_{j+k,1}(-1)^{j}\big)
\end{eqnarray}
for all $0\leq i,j,k\leq 1$.
\end{example}

\begin{proof}
Our proof is based on the coquasi-Hopf algebra isomorphism (\ref{eqn:D8iso}) $\mathbb{C}D_8\cong C\btd B^\ast$, where $C\btd B^\ast$ is the right partial dual of $H$ determined by the {\pams}
\begin{equation*}
\begin{array}{ccc}
\xymatrix{
B \ar@<.5ex>[r]^{\iota} & H \ar@<.5ex>@{-->}[l]^{\zeta} \ar@<.5ex>[r]^{\pi}
& C \ar@<.5ex>@{-->}[l]^{\gamma}  }
&\;\;\text{and}\;\;&
\xymatrix{
C^\ast \ar@<.5ex>[r]^{\pi^\ast}
& H^\ast \ar@<.5ex>@{-->}[l]^{\gamma^\ast} \ar@<.5ex>[r]^{\iota^\ast}
& B^\ast \ar@<.5ex>@{-->}[l]^{\zeta^\ast}  },
\end{array}
\end{equation*}
given in Lemma \ref{lem:K8pams}.

Recall also in Lemma \ref{lem:K8pams} that
$\zeta$ is the left $B$-module map satisfying
$\zeta(w)=1$, $\zeta\mid_B=\id_B$,
$\gamma(\overline{1})=1$ and
$\gamma(\overline{w})=w$. Thus we could make calculations
\begin{eqnarray*}
&& \sum \zeta(ww_{(1)})\otimes \zeta(w_{(2)})   \\
&=&
\zeta\Big( w\cdot\frac{1}{2}(1+bc)w\Big)\otimes \zeta(w)
+\zeta\Big(w\cdot\frac{1+\sqrt{-1}}{4}(1-bc)w\Big)\otimes \zeta(bw)    \\
&& +\zeta\Big(w\cdot\frac{1-\sqrt{-1}}{4}(1-bc)w\Big)\otimes \zeta(cw)   \\
&=&
\zeta\Big( \frac{1}{2}(1+bc)\Big)\otimes \zeta(w)
+\zeta\Big(\frac{1+\sqrt{-1}}{4}(1-bc)\Big)\otimes b\zeta(w)    \\
&& +\zeta\Big(\frac{1-\sqrt{-1}}{4}(1-bc)\Big)\otimes c\zeta(w)    \\
&=&
\frac{1}{2}(1+bc)\otimes 1
+\frac{1+\sqrt{-1}}{4}(1-bc)\otimes b
+\frac{1-\sqrt{-1}}{4}(1-bc)\otimes c
\end{eqnarray*}
and
\begin{eqnarray}\label{eqn:zetaw1zetaw2}
&& \sum \zeta(w_{(1)})\otimes \zeta(w_{(2)})   \nonumber  \\
&=&
\zeta\Big(\frac{1}{2}(1+bc)w\Big)\otimes \zeta(w)
+\zeta\Big(\frac{1+\sqrt{-1}}{4}(1-bc)w\Big)\otimes \zeta(bw)  \nonumber   \\
&&  +\zeta\Big(\frac{1-\sqrt{-1}}{4}(1-bc)w\Big)\otimes \zeta(cw)  \nonumber   \\
&=&
\frac{1}{2}(1+bc)\otimes 1
+\frac{1+\sqrt{-1}}{4}(1-bc)\otimes b
+\frac{1-\sqrt{-1}}{4}(1-bc)\otimes c,
\end{eqnarray}
and find that $\sum \zeta(ww_{(1)})\otimes \zeta(w_{(2)})=\sum \zeta(w_{(1)})\otimes \zeta(w_{(2)})$ holds.

Now let us consider the linear basis $\{\overline{1},\overline{w}\}$ of $C$ with dual basis $\{p_{\overline{1}},p_{\overline{w}}\}$ of $C^\ast$, and note that $p_{\overline{1}}+p_{\overline{w}}=\e$.
Then
we could apply the formula (\ref{eqn:phi^-1(MAMS)})
know that the associator $\pd{\phi}$ of the left partially dualized quasi-Hopf algebra $C^\ast\#B$ has inverse
\begin{eqnarray*}
\pd{\phi}^{-1}
&=&
\sum \big(\e\#\zeta[\gamma(\overline{1})\gamma(\overline{1})_{(1)}]\big)
\otimes \big(p_{\overline{1}}\#\zeta[\gamma(\overline{1})_{(2)}]\big)
\otimes \big(p_{\overline{1}}\#1\big)  \\
&&
+ \sum \big(\e\#\zeta[\gamma(\overline{w})\gamma(\overline{1})_{(1)}]\big)
\otimes \big(p_{\overline{w}}\#\zeta[\gamma(\overline{1})_{(2)}]\big)
\otimes \big(p_{\overline{1}}\#1\big)  \\
&&
+ \sum \big(\e\#\zeta[\gamma(\overline{1})\gamma(\overline{w})_{(1)}]\big)
\otimes \big(p_{\overline{1}}\#\zeta[\gamma(\overline{w})_{(2)}]\big)
\otimes \big(p_{\overline{w}}\#1\big)  \\
&&
+\sum \big(\e\#\zeta[\gamma(\overline{w})\gamma(\overline{w})_{(1)}]\big)
\otimes \big(p_{\overline{w}}\#\zeta[\gamma(\overline{w})_{(2)}]\big)
\otimes \big(p_{\overline{w}}\#1\big)  \\
&=&
 (\e\#\zeta(1))
\otimes (p_{\overline{1}}\#\zeta(1))
\otimes (p_{\overline{1}}\#1)
+  (\e\#\zeta(w))
\otimes (p_{\overline{w}}\#\zeta(1))
\otimes (p_{\overline{1}}\#1)  \\
&&
+ \sum (\e\#\zeta(w_{(1)}))
\otimes (p_{\overline{1}}\#\zeta(w_{(2)}))
\otimes (p_{\overline{w}}\#1)  \\
&&
+\sum (\e\#\zeta(ww_{(1)}))
\otimes (p_{\overline{w}}\#\zeta(w_{(2)}))
\otimes (p_{\overline{w}}\#1)  \\
&=&
(\e\#1)\otimes (p_{\overline{1}}\#1)\otimes (p_{\overline{1}}\#1)
+  (\e\#1)\otimes (p_{\overline{w}}\#1)\otimes (p_{\overline{1}}\#1)  \\
&&
+ \sum (\e\#\zeta(w_{(1)}))
\otimes (p_{\overline{1}}\#\zeta(w_{(2)}))
\otimes (p_{\overline{w}}\#1)  \\
&&
+\sum (\e\#\zeta(w_{(1)}))
\otimes (p_{\overline{w}}\#\zeta(w_{(2)}))
\otimes (p_{\overline{w}}\#1)  \\
&=&
(\e\#1)\otimes (\e\#1)\otimes (p_{\overline{1}}\#1)
+ \sum (\e\#\zeta(w_{(1)}))\otimes (\e\#\zeta(w_{(2)}))
\otimes (p_{\overline{w}}\#1)  \\
&\overset{(\ref{eqn:zetaw1zetaw2})}{=}&
(\e\#1)\otimes (\e\#1)\otimes (p_{\overline{1}}\#1)
+ \Big(\e\#\frac{1}{2}(1+bc)\Big)\otimes (\e\#1)\otimes (p_{\overline{w}}\#1)   \\
&&
+ \Big(\e\#\frac{1+\sqrt{-1}}{4}(1-bc)\Big)\otimes (\e\#b)\otimes (p_{\overline{w}}\#1)  \\
&&
+ \Big(\e\#\frac{1-\sqrt{-1}}{4}(1-bc)\Big)\otimes (\e\#c)\otimes (p_{\overline{w}}\#1).
\end{eqnarray*}

Moreover, it follows from the notations (\ref{eqn:b*andc*}) that
\begin{equation}\label{eqn:varphi^j psi^k(MAMS)}
\varphi^{j} \psi^{k}
=p_1+(-1)^kp_b+(-1)^jp_c+(-1)^{j+k}p_{bc}\;\;\;\;\;\;\;\;
(\forall 0\leq j,k\leq 1)
\end{equation}
holds, since $p_1,p_b,p_c,p_{bc}$ are orthonormal idempotents in $B^\ast$.
Then we regard $\pd{\phi}^{-1}$ as a linear function on $(C\btd B^\ast)^{\otimes\,3}$, and compute for any $i,i',i'',j,j',j'',k,k',k''\in\{0,1\}$ that
\begin{eqnarray*}
&& \big\langle \pd{\phi}^{-1},
  (\overline{w}{}^{i}\btd \varphi^{j} \psi^{k})
  \otimes (\overline{w}{}^{i'}\btd \varphi^{j'} \psi^{k'})
  \otimes (\overline{w}{}^{i''}\btd \varphi^{j''} \psi^{k''})
\big\rangle  \\
&=&
\big\langle p_{\overline{1}},\overline{w}{}^{i''}\big\rangle
+\big\langle \varphi^{j} \psi^{k},\frac{1}{2}(1+bc)\big\rangle
 \big\langle p_{\overline{w}},\overline{w}{}^{i''}\big\rangle   \\
&&
+ \big\langle \varphi^{j} \psi^{k},\frac{1+\sqrt{-1}}{4}(1-bc)\big\rangle
  \big\langle \varphi^{j'} \psi^{k'},b\big\rangle
  \big\langle p_{\overline{w}},\overline{w}{}^{i''}\big\rangle  \\
&&
+ \big\langle \varphi^{j} \psi^{k},\frac{1-\sqrt{-1}}{4}(1-bc)\big\rangle
  \big\langle \varphi^{j'} \psi^{k'},c\big\rangle
  \big\langle p_{\overline{w}},\overline{w}{}^{i''}\big\rangle  \\
&=&
\delta_{i'',0}
+ \frac{1}{2}[1+(-1)^{j+k}]\delta_{i'',1}  \\
&&+ \frac{1+\sqrt{-1}}{4}[1-(-1)^{j+k}](-1)^{k'}\delta_{i'',1}
+ \frac{1-\sqrt{-1}}{4}[1-(-1)^{j+k}](-1)^{j'}\delta_{i'',1} \\
&=&
\delta_{i'',0}+\delta_{i'',1}\Big(\delta_{j,k}+\delta_{j+k,1}
\Big[\frac{1+\sqrt{-1}}{2}(-1)^{k'}+\frac{1-\sqrt{-1}}{2}(-1)^{j'}\Big]\Big)  \\
&=&
\delta_{i'',0}+\delta_{i'',1}\Big(\delta_{j,k}+\delta_{j+k,1}
\Big[\frac{1-\sqrt{-1}}{2}(-1)^{j'}+\frac{1+\sqrt{-1}}{2}(-1)^{k'}\Big]\Big).
\end{eqnarray*}
As a conclusion, we could apply the isomorphism (\ref{eqn:D8iso}) $\mathbb{C}D_8\cong C\btd B^\ast$ to know that $\pd{\phi}$ given by (\ref{eqn:3cocycle}) is a normalized 3-cocycle on the group $D_8$,
and it is non-trivial due to the argument in Proposition \ref{prop:PDofK8genuine(MAMS)}.

Furthermore,
according to the constructions in \cite[Example 3.50]{BCPV19} for instance,
the quasi-bialgebra $(\mathbb{C}D_8)^\ast$ has an antipode $(\pd{S},\pd{\alpha},\pd{\beta})$ given by
$$\pd{S}(r^{i}s^{j}t^{k})=(r^{i}s^{j}t^{k})^{-1},\;\;\;\;
\pd{\alpha}=\pd{\e}\;\;\;\;\text{and}\;\;\;\;
\pd{\beta}(r^{i}s^{j}t^{k})=
\pd{\phi}\left(r^{i}s^{j}t^{k},(r^{i}s^{j}t^{k})^{-1},r^{i}s^{j}t^{k}\right){}^{-1}$$
for all  $0\leq i,j,k\leq 1$.
In specific, one could calculate that
\begin{eqnarray*}
&& \pd{S}(r^{i}s^{j}t^{k})
= (r^{i}s^{j}t^{k})^{-1}=t^{1-k}s^{1-j}r^{1-i}  \\
&=& \left\{\begin{array}{ll}
t^{1-k}s^{1-j}r & (\text{when}\;\;i=0) \\
t^{1-k}s^{1-j} & (\text{when}\;\;i=1)
\end{array}\right.
\overset{(\ref{eqn:D8(MAMS)})}=
\left\{\begin{array}{ll}
r^1s^{1-k}t^{1-j} & (\text{when}\;\;i=0) \\
r^0s^{1-j}t^{1-k} & (\text{when}\;\;i=1)
\end{array}\right.  \\
&=& \delta_{i,0}r^1s^{1-k}t^{1-j}+\delta_{i,1}r^0s^{1-j}t^{1-k},
\end{eqnarray*}
which coincides with the first formula listed in (\ref{eqn:D8antipode(MAMS)}).

As for the distinguished element $\pd{\beta}$, it follows that
\begin{eqnarray*}
&& \pd{\beta}(r^{i}s^{j}t^{k})
=
\pd{\phi}\left(r^{i}s^{j}t^{k},(r^{i}s^{j}t^{k})^{-1},r^{i}s^{j}t^{k}\right){}^{-1} \\
&=&
\delta_{i,0}\pd{\phi}
  \left(r^0s^{j}t^{k},r^1s^{1-k}t^{1-j},r^0s^{j}t^{k}\right){}^{-1}
+\delta_{i,1}\pd{\phi}
  \left(r^1s^{j}t^{k},r^0s^{1-j}t^{1-k},r^1s^{j}t^{k}\right){}^{-1}  \\
&\overset{(\ref{eqn:3cocycle})}=&
\delta_{i,0}+\delta_{i,1}
 \Big(\delta_{j,k}+\delta_{j+k,1}
\Big[\frac{1-\sqrt{-1}}{2}(-1)^{1-j}+\frac{1+\sqrt{-1}}{2}(-1)^{1-k}\Big]\Big)  \\
&=&
\delta_{i,0}+\delta_{i,1}
 \Big(\delta_{j,k}+\delta_{j+k,1}
\Big[-\frac{1-\sqrt{-1}}{2}(-1)^{j}+\frac{1+\sqrt{-1}}{2}(-1)^{j}\Big]\Big)  \\
&=&
\delta_{i,0}+\delta_{i,1}
 \big(\delta_{j,k}+\sqrt{-1}\delta_{j+k,1}(-1)^{j}\big),
\end{eqnarray*}
which coincides with the last formula listed in (\ref{eqn:D8antipode(MAMS)}).
\end{proof}

\begin{remark}
In fact, the form of distinguished elements $\pd{\alpha}$ and $\pd{\beta}$ could be also obtained by the constructions in Theorem \ref{thm:partialdual}(4) (or equivalently, the formulas in Remark \ref{rmk:equivDelta}(4)). Specifically, since the left partially dualized quasi-Hopf algebra $C^\ast\#B$ is commutative in this example,
it follows from the axiom
$$\sum\pd{\phi}^1\pd{\beta}\pd{S}(\pd{\phi}^2)\pd{\alpha}\pd{\phi}^3=\pd{e}$$
of quasi Hopf algebras that the distinguished elements $\pd{\alpha}$ and $\pd{\beta}$ are both invertible.
Thus, we could choose $\pd{\alpha}=\pd{\e}$ and
apply the formula in Remark \ref{rmk:equivDelta}(4) to calculate
$$\pd{\beta}=\pd{v}
=
\left(\e\#\overline{\zeta}[\overline{\gamma}(\overline{1})]\right)
(p_{\overline{1}}\#1)
+ \left(\e\#\overline{\zeta}[\overline{\gamma}(\overline{w})]\right)
(p_{\overline{w}}\#1) \\
=
p_{\overline{1}}\#\overline{\zeta}[\overline{\gamma}(\overline{1})]
+ p_{\overline{w}}\#\overline{\zeta}[\overline{\gamma}(\overline{w})],$$
since $\{\overline{1},\overline{w}\}$ is a linear basis of $C$ with dual basis $\{p_{\overline{1}},p_{\overline{w}}\}$ of $C^\ast$.

However, recall in Lemma \ref{lem:K8pams} and its proof that $\overline{1}$ and $\overline{w}$ are group-like in $C$, and hence
$$\overline{\gamma}(\overline{1})=\gamma(\overline{1})^{-1}=1
\;\;\;\;\text{and}\;\;\;\;
\overline{\gamma}(\overline{w})=\gamma(\overline{w})^{-1}=w^{-1}=w.$$
Therefore, we have
\begin{eqnarray*}
\pd{\beta}
&=&
p_{\overline{1}}\#\overline{\zeta}[\overline{\gamma}(\overline{1})]
+ p_{\overline{w}}\#\overline{\zeta}[\overline{\gamma}(\overline{w})]
~=~
p_{\overline{1}}\#\overline{\zeta}(1)
+ p_{\overline{w}}\#\overline{\zeta}(w)  \\
&\overset{(\ref{eqn:K8zetabar(MAMS)})}=&
p_{\overline{1}}\#1
+\frac{1}{2}\left[p_{\overline{w}}\#(1-\sqrt{-1}b+\sqrt{-1}c+bc)\right].
\end{eqnarray*}
Consequently, for all $0\leq i,j,k\leq 1$, note from (\ref{eqn:varphi^j psi^k(MAMS)}) that
$$\varphi^{j} \psi^{k}=p_1+(-1)^kp_b+(-1)^jp_c+(-1)^{j+k}p_{bc},$$
and thus
\begin{eqnarray*}
\big\langle\pd{\beta},\overline{w}{}^{i}\btd \varphi^{j} \psi^{k}\big\rangle
&=&
\big\langle p_{\overline{1}},\overline{w}{}^{i}\big\rangle
\big\langle \varphi^{j} \psi^{k},1\big\rangle
+\frac{1}{2}\big\langle p_{\overline{w}},\overline{w}{}^{i}\big\rangle
\big\langle \varphi^{j} \psi^{k},1-\sqrt{-1}b+\sqrt{-1}c+bc\big\rangle  \\
&=&
\delta_{i,0}
+\frac{1}{2}\delta_{i,1}\Big(1-\sqrt{-1}(-1)^k+\sqrt{-1}(-1)^j+(-1)^{j+k}\Big)  \\
&=&
\delta_{i,0}
+\frac{1}{2}\delta_{i,1}(\delta_{j,k}+\delta_{j+k,1})
\Big(1-\sqrt{-1}(-1)^k+\sqrt{-1}(-1)^j+(-1)^{j+k}\Big)  \\
&=&
\delta_{i,0}
+\frac{1}{2}\delta_{i,1}
\Big[\delta_{j,k}\Big(1-\sqrt{-1}(-1)^j+\sqrt{-1}(-1)^j+(-1)^{2k}\Big)  \\
&&  \;\;\;\;\;\;\;\;\;\;\;\;\;\;\;\;\;\;\;
+\delta_{j+k,1}\Big(1-\sqrt{-1}(-1)^{1-j}+\sqrt{-1}(-1)^j+(-1)^{1}\Big)\Big]  \\
&=&
\delta_{i,0}
+\frac{1}{2}\delta_{i,1}
\big[2\delta_{j,k}+\delta_{j+k,1}\big(2\sqrt{-1}(-1)^j\big)\big]  \\
&=&
\delta_{i,0}+\delta_{i,1}
\big(\delta_{j,k}+\sqrt{-1}\delta_{j+k,1}(-1)^j\big),
\end{eqnarray*}
which coincides with the last formula listed in (\ref{eqn:D8antipode(MAMS)}) as well.
\end{remark}

More generally, it might be possible to use the same method in this subsection, in order to construct genuine quasi-Hopf algebras which are not commutative. For this purpose, one could apply the (left) partial dualization to the following kinds of examples:
\begin{itemize}
\item
Non-group-theoretical semisimple Hopf algebras (\cite{Nik08,GNN09} etc.) which fit into certain extensions of Hopf algebras;

\item
Hopf algebras which fit into extensions of two non-semisimple Hopf algebras.
\end{itemize}
On the other hand, their left partially dualized quasi-Hopf algebras might be shown to be genuine with the help of some specific invariants (under categorical Morita equivalences).

\section*{Acknowledgements}

The author would like to thank Professors Kenichi Shimizu, Akira Masuoka, Gongxiang Liu, Shenglin Zhu and Zhimin Liu for motivations, suggestions as well as valuable discussions on the manuscript.

\end{document}